\documentclass[letterpaper,10pt]{amsart}

\usepackage{amsmath, amssymb, amsthm, xcolor, mathtools}
\usepackage[shortlabels]{enumitem}
\usepackage{verbatim}
\usepackage[hidelinks]{hyperref}
\usepackage{graphicx}
\usepackage[normalem]{ulem}
\usepackage{tikz-cd}

\usepackage{cancel}
\usepackage{enumitem}

\usepackage[american]{babel}

\let\oldtocsubsection=\tocsubsection
\renewcommand{\tocsubsection}[2]{\hspace{.75cm}\oldtocsubsection{#1}{#2}}

\DeclareRobustCommand{\SkipTocEntry}[5]{}

\newtheorem{theorem}{Theorem}[section]
\newtheorem{lemma}[theorem]{Lemma}
\newtheorem{proposition}[theorem]{Proposition}
\newtheorem{corollary}[theorem]{Corollary}
\newtheorem{properties}[theorem]{Properties}

\theoremstyle{definition}
\newtheorem{definition}[theorem]{Definition}
\newtheorem{notation}[theorem]{Notation}

\newtheorem{remark}[theorem]{Remark}
\newtheorem{question}[theorem]{Question}
\newtheorem*{remark*}{Remark}

\newtheorem{example}[theorem]{Example}

\usepackage{perpage}
\newcounter{mparcnt}
\MakePerPage{mparcnt}

\numberwithin{equation}{section}

\newcommand{\Q}{\mathcal Q} 
\newcommand{\Z}{\mathcal Z}

\newcommand{\bN}{\mathbb{N}}
\newcommand{\bC}{\mathbb{C}}

\newcommand{\id}{\mathrm{id}}
\newcommand{\M}{\mathcal{M}}

\newcommand{\Cu}{\mathrm{Cu}} 
\newcommand{\minusa}{\ensuremath{\mathrlap{\!\not{\phantom{\mathrm{a}}}}\mathrm{a}}}
\newcommand{\Ka}{\ensuremath{\overline{K}^{\mathrm{alg}}_{1}}} 
\newcommand{\ka}[1]{\ensuremath{[#1]_{\mathrm{alg}}}} 
 
\newcommand{\KK}{\ensuremath{KK}} 

\newcommand{\KL}{\ensuremath{KL}} 
\newcommand{\Ext}{\ensuremath{\mathrm{Ext}}} 
 
\newcommand{\totK}{\ensuremath{\underline{K}}}

\newcommand{\bcN}{\ensuremath{\overline{\mathbb{N}}}}

\newcommand{\Zn}[1]{\mathbb Z/#1}
\newcommand{\alg}{\mathrm{alg}}

\newcommand{\Hom}{\ensuremath{\mathrm{Hom}}}

\DeclareMathOperator{\Aff}{Aff}
\DeclareMathOperator{\Ad}{Ad}
\DeclareMathOperator{\im}{im}
\DeclareMathOperator{\ev}{ev}
\DeclareMathOperator{\Th}{Th} 
\DeclareMathOperator{\Ell}{Ell}
\newcommand{\inv}{\underline{K}T_u}

\title[Classifying $^*$-homomorphisms I]%
{Classifying $^*$-homomorphisms I:\\ Unital simple nuclear $C^*$-algebras}

\author[J. Carri\'on]{Jos\'e R. Carri\'on}
\address{\hspace{.5ex}Jos\'e R.\ Carri\'on, Department of Mathematics, Texas Christian
  University, Fort \linebreak Worth, Texas 76109, United States}
\email{j.carrion@tcu.edu}

\author[J. Gabe]{James Gabe}
\address{\hspace{.5ex}James Gabe, Department of Mathematics and Computer Science,
University of Southern Denmark, 5230 Odense, Denmark}
  \email{gabe@imada.sdu.dk}

\author[C. Schafhauser]{Christopher Schafhauser}
\address{\hspace{.5ex}Christopher Schafhauser, Department of  Mathematics, University of Nebraska-\linebreak Lincoln, Lincoln, Nebraska 68588, United States}
\email{cschafhauser2@unl.edu}

\author[A. Tikuisis]{Aaron Tikuisis}
\address{\hspace{.5ex}Aaron Tikuisis, Department of Mathematics and Statistics, University of
  Ottawa, 585 King Edward, Ottawa, ON, K1N 6N5, Canada}
\email{aaron.tikuisis@uottawa.ca}

\author[S. White]{Stuart White}
\address{\hspace{.5ex}Stuart White, Mathematical Institute, University of Oxford,
  Oxford, OX2 6GG, United Kingdom}
\email{stuart.white@maths.ox.ac.uk}

\thanks{Research partially supported by: NSF grant DMS-2451675 (JC); DFF starting grant 1054-00094B, ARC grant DP180100595 (JG); NSF grant DMS-2000129 (CS);  NSERC Discovery Grant (AT); EPSRC grants EP/R025061/1, EP/R025061/2 and EP/X026647/1 (SW).  Part of the research in this paper was undertaken at the American Institute of Mathematics as part of the SQuaRE \emph{von Neumann techniques in the classification of $C^*$-algebras.}  For the purpose of open access the authors have applied a CC-BY license to any author accepted manuscript arising from this submission.}


\begin{document}

\maketitle
\setcounter{tocdepth}{2}

\renewcommand*{\thetheorem}{\Alph{theorem}}

\begin{abstract}
  We classify the unital embeddings of a unital separable nuclear $C^*$-algebra satisfying the universal coefficient theorem into a unital simple separable nuclear $C^*$-algebra that tensorially absorbs the Jiang--Su algebra.
 
  This gives a new and essentially self-contained proof of the stably finite case of the unital classification theorem: unital simple separable nuclear $C^*$-algebras that absorb the Jiang--Su algebra tensorially and satisfy the universal coefficient theorem are classified by Elliott's invariant of $K$-theory and traces.
\end{abstract}

\addtocontents{toc}{\SkipTocEntry}
\section*{Overview of results}

This paper gives an abstract proof of the unital classification theorem (Theorem \ref{Main}) for $C^*$-algebras.  This says that a suitable class of unital simple separable nuclear $C^*$-algebras is classified by operator algebraic $K$-theory (the non-commutative version of Atiyah and Hirzebruch's topological $K$-theory) and its pairing with traces (thought of as non-commutative measures on $C^*$-algebras).
We put the hypotheses of the theorem, especially $\mathcal Z$-stability and the universal coefficient theorem (UCT) of Rosenberg and Schochet, into context in Section~\ref{Sect:Intro.1}.
At this point, we emphasize that these two conditions are both necessary: the classifiable class is maximal for such a result. Moreover, powerful tools exist to verify the hypotheses in a wide range of concrete situations (see Section~\ref{Intro:examples}).
The unital classification theorem may be regarded as the $C^*$-algebraic analog of the Connes--Haagerup classification of injective factors (\cite{Connes76,Haagerup87}), and in this way it provides a capstone for the large scale Elliott classification program.



\begin{theorem}[{The unital classification theorem; see Theorem~\ref{algebra-classification}}]\label{Main}
  Unital simple separable nuclear $\mathcal{Z}$-stable $C^*$-algebras satisfying Rosenberg and Schochet's universal coefficient theorem are classified by Elliott's invariant consisting of $K$-theory and traces.
\end{theorem}

This theorem first emerged (with slightly different, but equivalent hypotheses; see Sections~\ref{Intro:examples} and~\ref{subsect127}) as a combination of results announced in 2015, together with the Kirchberg--Phillips theorem \cite{Kirchberg95b,Phillips00} from the 1990s. In short, the 2015 work had two components: the combination of \cite{EGLN,TWW,Winter16a,Winter14,Winter12} reduced the problem to a strong asymptotic classification of morphisms between certain $C^*$-algebras satisfying a tracial approximation condition. Such a classification was the purpose of the long and intricate papers \cite{GLN-part1,GLN-part2} (originally made available as \cite{GLN-preprint}). All of this built on decades of work by large numbers of researchers, so the unital classification theorem is often attributed to ``many hands''; we discuss this history further in Section~\ref{subsec:brief-history}. However, the vast scale and technical complexity of the overall argument have made it challenging for researchers to comprehend in full. Here we give a comparatively self-contained, abstract proof of the unital classification theorem that directly relates to the von Neumann algebraic framework.


Beginning with Elliott's classification of approximately finite dimensional $C^*$-algebras, classification results for $C^*$-algebras have been built on classification results for morphisms.
Our strategy in this paper fits the version of this framework pioneered by R\o{}rdam in \cite{Rordam95}: the main goal of the paper is the following classification of unital embeddings, from which Theorem~\ref{Main} is quickly deduced. We discuss the need for the total invariant $\inv$ further below (in Sections~\ref{Sect:Intro.1}, \ref{subsec:brief-history}, and~\ref{subsec:methods-behind-thmB}), and describe its components in detail in Sections~\ref{SectAlgK1}, \ref{ss:totalKtheory}, \ref{sec:new-map}, and~\ref{ss:totalinv}.

\begin{theorem}[{Classification of unital embeddings; see
    Theorem~\ref{one-sided-classification}}]\label{Main2} 
  Let $A$ be a unital separable nuclear $C^*$-algebra satisfying the UCT and $B$ be a unital simple separable nuclear $\mathcal{Z}$-stable $C^*$-algebra. Up to approximate unitary equivalence, the unital embeddings $A \hookrightarrow B$ correspond bijectively with those morphisms of the total invariant $\inv$ --- consisting of enriched $K$-theoretic and tracial data --- mapping traces on $B$ to faithful traces on $A$.
\end{theorem}

Both Theorem~\ref{Main} and Theorem~\ref{Main2} admit a dichotomy thanks to a special case of a beautiful theorem of Kirchberg: a $C^*$-algebra $B$ as in Theorem~\ref{Main2} is either \emph{purely infinite} or \emph{stably finite}. Moreover, the tracial component of the invariants used in Theorems~\ref{Main} and~\ref{Main2} detects these conditions: purely infinite simple $C^*$-algebras are traceless, while a deep theorem of Haagerup (\cite{Haagerup14}) combines with work of Blackadar and Handelman (\cite{Haagerup14,Blackadar-Handelman82}) to give traces on unital simple nuclear stably finite $C^*$-algebras.
Accordingly, both theorems and their proofs split into these two cases.
In the purely infinite setting, Theorem~\ref{Main} was established independently by Kirchberg and by Phillips in the 1990s (\cite{Kirchberg95b,Kirchberg-Phillips00, Phillips00}; see also R\o{}rdam's expository account in \cite{Rordam02}).
Subsequent work on Theorem~\ref{Main} was then devoted to the stably finite case. The same is true for our proof. In this paper we only handle the stably finite aspect of Theorem~\ref{Main}, and one obtains the entire unital classification theorem by combining this with the Kirchberg--Phillips theorem.
Likewise, when $B$ is purely infinite, Kirchberg obtained a classification of embeddings; our contribution is the stably finite counterpart, allowing for the unified statement of Theorem~\ref{Main2}.

The generality of the stably finite case of Theorem~\ref{Main2} demonstrates the power of the abstract methods we introduce in this paper. It simultaneously unifies and generalizes earlier classifications of embeddings of nuclear $C^*$-algebras into simple nuclear stably finite $C^*$-algebras by $K$-theoretic data. The key point is the abstract nature of the hypotheses: the theorem allows all separable nuclear $C^*$-algebras in the UCT class as domains of embeddings, whereas prior results such as \cite{Lin07,Lin12a,Matui11} rely on internal structure for the domains (either commutativity, or explicit inductive limit or tracial approximation structure), and often also for the codomains.  While this paper was in preparation, Gong, Lin, and Niu independently obtained Theorem~\ref{Main2} when $A$ is additionally simple (\cite{Gong-Lin-etal23}, building on \cite{Lin-Niu14}). Their work uses the power of the unital classification theorem to obtain very precise internal structure on UHF-stabilizations of $A$.

In fact, Kirchberg's embedding theorem is more general, with some hypotheses transferred from the domain and codomain to the morphisms.
  In addition to covering non-unital classification, the second paper in this series will extend the stably finite aspect of Theorem~\ref{Main2} to Kirchberg's framework.
  Following the passing of Eberhard Kirchberg, his long-term project \cite{Kirchberg} describing these, and many more results, will sadly remain unfinished.
  JG gives a new approach to Kirchberg's classification theorems in \cite{Gabe20,Gabe-Preprint}.

The passage back from Theorem~\ref{Main2} to Theorem~\ref{Main} is via an Elliott intertwining argument, for which it is vital that Theorem~\ref{Main2} has both existence and uniqueness components. That is, Theorem~\ref{Main2} specifies exactly which maps between total invariants are realized by injective $^*$-homomorphisms, and then shows uniqueness of these. Moreover, although the total invariant $\inv$ contains more information than $K$-theory and traces, any map between $K$-theory and traces can be extended to $\inv$. In this way, the existence aspect of Theorem~\ref{Main2} yields the following:

\begin{corollary}\label{MainCor}
Let $A$ be a unital separable nuclear $C^*$-algebra satisfying the UCT and $B$ be a unital simple separable nuclear $\mathcal{Z}$-stable $C^*$-algebra. Any morphism from the Elliott invariant of $A$ to that of $B$ that maps traces on $B$ to faithful traces on $A$ is realized by a unital injective $^*$-homomorphism $A\to B$.
\end{corollary}

We will put the main ingredients in these theorems into context in Section~\ref{Sect:Intro.1} of the introduction, connecting these with von Neumann algebras and describing the range of examples covered by Theorem~\ref{Main}.  We follow this with a brief history in Section~\ref{subsec:brief-history} and then outline our methods in Section~\ref{subsec:methods-behind-thmB}.  

We first announced these results at NCGOA in M\"unster in 2018, and then in Oberwolfach (\cite{TikuisisWhite-OberwolfachReport}).  A summary is also given in the survey \cite{White:ICM}.

\tableofcontents

\numberwithin{theorem}{section}

\section{Introduction}

\subsection{Context and examples}\label{Sect:Intro.1} From the outset, the classification of simple nuclear $C^*$-algebras has been intimately tied to von Neumann algebras. Indeed, the first such result --- Glimm's classification of uniformly hyperfinite $C^*$-algebras (\cite{Glimm60}) --- is a direct analog of Murray and von Neumann's uniqueness of the hyperfinite II$_1$ factor (\cite{MvN.4}) allowing for the more refined invariants needed with $C^*$-algebras.

Later developments in $C^*$-algebra classification were similarly influenced by the spectacular successes on the von Neumann algebra side.  Connes' groundbreaking work on injective von Neumann algebras (\cite{Connes76}) gives a clean, testable, and abstract characterization of hyperfiniteness and is a landmark of 20th century mathematics.  It provides a game-changing enlargement of the scope of Murray and von Neumann's uniqueness theorem and led to a complete classification of injective infinite factors (combining \cite{Connes76} with \cite{Connes73,Connes75a,Connes75,Krieger76} and Haagerup's subsequent uniqueness of the hyperfinite III$_1$ factor in \cite{Haagerup87}). Beyond the impact in von Neumann algebras --- which continues to the present day, for example in subfactor theory (\cite{Jones83,Popa94,Jones-Morrison-etal14}) and Popa's deformation/rigidity theory (\cite{Popa07,Vaes18}) --- Connes' theorem inspired classification results for amenable measurable dynamical systems (\cite{Connes-Feldman-etal81}) and is vital in developing approximation properties for $C^*$-algebras.

The unital classification theorem is the $C^*$-counterpart of the von Neumann algebra classification theorems. The hypotheses of unitality, simplicity, separability, and nuclearity in Theorem~\ref{Main} correspond directly to the von Neumann algebra setting. All von Neumann algebras are unital, and factors are the simple von Neumann algebras. In the von Neumann algebra classification, one demands a separable predual; in both cases separability powers intertwining arguments (and is necessary, as shown by \cite{Farah-Katsura15}). Finally, nuclearity 
 corresponds to injectivity in Connes' theorem: a $C^*$-algebra $A$ is nuclear if and only if its bidual $A^{**}$ is injective (\cite{Choi-Effros76a,Choi-Effros77}). 

The other two hypotheses --- tensorial absorption of the Jiang--Su algebra $\mathcal Z$ and belonging to the UCT class --- are of a more subtle nature. We describe these next.

\subsubsection{$\mathcal Z$-stability}
A key step in Connes' work is that every separably acting injective II$_1$ factor $\mathcal M$ absorbs the hyperfinite II$_1$ factor $\mathcal R$ tensorially, i.e., $\mathcal M\cong\mathcal M\,\overline{\otimes}\,\mathcal R$. This property was first extensively studied in \cite{McDuff70}, and such II$_1$ factors are said to be  \emph{McDuff}. In contrast, not all infinite dimensional unital simple separable nuclear $C^*$-algebras factorize as a non-trivial tensor product (examples without such factorizations include R\o{}rdam's famous construction of a unital simple separable nuclear $C^*$-algebra with both an infinite and finite projection from \cite{Rordam03}). Yet tensorial absorption has been crucial in $C^*$-classification ever since Kirchberg's Geneva theorems, one of which shows that for simple separable nuclear $C^*$-algebras, pure infiniteness is characterized by tensorial absorption of the Cuntz algebra $\mathcal O_\infty$. Moreover, ``$\mathcal O_\infty$-stability'' is an important ingredient in the Kirchberg--Phillips theorem.

The Jiang--Su algebra $\mathcal Z$, introduced in \cite{Jiang-Su99}, is the stably finite counterpart of $\mathcal O_\infty$. A $C^*$-algebra $A$ is said to be $\Z$-stable if $A\cong A\otimes\Z$. Both $\mathcal O_\infty$ and $\Z$ are infinite dimensional, unital, simple, separable, nuclear, with the same $K$-theory as the complex numbers.  However, $\Z$ has a unique trace, so is stably finite, whereas $\mathcal O_\infty$ is purely infinite. Both are tensorially (strongly) self-absorbing. Moreover, $\Z$-stability is the mildest $C^*$-tensorial absorption condition akin to the McDuff property for II$_1$ factors (see \cite{Winter11}).
From this point of view, $\mathcal Z$ is a natural analog of $\mathcal R$.

Just as $\mathcal O_\infty$-stability underpins the Kirchberg--Phillips classification, we use $\Z$-stability to separate classifiable $C^*$-algebras from the higher dimensional examples of \cite{Villadsen98,Villadsen99,Rordam03,Toms08a,Toms-Winter13,Giol-Kerr10} pioneered by Villadsen. The use of $\Z$-stability in classification was explicitly predicted by Toms in 2005 (\cite{Toms08a}): ``Optimistically $\Z$-stability is an abstraction of slow dimension growth, and the Elliott conjecture will be confirmed for all simple separable nuclear $C^*$-algebras having this property.'' See also \cite{Elliott-Toms08}.  Not only does $\Z$ have the same $K$-theory and traces as $\mathbb C$, but for any $C^*$-algebra $A$, the $K$-theory and traces of $A$ and of $A \otimes \mathcal Z$ are isomorphic. For this reason, $\mathcal Z$-stability is for all intents and purposes necessary for classification by $K$-theory and traces. We describe the Jiang--Su algebra and collect the consequences of $\Z$-stability used throughout the paper in Section~\ref{Sect:Z}.

\subsubsection{The UCT}\label{Intro:UCTsubsect} Kasparov's $KK$-theory is a bivariant theory unifying $K$-theory and $K$-homology (\cite{Kasparov80,Kasparov84}). It has proved indispensable in both $C^*$-algebra classification and index theory (see for example \cite{Kasparov88,Higson98}). The Kirchberg--Phillips classification of stable simple separable nuclear purely infinite  $C^*$-algebras is expressed directly in terms of $KK$-theory: two such algebras $A$ and $B$ are isomorphic if and only if they are $KK$-equivalent.  Thinking of $KK$-equivalence as a very loose notion of homotopy between $C^*$-algebras, this has the spirit of Mostow's rigidity theorem in geometric topology.

Inspired by earlier work of Brown (\cite{Brown84}), Rosenberg and Schochet established their universal coefficient theorem (UCT), which computes $KK$-theory in terms of $K$-theory for a large class of $C^*$-algebras (\cite{Rosenberg-Schochet87}) with generous permanence properties. This class is precisely the $C^*$-algebras that are $KK$-equivalent to some abelian $C^*$-algebra. Such algebras are said to \emph{satisfy the UCT}.

In contrast to $\Z$-stability, where examples lacking the property are known, it remains a crucial open question whether all (simple) separable nuclear $C^*$-algebras satisfy the UCT. For naturally occurring, concrete examples, this is less problematic as there are very general tools for verifying the UCT. Notably, all amenable groupoid $C^*$-algebras satisfy the UCT (\cite{Tu05}). In the setting of group actions, the UCT is preserved by crossed products of nuclear $C^*$-algebras by countable torsion-free amenable groups (\cite{Meyer-Nest06}). Whether the same holds for finite groups is open (in fact, equivalent to the UCT problem). See Section~\ref{subsec:the-uct} and Remark~\ref{rmk:UCTcrossedproducts} for more details.

Among the hypotheses, the UCT is the only one that does not appear to have a von Neumann algebraic counterpart, due to its inherently topological nature.

\subsubsection{The invariant}
The data comprising the \emph{Elliott invariant} of a unital stably finite $C^*$-algebra evolved over the early stages of the classification program, from the dimension groups appearing in Elliott's classification of approximately finite dimensional $C^*$-algebras (\cite{Elliott76}), to its present form (\cite{Elliott93,Thomsen94}; see \cite{Rordam02}). This consists of ordered $K_0$, $K_1$, the class of the unit, traces, and the pairing map. In fact, the order on $K_0$ is redundant under the hypothesis of $\mathcal Z$-stability (see \cite[Corollary 4.10]{Rordam04}). While we have stated Theorem~\ref{Main} and Corollary~\ref{MainCor} in terms of Elliott's invariant, in the main body of the paper we use the (formally weaker) invariant $KT_u$ consisting of $K$-theory, the class of the unit, traces, and the pairing map, but not the order on $K_0$.  See the discussion in Section~\ref{sec:Elliott-invariant}.

The unital classification theorem is complemented by an important range of invariant result: all possible pairs of countable groups and metrizable Choquet simplices are realized by the $C^*$-algebras of Theorem \ref{Main} (see Remark \ref{rmk:RangeOfInvariant}). These range of invariant theorems were established via crossed product constructions (\cite{Rordam95}) in the purely infinite setting and inductive limit constructions (\cite{Elliott96}) in the stably finite one.
Understanding the range of the invariant leads to a wealth of other automatic structural results such as the existence of Cartan subalgebras inside the $C^*$-algebras covered by Theorem~\ref{Main} (\cite{Li20,Clark-Fletcher-etal}); equivalently via \cite{Renault08}, 
classifiable $C^*$-algebras all arise from twisted groupoids.

Turning to Theorem~\ref{Main2}, examples in the 1990s show that $K$-theory and traces alone are not enough to classify $^*$-homomorphisms (this was made explicit in \cite{NielsenThomsen}). This forces enlargement of the invariant, which we formalize as the \emph{total invariant}: $\inv$ (see Section \ref{ss:totalinv}). This is obtained by adjoining a Hausdorffized unitary algebraic $K_1$-group developed by Thomsen in \cite{Thomsen95}, which we denote $\Ka(A)$, and \emph{total $K$-theory} (Schochet's $K$-theory with $\mathbb Z/n\mathbb Z$-coefficients; \cite{Schochet84}) to $KT_u$.

While our approach does not need a range of invariant result (i.e., existence of objects) to classify $C^*$-algebras in the unital classification theorem, it is vital in the deduction of Theorem~\ref{Main} from Theorem~\ref{Main2} that the classification of unital embeddings does include a range of invariant result. Any morphism at the level of the invariant $\inv$ in Theorem~\ref{Main2} that is faithful on traces is realized by a unital $^*$-monomorphism.  This is more subtle than the previous sentence might initially suggest. Just as any morphism between the Elliott invariants of $C^*$-algebras is required to intertwine the pairing maps between $K$-theory and traces, morphisms between the enlarged invariant $\inv$ must intertwine the natural maps between its components.  Most of these are familiar to experts: the Bockstein operations connecting the various $K$-groups with coefficients, the de la Harpe--Skandalis determinant from \cite{Harpe-Skandalis84} relating $\Ka$ and traces, and the canonical map from $\Ka$ to $K_1$.  We review these in Section~\ref{SectAlgK1} and collect properties of the Hausdorffized algebraic $K_1$-group and total $K$-theory there as well.

There is one more natural family of maps
\begin{equation}\label{Intro.NewMaps}
\zeta_A^{(n)}\colon K_0(A;\mathbb Z/n\mathbb Z)\to \Ka(A) 
\end{equation}
through which the Bockstein maps $K_0(A;\mathbb Z/n\mathbb Z)\to K_1(A)$ factorize. When $K_1(A)$ is torsion-free, these maps are redundant, but in general, compatibility with $\zeta^{(n)}$ is necessary to construct a morphism between $C^*$-algebras realizing specified behavior on $K$-theory, traces, total $K$-theory, and $\Ka$.  In the setting of the classification of unital embeddings (Theorem~\ref{Main2}) this is the only additional compatibility needed.

We give an explicit construction of the maps \eqref{Intro.NewMaps} in Section~\ref{sec:new-map}. During the (lengthy) gestation of this paper, Gong, Lin, and Niu also discovered the need for additional compatibility between total and algebraic $K$-theory to produce morphisms (\cite{Gong-Lin-etal23}). They took a different, more abstract approach. See  Remark~\ref{NewMapsRemark}.

\subsubsection{Examples}\label{Intro:examples} To apply the classification theorem, it is important to be able to recognize $\Z$-stability in examples. Firstly, $\Z$-stability can be described in terms of approximately central sequences, without reference to the Jiang--Su algebra itself (\cite[Proposition 2.3]{Winter10a}, using \cite{Rordam-Winter10}, together with \cite{Kirchberg06} or \cite{Toms-Winter07}). This may be regarded as a suitable softening to positive elements of McDuff's criterion for absorption of $\mathcal R$ in terms of approximately central matrix embeddings from \cite{McDuff70}. 

Over the last 15 years, a major program of work on the Toms--Winter conjecture (see \cite{Elliott-Toms08,Toms-Winter09} and \cite[Section 5]{Winter19}) has provided  both a topological characterization of $\Z$-stability and Matui and Sato's powerful von Neumann techniques for establishing $\Z$-stability (\cite{Matui-Sato12}). \emph{Nuclear dimension} is a noncommutative generalization of topological covering dimension introduced by Winter and Zacharias (\cite{Winter-Zacharias10}). For unital simple separable nuclear non-elementary $C^*$-algebras, $\Z$-stability is equivalent to finiteness of the nuclear dimension (\cite{Winter12,Castillejos-Evington-etal21,Castillejos-Evington20}, building on \cite{Winter10a,Matui-Sato14,Sato-White-etal15,Bosa-Brown-etal15}).  In fact, the 2015 version of Theorem \ref{Main} was stated with a finite nuclear dimension hypothesis, though this was only used directly in one part of the argument, and elsewhere finite nuclear dimension was used to access $\Z$-stability through \cite{Winter12}. Today, the equivalence gives two paths to classifiability. For examples built from objects of bounded topological dimension, it has often been possible to estimate the nuclear dimension.  However, for many examples outside this setting, direct estimates on the nuclear dimension seem inaccessible, and von Neumann techniques provide a route to $\Z$-stability.

For example, a unital simple separable approximately subhomogeneous $C^*$-al\-ge\-bra $A$ is always in the UCT class. The nuclear dimension of subhomogeneous algebras is controlled by the dimension of the spectrum (\cite{Winter04}), so that when $A$ has no dimension growth it also has finite nuclear dimension. When the dimension grows, a direct estimate on the nuclear dimension fails, yet, using \cite{Winter12}, Toms proved that $A$ is $\Z$-stable precisely when it has slow dimension growth (\cite{Toms-11}).

Given a finitely generated amenable group $G$, every irreducible unitary representation $\pi$ generates a simple $C^*$-algebra precisely when $G$ is virtually nilpotent (this goes back to \cite{Moore-Rosenberg76, Poguntke81}; see also \cite{Echterhoff90}). In this case, $C^*_\pi(G)$ is $\Z$-stable whenever $\pi$ is infinite dimensional (\cite{Eckhard-Gillaspy-McKenney-19}, building on \cite{EckhardtMcKenney18}, an argument using $C^*$-classification and work on the Toms--Winter conjecture). Very shortly after the first version of this paper, Eckhardt established the UCT for these examples (\cite{Eckhardt:adv}, extending the earlier \cite{Eckhardt-Gillaspy-16} for nilpotent groups).

A vast class of examples arise from free minimal actions $G\curvearrowright X$ of a countable discrete amenable group on a compact metrizable space $X$.
Here, the associated crossed product $C^*$-algebra is unital, simple (by \cite[Corollary 5.16]{Effros-Hahn67}), separable, nuclear, and satisfies the UCT (by \cite{Tu05}). It is increasingly clear that $\mathcal Z$-stability is linked to the dimension of $X$, or more generally, of the action of $G$ on $X$.

When $X$ has finite covering dimension, direct estimates on the nuclear dimension are possible when $G$ is virtually nilpotent (\cite{SWZ-19}, going back to \cite{HWZ-15,Szabo-15}). For more general groups, Kerr generalized Matui's notion of almost finiteness from \cite{Matui-12} and used this together with Matui and Sato's von Neumann algebra transfer technique (\cite{Matui-Sato12}, in the form given in \cite{Orovitz-Hirshberg13}) to obtain his $\Z$-stability theorem in \cite{Kerr20} for free minimal almost finite actions of amenable groups on finite dimensional spaces. This has been used to show $\mathcal Z$-stability (hence classifiability) of such crossed products when $G$ has locally subexponential growth (\cite{Kerr-Szabo-20,DZ-17}), when $G$ is elementary amenable (\cite{Kerr-Naryshkin21,Naryshkin23}), or for generic actions of any amenable group $G$ on the Cantor set (\cite{CJKMST-D-18}).
It is conceivable that free minimal actions of countable discrete amenable groups on finite dimensional spaces give classifiable crossed products.

Outside the setting of finite dimensional base spaces, examples of non-$\Z$-stable free minimal crossed products by $\mathbb Z$ were shown to exist by Giol and Kerr (\cite{Giol-Kerr10}); these all have strictly positive mean dimension in the sense of Gromov, Lindenstrauss, and Weiss (\cite{Gromov99b,Lindenstrauss-Weiss00}). In the positive direction, Elliott and Niu showed that minimal $\mathbb Z$-actions of mean dimension zero give rise to a $\Z$-stable, and hence classifiable, crossed product (\cite{Elliott-Niu-17}).  This has been extended by Niu to $\mathbb Z^d$-actions in \cite{Niu19}. Combining the Phillips--Toms conjecture (see the introduction to \cite{Hirshberg-Phillips22}) and the last open implication in the Toms--Winter conjecture, crossed products associated to free minimal actions of amenable groups on compact metrizable spaces are predicted to be classifiable precisely when the mean dimension vanishes.

Another very natural class of examples arises from non-commutative dynamics: crossed products from outer actions $G\curvearrowright A$ of amenable groups on classifiable $C^*$-algebras.  When $A$ has a unique trace, Sato showed that the crossed product $A\rtimes G$ is $\Z$-stable (\cite{Sato19}), so is classifiable when it satisfies the UCT (see Remark~\ref{rmk:UCTcrossedproducts}). This result extends outside the monotracial setting (\cite{Sato19,Gardella-Hirshberg-Vaccaro, GGNV}), although at present there are restrictions both on the size of the tracial state space, and on the structure of the action of $G$ on traces.

\subsection{A history of the unital classification theorem}
\label{subsec:brief-history}

In the previous section, we took advantage of hindsight to set out the modern context for the unital classification theorem.  Here we explain in broad strokes how the subject got to this point. We proceed thematically (rather than strictly chronologically) and focus on classification rather than regularity. Therefore, we do not describe the history of the Toms--Winter conjecture (for which, see \cite{Elliott-Toms08} and \cite[Section 5]{Winter19}).

\subsubsection{Early days}

The first classification results for simple $C^*$-algebras were for inductive limits of finite dimensional $C^*$-algebras (AF algebras) from the late 1950s to the mid-1970s (\cite{Glimm60,Dixmier67,Bratteli72,Elliott74}), inspired by Murray and von Neumann's uniqueness of the hyperfinite II$_1$ factor.  The range of the invariant (ordered $K_0$) for AF algebras was described abstractly at the end of the 1970s in \cite{Elliott79, Effros-Handelman-etal80}.

Also at the end of the 1970s came Cuntz's work on his now eponymous $C^*$-algebras $\mathcal O_n$ (\cite{Cuntz77,Cuntz81Ann}, building on an example given by Dixmier \cite{Dixmier64}), and the introduction of simple purely infinite $C^*$-algebras.  These are $C^*$-analogues of the type III factors, and important examples include $\mathcal O_n$ and simple Cuntz--Krieger algebras (\cite{Cuntz-Krieger80}).

The 1980 Kingston meeting on operator algebras is widely regarded as a landmark event in the field. Cuntz's survey article from this conference gives a nice account of early progress and examples of simple $C^*$-algebras (\cite{Cuntz82}). A talk by Effros laid out a number of open problems (\cite{Effros82}) related to the structure and classification of (simple) nuclear $C^*$-algebras (particularly Problems 5--10). As we set out below, these proved influential over the next phase of the subject.

\subsubsection{Purely infinite $C^*$-algebras}

In \cite{Cuntz86}, Cuntz asked whether the Cuntz--Krieger algebras were classified by their $K_0$-groups (perhaps along with other invariants), following related  results in symbolic dynamics (\cite{Franks84}). Interest in such questions led to  classification theorems for inductive limits of building blocks constructed from Cuntz algebras (\cite{Rordam93,Lin-Phillips95c}), followed by breakthrough tensorial absorption results: $\mathcal O_2 \otimes M_{2^\infty} \cong \mathcal O_2$ (a problem posed by Cuntz in \cite{Cuntz82}) was obtained using dynamical methods in \cite{BSKR93}, and Elliott proved that $\mathcal O_2 \otimes \mathcal O_2 \cong \mathcal O_2$ (though his proof remains unpublished following R\o{}rdam's \cite{Rordam94}).  In \cite{Rordam95}, R\o{}rdam introduced an important strategy for classifying $C^*$-algebras by classifying embeddings, which has been crucial to subsequent results. He considered the class of unital simple separable nuclear purely infinite $C^*$-algebras satisfying the UCT, for which $KL$ classifies embeddings into simple purely infinite $C^*$-algebras, and showed that such algebras are classified by $K$-theory (keeping track of the unit). This class encompasses all simple Cuntz--Krieger algebras (\cite{RordamKth95}) and exhausts all possible combinations of $K$-groups (\cite{Elliott-Rordam95}). Accordingly, it is the largest class of simple separable nuclear purely infinite $C^*$-algebras satisfying the UCT that can be classified by $K$-theory.

R\o{}rdam's work reduced the classification problem (in the purely infinite setting) to classifying embeddings by $KL$. Checking this condition in concrete settings would be a herculean task --- an abstract approach was needed. The stage for this had been set by decades of advances, including Arveson's extension theorem (\cite{Arveson69}), 
the Choi--Effros lifting theorem (\cite{Choi-Effros76}), Voiculescu's theorem (\cite{Voiculescu76}), and Kasparov's development of $KK$-theory (\cite{Kasparov80, Kasparov84}).

Building on his deep work on tensor products and exactness (\cite{Kirchberg94, Kirchberg95}), Kirchberg announced his breakthrough ``Geneva theorems'' at the ICM satellite meeting in 1994: every separable exact $C^*$-algebra embeds into $\mathcal O_2$; $A\otimes \mathcal O_2 \cong \mathcal O_2$ for any unital simple separable nuclear $C^*$-algebra $A$; and such an $A$ is purely infinite if and only if $A\otimes \mathcal O_\infty \cong A$ (the first published proof was written in collaboration with Phillips in \cite{Kirchberg-Phillips00}). Using these results, Kirchberg (\cite{Kirchberg}) and Phillips (\cite{Phillips00}) independently showed that simple separable nuclear purely infinite $C^*$-algebras --- now known as \emph{Kirchberg algebras} --- satisfying the UCT are classified by $K$-theory.  Consequently, R\o{}rdam's class consists exactly of the UCT Kirchberg algebras, and \cite{Elliott-Rordam95} is a range of invariant result for the Kirchberg--Phillips classification.

\subsubsection{Approximately type I $C^*$-algebras} 
\label{sec:AIhistory}

Stepping back a bit, the late 1980s saw two of Effros' problems from \cite{Effros82} solved: Blackadar showed that the fixed-point algebra of a $\mathbb Z/2\mathbb Z$ action on an AF algebra need not be AF (\cite{Blackadar90}), and Evans and Kishimoto showed that removing an AF tensor factor from an AF algebra need not give an AF algebra (\cite{Evans-Kishimoto91}).
These and related developments brought renewed interest in classifying inductive limits of more general type I building blocks --- again one of Effros' problems --- particularly for building blocks of the form $C([0,1],F)$ and $C(\mathbb T,F)$, where $F$ is finite dimensional, which give rise to the classes of AI and A$\mathbb T$ algebras respectively.  Presenting the CAR algebra as an A$\mathbb T$ algebra (rather than using an AF model) is instrumental in Blackadar's construction \cite{Blackadar90}, and A$\mathbb T$ structure was found in important examples coming from dynamics (\cite{Putnam89,Elliott-Evans93}, the latter of which was also roughly conjectured in Effros' \cite{Effros82}).  In his pioneering classification of A$\mathbb T$ algebras of real rank zero (\cite{Elliott93}), Elliott enlarged his previous invariant, used to classify AF algebras, to include the $K_1$-group (in talks at the time, Elliott noted inspiration from the use of $K_1$ in the classification results of \cite{Pasnicu87} and in the proof that extensions of AF algebras are AF from \cite{Brown82a}). This led to his slogan ``$K$-theory suffices'', and arguably, the Elliott classification program began in earnest at this point.

It soon turned out that, outside the real rank zero setting, topological $K$-theory alone is not sufficient. In the first of many important contributions, Thomsen showed that any metrizable Choquet simplex can arise as the trace space of an AI algebra whose $K_0$-group is $\mathbb Q$ with the usual order (\cite{Thomsen94}). This led to the addition of traces and their pairing with $K_0$ to the invariant (the pairing is also necessary by examples which can be obtained from \cite{Goodearl92} --- see \cite[Proposition~3.1.8]{Rordam02} --- but the big revelation was the need for traces at all).  With this larger invariant, Elliott managed to classify simple AI algebras in \cite{Elliott93a}, building on \cite{Blackadar-Kumjian-etal92, Thomsen94}, and then extended this to the A$\mathbb T$ case in \cite{Elliott97}.  Thus the Elliott invariant reached its modern form, and Elliott stated his classification conjecture in  \cite{Elliott95}.

The need for additional data beyond the Elliott invariant --- total $K$-theory and the Hausdorffized unitary algebraic $K_1$-group --- to establish uniqueness theorems for $^*$-homomorphisms was first seen in the setting of inductive limits of building blocks with topological dimension one. The role of total $K$-theory emerged in \cite{Dadarlat-Loring96b,Eilers96} and was used by Dadarlat and Loring (\cite{Dadarlat-Loring96}) to classify morphisms between certain (not necessarily simple) real rank zero inductive limits previously considered by Elliott.
The new ingredient introduced in \cite{Dadarlat-Loring96} was the universal multicoefficient theorem (UMCT), which was used to relate maps at the level of total $K$-theory to $KK$-classes.  This has become a crucial tool in classification.

Subsequently, Nielsen and Thomsen gave a new approach to the classification of simple A$\mathbb T$ algebras (\cite{NielsenThomsen,Nielsen99}). They give examples to show that the Elliott invariant alone does not classify the morphisms and prove a uniqueness theorem for $^*$-homomorphisms by adjoining the unitary group modulo the closure of the commutators. As A$\mathbb T$ algebras have stable rank one, this additional data is precisely the Hausdorffized unitary algebraic $K_1$-group previously studied by Thomsen (\cite{Thomsen95}) using the de la Harpe--Skandalis determinant of \cite{Harpe-Skandalis84}.

Moving to higher dimensional building blocks,
large-scale work culminated in far-reaching classification results for approximately homogeneous (AH) $C^*$-algebras by Elliott and Gong (\cite{Elliott-Gong96}) and then by Elliott, Gong, and Li (\cite{Gong02, Elliott-Gong-etal07}; preprint versions were circulated in 1998). At this point, the key hypothesis entailing classifiability was given in terms of relative upper bounds on the topological dimension of the spectrum of the building blocks of the inductive limit --- \emph{very slow dimension growth} (see \cite[Chapter 3]{Rordam02} for a discussion and the definition) ---  which was used in a careful analysis to reduce to inductive limits of very specific building blocks. 

At the Fields Institute in 1995, Villadsen announced his work on exotic AH algebras with perforated $K$-theory (\cite{Villadsen98, Villadsen99}).
None of these examples can absorb $\Z$ tensorially, as shown by Gong, Jiang, and Su (\cite{Gong-Jiang-etal00}), who initiated the abstract study of $\Z$-stable $C^*$-algebras, continued in \cite{Rordam04}.
Through clever adaptations of Villadsen's constructions, R{\o}rdam (\cite{Rordam03}) and Toms (\cite{Toms08a}) produced strong counterexamples to Elliott's classification conjecture in 2002 and 2003. 
It is around this time that $\Z$-stability was first suggested as a potential abstract hypothesis for classification; see \cite[Section 5]{Toms08a} and \cite[Page 64]{Rordam04}.

Inductive limits of subhomogeneous building blocks, known as ASH algebras, provided further challenges (as a warning, some early papers use a more restrictive definition of ASH algebras). These were eventually overcome using the techniques described in Sections~\ref{intro:sectTAF} and~\ref{intro:sectLocalization}, and it was not until the unital classification theorem in 2015 when unital simple ASH algebras with slow dimension growth (or, equivalently by \cite{Toms-11,Winter12}, simple $\Z$-stable ASH algebras) were classified. Particularly important examples came earlier: Elliott exhausted the range of the invariant of stably finite classifiable $C^*$-algebras in \cite{Elliott96}, using certain ASH algebras (whose building blocks have dimension at most two), and Jiang and Su's classification of simple inductive limits of sums of dimension drop $C^*$-algebras underpins their work on $\Z$ and strong self-absorption (\cite{Jiang-Su99}, building on \cite{Elliott-Gong-etal97}).  It remains an open problem whether every simple separable nuclear stably finite $C^*$-algebra is ASH.

\subsubsection{Tracially AF $C^*$-algebras}\label{intro:sectTAF} The major task around the turn of the millennium was articulated by Lin in his foundational paper \cite{Lin01a}: ``\dots to establish a classification result for [simple nuclear stably finite] $C^*$-algebras that are not assumed to be direct limits of some special form.'' 

The impetus for the next phase was given by breakthrough results for quasidiagonal $C^*$-algebras (\cite{Blackadar-Kirchberg97,Popa97}).  Popa's approximation property, motivated in part by Effros' problem of abstractly characterizing AF algebras (\cite{Effros82}), obtains finite dimensional subalgebras from quasidiagonality in the presence of sufficiently many projections. For a unital simple quasidiagonal $C^*$-algebra $A$ of real rank zero there exist non-zero approximately central projections $p$ with the property that compressions to the resulting corners $pAp$ can be approximated by finite dimensional subalgebras. This was based on Popa's earlier 2-norm version of this idea for von Neumann algebras, where he could use maximality to take $p=1$, and recover Connes' Theorem that injective $\mathrm{II}_1$ factors are hyperfinite (\cite{Popa86}). Likewise, by definition, a $C^*$-algebra is AF if one can always take $p=1$ in Popa's condition.  Lin's innovation was to ask for $1-p$ to be small in the Cuntz semigroup (described as $p$ being ``tracially large''), introducing the class of \emph{tracially AF} (TAF) $C^*$-algebras (\cite{Lin01a}).

Lin obtained a classification theorem for unital simple separable nuclear tracially AF algebras satisfying the UCT in \cite{Lin01}, and using a range of the invariant result, showed that this class coincides with unital simple AH algebras of real rank zero and slow dimension growth (\cite{Lin03}).

The classification of TAF algebras instigated the systematic use of absorption in the stably finite setting, prompting the ``stable uniqueness'' and ``stable existence'' theorems of Dadarlat--Eilers (\cite{Dadarlat-Eilers02}) and Lin (\cite{Lin02}) together with a novel use of the UCT as a way of controlling the complexity of stabilizations in $KK$-theory.  Morally, Lin's TAF uniqueness theorem is obtained from stable uniqueness, hiding the stabilization in tracially small corners via the UCT.  The corresponding existence theorem is even more difficult: approximately multiplicative maps $A\to B$ are built in two stages by factoring through explicitly constructed model algebras $C$ with precise internal structure. Lin uses Blackadar and Kirchberg's work on simple nuclear quasidiagonal $C^*$-algebras (\cite{Blackadar-Kirchberg97, Blackadar-Kirchberg01}) to pass from $A$ to $C$ (\cite[Theorem~4.2]{Lin04}) and the classification of simple real rank zero AH algebras to get from $C$ to $B$ (\cite[Corollary~4.4 and Definition~3.1]{Lin01}). See also Dadarlat's streamlined approach (which obtained stronger results for embeddings, also factoring through models) in \cite{Dadarlat04}. 

To use Lin's theorem, one needs abstract tools for obtaining tracial approximations. A number of dynamical examples were brought within the scope of this classification by ad hoc means (for example \cite[Corollary 3.13, Theorem 3.6]{Lin03a}, \cite{Lin-Phillips10}). Winter provided the first systematic approach in 2005 (\cite{Winter06}), via  \emph{decomposition rank} (\cite{Kirchberg-Winter04}). This is another important non-commutative covering dimension which preceded  nuclear dimension. The two notions are subtly related, but having finite decomposition rank is a strong condition, entailing quasidiagonality and being challenging to verify directly, although it does cover ASH algebras with no dimension growth (\cite{Winter04}).

For a unital simple separable $C^*$-algebra with real rank zero and finite decomposition rank, Winter obtains a ``Popa approximation'' where the projection $p$ is bounded below in all traces in terms of the decomposition rank (\cite{Winter06}, noting that while that reference also requires $\Z$-stability, this is now redundant due to Winter's first $\Z$-stability theorem from \cite{Winter10a}). Repeating the argument inductively in the complementary corner leads to well-approximated corners of arbitrarily large trace, proving such algebras are tracially AF. In particular, Winter's work gave the first classification of simple ASH algebras of real rank zero and no dimension growth.

\subsubsection{Localization and rational tracial approximation}\label{intro:sectLocalization}

Not all simple nuclear stably finite $C^*$-algebras have good tracial approximations: indeed, they might have no projections. Winter's ``localization technique'' (\cite{Winter14}, first circulated as a preprint in 2007) overcomes this difficulty via an ingenious use of the Jiang--Su algebra $\mathcal Z$.  By viewing $\mathcal Z$ as an inductive limit of generalized dimension drop algebras with UHF fibers as in \cite{Rordam-Winter10}, Winter's strategy passes from a classification of UHF-stable algebras (where there are many projections, so tracial approximations are conceivable) to a classification of $\Z$-stable algebras. This localization technique comes at a price: at the UHF-stable level, one must classify $^*$-isomorphisms up to ``strong asymptotic unitary equivalence'' instead of approximate unitary equivalence --- that is, sequences of unitaries are replaced with continuous one-parameter families of unitaries starting at the unit.  There are considerable additional obstructions to this form of uniqueness --- for example, one must consider $KK$ instead of its Hausdorffization $KL$ --- and  correspondingly strong existence results are also required.  

Strong asymptotic classification results were nevertheless proved for $C^*$-algebras with tracial approximations: Lin and Niu achieved this for simple separable nuclear TAF $C^*$-algebras with the UCT in \cite{Lin08, Lin-Niu08}. There is a quantum leap in difficulty over approximate classification. The fact that $KK$-theory (in contrast to $KL$) fails to preserve inductive limits calls for highly intricate analysis involving careful use of ``basic homotopy lemmas'' (\cite{Lin10}), allowing approximately central unitaries to be connected by approximately central homotopies (a concept going back to \cite{Bratteli-Elliott-etal}).

In 2013, Matui and Sato used the Lin--Niu--Winter results together with breakthrough work on the Toms--Winter conjecture to give the first truly abstract definitive stably finite classification theorem: unital simple separable nuclear $\mathcal Z$-stable quasidiagonal monotracial $C^*$-algebras satisfying the UCT are classified by Elliott's invariant (\cite{Matui-Sato14}). (Via their earlier breakthrough \cite{Matui-Sato12}, one could replace $\Z$-stability by strict comparison here.) Matui and Sato show that the UHF-stabilizations of such a $C^*$-algebra $A$ (possibly without the UCT) are TAF ($A$ is said to be \emph{rationally TAF}). See \cite[Section 6]{Matui-Sato14}; in hindsight, this argument parallels the last part of Connes' proof of injectivity implies hyperfiniteness for II$_1$ factors (a related discussion is given in \cite[Section 6]{CETW22}). The Matui--Sato classification theorem includes $C^*$-algebras arising from a minimal uniquely ergodic homeomorphism of a finite dimensional compact metrizable space with infinitely many points, using  \cite{Toms-Winter13} and \cite{Pimsner83} to prove the regularity and quasidiagonality hypotheses, respectively.

\subsubsection{The Gong--Lin--Niu class and classification by embeddings}

Stepping back a bit, following Lin's classification of tracially AF $C^*$-algebras, there was rapid progress in the classification of $C^*$-algebras with tracial approximations by more complicated type I building blocks such as interval algebras (\cite{Lin07}). The new target of strong asymptotic classification promoted by localization was first obtained for TAI algebras in \cite{Lin11}. This sparked the line of research culminating in the Gong--Lin--Niu strong asymptotic classification (\cite{GLN-part1,GLN-part2}) of unital simple separable nuclear UHF-stable $C^*$-algebras satisfying the UCT that are tracially approximated by non-commutative $1$-dimensional CW-complexes. This was announced in Oberwolfach in 2012 (\cite{Gong-OberwolfachReport,Lin-OberwolfachReport}) and circulated as a preprint in 2015 (\cite{GLN-preprint}). The relevance of the Gong--Lin--Niu class is that it exhausts the values of the Elliott invariant on unital simple separable nuclear UHF-stable stably finite  $C^*$-algebras. No larger class of building blocks are needed.  

Gong, Lin, and Niu's argument is exceptionally technically demanding and involves detailed analysis and layers of quantification and approximation.  Just as in the setting of real rank zero inductive limits of no dimension growth, where the ASH case proved an order of magnitude more challenging than the AH case (and came a decade later through tracial approximations), the subhomogeneous nature of the building blocks present huge technical hurdles.  Highly bespoke homotopy lemmas and existence and uniqueness results between the building blocks are used throughout the paper.  The arguments also involve elaborately constructed explicit models, through which to factor their existence and uniqueness results. The vision behind this work represents an incredible achievement on the part of Gong, Lin and Niu, yet the daunting length and extreme technicality of \cite{GLN-part1,GLN-part2} have made it very difficult for researchers to fully absorb the entire argument.

Given the scope of Gong, Lin, and Niu's work, research on the stably finite case of the unital classification theorem pivoted to abstractly determining which UHF-stable $C^*$-algebras fall into the Gong--Lin--Niu class.  Shortly after Matui and Sato's abstract approach to rational TAF structure, Winter introduced a general technique of \emph{classification by embeddings} (\cite{Winter16a}), and applied it (together with \cite{Robert12,Strung}) to bring $C^*$-algebras associated to minimal diffeomorphisms on odd spheres within the scope of the Elliott program by means of rational TAI classification. Winter's result shows that a significantly weaker form of tracial approximation implies rational (Lin-style) tracial approximations in the presence of finite nuclear dimension.
Combined with a classification result by Robert (\cite{Robert12}), Winter's technique only requires appropriate external approximations in order to get (internal) rational tracial approximations. As a proof of concept, these external approximations arise immediately for quasidiagonal monotracial $C^*$-algebras. (See \cite[Corollary 2.4]{Winter16a}, which at the time was a special case of \cite[Theorem 6.1]{Matui-Sato14}).

Classification by embeddings was to become the route to the Gong--Lin--Niu class. For example, as precursors to the full classification theorem, it was used by Lin to classify simple $\mathbb Z$-crossed products of finite dimensional spaces (\cite{Lin15}) and by Elliott, Gong, Lin, and Niu to classify simple $\Z$-stable ASH algebras (\cite{Elliott-Gong-etal15}).

\subsubsection{Stable uniqueness across the interval and quasidiagonality}\label{subsect127}

Beyond these precursor results, major breakthroughs came quickly in 2015. Demonstrating a new technique for obtaining Winter-style external approximations, Niu announced joint work with Elliott (\cite{Elliott-Niu16a}) in April 2015, that $C^*$-algebras with finite decomposition rank and $K_0(A)\otimes\mathbb Q\cong \mathbb Q$ are rationally TAI. This was followed by Elliott, Gong, Lin, and Niu's general result (\cite{EGLN}; circulated as a preprint in July 2015) which removed the $K_0$ hypothesis, and therefore showed all unital simple separable nuclear $C^*$-algebras with finite decomposition rank satisfying the UCT, are rationally in the Gong--Lin--Niu class, and so classified by \cite{GLN-part1,GLN-part2}. The $\Z$-stability needed for classification is obtained from Winter's first $\Z$-stability theorem (\cite{Winter10a}).

The fundamental new idea in \cite{Elliott-Niu16a,EGLN} is a \emph{stable uniqueness across the interval} argument to connect two maps $\phi$ and $\psi$ into the universal UHF algebra (agreeing suitably on invariants) by a homotopy of approximately multiplicative maps with approximately constant tracial behavior.  The UCT enters to control the matrix size required in the stable uniqueness theorem.  This is then applied to maps arising from quasidiagonality and a range of invariant result to produce Winter-style approximations into (UHF-stabilizations of) Elliott--Thomsen building blocks.

As noted in Section~\ref{intro:sectTAF}, finite decomposition rank is much harder to verify in examples than finite nuclear dimension. In part, this is because it entails quasidiagonality. Right back in his foundational TAF paper (\cite{Lin01a}), Lin noted that ``recent developments suggest one may further assume simple (nuclear stably finite) $C^*$-algebras are quasidiagonal'', and indeed Elliott's examples exhausting the range of the stably finite invariant (\cite{Elliott96}) are all quasidiagonal, and moreover, all their traces are quasidiagonal (a condition first raised by Brown in his memoir \cite{Brown06}).

While there are many deep theorems giving quasidiagonality in examples, in the abstract setting this is a major challenge: Blackadar and Kirchberg's famous question (\cite[Question 7.3.1]{Blackadar-Kirchberg97}) of whether all nuclear stably finite $C^*$-algebras are quasidiagonal remains open. In fact, the Elliott--Gong--Lin--Niu theorem isolates quasidiagonality of traces as a key hypothesis: their theorem holds for unital simple separable finite $C^*$-algebras of finite nuclear dimension (via Winter's second $\Z$-stability theorem \cite{Winter12} in place of \cite{Winter10a}) satisfying the UCT, such that all traces are quasidiagonal. Later, through \cite{Bosa-Brown-etal15} and \cite{Castillejos-Evington-etal21,Castillejos-Evington20}, it was shown, without reference to the UCT, that a simple separable $C^*$-algebra $A$ of finite nuclear dimension has finite decomposition rank precisely when both $A$ and all its traces are quasidiagonal.

A little earlier in 2014, Ozawa, R\o{}rdam, and Sato found an unexpected route to quasidiagonality powered by classification, using Matui and Sato's \cite{Matui-Sato14} to show that elementary amenable groups have quasidiagonal $C^*$-algebras (\cite{Ozawa-Rordam-etal15}). Inspired by this, the Kirchberg--Phillips theorem, and talks on \cite{Elliott-Niu16a,EGLN}, AT, SW, and Winter used a  stable uniqueness across the interval argument to prove their quasidiagonality theorem (\cite{TWW}) over the summer of 2015: faithful traces on a separable nuclear $C^*$-algebra satisfying the UCT are quasidiagonal. This established Rosenberg's conjecture --- $C^*$-algebras associated to discrete amenable groups are quasidiagonal --- and removed the quasidiagonality hypothesis from \cite{EGLN}. In this way, the combination of the results in \cite{GLN-part1,GLN-part2,EGLN,TWW,Winter16a,Winter12,Winter14} (and the large body of work these papers rely on) comes together as the unital stably finite classification theorem (with the hypothesis of finite nuclear dimension replacing that of $\Z$-stability).

\subsubsection{Post-2015: $\Z$-stability and abstract classification}

Developments continued apace post 2015 in numerous directions, including classification results in the non-unital setting, a vast body of work bringing examples within the scope of the unital classification theorem as discussed in Section~\ref{Intro:examples}, and the continued large-scale study of fine structure in the Cuntz semigroup. We do not describe these here.

The present paper has its origins in two results announced at BIRS in 2017 (\cite{BanffReport17}).   Firstly Castillejos, Evington, AT, SW, and Winter showed that unital simple separable nuclear $\Z$-stable $C^*$-algebras have nuclear dimension at most one (\cite{Castillejos-Evington-etal21}).  This allows the unital classification theorem to be accessed from $\Z$-stability. (As explained in \cite[Section 6]{CETW22}, the methods of \cite{Castillejos-Evington-etal21} can be combined with those of \cite{Matui-Sato14} to give a direct proof of Winter's classification by embeddings theorem from $\Z$-stability, bypassing the only use of  nuclear dimension in the 2015 proof of the unital classification theorem.) Secondly, building on his abstract approach to the quasidiagonality theorem (\cite{Schafhauser17}), CS announced a precursor of what would become his AF-embedding theorem (\cite{Schafhauser18}): a classification of unital full nuclear embeddings of unital separable nuclear UCT $C^*$-algebras into the ultrapower of the universal UHF algebra. The techniques provided an abstract classification of  unital simple separable nuclear UHF-stable monotracial $C^*$-algebras satisfying the UCT which does not rely on obtaining any form of tracial approximation. (The version published in \cite{Schafhauser18} assumes $\mathcal Q$-stability and trivial $K_1$-group as this allows one to bypass some technicalities and is sufficient to obtain the headline AF-embedding result.)  This was the first abstract stably finite classification theorem which does not pass through some kind of internal $C^*$-algebra structure in the proof.

\subsection{Discussion of the methods behind Theorem~\ref{Main2}}
\label{subsec:methods-behind-thmB}

The overarching strategy of our paper is to combine the ideas in  \cite{Schafhauser18} and \cite{Castillejos-Evington-etal21} with a new $KK$-uniqueness theorem (Theorem~\ref{intro-KK-unique}). This avoids Winter's localization technique and directly classifies $\Z$-stable $C^*$-algebras without first classifying UHF-stable algebras.

\subsubsection{Classifying full approximate embeddings}

The existence portion of Theorem~\ref{Main2} requires producing $^*$-ho\-mo\-mor\-phisms between abstract $C^*$-algebras. Constructing these directly is hard. Instead, the usual strategy is to classify approximate $^*$-ho\-mo\-mor\-phisms --- these are sequences of maps that become $^*$-homomorphisms in the limit.  These are often easier to produce. Indeed, for example, quasidiagonality provides a rich supply of approximate $^*$-ho\-mo\-mor\-phisms into matrix algebras.

Let $A$ and $B$ be as in Theorem~\ref{Main2}. To avoid the need to carefully keep track of quantifiers, we work with the sequence algebra $B_\infty\coloneqq \ell^\infty(B)/c_0(B)$, and $^*$-ho\-mo\-mor\-phisms $A\to B_\infty$; these encode approximate $^*$-homomorphisms. 
Unitary equivalence of $^*$-homomorphisms into $B_\infty$ corresponds to a quantified approximate unitary equivalence of approximate $^*$-homomorphisms into $B$.
Although $B$ is simple, $B_\infty$ is not, so we impose a simplicity condition on the allowed morphisms $\phi\colon A\to B_\infty$: we require that $\phi(a)$ generates $B_\infty$ as an ideal for each non-zero $a\in A$.  Such maps are called \emph{full}, and (under our hypotheses on $B$) fullness is readily tested using traces via Cuntz comparison.  

With the setup above, the following theorem is the main objective of the paper.

\begin{theorem}[Classification of unital full approximate embeddings; see Theorem~\ref{approximate-classification}]\label{Main3}
Let $A$ be a unital separable nuclear $C^*$-algebra satisfying the UCT and $B$ be a unital simple separable nuclear $\Z$-stable $C^*$-algebra.  Up to unitary equivalence, the unital full embeddings $A \rightarrow B_\infty$ are classified by faithful morphisms on $\inv$.
\end{theorem}

The deduction of Theorem~\ref{Main2} from Theorem~\ref{Main3} is a standard intertwining argument. For this, it is vital that the classification in Theorem~\ref{Main3} contains both uniqueness and existence statements. This is set out in Section~\ref{sec:main-results}.

As with Theorem~\ref{Main2}, the purely infinite case of Theorem~\ref{Main3} is a known consequence of the methods behind the Kirchberg--Phillips classification theorem, while our proof handles the finite (i.e., tracial) situation. We proceed by dividing  $B_\infty$ into two pieces: a ``tracially small'' part,
\begin{equation} J_B \coloneqq \{ b \in B_\infty : \tau(b^*b) = 0 \text{ for all } \tau \in T(B_\infty) \} \lhd B_\infty, \end{equation}
and a ``tracially large'' part, $B^\infty \coloneqq B_\infty / J_B$, known as the \emph{trace-kernel ideal} and the \emph{trace-kernel quotient}, respectively. As such, $B_\infty$ fits into the \emph{trace-kernel extension} \begin{equation}\label{eq:trace-kernel-intro}
\begin{tikzcd}
	0 \arrow{r} & J_B \arrow{r}{j_B} & B_\infty \arrow{r}{q_B} & B^\infty \arrow{r} & 0,
\end{tikzcd}
\end{equation}
an ultrapower version of which first came to the fore in Matui and Sato's work on the Toms--Winter conjecture (\cite{Matui-Sato12,Matui-Sato14}; see also \cite{Kirchberg-Rordam14}).

Just as in CS's $\mathcal Q$-stable monotracial classification theorem (\cite{Schafhauser18}), the trace-kernel extension organizes the proof of Theorem~\ref{Main3} into two conceptual steps: classify morphisms into $B^\infty$, and then classify lifts of these morphisms to $B_\infty$.

\begin{equation}\label{overview-diag}
\begin{tikzcd}
&&&A\ar[d,"\theta\text{, classified by traces}"]\ar[dl,dashed,swap,"\substack{\text{lifts are classified by} \\ \text{$K$-theoretic data}}"]\\0\ar[r]&J_B\ar[r,"j_B"]&B_\infty\ar[r,"q_B"]&B^\infty\ar[r]&0
\end{tikzcd}
\end{equation}

The first step only uses tracial data and follows from \cite{Castillejos-Evington-etal21}.  The major work we undertake here is the classification of lifts.  The precise data needed for the existence of a lift lives in $KK(A,B_\infty)$.  On the other hand,  the data needed for uniqueness resides in the group $KL(A,J_B)$ (a Hausdorffization of $KK(A,J_B)$).

\begin{theorem}[Classification of unital lifts along the trace-kernel extension; see Theorems~\ref{thm:ClassifyingUnitalLifts}\ref{thm:ClassifyingLifts.C1} and \ref{thm:ClassifyingUnitalLiftsIntroRepeat}]\label{intro-lifts}
Let $A$ be a unital separable nuclear $C^*$-algebra and $B$ be a unital simple separable nuclear $\mathcal Z$-stable $C^*$-algebra with $T(B) \neq \emptyset$. Suppose that $\theta \colon A \rightarrow B^\infty$ is a unital full embedding.
	\begin{enumerate}
		\item\label{it:intro-lifts.1} There is a unital full embedding $A \rightarrow B_\infty$ lifting $\theta$ if and only if there is a lift of $[\theta]$ in $KK(A,B_\infty)$ which preserves the unit of $K_0$.\label{intro-lifts-1}
		\item\label{it:intro-lifts.2} Any choice of a unital full embedding $A \rightarrow  B_\infty$ lifting $\theta$ determines a bijection between the unitary equivalence classes of unital full lifts of $\theta$ and the elements of $KL(A, J_B)$ that annihilate the unit of $K_0(A)$.
  \label{intro-lifts-2}
	\end{enumerate}
\end{theorem}

The invariants in Theorem~\ref{intro-lifts} are somewhat unwieldy. Once Theorem \ref{intro-lifts} is proved, our last major task will be to compute $KL(A,J_B)$ using the UCT, allowing the classification of lifts to be reinterpreted in terms of the total invariant.

Next, we provide more detail regarding these components and how they come together to prove Theorem~\ref{Main3}. We follow this by discussing the role of the UCT and $\Z$-stability, and a brief comparison of our methods with the original approach.

\subsubsection{Classification into the trace-kernel quotient}\label{Intro:ClassTKQ}

Maps from separable nuclear $C^*$-algebras $A$ into finite von Neumann algebras are classified by traces as a consequence of Connes' theorem. When $B$ has a unique trace, this provides the classification of maps into the trace-kernel quotient. This is more often stated with the tracial ultrapower $B^\omega$ --- which is a II$_1$ factor --- in place of $B^\infty$, but it is routine to convert between these situations.  In general, neither $B^\infty$ nor $B^\omega$ is a von Neumann algebra, but the techniques introduced by Castillejos, Evington, AT, SW, and Winter in \cite{Castillejos-Evington-etal21} implicitly show that when $B$ is nuclear and $\mathcal Z$-stable, $B^\infty$ behaves enough like an ultraproduct of von Neumann algebras to give a classification of maps into $B^\infty$ by traces. The precise result we need was recorded in the short sequel \cite{CETW21} to \cite{Castillejos-Evington-etal21} for use in this paper.  The nuclearity of the codomain $B$ required to use \cite{Castillejos-Evington-etal21,CETW21} is the main reason for this hypothesis in our classification theorems  (the other use being to ensure that quasitraces are traces on $B$ by \cite{Haagerup14}).

\begin{theorem}[Classification of unital morphisms into the trace-kernel quotient; see Theorem~\ref{thm:B^inftyClassification}]\label{intro:classtraces}
	Let $A$ be a unital separable nuclear $C^*$-algebra and $B$ be a unital simple separable  nuclear $\mathcal Z$-stable $C^*$-algebra with $T(B) \neq \emptyset$.  Up to unitary equivalence, the unital $^*$-homomorphisms $A \rightarrow B^\infty$ are classified by traces.
\end{theorem}

We describe the trace-kernel extension in Section~\ref{sec:trace-kern-ext}, giving a little more detail about how one thinks of $B^\infty$ as being ``von Neumann-like''.

\subsubsection{Classification of lifts along the trace-kernel extension (Theorem~\ref{intro-lifts})} The strategy for the proof of Theorem~\ref{intro-lifts} builds on CS's work on the non-stable $KK$-theory of the trace-kernel ideal $J_B$ from \cite{Schafhauser17,Schafhauser18} which essentially proves Theorem~\ref{intro-lifts} under the additional hypotheses that $B$ is $\mathcal Q$-stable and monotracial with $K_1(B)=0$ (see Remark~\ref{rmk:LiftsContext}). Absorption --- in the spirit of Voiculescu's Weyl--von Neumann theorem --- is the fundamental tool both in \cite{Schafhauser17,Schafhauser18} and our arguments.

For a unital full embedding $\theta \colon A \rightarrow B^\infty$, we may take the pullback of the trace-kernel extension along $\theta$ to form a commutative diagram
\begin{equation}
\begin{tikzcd}\label{eq:pullback-intro}
\mathsf{e}\colon 0 \arrow{r} & J_B \arrow{r} \arrow[equals]{d} & E \arrow{r} \arrow{d} & A \arrow{r} \arrow{d}{\theta} & 0 \\
 \phantom{\mathsf{e}\colon} 0 \arrow{r} & J_B \arrow{r}{j_B} & B_\infty \arrow{r}{q_B} & B^\infty \arrow{r} & 0.
\end{tikzcd}
\end{equation}
The extension $\mathsf{e}$ in the top row of \eqref{eq:pullback-intro} defines a class in $\Ext(A,J_B)$. As $J_B$ is neither separable, nor even $\sigma$-unital, and it is also not stable, one must take some care with $KK$-theory. Fortunately it is \emph{separably stable} (every separable subalgebra of $J_B$ is contained in a stable separable subalgebra).  We will pretend in this introduction that $J_B$ is separable and stable, but the analysis that follows will really be performed on a suitable separable and stable subextension of the top row of \eqref{eq:pullback-intro}.

Given a $KK$-lift of $\theta$ in $KK(A, B_\infty)$ as in the statement of Theorem~\ref{intro-lifts}\ref{intro-lifts-1}, the class of $\mathsf{e}$ in $\mathrm{Ext}$ vanishes, so $\mathsf{e}$ stably splits, i.e., there exists a trivial extension $\mathsf{t}$ such that $\mathsf{e}\oplus\mathsf{t}$ splits.  This is where absorption first enters in Theorem~\ref{intro-lifts}\ref{it:intro-lifts.1}: if we can show $\mathsf{e}$ is absorbing, then $\mathsf{e}\cong\mathsf{e}\oplus\mathsf{t}$ so that $\mathsf{e}$ itself splits, providing a lift of $\theta$.  This strategy is an abstract version of CS's proof in \cite{Schafhauser17} of the quasidiagonality theorem of \cite{TWW}. Absorption is used again in correcting the $KK$-class of the lift.

Absorption is equally crucial in Theorem~\ref{intro-lifts}\ref{it:intro-lifts.2}, but appears in a slightly different guise. A pair of lifts $\phi,\psi\colon A\to B_\infty$ of $\theta$ gives rise to a class $[\phi,\psi]_{KK(A,J_B)}$ in $KK(A,J_B)$. Our objective is to show that $\phi$ and $\psi$ are unitarily equivalent if and only if $[\phi,\psi]_{KL(A,J_B)}$ vanishes.  This is achieved through a ``$\mathcal Z$-stable $KL$-uniqueness theorem''.   To set this up, note that the trace-kernel extension induces a canonical map $\mu\colon B_\infty \to\mathcal M(J_B)$, where $\mathcal M(J_B)$ is the multiplier algebra of $J_B$.
\begin{equation}\label{intro-kk-class}
\begin{tikzcd}
&&& A\ar["\psi",shift left=.2ex]{dl}\ar["\phi",shift right=.6ex,swap]{dl}\ar["\theta"]{d} \\
 0 \arrow{r} & J_B \arrow["j_B",swap]{r} & B_\infty\ar[d] \arrow["q_B",swap]{r}\ar["\mu"]{d} & B^\infty \arrow{r}\ar[d] & 0\\0\ar[r]&J_B\ar[equals]{u}\ar[r]&\mathcal M(J_B)\ar[r]&\mathcal Q(J_B)\ar[r]&0
\end{tikzcd}
\end{equation}

Uniqueness theorems for $KK$ and $KL$ intimately involve absorption.  For example, via the \emph{stable uniqueness theorems} of Dadarlat--Eilers and Lin from \cite{Dadarlat-Eilers02} and \cite{Lin02} (and continuing to pretend that $J_B$ is separable and stable), if $[\phi,\psi]_{KL(A,J_B)}$ vanishes, then $\mu\circ\phi$ and $\mu\circ\psi$ become properly approximately unitarily equivalent after adjoining any absorbing representation $\eta\colon A\to \mathcal M(J_B)$; i.e., $(\mu\circ\phi)\oplus\eta$ and $(\mu\circ\psi)\oplus\eta$ are approximately unitarily equivalent via unitaries in the minimal unitization of $J_B$.  If $\mu\circ\phi$ and $\mu\circ\psi$ are already known to be absorbing, then we can use these for $\eta$ in two applications of stable uniqueness to obtain proper approximate unitary equivalence of $(\mu\circ\phi)\oplus(\mu\circ\phi)$ and $(\mu\circ\psi)\oplus(\mu\circ\psi)$.  Assuming further that $B_\infty$ and $J_B$ absorb the UHF algebra $M_{2^\infty}$ tensorially (really one must work with ``separable $M_{2^\infty}$-stablity'') the two-fold amplification can be removed, giving rise to approximate unitary equivalence of $\phi$ and $\psi$ with unitaries from the minimal unitization of $J_B$, and hence in $B_\infty$, not just in $\mathcal M(J_B)$.  This ``UHF-stable $KL$-uniqueness theorem'' is an essential trick from \cite{Schafhauser18}, and with enough absorption achieves our objective whenever $B$ absorbs a UHF algebra of infinite type.

In this paper, some major new ingredients are $KK$- and $KL$-uniqueness theorems using $\mathcal Z$-stability in place of UHF-stability. We state the $KK$-version below; the $KL$-version is similar (characterizing proper \emph{approximate} unitary equivalence).  The strategy behind our $KK$-uniqueness theorem goes back to ideas in Dadarlat and Eilers' stable uniqueness theorem (\cite{Dadarlat-Eilers02}), as discussed further in Section~\ref{sec:KK-Uniqueness}.

\begin{theorem}[$\mathcal Z$-stable $KK$-uniqueness theorem; see Theorem~\ref{thm:KK-Uniqueness}\ref{KK-Uniqueness}]\label{intro-KK-unique} Let $A$ be a separable $C^*$-algebra and $I$ be a $\sigma$-unital and stable $C^*$-algebra. Suppose that $\phi,\psi\colon A\to \mathcal M(I)$ are absorbing and $\phi(a)-\psi(a)\in I$ for all $a\in A$. Then the following are equivalent:
\begin{enumerate}
    \item\label{kk-unique1}$[\phi,\psi]_{KK(A,I)}=0$
    \item \label{kk-unique2}$\phi\otimes 1_\Z,\psi\otimes1_\Z\colon A\to\mathcal M(I\otimes\Z)$ are properly asymptotically unitarily equivalent.
\end{enumerate}
\end{theorem}
The Dadarlat--Eilers stable uniqueness theorem gives the equivalence of~\ref{kk-unique1} and
\emph{\begin{enumerate}
\setcounter{enumi}{2}
   \item\label{kk-unique3} $\phi\oplus \eta,\psi\oplus\eta$ are properly asymptotically unitarily equivalent for all absorbing $\eta\colon A\to \mathcal M(I)$.
\end{enumerate}}

Using $\Z$-stability of $B$ to remove the tensor factor of $1_\Z$ (taking due care with issues around separability and tensorial absorption), and lifting the outcome of the $\Z$-stable $KL$-uniqueness theorem back to $B_\infty$, is exactly what we need to pass from $[\phi,\psi]_{KK(A,J_B)}=0$ to the desired unitary equivalence --- provided we can show that the maps $\mu\circ\phi$ and $\mu\circ\psi$ from \eqref{intro-kk-class} are suitably absorbing.

\subsubsection{Obtaining absorption}
Elliott and Kucerovsky's vast generalization (\cite{Elliott-Kucerovsky01}) of Voiculescu's theorem provides a criterion --- ``pure largeness'' --- that abstractly characterizes absorption in the presence of nuclearity.  The precise definition of pure largeness is not so important to our paper as it can be identified in terms of fullness under mild regularity conditions (much weaker than $\mathcal Z$-stability). The relevant result for us is obtained by combining \cite{Elliott-Kucerovsky01,Kucerovsky-Ng06,Ortega-Perera-etal12,Rordam04}: an extension of a separable nuclear $C^*$-algebra by a $\sigma$-unital stable $\mathcal Z$-stable ideal is absorbing if and only if the (forced) unitization of the extension is full.

Using the version of the Elliott--Kucerovsky theorem described above requires separable stability and $\mathcal Z$-stability of $J_B$.  The latter condition follows directly from the $\mathcal Z$-stability of $B$.  Hjelmborg and R{\o}rdam characterized stability for $\sigma$-unital $C^*$-algebras in \cite{Hjelmborg-Rordam98}; this also characterizes separable stability.  In this paper we use Cuntz semigroup methods to verify the Hjelmborg--R{\o}rdam conditions for $J_B$ from strict comparison (and hence from $\Z$-stability). This replaces the ad hoc notion of an admissible kernel CS introduced to obtain separable stability of $J_B$ (also via Hjelmborg--Rørdam) when $B=\mathcal Q$ (\cite{Schafhauser17}) and more generally, when $B$ is a unital simple $\mathcal Q$-stable AF algebra with unique trace (\cite{Schafhauser18}).

The final detail is a ``de-unitization'' trick.  Unital extensions can never be absorbing, as they do not absorb non-unital extensions. As in \cite{Schafhauser17,Schafhauser18}, we resolve this by taking the direct sum with the zero extension --- see Section~\ref{sect:unital-lifts}.  

It is tempting to try and replace the trace-kernel quotient $B^\infty$ with a von Neumann algebra $\M$, such as an ultrapower of the finite part of the bidual of $B$.  In this case, $\M$ would encode all the traces on $B$, and one has the classification of maps $A\to \mathcal M$ by traces directly from Connes' theorem, and Theorem~\ref{intro:classtraces} (and the work of \cite{Castillejos-Evington-etal21}) would not be needed.  However, the ideal in the ultrapower $B_\omega$ corresponding to the quotient onto $\mathcal M$ is not  separably stable, obstructing the route to absorption (see Remark~\ref{rem6.12}). This prevents us from lifting the classification back from $\M$ to $B_\omega$. Loosely speaking, we must work with $2$-norm estimates that are uniform over the traces in order to take advantage of strict comparison in $B$.

\subsubsection{Computing $KL$-theory of the trace-kernel ideal}\label{subsubsectKL}

The major remaining task is to compute the group $KL(A, J_B)$  by combining the universal multicoefficient theorem (UMCT) with total $K$-theory computations for the trace-kernel extension.

For any $C^*$-algebra $D$, the Kasparov product induces a natural map \begin{equation}\label{eq:UMCT-intro}
\begin{tikzcd}
	KL(A,D) \arrow{r} & \Hom_\Lambda(\totK(A), \totK(D)),
\end{tikzcd}
\end{equation}
where $\Hom_\Lambda(\totK(A), \totK(D))$ is the group of morphisms in total $K$-theory.   Dadarlat and Loring's UMCT shows that \eqref{eq:UMCT-intro} is an isomorphism when $A$ satisfies the UCT.  Applying this to the trace-kernel extension allows $KL(A,J_B)$ to be computed using the resulting exact sequence:
	\begin{equation}\label{eq:KKJ}
	\begin{tikzcd}
		0 \arrow{rr} &[-3ex] &[-28ex] \ker \mathrm{Hom}_\Lambda\big(\totK(A),\totK(j_B)\big) \arrow{r} 	\arrow[dl, phantom, ""{coordinate, name=Z}] &[-7ex] KL(A, J_B) \arrow[dll, rounded corners, to path={-- ([xshift=3ex]\tikztostart.east) |- (Z) [near end]\tikztonodes -| ([xshift=-3ex]\tikztotarget.west) -- (\tikztotarget)}] \\ & \Hom_\Lambda \big(\totK(A), \totK(B_\infty)\big) \arrow{rr} & & \Hom_\Lambda \big(\totK(A), \totK(B^\infty) \big).
	\end{tikzcd}
	\end{equation}

The first and fourth terms above can be simplified using the von Neumann algebra-like behavior of $B^\infty$:  $K_1(B^\infty)=0$, and the pairing map $\rho_{B^\infty}\colon K_0(B^\infty)\to\Aff T(B^\infty)$ is an isomorphism just as the $K_0$-group of a II$_1$ factor is identified with $\mathbb R$ using the trace.  Consequently, $\totK(B^\infty)$  and $\ker\totK(j_B)$ are concentrated in their $K_0$ and $K_1$ components respectively, so the last entry in \eqref{eq:KKJ} is isomorphic to $\mathrm{Hom}(K_0(A),\Aff T(B^\infty))$ (see Proposition~\ref{Prop:KK-trace-kernel-quotient} and \eqref{eq:KTheoryIsoagain}, noting that in the main body we identify $\Aff T(B_\infty)$ and $\Aff T(B^\infty)$ as the regularity hypothesis forces $\Aff T(q_B)$ to be an isomorphism).  This allows us to determine $KL(A,J_B)$ by computing the maps in \eqref{eq:KKJ} explicitly.

These computations can be cleanly interpreted in terms of a Hausdorffized unitary algebraic $K_1$-group developed by Thomsen.  For a $C^*$-algebra $D$, $\Ka(D)$ is defined as the quotient of the infinite unitary group of $D$ by the closure (in the inductive limit topology) of its commutator subgroup, and there is a natural surjection $\minusa_D\colon \Ka(D)\to K_1(D)$. Thomsen used the de la Harpe--Skandalis determinant to provide a natural sequence
\begin{equation}\label{eq:thomsen-intro}
    \begin{tikzcd}
        K_0(D)\arrow[r,"\rho_D"]&\Aff T(D)\arrow[r,"\Th_D"]&\Ka(D)\arrow[r,"\minusa_D"]&K_1(D)\arrow[r]&0
    \end{tikzcd}
\end{equation}
that is exact at $\Ka(D)$ and satisfies $\ker \Th_D = \overline{\im \rho_D}$.  Here, $\rho_D$ is the pairing between $K_0$ and traces, and the \emph{Thomsen map} $\Th_D$ can be constructed from an inverse of the de la Harpe--Skandalis determinant (see Section~\ref{SectAlgK1}, particularly \eqref{eq:detInverse}).  When we take $D=B_\infty$, a fragment of the six-term exact sequence of the trace-kernel extension corresponds to Thomsen's sequence as follows.

\begin{theorem}[Calculating $K_1(J_B)$; see Theorem~\ref{Thm:K1J}]\label{intro:calcKJB}
Let $B$ be a unital simple separable nuclear $\Z$-stable $C^*$-algebra with $T(B)\neq\emptyset$.  Then there is a natural isomorphism $\omega_B\colon K_1(J_B)\to \Ka(B_\infty)$ that (together with the computations $K_0(B^\infty)\cong \Aff T(B_\infty)$ and $K_1(B^\infty)=0$) gives the commutative diagram
\begin{equation}
\label{eq:introcalcKJB}
    \begin{tikzcd}
K_0(B_\infty)\arrow[equal,d]\arrow[r,"K_0(q_B)"]&K_0(B^\infty)\arrow[d,"\cong"]\arrow[r,"\partial"]&K_1(J_B)\arrow[d,"\omega_B"]\arrow[r,"K_1(j_B)"]&K_1(B_\infty)\arrow[d,equal] \arrow[r] & 0\phantom{.} \\
        K_0(B_\infty)\arrow[r, "\rho_{B_\infty}"]&\Aff T(B_\infty)\arrow[r,"\Th_{B_\infty}"]&\Ka(B_\infty)\arrow[r,"\minusa_{B_\infty}"]&K_1(B_\infty)\arrow[r]&0.
    \end{tikzcd}
\end{equation}
\end{theorem}

Putting this all together, and keeping track of the maps explicitly, allows us to compute $KL(A,J_B)$ in terms of the enriched invariant $\inv(\,\cdot\,)$.

\begin{theorem}[Calculating $KL(A, J_B)$; see Theorem~\ref{thm:calcKL}]\label{intro:calcKL}
Let $A$ be a unital separable $C^*$-algebra satisfying the UCT and let $B$ be a unital simple separable nuclear $\mathcal Z$-stable $C^*$-algebra with $T(B) \neq \emptyset$.  Then the sequence in \eqref{eq:KKJ} is exact.  Moreover:

\begin{enumerate}
    \item The third arrow  in \eqref{eq:KKJ} is a group homomorphism from $KL(A,J_B)$ onto the set of $\inv$-morphisms $\inv(A)\to\inv(B_\infty)$ that vanish on the trace.\label{intro:calcKL1}
    \item \label{intro:calcKL.2} Given a full $^*$-homomorphism $\psi\colon A \to B_\infty$, the kernel of the above homomorphism identifies with the set of $\inv$-morphisms $\inv(A) \to \inv(B_\infty)$ that agree with $\psi$ on total $K$-theory and traces.\label{intro:calcKL2}
\end{enumerate}
\end{theorem}

Compatibility with $\zeta^{(n)}$ is essential in establishing surjectivity in Theorem~\ref{intro:calcKL}\ref{intro:calcKL2}.  In slightly more detail, given a $\inv$-morphism $(\underline{\alpha},\beta,\gamma)\colon \inv(A) \rightarrow \inv(B_\infty)$ that agrees with $\psi$ on total $K$-theory and traces, $\beta-\Ka(\psi)$ gives rise to a group morphism $r \colon K_1(A) \rightarrow \ker K_1(j_B)$ via \eqref{eq:introcalcKJB}.  In order to extend this map to total $K$-theory, as in the first term of \eqref{eq:KKJ}, $r$ must vanish on torsion.  This is equivalent to compatibility with $\zeta^{(n)}$.
See Sections~\ref{sec:new-map} and~\ref{Sect:RotationMap}.

Theorem~\ref{Main3} follows by combining the existence and uniqueness parts of Theorems~\ref{intro-lifts}, \ref{intro:classtraces},  and~\ref{intro:calcKL}. 

\subsubsection{The role of the UCT and $\Z$-stability}\label{SS:roleUCTZ}

The UCT appears twice in the proof.  Firstly, it is used to compute $KL(A, J_B)$ in terms of $KK(A, B_\infty)$ and traces, via the UMCT, as described in Section~\ref{subsubsectKL} above.
Secondly, the UCT is used to show that the $\Ext(A,J_B)$ obstruction in Theorem~\ref{intro-lifts}\ref{intro-lifts-1} vanishes, by producing a required $KK$-lifting.  This second  application of the UCT is analogous to the role it plays in CS's approach (\cite{Schafhauser17}) to the quasidiagonality theorem (see Remark~\ref{r:qdlift}). Assuming the existence of a ``trace-compatible'' element of $KK(A, B_\infty)$ does not appear to be sufficient for this step, as the UCT would still be needed to show this element lifts a given morphism $A \rightarrow B^\infty$ in $KK(A, B_\infty)$.

In the classification of UCT Kirchberg algebras by $K$-theory, the UCT is only needed to lift an isomorphism at the level of $K$-theory to a $KK$-equivalence.  The more subtle uses of the UCT in the stably finite setting described above mean that classification currently requires the UCT and cannot (yet) be formulated in terms of a combination of $KK$-theory and traces.

In contrast, regularity --- i.e., $\Z$-stability of the codomain $B$ --- is ever-present in the argument. Jiang--Su stability is a rich hypothesis, with many ramifications. In the main body of the paper, we will typically state results with $\Z$-stability as the hypothesis, even if only a consequence of $\Z$-stability is needed.  Here, we describe which aspects of $\Z$-stability are used at different points in the argument.

Two consequences of $\Z$-stability are pertinent: unital simple finite $\Z$-stable $C^*$-algebras have \emph{strict comparison} and \emph{uniform property $\Gamma$}. Both are analogs of properties of von Neumann algebras. For a unital simple separable nuclear $C^*$-algebra $B$ with $T(B)\neq \emptyset$, $\Z$-stability is now known to be equivalent to the combination of strict comparison and uniform property $\Gamma$ (using \cite{Rordam04} and \cite{CETW22}, which builds on \cite{Matui-Sato12,Sato12,Toms-White-Winter15,Kirchberg-Rordam14,Ozawa13}).  Strict comparison is a Cuntz semigroup condition, which can be used to obtain properties of the trace-kernel ideal --- in particular, the route to absorption really runs through comparison --- and to control the traces on $B_\infty$. In the main body of the paper, we use separable $\Z$-stability of $J_B$ (inherited from $\Z$-stability of $B$) as one of the two ingredients to access absorption via Theorem~\ref{thm:absorbing-z-stable}, but this is a matter of convenience. We could have worked with the weaker condition of unperforation of the Cuntz semigroup of $J_B$ and obtained this from strict comparison of $B$. Likewise, we use $\Z$-stability to control traces on $B_\infty$  via Ozawa's theorem (\cite{Ozawa13}), but we could have used the harder \cite[Theorem~1.2]{Ng-Robert16} to obtain this from strict comparison. Similarly, although it doesn't appear explicitly in this paper, it is really uniform property $\Gamma$ rather than the full force of $\Z$-stability that gives $B^\infty$ its von Neumann-like behavior.  We also use that unital simple finite $\Z$-stable $C^*$-algebras have stable rank one as a technical convenience (to access cancellation) in the proof.

In contrast, our proof of the $\Z$-stable $KK$- and $KL$-uniqueness theorems use the full force of $\Z$-stability, not its consequences.  We use both an explicit construction of $\Z$ (in the proof of Jiang's theorem; see Section~\ref{SSK1Inject}) and strong self-absorption.  This is the only place where we directly use strong self-absorption of $\Z$.

\subsubsection{Comparison with classification via tracial approximations}

The $\Z$-stable $KK$- and $KL$-uniqueness theorem provides a major shortcut. It allows us to bypass Winter's localization techniques, and obtain a classification of morphisms into $\Z$-stable codomains up to approximate unitary equivalence directly.  A significant amount of the work in Gong, Lin, and Niu's argument (\cite{GLN-part1,GLN-part2}) is obtaining the  strong asymptotic classification required for localization. As we do not need asymptotic classification results, we do not prove any in this paper (with some extra work and invariants, it is possible to do this abstractly, and we will return to this point in future) allowing us to work with simpler invariants throughout.  

Another fundamental difference is that our approach does not require explicit models of classifiable $C^*$-algebras nor internal approximation structure such as inductive limits or tracial approximations. Indeed, the only place where internal structure arises in the argument is through Connes' theorem that injective von Neumann algebras (in this case, the bidual of the nuclear domain $C^*$-algebras) are hyperfinite, and so are inductive limits of finite dimensional algebras. This is crucial in the folklore result that maps from separable nuclear $C^*$-algebras into finite von Neumann algebras are classified (up to strong$^*$-approximate unitary equivalence) by traces. The classification into the trace-kernel quotient (Theorem~\ref{intro:classtraces}) is obtained by gluing this observation uniformly across the traces.
 
The absence of tracial approximations allows us to obtain our embedding results in vast generality: in this paper, a separable nuclear $C^*$-algebra satisfying the UCT is allowed as a domain, as opposed to the simplicity requirement found in \cite{Gong-Lin-etal23} (see the discussion in Remark~\ref{rmk:MatuiClassification}). In future work we will obtain embedding results for exact domains and quite general codomains, with the simplicity and nuclearity requirements at the level of the map.  With our approach, internal $C^*$-structure --- such as arising as an ASH algebra or enjoying tracial approximations --- can be viewed as a consequence of $C^*$-classification, rather than an ingredient.

Our proof essentially contains a copy of CS's proof (\cite{Schafhauser17}) of the quasidiagonality theorem, and so we no longer need this result nor the UCT-powered stable uniqueness across the interval arguments used to obtain tracial approximations \cite{EGLN} (and in the original proof of the quasidiagonality theorem, \cite{TWW}). These are replaced by the abstract lifting result Theorem~\ref{intro-lifts}\ref{intro-lifts-1} in combination with the UCT to verify the hypotheses. The use of the UCT in \cite{Schafhauser17} is more direct than that in \cite{TWW} and is similar to its use here.

Our work also has points of contact with the work of Gong, Lin, and Niu, and many others in the long history of $C^*$-classification.  Classifying into the trace-kernel quotient is spiritually connected to classifying tracially large maps into $C^*$-algebras with good tracial approximations, and our use of the Elliott--Kucerovsky absorption theorem plays a role analogous to the use of stable uniqueness theorems in \cite{GLN-part1,GLN-part2}. As with \cite{GLN-part1,GLN-part2}, our existence result is proven in two stages.  We first prove the existence of an approximate embedding with prescribed tracial information (through Theorem~\ref{intro-lifts}\ref{intro-lifts-1}), and then  modify the $K$-theoretic behavior (in Theorem~\ref{intro-lifts}\ref{intro-lifts-2}).  This two-step process dates back to Lin's classification of TAF algebras (\cite{Lin01}), which is made explicit in Dadarlat's account (\cite{Dadarlat04}).

Finally, while we do not use any of the results from the major papers establishing the 2015 classification theorem during our work, we carefully studied the invariants and rotation maps from \cite{GLN-preprint}. This led us to expect that a suitable combination of total $K$-theory, traces, and Hausdorffized unitary algebraic $K_1$ would provide the classification invariant found in Theorem~\ref{Main2}.

\subsection{Structure of the paper, notation, and references}  We hope to clarify the new ingredients of the invariants used in our classification in the next two sections and review their essential properties (with certain facts about total $K$-theory being collated in Appendix~\ref{appendix:totalappendix} for completeness).  Section~\ref{s:totalinvariant} is devoted to the maps $\zeta^{(n)}\colon K_0(\,\cdot\,;\Zn{n})\to \Ka(\,\cdot\,)$, and formalizes the total invariant $\inv$. Section~\ref{Sect:Z} is concerned with the Jiang--Su algebra $\mathcal Z$. We give a short self-contained proof of Jiang's unpublished theorem that $\Z$-stable $C^*$-algebras are $K_1$-injective. This result is a vital component of the $\Z$-stable $KK$- and $KL$-uniqueness theorems. We also collect results on strict comparison and the Cuntz semigroup in Section~\ref{Sect:Z}, which are used later in absorption results.

Absorption and $KK$-theory strike back in Section~\ref{sec:KK-classification}.  Since we wish to use $KK(A,J_B)$ and $KL(A,J_B)$ and the trace-kernel ideal is not $\sigma$-unital, we will separabilize and treat these $KK$-groups as limits of the $KK$- and $KL$-groups corresponding to separable subalgebras $J$ of $J_B$.  This is explained in Section~\ref{subsec:elem-of-kk}, with the proofs of facts that are standard in the $\sigma$-unital case deferred to Appendix~\ref{sec:kkappendix}.  The rest of Section~\ref{sec:KK-classification} is devoted to absorption and the proof of the $\Z$-stable $KK$- and $KL$-uniqueness theorems. 

In Section~\ref{sec:trace-kern-ext}, we describe the trace-kernel extension, and compute the $K$-theory of the trace-kernel quotient.  This enables us to prove Theorem~\ref{intro:calcKJB}.  Section~\ref{SSTKSepStab}  shows that the trace-kernel extension $J_B$ is separably stable for use in absorption.  Section~\ref{sec:class-lifts-trace} analyzes lifts along the trace-kernel extension, proving Theorem~\ref{intro-lifts} (which we deduce from non-unital versions).

The final act of the paper sees the return of the UCT and its cousin, the UMCT. We review these results in Section~\ref{sec:uct-rot}, and compute $KL(A,J_B)$ to prove Theorem~\ref{intro:calcKL}. In particular, we do not use that our domain algebras $A$ satisfy the UCT in any results before Section~\ref{sec:uct-rot}.  Finally we put everything together in Section~\ref{sec:main-results}, which proves the classification of full approximate $^*$-homomorphisms (Theorem~\ref{Main3}) and deduces Theorems~\ref{Main} and~\ref{Main2}.

\subsubsection{Notation}

Throughout the paper we endeavor to consistently use $A$ for $C^*$-algebras appearing as the domain of the classification of unital embeddings (Theorem~\ref{Main2}). When $A$ is used in other results (for example, appearing as the first variable in $KK$-theory) it is because we envisage applying these results to the domains of Theorem~\ref{Main2}. Likewise, we use $B$ for the codomains of the classification of embeddings, with the corresponding notation for the trace-kernel extension in \eqref{eq:trace-kernel-intro}.  For that reason we typically write $D$ or $E$ for a general $C^*$-algebra, and $I$ or $J$ for a general ideal (often the second variable in $KK$-theory). We will write $\mathcal M(I)$ for the multiplier algebra of $I$, and $\mathcal Q(I)$ for its corona.  The symbols $\alpha,\beta,\gamma$ will be reserved for the components of a morphism between invariants --- see Section~\ref{ss:totalinv}. We also use $\kappa$ for a $KK$-element, $\lambda$ for a $KL$-element, and typically use $\phi,\psi,\theta$ for $^*$-homomorphisms between $C^*$-algebras ($\kappa^{(m,n)}_*$ also appears briefly as one of the Bockstein maps; no ambiguity is possible). For $D \subseteq E$ an inclusion of $C^*$-algebras, we write $\iota_{D\subseteq E}\colon D \to E$ for the inclusion map, and $D^\dagger$ denotes the forced unitization of $D$ (when $D$ is already unital, a new unit is added).

\subsubsection{References underpinning the argument}\label{section:referencesneeded}

We have attempted to provide both context throughout, and also to track down the intellectual origins of the ideas we use. However, from a formal point of view, the vast majority of the (many) papers we cite do not form part of our argument for the stably finite parts of Theorems~\ref{Main} and~\ref{Main2}, which are relatively self-contained.  Concretely, in addition to material in textbooks and short stand-alone lemmas which will be found straightforward by experts (this includes every instance where we appeal to one of \cite{Bosa-Brown-etal15,Schafhauser18,TWW} for a result rather than for context), we only rely on the following works:
\begin{itemize}
    \item 
\cite{Harpe-Skandalis84,Thomsen95} for properties of algebraic $K_1$ and the de la Harpe--Skandalis determinant;
\item
\cite{Dadarlat-Loring96,Schochet84} for total $K$-theory;
\item \cite{Jiang-Su99,Ozawa13,Rordam04,Rordam-Winter10,Toms-Winter07} for the Jiang--Su algebra and consequences of $\Z$-stability;
\item \cite{Haagerup14} for the fact that quasitraces on exact $C^*$-algebras are traces (we only need this for the simple nuclear $\mathcal Z$-stable codomain algebras $B$; there is a short proof of this fact for $C^*$-algebras with finite nuclear dimension \cite{Brown-Winter}, although we don't know how to do this from $\mathcal Z$-stability directly);
\item\cite{Hjelmborg-Rordam98,Ortega-Perera-etal12,Rordam92} for the Cuntz semigroup, strict comparison, and its use in obtaining stability, and the corona factorization property;
\item \cite{Dadarlat-Eilers01,Dadarlat-Eilers02,Elliott-Kucerovsky01,Gabe16,Kucerovsky-Ng06} for the absorption theorem; \cite{Dadarlat00b,Dadarlat-Eilers01} for ingredients in the $KK$- and $KL$-uniqueness theorems (including a convenient form of Paschke duality); 
\item \cite{CETW21,Castillejos-Evington-etal21} for the von Neumann-like behavior of $B^\infty$ (i.e., the complemented partition of unity technique from \cite[Sections 2,3]{Castillejos-Evington-etal21}) and Theorem~\ref{intro-lifts} (the full force of the finite nuclear dimension theorem from \cite{Castillejos-Evington-etal21} is not required);
\item \cite{Dadarlat05,Dadarlat-Loring96} for the universal multicoefficient theorem, and aspects of $KL$;
\item \cite{Gabe20} for intertwining via reparameterizations.
\end{itemize}

\subsection{Addendum}\label{addendum} Since the first version of this paper, $KK$-uniqueness has seen major progress.  Firstly, Farah and Szab\'o (building on Farah's \cite{Farah23}) established $\mathcal Z$-saturation for the corona algebra of a $\mathcal Z$-stable algebra.
Combining this with our approach to Theorems \ref{intro-KK-unique} and \ref{thm:KK-Uniqueness}\ref{KK-Uniqueness}, their result improves the $\Z$-stable $KK$-uniqueness theorem to give a clean statement assuming the ideal $I$ is $\mathcal Z$-stable (\cite[Theorem 5.5]{FarahSzabo24}). Using this in place of Theorem \ref{thm:KK-Uniqueness} makes things a little bit tidier, but does not result in major simplifications.

Secondly, in \cite{Szabo26}, Szab\'o obtained a definitive $KK$-uniqueness theorem, answering Question \ref{q:kk-uniqueness} below positively. Following this breakthrough, we believe that the best abstract approach to the unital classification theorem is now to use \cite[Theorem 2.10]{Szabo26} in place of Theorem \ref{intro-KK-unique}. Doing so enables one to use the combination of strict comparison and the existence of complemented partitions of unity as codomain hypotheses in place of $\Z$-stability. As such, it will be possible to avoid any use of strong self-absorption of $\Z$ (or indeed an explicit construction of $\Z$) in the unital classification theorem (very much in the spirit of the recent work of Bouwen and JG \cite{Bouwen-Gabe24} in the purely infinite case).

Additionally, after we first arXived this paper, together with Castillejos and Evington, we have extended the gluing technique from \cite{Castillejos-Evington-etal21} to apply without requiring nuclearity of the codomain in the preprint \cite{CCEGSTW}. This will be particularly relevant in combination with Szab\'o's $KK$-uniqueness theorem. We will set out the effect of all the advances above on the classification of embeddings theorem in the next paper in this series.

We have not modified the main body of the paper to take account of these advances; the argument given here still uses our $\Z$-stable $KK$-uniqueness theorem, not least because the material in Section \ref{sec:KK-Uniqueness} is the primary motivation for \cite{Szabo26}. However, we have removed our original Section 5.5, as the now obsolete $\Z$-stable $KK_{\mathrm{nuc}}$- and $KL_{\mathrm{nuc}}$-uniqueness theorems there were designed for further work.

In another direction, CS has built on the ideas in this paper, to take the first steps towards a classification in terms of $KK$-theory paired with traces (\cite{Schafhauser24}). One result is that finite strongly self-absorbing $C^*$-algebras are isomorphic if they are unitally $KK$-equivalent (\cite[Corollary C]{Schafhauser24}).

\subsection{Acknowledgements}

We would like to thank the organizers and funding agencies of those conferences where we have undertaken work on this paper: BIRS (2017), NCGOA in M\"unster (2018), Texas A\&M (2018), Oberwolfach (2019, 2022), and Sde Boker (2022). We also thank the TCU Department of Mathematics and the University of Ottawa for hosting some of the authors during the early stages of this work.

Throughout this project we have benefited extensively from the insight and expertise of the community over multiple conversations and feedback following presentations.  We thank you all.  We would like to especially thank:
\begin{itemize}
\item Guihua Gong, Huaxin Lin, and Zhuang Niu for valuable conversations on their work, which have been consequential in the development of this paper;
\item  George Elliott, who was present in many presentations of this work --- including a semester-long lecture series in the Fields given by CS --- and always had something illuminating to say;
\item  Bruce Blackadar, Joachim Cuntz, Marius Dadarlat, S\o{}ren Eilers, George Elliott, Huaxin Lin, Mikael R\o{}rdam, Klaus Thomsen, Stefaan Vaes, and Wilhelm Winter for their helpful comments on the introduction, especially regarding the history;
\item  Mikkel Munkholm, Robert Neagu, and Pawel Sarkowicz for their careful proofreading;
\item everyone who provided feedback on the arXiv version of this paper, including Shanshan Hua, Yifan Zhang, Ben Bouwen, Karl Bernhard Nielsen; 
\item the referee for careful proofreading and valuable feedback; and
\item the handling editor for all their work managing this paper.
\end{itemize}

Finally, we thank the community for their collective patience. The second paper in this series on non-unital classification will follow more promptly.


\section{Invariants for unital $C^*$-algebras}
\label{s:invariant}
In this section, we review the invariants used for the classification of unital $C^*$-algebras and their $^*$-homomorphisms. 

\subsection{$K$-theory and traces}\label{sec:Elliott-invariant}

Elliott's invariant for the classification of simple nuclear
$C^*$-algebras consists of ordered $K$-theory and traces. Write
$K_*=(K_0,K_1)$ for the operator algebraic $K$-theory functor from
$C^*$-algebras to pairs of abelian groups (see \cite{Rordam-Larsen-etal00} for an overview). When $A$ is a unital
$C^*$-algebra, the element $[1_A]_0$ represents a distinguished
element in $K_0(A)$. The positive elements $K_0(A)_+$ of $K_0(A)$ are
precisely those classes that can be realized by a projection in
$M_n(A)$ for some $n\geq 1$, so that $(K_0(A), K_0(A)_+)$ is an
ordered group whenever $A$ is stably finite and unital (see
\cite[Proposition~6.3.3]{Blackadar98}, for example).

Since our focus is unital $C^*$-algebras, by a
\emph{trace} we always mean a tracial state.  The set of all traces on
a unital $C^*$-algebra $A$ is denoted by $T(A)$. This is convex, and
unitality ensures that $T(A)$ is compact in the relative
weak$^*$-topology. In fact, though this plays no role in our approach to classification, $T(A)$ is a Choquet simplex (see
  \cite[Theorem~3.1.18]{Sakai98}, for example) that is metrizable when
  $A$ is separable.  Since $A\mapsto T(A)$ is
contravariant, we prefer to work with the
function system $\Aff T(A)$ of continuous real-valued affine functions
on $T(A)$ rather than with $T(A)$ itself. The duality between compact
convex subsets of locally convex spaces and Archimedean order unit
spaces (which are concretely realized as function systems) is known as \emph{Kadison duality}, and was set out in \cite{Kadison51}.  This works as follows: every
continuous affine map $\gamma^*\colon T(B)\rightarrow T(A)$ induces a
unital positive linear map $\gamma\colon\Aff T(A)\rightarrow \Aff
T(B)$ given by $\gamma(f)(\tau)\coloneqq f(\gamma^*(\tau))$ for $f\in \Aff
T(A)$ and $\tau\in T(B)$. Conversely, $T(A)$ is naturally recovered as
the states on $\Aff T(A)$ (see \cite[Theorem 7.1]{Goodearl86}). That is, $T(A)$ is identified with order preserving group homomorphisms $\Aff T(A)\to\mathbb R$ sending $1$ to $1$.  Therefore, unital positive linear
maps $\gamma\colon\Aff T(A)\rightarrow \Aff T(B)$ give rise to
continuous affine maps $\gamma^*\colon T(B)\rightarrow T(A)$.

The relationship between the function system $\Aff T(A)$ and the space $A_{\mathrm{sa}}$ of self-adjoint elements of $A$ goes back to the work of Cuntz and Pedersen (\cite{Cuntz-Pedersen79}). Every $a\in A_{\mathrm{sa}}$ induces a continuous affine function $\hat{a}\colon T(A)\to\mathbb R$ by $\hat{a}(\tau)=\tau(a)$ for $\tau\in T(A)$. Write $[a,b]\coloneqq ab-ba$ and $[A,A]$ for the span of the set $\{ ab - ba : a, b \in A\}$, so a trace is precisely a state that vanishes on $[A,A]$. We need the following result, which is often attributed to \cite{Cuntz-Pedersen79} but does not appear
there explicitly (a strategy for deducing it from \cite{Cuntz-Pedersen79} is indicated at the beginning of the proof of \cite[Theorem 3.1]{Thomsen95}). We give a short proof for completeness (alternatively, one can deduce the result from \cite[Lemma~6.2(iii)]{Kirchberg-Rordam14} or \cite[Theorem 5]{Ozawa13}, both of which factor through the enveloping von Neumann algebra). 

\begin{proposition}
  \label{prop:cuntz-pedersen}
  Suppose $A$ is a unital $C^*$-algebra.
  \begin{enumerate}
  \item\label{cp1} If $f\in \Aff T(A)$, then there exists $a\in A_{\mathrm{sa}}$
    such that $f=\hat{a}$ (i.e., $f(\tau)=\tau(a)$ for all $\tau\in
    T(A)$), and $a$ is unique modulo $\overline{ [A,A] }$.  Moreover,
    given $\epsilon > 0$, we may choose $a$ with $\|a\| <\|f\| +
    \epsilon$.
  \item\label{cp2} $[A, A] \cap A_{\mathrm{sa}}$ is spanned by $\{ [x, x^*] : x\in A\}$.
  \end{enumerate}
\end{proposition}

\begin{proof}
  Let $T_\mathbb C(A)\coloneqq \{ \tau \in A^* : \tau|_{[A, A]}=0\}$ be the space of bounded tracial functionals on $A$, so that there is a canonical isometric isomorphism $T_\mathbb C(A) \rightarrow (A/ \overline{ [A,A] })^*$ preserving the weak$^*$ topology.
        
  Given a continuous affine function $f\colon T(A)\to \mathbb{R}$, we
  may use the Jordan decomposition to extend $f$ to a bounded
  self-adjoint linear functional $\tilde{f}\in (T_\mathbb{C}(A))^* \cong (A/ \overline{ [A,A]
  })^{**}$.
  Note that $\tilde{f}$ is also weak$^*$-continuous: indeed, the intersection of $\ker\tilde{f}$ with any multiple of the closed unit ball in $T_{\mathbb{C}}(A)$ is weak$^*$-closed (using weak$^*$-compactness, one may replace a bounded weak$^*$ convergent net in $T_\mathbb C(A)$ with one whose Jordan decompositions also converge), so that $\ker\tilde{f}$ is weak$^*$-closed by the Krein--Smulian
  Theorem (\cite[Corollary~V.12.6]{Conway90}).
  
 As $\tilde{f}$ is weak$^*$-continuous, there is
  an $a \in A$ with $\tilde{f}(\tau) = \tau(a)$ for all $\tau \in T(A)$.  In
  fact, because $\tilde{f}$ is self-adjoint, we may choose $a$ to be
  self-adjoint. Uniqueness in \ref{cp1} follows from $(T_\mathbb{C}(A))^*\cong (A/ \overline{ [A,A] })^{**}$.
 
 The claim in \ref{cp2} follows from the
  polarization identity applied to the sesquilinear form $(a, b)
  \mapsto [a, b^*]$.
\end{proof}

The relationship between $K_0$ and traces is encoded in the natural \emph{pairing map} $\rho_A\colon K_0(A) \to
\mathrm{Aff}\, T(A)$. For $n\in\mathbb N$ and $\tau\in T(A)$, let $\tau_n$ be the canonical extension of $\tau$ to a tracial functional on $M_n(A)$ with $\tau_n(1_{M_n(A)})=n$, and define
\begin{equation}\label{DefPairingMap}
  \rho_A([p]_0-[q]_0)(\tau) \coloneqq \tau_n(p-q)
\end{equation}
for all $\tau\in T(A)$ and all projections $p,q\in M_n(A)$.

When $A$ is unital and exact, every state on
$(K_0(A),K_0(A)_+,[1_A]_0)$ arises from a trace (i.e., every order-preserving homomorphism $\phi:K_0(A)\to\mathbb R$ with $\phi([1_A]_0)=1$ is of the form
$\phi(x)= \rho_A(x)(\tau)$ for some $\tau\in T(A)$). This combines
 Haagerup's celebrated \cite{Haagerup14} with \cite{Blackadar-Rordam92}. In particular,
unital nuclear stably finite $C^*$-algebras always have traces. 

Combining all of these ingredients gives the Elliott invariant of a unital stably finite $C^*$-algebra. By using $\Aff T(A)$ to record the tracial component of the invariant, $\Ell(\cdot)$ is a covariant functor.

\begin{definition}\label{DefEll}
  If $A$ is a unital stably finite $C^*$-algebra, then its
  \emph{Elliott invariant, $\Ell(A)$}, is
  \begin{equation}
    \Ell(A)\coloneqq
    ( K_0(A), K_0(A)_+, [1_A]_0, K_1(A), \Aff T(A), \rho_A ).
\end{equation}
\end{definition}

Historically the Elliott invariant emerged from Elliott's classification of AF-algebras, and so includes the order on $K_0$. For reasons explained below, we prefer to work with a weaker
invariant, $KT_u$ (defined below), that does not keep track of the order on $K_0$. While we are primarily interested in the case of $C^*$-algebras
with traces, the definition also makes sense when
$T(A)=\emptyset$. In this case, $\Aff T(A)=0$, and $KT_u(A)$ reduces
to $K_*(A)$ together with the position of the unit --- which is
precisely the information used in the classification of unital
Kirchberg algebras satisfying the UCT by Kirchberg (\cite{Kirchberg95b})
and by Phillips (\cite{Phillips00}).

\begin{definition}
  \label{def:KTu}
  Let $KT_u$ be the functor on the category of unital $C^*$-algebras
  with unital $^*$-homomorphisms that assigns to a unital
  $C^*$-algebra $A$ the quadruple
  \begin{equation}
    KT_u(A)=(K_*(A),[1_A]_0,\Aff T(A),\rho_A).
  \end{equation}
  A morphism
  $(\alpha_*,\gamma)\colon KT_u(A)\to KT_u(B)$ consists of a pair
  $\alpha_*=(\alpha_0,\alpha_1)$ of homomorphisms $\alpha_i\colon
  K_i(A)\rightarrow K_i(B)$ (for $i=0,1$) with
  $\alpha_0([1_A]_0)=[1_B]_0$, and a positive unital linear map
  $\gamma\colon \Aff T(A)\rightarrow \Aff T(B)$ such that
    \begin{equation}\label{Def:KTU.Diag}
      \begin{tikzcd}
        K_0(A) \ar[r, "\rho_A"] \ar[d, "\alpha_0"] &
        \Aff T(A)\ar[d,"\gamma"] \\
        K_0(B) \ar[r, "\rho_B"] &
        \Aff T(B)       \end{tikzcd}
    \end{equation}
commutes. Formally, the objects of the target category of $KT_u$ are
      quadruples $((G_0,G_1),e,X,\rho)$, where $G_0$ and $G_1$ are abelian
      groups with $e\in G_0$, $X$ is an Archimedian order unit space
      (the order unit is part of the data), and $\rho\colon
      G_0\rightarrow X$ a group homomorphism mapping $e$ to the order unit. Morphisms between general objects are pairs $(\alpha_*,\gamma)$ satisfying the analogous commutative diagram to \eqref{Def:KTU.Diag}.
\end{definition}

In the case where $A$ is exact and has real rank zero, $T(A)$ can be
recovered as states on the ordered group $(K_0(A), K_0(A)_+, [1_A]_0)$
(see \cite[Theorem 1.1.11 and Proposition 1.1.12]{Rordam02}), so
$K_0(A)_+$ can be used in place of $\Aff T(A)$ and $\rho_A$.  Outside
this setting, traces are needed for classification, and the pairing map, which first arose explicitly
in the invariant in \cite{Elliott93a}, carries data that cannot always be recovered from $T(A)$ and $(K_0(A), K_0(A)_+)$ (see the example on page 29 of \cite{Rordam02}).  On the other hand, for
unital simple nuclear finite $\Z$-stable $C^*$-algebras, the order on
$K_0(A)$ is induced by the pairing map, (this is a
  direct consequence of  Proposition~\ref{p:ZStableUnperforated}\ref{p:ZStableUnperforated.2} given below, and due to R\o{}rdam from \cite{Rordam04}), and so $KT_u$ and the
Elliott invariant carry exactly the same information for these algebras.
  
  In
general, the order on $K_0(A)$ may fail to be weakly unperforated, in
which case, it cannot arise from a pairing map. This phenomenon was
first exhibited for a simple nuclear stably finite $C^*$-algebra by
Villadsen in \cite{Villadsen98}.  We prefer to use the invariant
$KT_u$ for two reasons: to emphasize that our proof makes no explicit
reference to the order structure on $K_0(A)$ and because
$KT_u(A)\cong KT_u(A\otimes\mathcal Z)$ canonically for any $C^*$-algebra $A$  (see Proposition~\ref{prop:ZKKequiv} below)
 whereas, for unital simple $A$, $\Ell(A\otimes \Z)\cong \Ell(A)$ if and only if
$(K_0(A),K_0(A)_+)$ is weakly unperforated
(\cite[Theorem~1]{Gong-Jiang-etal00}).

 Recall that a trace $\tau$ is \emph{faithful} if $\tau(a^*a)=0$ implies $a=0$. When $B$ is simple,  an embedding $\theta\colon A\hookrightarrow B$ and a trace $\tau\in T(B)$, gives rise to a faithful trace $\tau\circ\theta$ on $A$.  This gives a necessary condition for a map between invariants to arise from an
embedding into a simple $C^*$-algebra.

\begin{definition}
  Let $A$ and $B$ be unital $C^*$-algebras. We say that
  $(\alpha_*,\gamma)\colon{}KT_u(A)\rightarrow KT_u(B)$ is
  \emph{faithful} if the induced map $\gamma^*\colon{}T(B)\rightarrow
  T(A)$ (given by Kadison duality) satisfies that $\gamma^*(\tau)$ is
  faithful for all $\tau\in T(B)$. This can be phrased
    directly in terms of $\gamma$ as follows: if $f\in \Aff T(A)$ is
    strictly positive on faithful traces on $A$, then $\gamma(f)$ is
    strictly positive (as if there is $a\in A_+\setminus\{0\}$ with $\gamma^*(\tau)(a)=0$, then $\hat{a}\in \Aff T(A)$ is strictly positive on faithful traces but
   $\gamma(\hat{a}) \in \Aff T(B)$ is not).  

\end{definition}

\begin{remark}
\label{rmk:RangeOfInvariant}
As our approach to classification does not require comparing
$C^*$-algebras with concrete models, the range of the invariant $KT_u$
does not play a role in our proof.  However, 
(in contrast to the Elliott invariant), the range of $KT_u$ on unital $C^*$-algebras is
completely understood. When $T(A)=\emptyset$, $K_0(A)$ and $K_1(A)$ can be arbitrary countable abelian groups, $[1_A]_0$ can be any element of $K_0(A)$, and the pairing map can be disregarded since there are no traces. When $T(A)\neq \emptyset$, $K_1(A)$ can be an arbitrary countable abelian group, $K_0(A)$ is any countable abelian group with a specified element $[1_A]_0$ of infinite order, $T(A)$ can be an arbitrary metrizable Choquet simplex, and the pairing map $\rho_A \colon K_0(A) \rightarrow \Aff T(A)$ can be any unit-preserving group homomorphism.  Moreover, each possible invariant is realized by some unital simple separable nuclear $\Z$-stable $C^*$-algebra in the UCT class.

 The results described above reinterpret the range of invariant theorems from \cite{Rordam95,Elliott96}, which predate $\Z$-stability. While all the examples of \cite{Rordam95,Elliott96} are $\Z$-stable (through $\mathcal O_\infty$-stability for \cite{Rordam95} and applying Winter's $\Z$-stability theorem, \cite{Winter12}, for \cite{Elliott96}), one can avoid checking this by replacing them with their $\Z$-stabilizations.  It is in realizing the range of the invariant that it is important that $T(A)$ is a Choquet simplex, and hence an inverse limit of finite dimensional simplices (\cite{LazarLindenstrauss}). This allows the models of 
\cite{Elliott96} to be built as inductive limits.
\end{remark}

\subsection{Unitary algebraic $K_1$}\label{SectAlgK1}The role of the
quotient of the infinite unitary group by the closure of the
commutator subgroup in the classification program goes back to work of
Nielsen and Thomsen (\cite{NielsenThomsen}), who showed that it is needed
to determine the approximate unitary equivalence classes of maps
between simple A$\mathbb T$ algebras.  Although this group and its
structure is known to experts, we provide a guide for less familiar
readers, both to clarify some topological points and to collect a
number of details of the construction  ---  notably, the de la
Harpe--Skandalis determinant  ---  for subsequent use.

For a unital $C^*$-algebra $A$, write $U_n(A)$ for the group of
unitaries in $M_n(A)$ with the norm topology and $DU_n(A)$ for the
derived subgroup of $U_n(A)$, i.e., the subgroup generated by
commutators $uvu^*v^*$ for $u,v\in U_n(A)$ (which is
automatically normal). Each $U_n(A)$ is included into $U_{n+1}(A)$ via
the standard connecting maps, and we write $U_\infty(A)$ for the
direct limit, which is equipped with the topological inductive limit
topology: $V\subseteq U_\infty(A)$ is open if and only if $V\cap U_n(A)$
is open for all $n\in\mathbb N$.  Let
$DU_\infty(A)\coloneqq\bigcup_nDU_n(A)$ and note that $DU_\infty(A)$
is contained in $U^{(0)}_\infty(A)$. Moreover, any commutator in $U_\infty(A)$ is
already a commutator of elements in $U^{(0)}_\infty(A)$: for $u, v \in U_\infty(A)$, one has \begin{equation}
uvu^*v^*=\mathrm{diag}(u,u^*,1)\mathrm{diag}(v,1,v^*)\mathrm{diag}(u^*,u,1)\mathrm{diag}(v^*,1,v).
    \end{equation}

As multiplication is separately continuous on $U_\infty(A)$ and
inversion is continuous, the closure $\overline{DU_\infty(A)}$ of
$DU_\infty(A)$ in the inductive limit topology is a normal subgroup of
$U_\infty(A)$.  Taking the quotient gives the version of algebraic
$K_1$ we need for the uniqueness component of the approximate
classification of embeddings.

\begin{definition}\label{DefKa}
Write
\begin{equation}
\Ka(A)\coloneqq U_\infty(A)/\overline{DU_\infty(A)}
\end{equation}
and call this the \emph{Hausdorffized unitary algebraic $K_1$-group}
of $A$.
      Write $\ka{u}$ for the class of $u\in U_\infty(A)$ in $\Ka(A)$. For a unital $^*$-homo\-morphism $\phi\colon A\rightarrow B$, we get an induced map $\Ka(\phi)\colon \Ka(A)\rightarrow\Ka(B)$ given by $\Ka(\ka{u})\coloneqq \ka{\phi^{(n)}(u)}$ for $u\in U_n(A)$. In this way $\Ka$ is a functor from unital $C^*$-algebras to abelian groups that is invariant under approximate unitary equivalence of morphisms.%
\end{definition}

\begin{remark}
  The algebraic $K_1$ group of a unital ring $A$ is $GL_\infty(A)/DGL_\infty(A)$. In the $C^*$-context, it is natural to use unitaries in place of invertibles. The other difference is the quotient by the closure of $DU_\infty(A)$, to ensure invariance under approximate unitary equivalence of morphisms. As an example comparison, the algebraic $K_1$-group of $\mathbb C$ is $\mathbb C^\times$, whereas $\Ka(\mathbb C)=\mathbb T$.  The Hausdorffized algebraic $K_1$-group (using invertibles) is naturally recovered from $\inv$; see \cite{Elliott22,Sarkowicz-Tikuisis} for this and other relations between these algebraic $K_1$-groups.
\end{remark}

\begin{remark}\label{rem2.8}
When the natural maps $\pi_0(U(A)) \to K_1(A)$ and $\pi_1(U(A)) \to K_0(A)$ are bijections, then so is $U(A)/\overline{DU(A)} \to \Ka(A)$ (\cite[Corollary 3.4]{Thomsen95}).
This is the case, in particular, for $C^*$-algebras with stable rank one  (\cite[Theorem 3.3]{Rieffel87}) and for $\Z$-stable $C^*$-algebras (\cite[Theorem~3]{Jiang97}).  This is typically why $U(\,\cdot\,)/\overline{DU(\,\cdot\,)}$ is used in classification results (for example, in \cite{GLN-part1}).   However, while our codomain algebras from Theorem~\ref{Main2} are $\Z$-stable, there is no such restriction on the domains, and so, in general, we cannot work with $U(A)/\overline{DU(A)}$ (or even $U_k(A)/\overline{DU_k(A)}$) in place of $\Ka(A)$. \end{remark}

The structure of $\Ka(A)$ was investigated by Thomsen in \cite{Thomsen95}, building on de la Harpe and Skandalis' determinant map and analysis of quotients by commutator subgroups in $GL_\infty(A)$ (\cite{Harpe-Skandalis84}). Importantly, it is related to $\Aff T(A)$ and $K_1(A)$ through natural group homomorphisms defined below:
\begin{equation}\label{ThomsenExtension2}
\begin{tikzcd}
\Aff T(A) \arrow{r}{\Th_A} & \Ka(A) \arrow{r}{\minusa_A} & K_1(A).
\end{tikzcd}
\end{equation}

\begin{definition}\label{DefMinusA}
Let $A$ be a unital $C^*$-algebra.  The map $\minusa_A\colon\Ka(A)\to
K_1(A)$ is the canonical surjection given by $\minusa_A(\ka{u})
\coloneqq [u]_1$ for $u\in U_\infty(A)$ (noting that $\overline{DU_\infty(A)}\subseteq U^{(0)}_\infty(A)$). The \emph{Thomsen map} $\Th_A$ is given by writing an element $f\in\Aff T(A)$ as $f=\hat{a}$ for some $a\in A_{\mathrm{sa}}$, and then setting 
\begin{equation}\label{DefThomsen}
    \Th_A(f)\coloneqq[e^{2\pi ia}]_\alg,\quad f=\hat{a},\ a\in A_{\mathrm{sa}}.
\end{equation}
\end{definition}

To see that the Thomsen map is well-defined, first note that
\begin{equation}\label{ExpCommute}
e^{ia}e^{ib}=e^{i(a+b)}\text{ modulo }\overline{DU(A)}
\end{equation}
for $a,b\in A_{\mathrm{sa}}$ (using the Lie--Trotter formula as in
\cite[Equation~(1.1)]{Thomsen95}), so that the map $a \mapsto
\ka{e^{2\pi i a}}$ is a continuous group homomorphism $A_{\mathrm{sa}} \to
\Ka(A)$.  It then suffices to show 
$\ka{e^{2\pi i( v^*v - vv^*)}}=0$ for every $v\in A$ (using Proposition~\ref{prop:cuntz-pedersen}).                  
This follows by combining \eqref{ExpCommute} with the polar decomposition trick found in the last few lines of the proof of \cite[Lemma~3.1]{Thomsen95}, which shows that for any $v\in M_n(A)$, $e^{2\pi i(v^*v-vv^*)}=1_A$ modulo $\overline{DU_n(A)}$. Note that \eqref{ExpCommute} also shows that the Thomsen map is a homomorphism. Since $\ker\minusa_A$ is generated by exponentials, the definition \eqref{DefThomsen} ensures that \eqref{ThomsenExtension2} is exact, and we have established the first two out of the following three  fundamental properties of the Thomsen map. (One of the inclusions in the last property will be established immediately below, and the second is given after Proposition \ref{Determinant}).

\begin{properties}\label{ThomsenProp}
For a unital $C^*$-algebra $A$,
\begin{enumerate}
  \item\label{ThomsenProp.1} $\Th_A(\hat a) = \ka{e^{2\pi i a}}$ for $a \in A_\mathrm{sa}$,
    where $\hat{a}$ is the continuous affine function $\hat{a}(\tau)=\tau(a)$ for $\tau\in T(A)$;
  \item\label{ThomsenProp.2} $\im \Th_A = \ker \minusa_A$;
  \item\label{ThomsenProp.3} $\ker \Th_A = \overline{\im \rho_A}$.
\end{enumerate}
\end{properties}

For the inclusion of $\overline{\im\rho_A}\subseteq\ker\Th_A$ in \ref{ThomsenProp.3}, note we can work equally well with elements $a\in M_n(A)_{\mathrm{sa}}$ in \eqref{DefThomsen} (following our scaling conventions so that $\hat{1}_{M_n(A)}$ is the constant function $n$ on $T(A)$). Then, given projections $p,q\in M_n(A)$, we have $\widehat{p-q}=\rho_A([p]_0-[q]_0)$, and so 
\begin{equation} 
\label{eq:ThRho0}
\begin{split}
  \Th_A(\rho_A([p]_0-[q]_0))=\ka{e^{2\pi i(p-q)}}
  =\ka{e^{2\pi i p}} + \ka{e^{-2\pi iq}}=0.
\end{split}\end{equation}
 Given $f\in \Aff T(A)$, one can control the norm of a choice of $a$ satisfying $\hat a =f$ (see Proposition~\ref{prop:cuntz-pedersen}), and therefore
 the Thomsen map is continuous as a map into $U_1(A) / \overline{DU}_1(A)$, so $\overline{\im\rho_A}\subseteq\ker\Th_A$.

The reverse inclusion $\ker\Th_A\subseteq\overline{\im\rho_A}$ uses the de la Harpe--Skandalis determinant from \cite{Harpe-Skandalis84}. 
For a piecewise smooth path $u \colon [0,1] \to  U_n(A)$, de la Harpe and Skandalis defined $\widetilde \Delta_A(u) \in \Aff T(A)$ by
\begin{equation}
\label{eq:dlHSdet}
\widetilde \Delta_A(u)(\tau) \coloneqq  \frac{1}{2\pi i} \int_0^1 \tau_n\big(\big(\tfrac{d}{d t} u(t)\big) u(t)^* \big) \, dt, \quad \tau\in T(A).
\end{equation}
We record the basic properties of this determinant in a slightly unusual fashion below, designed for use in Section~\ref{sec:new-map} (note that in Proposition \ref{Determinant}\ref{Determinant.C4} the proof shows that $\det_A$ actually factors through the more precise quotient $\Aff T(A) / \im \rho_A$; we do not need this). 

\begin{proposition}\label{Determinant}
Let $A$ be a unital $C^*$-algebra.  The de la Harpe--Skandalis determinant map gives a continuous group homomorphism
\begin{equation}
\widetilde\Delta_A\colon U_\infty(C([0,1],A))\longrightarrow \Aff T(A)
\end{equation}
such that:
\begin{enumerate}
  \item\label{Determinant.C1} $\widetilde{\Delta}_A(u)$ depends only on the class of $u$ through homotopies of maps $v : C([0,1])\to U_\infty(A)$ with $v(0)=u(0)$ and $v(1)=u(1)$. We call these \emph{endpoint-fixed homotopies}.
  \item\label{Determinant.C2} for a self-adjoint $h\in M_n(C([0,1],A))$,
\begin{equation}\label{DefDeterminant}
\widetilde\Delta_A(e^{2\pi i h})(\tau)=\tau_n(h(1)-h(0)),\quad \tau\in T(A)
\end{equation}
where $\tau_n$ is the canonical non-normalized extension of $\tau$ to $M_n(A)$;
\item\label{Determinant.C3} identifying homotopy classes of $U_\infty(C_0((0,1),A)^\dag)$ with $K_0(A)$ via the Bott map, the homomorphism $U_\infty(C_0((0,1),A)^\dag)\to\Aff T(A)$ induced by \ref{Determinant.C1} is precisely the pairing map $\rho_A$;
\item\label{Determinant.C4} there is a continuous group homomorphism
\begin{equation}
  \textstyle\det_A\displaystyle \colon U_\infty^{(0)}(A) \longrightarrow \Aff T(A) / \overline{\im \rho_A}
\end{equation}
given by $\det_A(u) \coloneqq  \widetilde\Delta_A(v) + \overline{\im \rho_A}$, where $v\in U_\infty(C([0, 1], A))$ has $v(0) = 1_A$ and $v(1)=u$.
\end{enumerate}
\end{proposition}

\begin{proof}
Observe that $U_\infty(C([0,1],A))$ and $C([0,1],U_\infty(A))$ agree, as every compact subset of $U_\infty(A)$ lies in $U_n(A)$ for some $n$ (see \cite[Proposition A.1]{Hatcher02}). 
Therefore, \eqref{eq:dlHSdet} defines $\widetilde\Delta_A(u)$ for piecewise smooth paths in $U_\infty(C([0,1],A))$.

For such a path, \cite[Lemme~1(c)]{Harpe-Skandalis84} shows that $\widetilde\Delta_A(u)(\tau)$ depends only on the homotopy class of $u$ (through endpoint-fixed homotopies, but not necessarily made up of piecewise smooth paths). 
Moreover, by dividing $[0,1]$ into small subintervals where $\|u(s_1)-u(s_2)\|<2$, every $u \in C([0,1],U_\infty(A))$ has an endpoint-fixed homotopy to a piecewise smooth path.

Therefore, the definition of $\widetilde{\Delta}_A$ extends to all continuous paths $u\colon [0,1]\to U_n(A)$, by applying $\widetilde{\Delta}_A$ to a piecewise smooth path homotopic to $u$ (again with endpoints fixed).
Evidently, \ref{Determinant.C1} holds. 
Moreover, this definition of $\widetilde{\Delta}_A$ is a homomorphism by \cite[Lemme~1(a)]{Harpe-Skandalis84}.  

Condition~\ref{Determinant.C2} is immediate when $h$ is piecewise smooth, and so follows in general by replacing $h$ with a piecewise smooth path homotopic to $h$ with the same endpoints.  Note that the Bott map associates a projection $p \in M_n(A)$ with the unitary $u \in U_n((C_0(0,1), A)^\dag)$ given by $u(t) \coloneqq  e^{2 \pi i t p}$ for $t \in [0, 1]$, so \ref{Determinant.C3} follows from \ref{Determinant.C2}.  Continuity of $\widetilde{\Delta}_A$ on the collection of piecewise smooth paths in $U_n(A)$ is immediate, and then extends to $U_n(C([0,1],A))$. As this holds for all $n$, $\widetilde{\Delta}_A$ is continuous on $U_\infty(C([0,1],A))$.

Condition~\ref{Determinant.C4} is essentially \cite[Lemme~1(d)]{Harpe-Skandalis84} and the subsequent definitions. Once one knows $\det_A$ is well-defined, it is a continuous homomorphism because $\widetilde\Delta_A$ is.  Since $U^{(0)}_\infty(A)=\bigcup_n U_n^{(0)}(A)$ (as every compact subset of $U_\infty(A)$ is contained in some $U_n(A)$) given $u\in U^{(0)}_\infty(A)$, we can find $n\in\mathbb N$ and $v \in U_n(C[0, 1], A)$ with $v(0) = 1_{M_n(A)}$ and $v(1) = u$.  If $v_1, v_2 \in U_n(C([0, 1], A))$ are two such paths, then $v_1v_2^*$ can be viewed as an element of $U_n(C_0((0,1),A)^\dag)$ and so, by \ref{Determinant.C3}, $\widetilde{\Delta}_A(v_1v_2^*)\in\im\rho_A$.  Therefore $\widetilde{\Delta}_A(v_1)+ \overline{\im \rho_A}=\widetilde{\Delta}_A(v_2)+\overline{\im \rho_A}$.
\end{proof}

With this in place, we complete the proof of Properties \ref{ThomsenProp}.
\begin{proof}[Proof that $\ker\Th_A\subseteq\overline{\im\rho_A}$ in Property \ref{ThomsenProp}\ref{ThomsenProp.3}]
Since $\det_A$ is a continuous group homomorphism, $\overline{DU^{(0)}_\infty(A)}\subseteq\ker\det_A$, and hence $\det_A$ induces a continuous group homomorphism
\begin{equation}\label{eq:defoverlinedet}
  \overline{\det}_A \colon \ker \minusa_A = U^{(0)}_\infty(A) / \overline{DU^{(0)}_\infty(A)} \longrightarrow \Aff T(A) / \overline{\im \rho_A}.
\end{equation}
A computation shows $\overline{\det}_A(\Th_A(f)) = f + \overline{\im \rho_A}$ for all $f \in \Aff T(A)$ (using that $f = \hat a$ for some $a \in A_\mathrm{sa}$).  This implies $\ker \Th_A \subseteq \overline{\im \rho_A}$.
\end{proof}

Combining this with Properties \ref{ThomsenProp}(\ref{ThomsenProp.2}), we obtain an isomorphism
\begin{equation}
  \label{thomsen-bar-def}
\begin{tikzcd}
\overline{\Th}_A\colon \Aff T(A)/\overline{\im\rho_A}\arrow{r}{\cong} & U^{(0)}_\infty(A)/\overline{DU_\infty(A)}\subset\Ka(A),
\end{tikzcd}
\end{equation}
and the computation in the proof above shows that 
\begin{equation}
\label{eq:detInverse}
\overline{\det}_A=\overline{\Th}_A^{\,\,-1}.
\end{equation}
In this way there is a natural assignment of a short exact sequence to a unital $C^*$-algebra.
\begin{definition}
The \emph{Thomsen extension} of a unital $C^*$-algebra $A$ is the short exact sequence
\begin{equation}\label{ThomsenExtension}
\begin{tikzcd}
0 \arrow{r} & \Aff
T(A)/\overline{\im\rho_A}\arrow{r}{\overline{\Th}_A} &
\Ka(A) \arrow{r}{\minusa_A} & K_1(A)\arrow{r} & 0.
\end{tikzcd}
\end{equation}
\end{definition}

Note that $\Aff T(A)/\overline{\im\rho_A}$ is a divisible group (for every $f+\overline{\im\rho_A}$, and $n\in\mathbb N$, we have $n\cdot(f/n+\overline{\im\rho_A})=(f+\overline{\im\rho_A})$).  For an abelian group, this is equivalent to injectivity (see \cite[Corollary~2.3.2]{Weibel94}, for example), and so there is a splitting $\Ka(A)\to\Aff T(A)/\overline{\im\rho_A}$. In this way \eqref{ThomsenExtension} splits unnaturally, and  $\Ka(A)$ can be identified as
\begin{equation}
\label{eq:K1algSplitting}
\Ka(A)\cong \Aff T(A)/\overline{\im\rho_A}\oplus K_1(A)
\end{equation} (see \cite[Corollary~3.3]{Thomsen95}).  Consequently,
any morphism between $K$-theory and traces can be extended to a map
between the $\Ka$-groups (as in the following proposition), albeit in a non-canonical way. Therefore, the additional data encoded by $\Ka$, beyond that held in $KT_u$, is found not in $\Ka(A)$ but in the behavior of $\Ka$ on morphisms. To our knowledge, the first explicit result extending $KT_u$-morphisms to $\Ka$ was set out by Nielsen in \cite[Section 3]{Nielsen99} for $U(A)/DU(A)$ in the context of simple A$\mathbb T$ algebras (see also \cite[Lemma 3.2]{NielsenThomsen}), although the argument works generally. As it is essential in deducing the unital classification theorem (Theorem~\ref{Main}) from the classification of embeddings (Theorem~\ref{Main2}), we give a short proof of this extension result for completeness.

\begin{proposition}\label{PropExtendAlgK1}
Let $A$ and $B$ be unital $C^*$-algebras and $(\alpha_*,\gamma)\colon
KT_u(A)\rightarrow KT_u(B)$ be a morphism.  Then there exists a group homomorphism
$\beta\colon \Ka(A)\rightarrow\Ka(B)$ such that
    \begin{equation}\label{building:beta_1}
      \begin{tikzcd}
        K_0(A) \ar[r, "\rho_A"] \ar[d, "\alpha_0"] &
        \Aff T(A) \ar[r, "\Th_A"] \ar[d, "\gamma"] &
        \Ka(A) \ar[r, "\minusa_A"] \ar[d, "\beta"] &
        K_1(A) \ar[d, "\alpha_1"] \\
        K_0(B) \ar[r, "\rho_B"] &
        \Aff T(B) \ar[r, "\Th_B"] &
        \Ka(B) \ar[r, "\minusa_B"] &
        K_1(B)
      \end{tikzcd}
    \end{equation}
commutes.  Moreover, if $(\alpha_*, \gamma)$ is an isomorphism, then any such $\beta$ is an isomorphism.
\end{proposition}

\begin{proof}
As $\gamma$ is continuous, since it is unital, positive, and linear, the maps $(\alpha_*,\gamma)$ induce the diagram
  \begin{equation}\label{eq:building-beta}
  \begin{tikzcd}
  0\arrow[r]&  \Aff T(A) / \overline{\im \rho_A} \arrow{r}{\overline{\Th}_A} \arrow{d}{\bar{\gamma}} &
    \Ka(A)\arrow{r}{\minusa_A} \arrow[dashed]{d}{\beta} & K_1(A) \arrow{d}{\alpha_1}\arrow[r]&0 \\
    0\arrow[r]&\Aff T(B) / \overline{\im \rho_B} \arrow{r}{\overline{\Th}_B} & \Ka(B) \arrow{r}{\minusa_B} &
    K_1(B)\arrow[r]&0
  \end{tikzcd}
  \end{equation}
with exact rows. Both rows of this diagram split because $\Aff T(A) / \overline{\im \rho_A}$ and $\Aff T(B) / \overline{\im \rho_B}$ are
  divisible abelian groups (as discussed above).  Let
  \begin{equation}
    p_A \colon \Ka(A) \to \Aff T(A) / \overline{\im \rho_A} \quad \text{and} \quad i_B \colon K_1(B) \to \Ka(B)
  \end{equation}
  be splittings of $\overline{\Th}_A$ and $\minusa_B$, respectively. Defining the dashed arrow in \eqref{eq:building-beta} by
  \begin{equation}
    \beta \coloneqq
    i_B \circ \alpha_1 \circ \minusa_A
    + \overline{\Th}_B \circ \bar{\gamma} \circ p_A
  \end{equation}
ensures that \eqref{eq:building-beta}, and hence
\eqref{building:beta_1}, commutes.  If $(\alpha_*,\gamma)$ is an
isomorphism, then $\gamma$ is an isomorphism and thus $\overline{\gamma}$ is an isomorphism. By the five lemma, any map $\beta$ making \eqref{eq:building-beta} commute is necessarily an isomorphism.
\end{proof}

We end this section by noting a few topological facts. Firstly, as pointed out to us by George
  Elliott, $DU_\infty(A)$ is norm-dense in $U^{(0)}_\infty(A)$ (this can be seen using a trick in the proof of
  \cite[Theorem~2.4.7]{Higson88} that goes back to 
  \cite{PearcyTopping71}). Thus, it is important to take the closure of $DU_\infty(A)$ in the inductive limit topology when defining $\Ka(A)$.  Secondly, while the inverse map on $U_\infty(A)$ is continuous, multiplication is in general only
separately continuous, so $U_\infty(A)$ is not a
topological group with the inductive limit 
topology (in fact, using  \cite[Theorem~2.7]{Tatsuuma98} and \cite[Theorem~4]{Yamasaki98} it follows that $U_\infty(A)$ is a topological group in the inductive limit topology precisely when $A$ is finite dimensional).  In contrast, as shown in the next proposition, $\Ka(A)$ is naturally a topological group. However the second part of the proposition shows that is not necessary to keep track of its topological data for the classification
of full approximate embeddings, as it is recaptured by the compatibility conditions 
\eqref{building:beta_1} on $\inv$-morphisms. Hence we will treat $\Ka(A)$ purely algebraically throughout the rest of the paper.

\begin{proposition}\label{prop:ka-topology}
Let $A$ be a unital $C^*$-algebra.  Then $\Ka(A)$ is a Hausdorff topological group in the topology induced by the inductive limit topology on $U_\infty(A)$.  Moreover, given a unital $C^*$-algebra $B$ and a $KT_u$-morphism $(\alpha_*, \gamma) \colon KT_u(A) \rightarrow KT_u(B)$, any group homomorphism $\beta \colon \Ka(A) \rightarrow \Ka(B)$ making  \eqref{building:beta_1} commute is automatically continuous.
\end{proposition}

\begin{proof}
From \eqref{ThomsenExtension2}, we have a group isomorphism
\begin{equation}\label{eq:thomsen-homeo}
   \overline{\Th}_A \colon \Aff T(A) / \overline{\im \rho_A} \longrightarrow \ker\minusa_A,
\end{equation}
that is continuous because $\Th_A$ is.  It follows from
\eqref{eq:detInverse} that its inverse is the continuous map
$\overline{\det}_A$.  Since addition is separately continuous on
$\Ka(A)$, to show addition is jointly continuous, it suffices to show
that if $(x_i), (y_i) \subseteq \Ka(A)$ are nets with $x_i \rightarrow
0$ and $y_i \rightarrow 0$, then $x_i + y_i \rightarrow 0$.  Since
$\minusa_A$ is continuous and $K_1(A)=U_\infty(A)/U^{(0)}_\infty(A)$
is discrete, we have $x_i, y_i \in \ker \minusa_A$ for all large $i$.
Since $\overline{\det}_A$ is a homeomorphic group isomorphism and the
addition on $\Aff T(A) / \overline{\im \rho_A}$ is jointly continuous,
by \eqref{eq:detInverse} we have $x_i + y_i \rightarrow 0$.  This
shows $\Ka(A)$ is a topological group.  As $\minusa_A$ and
$\overline{\det}_A$ are continuous, $\{0\}$ is closed in
$\Ka(A)$, and hence $\Ka(A)$ is Hausdorff.

For the second statement, note that any map $\beta$ making \eqref{building:beta_1} commute also makes \eqref{eq:building-beta} commute, so that by applying the continuous map $\overline{\det}_A$, it follows that $\beta|_{\ker\minusa_A}$ is continuous. Since $\ker\minusa_A$ is open and $\Ka(A)$ is a topological group, it follows that $\beta$ is continuous.
\end{proof}

\begin{remark}\label{RemTopGroupInd}
It is perhaps more
natural to endow $U_\infty(A)$ with the topological group inductive limit topology $\mathcal T_g$, rather than the inductive limit topology $\mathcal T$, so that $(U_\infty(A),\mathcal T_g)$ becomes a topological group. Using \cite[Theorem 4]{Yamasaki98}, one can show that $\mathcal T\neq \mathcal T_g$ on $U_\infty(A)$ whenever $A$ is infinite dimensional. However, as we sketch below, using either topology gives rise to the same Hausdorffized algebraic $K_1$-group, i.e., $\overline{K}^{\mathrm{alg},g}_1(A)\coloneqq (U_\infty(A),\mathcal T_g) / \overline{DU_\infty(A)}^{\mathcal T_g}$ is canonically isomorphic to $\Ka(A)$.   

The topology 
$\mathcal T_g$ is defined as the strongest topological group topology on $U_\infty(A)$ for which all the inclusions $U_n(A) \rightarrow
(U_\infty(A),\mathcal T_g)$ are continuous, and consequently has the universal property that if $H$ is a
topological group and $\phi_n \colon U_n(A) \rightarrow H$ are
compatible continuous group homomorphisms, then there is a unique
$\mathcal T_g$-continuous group homomorphism $\phi \colon U_\infty(A) \rightarrow H$
extending the $\phi_n$'s. Taking $H=\Ka(A)$, the resulting $\phi$ annihilates $\overline{DU_\infty(A)}^{\mathcal T_g}$ (as $\Ka(A)$ is abelian and Hausdorff), so descends to a continuous group homomorphism $\bar{\phi}\colon\overline{K}^{\mathrm{alg},g}_1(A)\to\Ka(A)$.  The inverse arises as the maps $U_n(A)\to \overline{K}^{\mathrm{alg},g}_1(A)$ are continuous group homomorphisms and therefore they extend to a continuous map $(U_\infty(A),\mathcal T)\to\overline{K}^{\mathrm{alg},g}_1(A)$ (by the universal property of the inductive limit topology), which is a group homomorphism, and hence induces a continuous group homomorphism $\Ka(A)\to \overline{K}^{\mathrm{alg},g}_1(A)$ inverse to $\bar{\phi}$.
\end{remark}

\subsection{Total $K$-theory}
\label{ss:totalKtheory}

Total $K$-theory comprises $K$-theory, $K$-theory with coefficients in $\mathbb Z/n\mathbb Z$ (written $\Zn{n}$ from now on) for all $n\geq 2$, and the natural connecting maps between these groups --- the \emph{Bockstein operations} --- satisfying a number of $6$-term exact sequences. To our knowledge, the first explicit use of $K$-theory with coefficients in the classification program was by Dadarlat and Loring as an obstruction to the classification of certain non-simple $C^*$-algebras using the Elliott invariant (\cite{Dadarlat-LoringAIF}).  Eilers subsequently used $K$-theory with coefficients (equipped with an order structure) to classify certain AD algebras \cite{EilersPhD}. 

For $n\geq 2$ and $i\in \{0,1\}$, the groups $K_i(A;\Zn{n})$ and the Bockstein operations were introduced by Schochet in \cite{Schochet84} and have subsequently been expressed in a number of equivalent forms. We will use two pictures of total $K$-theory in the paper. For the purpose of constructing our new pairing map connecting $K_0(A;\Zn{n})$ to $\Ka(A)$ in Section~\ref{sec:new-map}, we use the definition
\begin{equation}\label{DefK-ThyCoefficients}
  K_i(A;\Zn{n})\coloneqq K_{1-i}(A\otimes\mathbb I_n),
\end{equation}
where $\mathbb I_n$ is the algebra
\begin{equation}\label{eq:I_n}
  \mathbb I_n \coloneqq \{ f\in C([0,1], M_n) : f(0) \in \mathbb C
  1_{M_n}  , \, f(1) = 0\}.
\end{equation}
This form of the definition is obtained in the joint work of Cuntz
  and Schochet found in \cite[Section~6]{Schochet84}; as noted there,
  one can define $K_i(\,\cdot\,;\Zn{n}) = K_i(\,\cdot\,\otimes C_n)$
  using any separable nuclear $C^*$-algebra $C_n$ in the UCT class
  with $K_\ast(C_n) \cong (\mathbb Z/n , 0)$. In our case, we use $C_n
  \coloneqq C_0((0,1))\otimes \mathbb I_n$ and use Bott periodicity to replace suspensions by index-shifts in \eqref{DefK-ThyCoefficients}.   
  Other possible choices for
  $C_n$ include the Cuntz algebras $\mathcal O_{n+1}$ (as we shall briefly use in Lemma~\ref{lem:inv-iota-injective}) or --- as in
  Schochet's original definition --- the mapping cone of a degree $n$
  map on $C_0(0,1)$.

Later in the paper, when we need Dadarlat and Loring's universal multicoefficient theorem, we will use the formulation in terms of $KK^i(\mathbb I_n,A)$ from \cite[Section 1]{Dadarlat-Loring96}. While it is well-known to experts that all pictures of total $K$-theory are naturally isomorphic, this is not easily extracted from the literature.  We handle this point in Appendix~\ref{appendix:totalappendix}.

Using the definition in \eqref{DefK-ThyCoefficients} one
obtains natural Bockstein maps
\begin{equation}
\mu^{(n)}_{i,A}\colon K_i(A)\longrightarrow K_i(A;\Zn{n})\text{ and }\nu^{(n)}_{i,A}\colon K_i(A;\Zn{n})\longrightarrow K_{1-i}(A)
\end{equation}
from the short exact sequence $0 \to C_0((0,1),M_n) \to \mathbb I_n
\to \mathbb C \to 0$, which after tensoring with $A$ induces
a $6$-term exact sequence
\begin{equation}
  \label{eq:bockstein-2}
  \begin{tikzcd}
    K_0(A) \ar[r, "\mu^{(n)}_{0,A}"]
      & K_0(A;\Zn{n}) \ar[r, "\nu^{(n)}_{0,A}"]
      & K_1(A) \ar[d, "\times n"] \\
    K_0(A) \ar[u, "\times n"]
      & K_1(A;\Zn{n}) \ar[l, "\nu^{(n)}_{1,A}"]
      & K_1(A)\ar[l, "\mu^{(n)}_{1,A}"]
  \end{tikzcd}
 \end{equation}
 that is natural in $A$.  For this, apply the natural identification
  $K_i(C_0(0,1) \otimes M_n \otimes A) \cong K_{1-i}(A)$ from Bott
  periodicity and stability of $K$-theory. It is easily checked that
  the boundary maps are multiplication by $n$. These Bockstein maps
  are often denoted by $\rho$ and $\beta$ in the literature, but as we
  consistently use $\rho$ and $\beta$ for other maps, we denote them by $\mu$ and $\nu$. 
  
  For $n,m \geq 2$, the remaining Bockstein
 operations
\begin{equation}
\label{eq:kappaDef}
    \begin{split}
  &\kappa^{(nm,n)}_{i,A} \colon K_i(A; \Zn{n}) \to K_i(A; \Zn{nm}) \quad \text{and} \\
  &\kappa^{(n,nm)}_{i,A} \colon K_i(A; \Zn{nm}) \to K_i(A; \Zn{n})
\end{split} 
\end{equation}
are induced by the canonical inclusions $\mathbb I_n
\xrightarrow{\cong} \mathbb I_n \otimes 1_{M_m} \hookrightarrow
\mathbb I_{nm}$ and $\mathbb I_{nm} \hookrightarrow \mathbb I_n
\otimes M_m$ respectively. These maps interact with the rows of
\eqref{eq:bockstein-2} through the following commutative diagram (which is easily checked directly from the definitions):
\begin{equation}
	\label{eq:bockstein-commute}
	\begin{tikzcd}[cramped]
		K_i(A) \ar[r, "\times m"] \ar[d, "\mu_{i,A}^{(n)}"] &
		K_i(A) \ar[r,equals] \ar[d, "\mu_{i,A}^{(nm)}"]&
		K_i(A) \ar[d, "\mu_{i,A}^{(n)}"] \\
		K_i(A; \Zn{n}) \ar[r, "\kappa_{i,A}^{(nm,n)}"]
			\ar[d, "\nu_{i,A}^{(n)}"] &
		K_i(A; \Zn{nm}) \ar[r, "\kappa_{i,A}^{(n,nm)}"]
			\ar[d, "\nu_{i,A}^{(nm)}"] &
		K_i(A; \Zn{n}) \ar[d, "\nu_{i,A}^{(n)}"] \\
		K_{1-i}(A) \ar[r,equals] & 
		K_{1-i}(A) \ar[r, "\times m"] & 
		K_{1-i}(A) 
	\end{tikzcd}
\end{equation}

\begin{definition}
  The \emph{total $K$-theory}, $\underline{K}(A)$, of a $C^*$-algebra
  $A$ is the collection of abelian groups $K_i(A)$ and $K_i(A; \Zn{n})$ for $i=0,1$ and $n=2,3,\dots$, together with all the Bockstein maps 
\begin{equation} 
\big\{\mu^{(n)}_{i,A},
  \nu^{(n)}_{i,A}, \kappa_{i,A}^{(n,nm)}, \kappa_{i,A}^{(nm,n)} : n,
  m\geq 2, i\in\{0,1\} \big\}.
\end{equation}
A \emph{$\Lambda$-morphism} $\underline{\alpha}\colon
\underline{K}(A)\to\underline{K}(B)$ consists of homomorphisms
\begin{equation}
\alpha_i\colon K_i(A)\to K_i(B)\quad\text{ and }
\quad \alpha_i^{(n)}\colon K_i(A;\Zn{n})\to K_i(B;\Zn{n})
\end{equation}
for all
$i\in\{0,1\}$ and $n\geq 2$ that intertwine all the Bockstein
operations. The collection of these $\Lambda$-morphisms is written
$\mathrm{Hom}_{\Lambda}(\underline{K}(A),\underline{K}(B))$.
\end{definition}

As noted in \cite[Proposition~1.8]{Schochet84}, the sequence
\eqref{eq:bockstein-2} collapses to a short exact
sequence
\begin{equation}
  \label{eq:bockstein-new}
  0 \longrightarrow K_i(A)\otimes  \Zn{n} \xrightarrow{\overline \mu^{(n)}_{i,A}} K_i(A; \Zn{n}) \xrightarrow{\overline \nu^{(n)}_{i,A}} \mathrm{Tor}(K_{1-i}(A), \Zn{n}) \longrightarrow 0.
\end{equation}
Alternatively, as $\mathbb I_n$ satisfies the UCT,
  \eqref{eq:bockstein-new} can be obtained from the K{\"u}nneth exact
  sequence for $K_{1-i}(A\otimes \mathbb I_n)$ from \cite{Rosenberg-Schochet87}.
  
Explicitly, $\mathrm{Tor}(K_{1-i}(A),\Zn{n})$ consists of those
elements $x\in K_{1-i}(A)$ with $nx=0$. The sequence
\eqref{eq:bockstein-new} splits unnaturally
(\cite[Proposition~2.4]{Schochet84}).  Moreover, as is known to
experts, these splittings may be chosen so that they preserve the
Bockstein operations. This goes back to B{\"o}digheimer
(\cite{Bodigheimer79,Bodigheimer80}) in a purely abstract framework and
is noted in the $C^*$-algebraic setting in
\cite[Lemma~2.2.8]{EilersPhD} and \cite[Remark~4.6]{Eilers97},
resulting in Proposition~\ref{TotalKExtend} below.  We explain how to
extract this precise statement from the literature in Appendix \ref{sec:bodigheimer} (the
proof follows that of Lemma~\ref{l:Bodigheimer}) since it is not
entirely straightforward.

In order to be clear
  about exactly where the UCT assumption is used in our approach
  to the classification theorem, we highlight that Proposition~\ref{TotalKExtend} is a purely algebraic result, and does not require the UCT.   However, in the proof of the   classification theorem, we will only apply
  the proposition when $A$ satisfies the UCT. In this
  case, as will be familiar to many readers, the result quickly
  follows from the UCT assumption: briefly, any $\alpha_*\colon
  K_*(A)\to K_*(B)$ lifts to an element of $KK(A,B)$ by the UCT, and this
  induces a $\Lambda$-morphism $\underline{\alpha}\colon\totK(A)\to\totK(B)$, which restricts to $\alpha_*$.  The map $\underline{\alpha}$ is an isomorphism if $\alpha_*$
  is (by applying the five lemma to
  \eqref{eq:bockstein-2}).  For this reason we do not include \cite{Bodigheimer79,Bodigheimer80} in Section~\ref{section:referencesneeded} in the list of references we rely on for our proofs of Theorems~\ref{Main} and \ref{Main2}.

\begin{proposition}\label{TotalKExtend}
  Let $A$ and $B$ be $C^*$-algebras.  Any morphism $\alpha_*\colon{}K_*(A)\to
  K_*(B)$ can be extended to a $\Lambda$-morphism
  $\underline{\alpha}\colon{}\underline{K}(A)\to\underline{K}(B)$,
  which is an isomorphism when $\alpha_*$ is.
\end{proposition}

As with $\Ka(A)$, the additional information
captured by $\underline{K}$ --- compared with $K_\ast$ --- lies in its
behavior on morphisms, and not in the constituent groups of
$\underline{K}(A)$ (which are completely determined by $K_*(A)$).

We end this subsection with the following calculation using the Bockstein maps, which will be applied to the trace-kernel extension as part of the proof of Theorem~\ref{intro:calcKL}. The relevant hypotheses will be verified in Section~\ref{SSTKKThy}.  If $G$ is an abelian group, we write $\mathrm{Tor}(G)$ for the
torsion subgroup of $G$ consisting of those elements in $G$ of finite
order; i.e., $\mathrm{Tor}(G)=\bigcup_{n\geq 2}\mathrm{Tor}(G,\Zn{n})$.
\begin{lemma}
  \label{lem:hom-Lambda-computation}
  Let $A$ be a $C^*$-algebra and let
  \begin{equation}
  \begin{tikzcd}
      0 \arrow{r} & I \arrow{r}{j} & E \arrow{r}{q} & D \arrow{r} & 0 
  \end{tikzcd}
  \end{equation}
  be an extension of $C^*$-algebras such that $K_1(D)=0$ and $K_0(D)$ is uniquely divisible, that is, for each $x\in K_0(D)$ and $n\in\mathbb Z$ with $n\neq 0$, there exists a unique $y\in K_0(D)$ with $ny=x$.
  Then a $\Lambda$-morphism $\underline{\alpha}\colon{}\underline{K}(A)\to \underline{K}(I)$ belongs to $\ker\Hom_\Lambda\big( \underline{K}(A),\underline{K}(j) \big)$ if and only if the only non-zero component of $\underline{\alpha}$ is $\alpha_1$, $K_1(j)\circ \alpha_1 = 0$, and $\alpha_1$ vanishes on $\mathrm{Tor}(K_1(A))$. In fact, there is an
  isomorphism
  \begin{equation}
    \label{eq:KTheoryIso1}
    \ker\Hom_\Lambda\big(
    \underline{K}(A),\underline{K}(j)
    \big)
    \stackrel\cong\longrightarrow
    \Hom \big(
    K_1(A)/\mathrm{Tor}\big(K_1(A)\big), \ker K_1(j)
    \big)
  \end{equation}
  which is natural in both $A$ and in extensions satisfying the specified hypotheses.
\end{lemma}

\begin{proof}
  Since $K_0(D)$ is uniquely divisible, multiplication by $n$ is an
  automorphism of $K_0(D)$ for all $n \geq 2$.  This, together with
  $K_1(D)=0$ and the exactness of \eqref{eq:bockstein-2}, gives $K_i(D
  ; \Zn{n}) = 0$ for all $i = 0, 1$ and $n \geq 2$.  Since
  $K_\ast(\,\cdot\, ; \Zn{n})$ takes short exact sequences to 6-term
  exact sequences (this follows from its definition as
  $K_{1-\ast}(\,\cdot\, \otimes\mathbb I_n)$), it follows that $\ker
  K_i(j; \Zn{n}) =0$ for all $n\geq 2$ and $i\in 0,1$. Similarly,
  $\ker K_0(j)=0$ since $K_1(D) = 0$.

  Let $\underline \alpha \colon \underline K(A) \to \underline K(I)$
  be in $\ker \Hom_\Lambda(\underline{K}(A),\underline{K}(j))$ with
  components $\alpha_i$ and $\alpha^{(n)}_i$ for $i=0,1$ and $n\geq
  2$. As these homomorphisms factor through $\ker K_i(j)$ and $\ker
  K_i(j;\Zn{n})$ respectively, it follows that $\alpha_0 = 0$ and
  $\alpha_i^{(n)} = 0$, for $i=0,1$. Because $\underline\alpha$ is a $\Lambda$-morphism, we have
  $\alpha_1 \circ \nu^{(n)}_{0,A} = \nu^{(n)}_{0,I} \circ
  \alpha^{(n)}_0 = 0$.  It follows that $\alpha_1$ vanishes on the image
  of $\nu^{(n)}_{0,A}$ for every $n\geq 2$. By exactness of
  \eqref{eq:bockstein-2}, the image of $\nu^{(n)}_{0,A}$ consists
  precisely of the $n$-torsion elements of $K_1(A)$, and thus
  $\alpha_1$ vanishes on $\mathrm{Tor}(K_1(A))$. Thus the map sending
  $\underline{\alpha}\in \ker
  \Hom_\Lambda(\underline{K}(A),\underline{K}(j))$ to the map
  $K_1(A)/\mathrm{Tor}(K_1(A)) \to K_1(I)$ induced by $\alpha_1$ is a
  well-defined injection.

  For surjectivity in \eqref{eq:KTheoryIso1}, suppose $\alpha_1'
  \colon K_1(A)/\mathrm{Tor}(K_1(A)) \to \ker K_1(j)$ is a
  homomorphism, and let $\alpha_1$ be the composition
  \begin{equation}
  \begin{tikzcd}
    K_1(A) \ar[r] & K_1(A)/\mathrm{Tor}(K_1(A)) \ar[r, "\alpha_1'"] & \ker K_1(j) \subseteq K_1(I).
  \end{tikzcd}
  \end{equation}
  We define $\underline{\alpha}$ by setting the other components, that is,
  $\alpha_0 \colon K_0(A) \to K_0(I)$ and $\alpha_i^{(n)} \colon
  K_i(A;\Zn{n})\to K_i(I; \Zn{n})$, to be the zero maps. This will
  certainly lie in the kernel of
  $\Hom_\Lambda(\underline{K}(A),\underline{K}(j))$ once we have shown
  that these choices of components do intertwine the Bockstein
  operations, i.e., that $\underline{\alpha}$ is a $\Lambda$-morphism. 

  Since $\alpha_0=0$ and $\alpha_i^{(n)} =0$, we only need to verify
  intertwinings involving $\alpha_1$.  That is, we must verify the
  commutativity of
  \begin{equation}\label{eq:KTheoryIso.NewEq2}
    \begin{tikzcd}
      K_1(A) \ar[r, "\mu^{(n)}_{1,A}"] \ar[d, "\alpha_1"]
      & K_1(A; \Zn{n}) \ar[d, "\alpha^{(n)}_1 = 0"] \\
      K_1(I) \ar[r, "\mu^{(n)}_{1,I}"]
      & K_1(I; \Zn{n})
    \end{tikzcd}
    \quad\text{ and }\quad
    \begin{tikzcd}
      K_0(A; \Zn{n}) \ar[d, "\alpha^{(n)}_0 = 0",swap] \ar[r,
      "\nu^{(n)}_{0,A}"]
      & K_1(A) \ar[d, "\alpha_1"]\\
      K_0(I; \Zn{n}) \ar[r, "\nu^{(n)}_{0,I}"] &
      K_1(I) \ .
    \end{tikzcd}
  \end{equation}
  As noted above, the image of $\nu^{(n)}_{0,A}$ consists of
  $n$-torsion elements so is contained in
  $\mathrm{Tor}(K_1(A))$. Thus $\alpha_1 \circ \nu^{(n)}_{0,A} = 0$ and the
  right square in \eqref{eq:KTheoryIso.NewEq2} commutes. For
  commutativity of the left square, we note that since $\alpha_1$
  takes values in $\ker K_1(j)$, naturality of the Bockstein maps
  $\mu_{1}^{(n)}$ gives
  \begin{equation}
    K_1(j; \Zn{n}) \circ \mu^{(n)}_{1,I} \circ \alpha_1 =
    \mu^{(n)}_{1,E} \circ K_1(j) \circ \alpha_1 = 0.
  \end{equation}
  By injectivity of $K_1(j; \Zn{n})$ (as observed at the beginning the proof),
  it follows that $\mu^{(n)}_{1,I} \circ \alpha_1 = 0$, so the left hand
  square of \eqref{eq:KTheoryIso.NewEq2} also commutes.
\end{proof}

\section{The total invariant}\label{s:totalinvariant}

This section assembles the total invariant $\inv(\,\cdot\,)$. We start out by constructing natural maps
\begin{equation}\label{NewMap}
\zeta^{(n)}_A \colon K_0(A; \Zn{n}) \to \Ka(A)
\end{equation}
relating total and algebraic $K$-theory in Section~\ref{sec:new-map}.  Such a system of maps was independently constructed by Gong, Lin, and Niu in \cite{Gong-Lin-etal23} by other means; see Remark~\ref{NewMapsRemark}. Section~\ref{ss:totalinv} formalizes $\inv(\,\cdot\,)$ and proves an ``extension of invariants'' theorem (Theorem~\ref{ThmExtension}): (iso)morphisms at the level of $KT_u(\,\cdot\,)$ lift to $\inv(\,\cdot\,)$. This is crucial in deducing Theorem~\ref{Main} and Corollary~\ref{MainCor} from Theorem~\ref{Main2}.  We end by observing that the inclusion of a $C^*$-algebra into its sequence algebra is injective on invariants (Lemma~\ref{lem:inv-iota-injective}), for use in establishing Theorem~\ref{Main2} from Theorem~\ref{Main3}.

\subsection{Factorizing the Bockstein maps}\label{sec:new-map}

Let $A$ be a unital $C^*$-algebra.  There is a well-known bijection between projections $p\in A$ and self-adjoint unitaries $u\in A$ given by $p\mapsto 1-2p=e^{2\pi i p/2}$.  Since self-adjoint unitaries are trivial in $K_1$, this identification is not interesting at the level of $K$-theory.  But the map $[p]_0\mapsto\ka{1-2p}=\ka{e^{2\pi ip/2}}=\Th_A(\rho_A([p]_0)/2)$ does give a well-defined natural map $K_0(A)\to\Ka(A)$ that factorizes through $K_0(A)\otimes \Zn{2}$. More generally, for any $n\geq 2$, we have a map $K_0(A)\to\Ka(A)$ given by $[p]_0\mapsto \ka{e^{2\pi ip/n}}$ that factors through $K_0(A)\otimes \Zn{n}$.
Recall that the Bockstein maps induce the short exact sequence
\begin{equation}
\begin{tikzcd}[column sep = 5ex]
0 \arrow{r} & K_0(A)\otimes  \Zn{n} \arrow{r}{\bar\mu^{(n)}_{0,A}} & K_0(A; \Zn{n}) \arrow{r}{\bar \nu^{(n)}_{0,A}} & \mathrm{Tor}(K_{1}(A), \Zn{n}) \arrow{r} & 0
\end{tikzcd}
\end{equation}
that splits unnaturally (as a consequence of the unnatural splitting of the K\"unneth formula, \cite[Theorem 23.1.3]{Blackadar98}).
This allows the maps above to be extended (a priori unnaturally) from $K_0(A)\otimes\Zn{n}$ to $K_0(A;\Zn{n})$.

In the next two paragraphs and Proposition~\ref{prop:zeta-def}, we show that this can, in fact, be done naturally.  Then, in Proposition~\ref{p:newmap}, we show that the resulting map is compatible with the other structure maps.

Fix $n\geq 2$. We identify $(A \otimes \mathbb I_n)^\dagger$ with
\begin{equation}
\{ f\in C([0,1], A\otimes M_n) : f(0) \in A\otimes 1_{M_n}, \,  f(1) \in \mathbb C1_{A\otimes M_n} \}.
\end{equation}
Write $\ev^{(0,n)}_A \colon (A\otimes \mathbb I_n)^\dagger \to A$ for the map induced by evaluation at $0$.
(In particular, note that $\ev^{(0,n)}_A (1_{(A\otimes \mathbb I_n)^\dagger})=1_A$ and not $1_{A\otimes M_n}$.)
We will abuse notation and also write $\ev_A^{(0,n)}$ for the induced map when taking matrix amplifications.  Then, using the definition $K_0(A;\Zn{n})= K_1(A\otimes \mathbb I_n)$, the Bockstein map $\nu^{(n)}_{0,A}\colon K_0(A;\Zn{n})\to K_1(A)$ is the map $K_1(\ev^{(0,n)}_A)$.

Let $u\in U_\infty((A\otimes \mathbb I_n)^\dagger)$;
in spirit, we would like to define $\zeta^{(n)}_A([u]_1)$ to be $\ka{\ev^{(0,n)}_A(u)}$.
However, this would not be well-defined since the class $\ka{\ev^{(0,n)}_A(u)}$ in $\Ka(A)$ is not invariant under homotopy, and thus we need a correcting term built using the de la Harpe--Skandalis determinant.  Write $\ev^{(1,n)}_A \colon (A\otimes \mathbb I_n)^\dagger \to \mathbb C1_A \subseteq A$ for the map induced by evaluation at $1$ (again with the normalization convention $\ev^{(1,n)}_A(1_{(A\otimes\mathbb I_n)^\dagger})=1_A$).  Regarding $u$ as element of $U_\infty(C([0,1],A))$, one obtains $\widetilde\Delta_A(u)\in\Aff T(A)$ as in Proposition~\ref{Determinant}. Now define $\widetilde{\zeta}^{(n)}_A\colon U_\infty((A\otimes \mathbb I_n)^\dagger) \to \Ka(A)$ by
\begin{equation}\label{NewMap.1}
\widetilde{\zeta}^{(n)}_A(u) \coloneqq [\ev^{(0,n)}_A(u)]_\alg - [\ev^{(1,n)}_A(u)]_\alg + \Th_A\big( \tfrac{1}{n} \widetilde \Delta_A(u) \big).
\end{equation}
The factor of $\frac1n$ arises in \eqref{NewMap.1} to account for the fact that when $u\in (A\otimes \mathbb I_n)^\dagger$, the computation in the last term is done viewing $u$ as a path in $M_{n}(A)$, while the normalization conventions use elements of $A$ for the first two terms.

Note that the last two terms in \eqref{NewMap.1} lie in the kernel of $\minusa_A$ (as $\ev^{(1,n)}_A(u)\in U_\infty(\mathbb C)$, and $\Th_A$ takes values in $\ker\minusa_A$).  Accordingly,
\begin{equation}
\label{eq:zetanurelation}
\minusa_A\big( \widetilde{\zeta}^{(n)}_A(u) \big) = [\ev^{(0,n)}_A(u)]_1 = \nu_{0, A}^{(n)}([u]_1).
\end{equation}

\begin{proposition}
  \label{prop:zeta-def}
  Let $A$ be a unital $C^*$-algebra and let $n\geq 2$. Then the map $\widetilde{\zeta}^{(n)}_A$ in \eqref{NewMap.1} 
  induces a group homomorphism $\zeta_A^{(n)}\colon K_0(A;\Zn{n}) \to \Ka(A)$ that is
  natural in $A$ and factorizes the Bockstein map $\nu^{(n)}_{0,A}$ through $\Ka(A)$ as $\nu^{(n)}_{0,A}=\minusa_A\circ\zeta^{(n)}_A$.
\end{proposition}

\begin{proof}
  As $\widetilde{\Delta}_A$ gives a homomorphism from
  $U_{\infty}(C([0,1],A))$ to $\Aff T(A)$ by Proposition~\ref{Determinant}, we have
  $\widetilde\zeta^{(n)}_A(uv)=\widetilde\zeta^{(n)}_A(u)+\widetilde\zeta^{(n)}_A(v)$
  for all $u,v\in U_\infty((A\otimes \mathbb
  I_n)^\dagger)$. Therefore, to show $\widetilde\zeta^{(n)}_A$ induces
  a well-defined homomorphism $\zeta^{(n)}_A$ on $K_0(A;\Zn{n})$ it
  suffices to check that $\widetilde{\zeta}^{(n)}_A(e^{2\pi ih})=0$ in
  $\Ka(A)$ for a self-adjoint $h\in M_k(A\otimes\mathbb I_n)^\dagger$ and $k\in\mathbb N$.
  For $\tau\in T(A)$, Proposition~\ref{Determinant}\ref{Determinant.C2} gives
  \begin{equation}
    \begin{aligned}
      \tfrac{1}{n}\widetilde{\Delta}_A(e^{2\pi ih})(\tau)
      =\tfrac{1}{n}\tau_{nk}\big(h(1)-h(0)\big)
      = \tau_k\big( \ev^{(1,n)}_A(h)-\ev^{(0,n)}_A(h) \big)
    \end{aligned}
  \end{equation}
  where $\tau_{nk}$ and $\tau_k$ are the non-normalized extensions of
  $\tau$ to $M_{kn}(A)$ and $M_k(A)$, and the second equality is a
  consequence of our normalization conventions. Thus,
  $\ev^{(1,n)}_A(h)-\ev^{(0,n)}_A(h)$ induces the affine function $
  \tfrac{1}{n}\widetilde{\Delta}_A(e^{2\pi ih})$, and so the definition of the
  Thomsen map in \eqref{DefThomsen} (working at the matrix
  amplification level as discussed in the text above \eqref{eq:ThRho0}) 
  gives
  \begin{equation}
    \begin{array}{rcl}
      \Th_A(\tfrac{1}{n}\widetilde{\Delta}_A(e^{2\pi ih}))&=& \ka{e^{2\pi i(\ev_A^{(1,n)}(h)-\ev_A^{(0,n)}(h))}}\\
      &\stackrel{\eqref{ExpCommute}}{=}& \ka{\ev^{(1,n)}_A(e^{2\pi
          ih})}-\ka{\ev^{(0,n)}_A(e^{2\pi ih})}.
    \end{array}
  \end{equation}
  Hence $\tilde\zeta^{(n)}_A(e^{2\pi ih})=0$ and so $\zeta^{(n)}_A$ is well-defined. Now
  $\nu^{(n)}_{0,A}=\minusa_A\circ\zeta_A^{(n)}$ follows from \eqref{eq:zetanurelation}.

  Naturality of $\zeta^{(n)}$ follows as all components of the
  construction of $\tilde\zeta^{(n)}$ in \eqref{NewMap.1} are natural.
\end{proof}

The next proposition shows that the maps $\zeta_A^{(n)}$ satisfy three
further compatibility relations with the other structure maps.

\begin{proposition}
  \label{p:newmap}
  If $A$ is a unital $C^*$-algebra and $m,n\geq 2$, then
  the diagrams
   \begin{equation}\label{eq:newmap1}
    \begin{tikzcd}
      K_0(A) \ar[r, "\mu^{(n)}_{0,A}"] \ar[d, "\tfrac{1}{n} \rho_A"]
      &
      K_0(A; \Zn{n}) \ar[d, "\zeta^{(n)}_A"] \ar[r, "\nu^{(n)}_{0,A}"]
      &
      K_1(A) \ar[d, equals]
      \\
      \Aff T(A) \ar[r, "\Th_A"]
      &
      \Ka(A) \ar[r, "\minusa_A"]
      &
      K_1(A)
    \end{tikzcd}
  \end{equation}
  and
  \begin{equation}\label{eq:newmap2}
    \begin{tikzcd}[column sep=large]
      K_0(A; \mathbb Z/nm) \ar[r, "\kappa_{0,A}^{(n,nm)}"]
      \ar[dr, "m\zeta^{(nm)}_A",swap]
      &
      K_0(A; \Zn{n}) \ar[r, "\kappa^{(nm,n)}_{0,A}"]
      \ar[d, "\zeta^{(n)}_A"]
      &
      K_0(A; \Zn{nm}) \ar[dl, "\zeta^{(nm)}_A"]
      \\
      &
      \Ka(A)
      &
    \end{tikzcd}
  \end{equation}
  commute.
\end{proposition}
\begin{proof}
Proposition~\ref{prop:zeta-def} already shows that the right square in \eqref{eq:newmap1} commutes.

We first show the commutativity of the left square of \eqref{eq:newmap1}.  Recall that after identifying $K_0(A) \cong K_1(C_0((0,1),A))$ by Bott periodicity, $\mu^{(n)}_{0,A}$ is given by the inclusion of $C((0,1),A\otimes M_n)$ into $A\otimes\mathbb I_n$. Therefore, given a unitary $u\in U_\infty(C((0,1),A)^\dagger)$, and regarding it as a path $[0,1]\to U_\infty(A)$ with $u(0)=u(1)\in U_\infty(\mathbb C)$,  $\mu^{(n)}_{0,A}([u]_1)$ is the class of $u$ in $K_1(A\otimes \mathbb I_n) = K_0(A;\Zn{n})$. By Proposition~\ref{Determinant}\ref{Determinant.C3}, $\widetilde{\Delta}_A(u)=\rho_A([u]_1)$ (where we again identify $K_0(A) \cong K_1(C_0((0,1),A))$). Since $\ev^{(0,n)}_A(u)=\ev^{(1,n)}_A(u)$, we have
\begin{equation}
  \zeta^{(n)}_A (\mu^{(n)}_{0,A}([u]_1))
  = \widetilde \zeta^{(n)}_A(u) = \Th_A \big(\tfrac{1}{n} \rho_A([u]_1)\big).
\end{equation}

Next we deal with the right triangle of  \eqref{eq:newmap2}.  Recall that $\kappa^{(nm,n)}_A$ is induced by the canonical inclusion $\iota\colon \mathbb I_n \stackrel{\cong}{\longrightarrow} \mathbb I_n\otimes 1_{M_m} \hookrightarrow \mathbb I_{nm}$. For $u\in U_\infty((A\otimes \mathbb I_{n})^\dagger)$ and $i=0,1$, we have $\ev^{(i,nm)}_A(\iota(u)) = \ev^{i,n}_A(u)$. Since $\widetilde \Delta_A(u\otimes 1_{M_m}) = m \widetilde \Delta_A(u)$, it follows that $(\zeta^{(nm)}_A\circ \kappa^{(nm,n)}_A)([u]_1)= \zeta^{(n)}_A([u]_1)$.

Finally, we handle the left-hand triangle of  \eqref{eq:newmap2}. Since $\kappa^{(n,nm)}_A$ is induced by the inclusion $\mathbb I_{nm} \subseteq \mathbb I_n\otimes M_m$, for $u\in U_k((A\otimes \mathbb I_{nm})^\dagger)$, our normalization conventions give $\ev^{(i,n)}_{M_m(A)} (u) = \ev^{(i,nm)}_A(u) \otimes 1_{M_m}$ for $i=0,1$. Then
\begin{gather}
  \raisetag{-50pt}\hspace{30pt}\begin{multlined}
    \zeta^{(n)}_A \left(\kappa^{(n,nm)}_A\big([u]_1\big)\right) \\
    \begin{aligned}
      \qquad & =\big[ \ev^{(0,n)}_{M_m(A)}(u) \big]_\alg -
      \big[\ev^{(1,n)}_{M_m(A)}(u) \big]_\alg + \Th_A\big(\tfrac{1}{n}\widetilde \Delta_A(u)\big) \\
  &= m \Big(\big[\ev^{(0,nm)}_A(u)\big]_\alg - \big[\ev^{(1,nm)}_A(u)\big]_\alg + \Th_A\big(\tfrac{1}{nm}\widetilde \Delta_A(u)\big)\Big) \\
  &= m \zeta^{(nm)}_A\big([u]_1\big),
\end{aligned}
\end{multlined}
\end{gather}

\vspace*{-1.5em}
\end{proof}

\begin{remark}
\label{rmk:NewMapProjections}
As $\zeta_A^{(n)}\circ \mu_{0,A}^{(n)}  =  \Th_A \circ \left(\tfrac{1}{n} \rho_A\right) \colon K_0(A) \to \Ka(A)$ by Proposition~\ref{p:newmap}, it follows that if $p\in M_n(A)$ is a projection, then
\begin{equation}\label{NewMapEasyExample}
\zeta_A^{(n)}\big(\mu_{0,A}^{(n)}([p]_0) \big) = \Th_A (\tfrac{1}{n} \hat p) = [e^{2\pi i p/n}]_\alg.
\end{equation}
This shows that $\zeta_A^{(n)}\circ\overline{\mu}^{(n)}_{0,A}$ is the map $K_0(A)\otimes\Zn{n}\to\Ka(A)$ described in the opening paragraph of this subsection. When $K_1(A)$ has no $n$-torsion elements, $\bar\mu^{(n)}_{0,A}$ is an isomorphism, so $\zeta^{(n)}_A$ is completely determined by \eqref{NewMapEasyExample}. In particular, if $A\coloneqq \mathbb C$ (or $A\coloneqq \Z$), then identifying $K_0(A;\Zn{n})\cong \Zn{n}$ and $\Ka(A)\cong\mathbb T$, the map $\zeta^{(n)}_A$ is the embedding of $n$-th roots of unity in the circle.
\end{remark}

\begin{remark}\label{NewMapsRemark}
As noted in the introduction, while this paper was being written, Gong, Lin, and Niu independently discovered the need for compatibility with natural maps $K_0(A;\Zn{n})\to \Ka(A)$ in \cite{Gong-Lin-etal23} --- see Remark 6.6 therein. The notation in \cite{Gong-Lin-etal23} is in terms of the unitary group modulo the closure of the commutators, and there is often a finite stable rank condition which allows them to work with $U_n(A)$ instead of $U_\infty(A)$ (see Remark \ref{rem2.8}). (It is a happy coincidence that the same symbol $\zeta$ was adopted in \cite{Gong-Lin-etal23} and by us.)

Gong, Lin, and Niu construct their maps abstractly through a splitting of the Thomsen extension for $A\otimes \mathbb I_n$. It turns out that the two maps are identical (taking care of our different conventions regarding the orientation of the interval in $\mathbb I_n$). 
In fact, since this paper was submitted, Munkholm has shown a general uniqueness statement for maps making the right square of \eqref{eq:newmap1} commute (\cite{Munkholm25}), which immediately demonstrates that our maps agree with those of Gong, Lin, and Niu.
\end{remark}

\subsection{The total invariant}\label{ss:totalinv}

With all the ingredients in place, we now formalize the invariant needed to classify embeddings.
Recall our convention that the Bockstein maps $\mu^{(n)}_{i,A},\nu^{(n)}_{i,A},\kappa^{(m, n)}_{i,A}$ form part of $\underline{K}(A)$.

\begin{definition}\label{DefInv}
Let $A$ be a unital $C^*$-algebra.  The \emph{total invariant} of $A$, denoted $\inv(A)$, is 
\begin{equation}\label{DefInvE.1}
\inv(A)\coloneqq (\underline{K}(A),\Aff T(A),\Ka(A),[1_A]_0,\rho_A,\Th_A,\minusa_A,(\zeta^{(n)}_A)_{n\geq 2})
\end{equation}
with the convention that $\Aff T(A)=0$ if $T(A)=\emptyset$. A morphism $(\underline{\alpha},\beta,\gamma):\inv(A)\to\inv(B)$ consists of
\begin{itemize}
\item $\underline{\alpha}\in\Hom_\Lambda(\underline{K}(A),\underline{K}(B))$ with $\alpha_0([1_A]_0)=[1_B]_0$,
\item a positive linear map $\gamma\colon \Aff
  T(A)\rightarrow\Aff T(B)$, and
\item a group homomorphism $\beta\colon \Ka(A)\rightarrow\Ka(B)$,
\end{itemize}
such that
    \begin{equation}\label{eq:compatibility}
      \begin{tikzcd}
        K_0(A) \ar[r, "\rho_A"] \ar[d, "\alpha_0"] &
        \Aff T(A) \ar[r, "\Th_A"] \ar[d, "\gamma"] &
        \Ka(A) \ar[r, "\minusa_A"] \ar[d, "\beta"] &
        K_1(A) \ar[d, "\alpha_1"] \\
        K_0(B) \ar[r, "\rho_B"] &
        \Aff T(B) \ar[r, "\Th_B"] &
        \Ka(B) \ar[r, "\minusa_B"] &
        K_1(B)
      \end{tikzcd}
    \end{equation}
    and
    \begin{equation}
      \label{eq:compatibility2}
      \begin{tikzcd}
        K_0(A; \Zn{n}) \ar[r, "\zeta^{(n)}_A"] \ar[d, "\alpha^{(n)}_0"] & \Ka(A) \ar[d, "\beta"] \\
        K_0(B;\Zn{n}) \ar[r, "\zeta^{(n)}_B"] & \Ka(B)
      \end{tikzcd}
    \end{equation}
commute.

As with $KT_u$-morphisms, we say that
$(\underline{\alpha},\beta,\gamma)\colon \inv(A)\rightarrow\inv(B)$ is
\emph{faithful} if the induced map $\gamma^*\colon T(B)\rightarrow
T(A)$ (given by Kadison duality) satisfies that $\gamma^*(\tau)$ is faithful
for all $\tau\in T(B)$.
\end{definition}

Note that the tracial component of a $\inv$-morphism $(\underline{\alpha},\gamma)$ is unital as \begin{equation} \gamma(1)=\gamma(\rho_A([1_A]_0))\stackrel{\eqref{eq:compatibility}}=\rho_B(\alpha_0([1_A]_0))=\rho_B([1_B]_0)=1. \end{equation}

It is important for us that $\inv$ is a functor; the point is not the exact specification of the target category (this can be done in similar fashion to the end of Definition \ref{def:KTu}), but rather that unital $^*$-homomorphisms do induce $\inv$-morphisms by naturality of the $\zeta^{(n)}$ (together with functoriality of $KT_u$).  This is vital in the proof of the existence portion of the classification of full approximate embeddings in the presence of torsion in $K_1$.  Moreover, $\inv(\,\cdot\,)$ is invariant under approximate unitary equivalence as each component is.

\begin{proposition}\label{Inv:AFunctor}
Given a unital $^*$-homomorphism $\phi\colon A\to B$,
\begin{equation}
\inv(\phi)\coloneqq (\underline{K}(\phi),\Ka(\phi),\Aff T(\phi))\colon\inv(A)\to\inv(B)
\end{equation} 
is a $\inv$-morphism. In this way, the assignment of $\inv(A)$ to a unital $C^*$-algebra $A$ provides a functor from the category of unital $C^*$-algebras and unital $^*$-ho\-mo\-morphisms. Furthermore, if $\phi,\psi\colon A\to B$ are approximately unitarily equivalent, then $\inv(\phi)=\inv(\psi)$.
\end{proposition}

\begin{remark}
\label{rmk:TotInvPI}
Let $B$ be a unital simple purely infinite $C^*$-algebra.  Then
$T(B)=\emptyset$ and exactness of \eqref{ThomsenExtension} shows that $\minusa_B$ is an isomorphism. Thus, for these
codomains, the functors $\inv(\,\cdot\,)$ and $(\underline{K}(\,\cdot\,),[1]_0)$
carry the same information.
More precisely, given any unital
$C^*$-algebra $A$, there is a one-to-one correspondence between
$\underline{\alpha}\colon\underline{K}(A)\to\underline{K}(B)$ with
$\alpha_0([1_A]_0)=[1_B]_0$ and morphisms
$(\underline{\alpha},\beta,\gamma)\colon\inv(A)\to\inv(B)$, given by
defining $\beta\coloneqq\minusa_B^{-1}\circ\alpha_1\circ\minusa_A$ and
$\gamma\coloneqq 0$. This allows the classification of nuclear embeddings of
unital separable exact $C^*$-algebras with the UCT into unital
Kirchberg algebras by total $K$-theory (see \cite{Rordam95,Phillips00,Kirchberg,Gabe-Preprint} --- in this level of generality, the result is recorded as \cite[Theorem 8.12]{Gabe-Preprint}) to be viewed as a classification using $\inv$.
\end{remark}

\begin{remark}
  For each integer $n \geq 2$, there are examples of triples
  $(\underline{\alpha}, \beta, \gamma)$ that satisfy
  \eqref{eq:compatibility} but not \eqref{eq:compatibility2}.  Set $A
  \coloneqq \mathbb I_n^\dag$ and let
  $B \coloneqq \Z$.  Let $\psi \colon \mathbb I_n^\dag \rightarrow
  \mathcal Z$ be an embedding inducing the Lebesgue trace on $\mathbb
  I_n^\dag$ (such an embedding exists, for example, by \cite[Theorem 4.11]{Schemaitat19}).
 Define $\underline{\alpha}
  \coloneqq \underline{K}(\psi)$ and $\gamma \coloneqq \Aff T(\psi)$.  As
  \begin{equation}
    K_1(\mathbb I_n^\dag) \cong \mathbb Z /n \quad \text{and} \quad
    \Aff T(\mathcal Z)/\overline{\im\rho_\mathcal Z} \cong \mathbb R / \mathbb Z,
  \end{equation}
  there is an embedding $r \colon K_1(\mathbb I_n^\dag) \rightarrow
  \Aff T(\mathcal Z)/\overline{\im\rho_\mathcal Z}$.  Define
  \begin{equation}
    \beta \coloneqq \Ka(\psi) + \overline{\Th}_\mathcal Z \circ r \circ \minusa_{\mathbb I_n^\dag} \colon \Ka(\mathbb I_n^\dag) \rightarrow \Ka(\mathcal Z).
  \end{equation}
  Then $\beta$ satisfies \eqref{eq:compatibility}.  Also, by \eqref{eq:newmap1} and the fact that $\inv(\psi)$ is an $\inv$-morphism (Proposition~\ref{Inv:AFunctor}), we have
  \begin{equation}
    \beta \circ \zeta_{\mathbb I_n^\dag}^{(n)} = \zeta_\mathcal Z^{(n)} \circ \alpha_0^{(n)} + \overline{\Th}_\mathcal Z \circ r \circ \nu_{0, \mathbb I_n^\dag}^{(n)}.
  \end{equation}
  Since $K_1(\mathbb I_n^\dag) \cong \mathbb Z/n$, exactness of \eqref{eq:bockstein-2} implies that $\nu_{0, \mathbb I_n^\dag}^{(n)}$ is surjective.
  As $\overline{\Th}_\mathcal Z$ is injective and $r\neq 0$, it follows that $\overline{\Th}_{\mathcal Z} \circ r \circ \nu_{0, \mathbb I_n^\dag}^{(n)} \neq 0$.
   Hence $\beta \circ \zeta_{\mathbb I_n^\dag}^{(n)} \neq \zeta_\mathcal Z^{(n)} \circ \alpha_0^{(n)}$; i.e., \eqref{eq:compatibility2} does not commute.
\end{remark}

For unital $C^*$-algebras $A$ and $B$, a morphism
$(\underline{\alpha},\beta,\gamma)\colon \inv(A)\to\inv(B)$ certainly
restricts to a morphism $(\alpha_*,\gamma)\colon KT_u(A)\to KT_u(B)$.
To obtain classification results for $C^*$-algebras using $KT_u$ from
those for $\inv$, we need to be able to go in the other direction and
extend (in a non-canonical fashion) $KT_u$-(iso)morphisms to
$\inv$-(iso)morphisms.  Results of this form (often using the UCT to
extend morphisms from $K$-theory to total $K$-theory rather than
via B{\"o}digheimer's result, Proposition~\ref{TotalKExtend}), have been
used in the classification program previously in cases where
$\underline{K}$ or $\Ka$ is needed to obtain uniqueness of
morphisms --- see the proof of  \cite[Theorem~21.9]{GLN-part1}, for example.
The new ingredient we need is to construct an extension that intertwines the maps $\zeta^{(n)}$ of the previous
subsection.

\begin{theorem}[Extension of (iso)morphisms from $KT_u$ to
  $\underline{K}T_u$]
  \label{ThmExtension} Let \allowbreak$A$ and $B$ be unital $C^*$-algebras. Then any morphism
  $(\alpha_\ast , \gamma) \colon KT_u(A) \to KT_u(B)$ can be extended
  to a morphism $(\underline \alpha, \beta, \gamma) \colon \inv(A) \to
  \inv(B)$. In addition, if $(\alpha_\ast, \gamma)$ is an isomorphism,
  then any such extension $(\underline \alpha,\beta,\gamma)$ is an
  isomorphism.
\end{theorem}
\begin{proof}
By Proposition~\ref{TotalKExtend}, we may extend $\alpha_\ast$ to $\underline \alpha \colon \underline K(A) \to \underline K(B)$, and by Proposition~\ref{PropExtendAlgK1}, there exists $\beta' \colon \Ka(A) \to \Ka(B)$ such that
\begin{equation}\label{Extension:E1}
\beta' \circ \Th_A = \Th_B \circ \gamma\quad\text{ and }\quad\minusa_B \circ \beta' = \alpha_1 \circ \minusa_A.
\end{equation}
Therefore, $\beta'$ makes the first diagram \eqref{eq:compatibility} commute
from the definition of a $\underline{K}T_u$-morphism, but it need not
make \eqref{eq:compatibility2} commute.  The rest of the proof
consists of identifying an appropriate correcting term to address this.

For $n\geq 2$, consider the maps
\begin{equation}\label{eq:zetacomm}
\beta' \circ \zeta^{(n)}_A - \zeta^{(n)}_B \circ \alpha^{(n)}_0 \colon K_0(A;\Zn{n}) \to \Ka(B).
\end{equation}
The plan is to prove that these maps factor through maps $s_n \colon \mathrm{Tor}(K_1(A),
\Zn{n}) \to \ker\minusa_B$ as compositions $s_n\circ \overline \nu^{(n)}_{0,A}$, and moreover that the resulting $s_n$ are compatible, and so define a 
homomorphism $s \colon \mathrm{Tor}(K_1(A)) \to \ker\minusa_B$. This will provide the needed correcting term.

Fix $n\geq 2$ for the next two paragraphs. Since $\underline \alpha$ is a $\Lambda$-homomorphism and
$(\alpha_\ast , \gamma)$ is a morphism on $KT_u$, we get
\begin{equation}
\begin{array}{rcl}
  \zeta_B^{(n)}\circ\alpha_0^{(n)} \circ\mu_{0,A}^{(n)}
  &\stackrel{(\Lambda)}=&
  \zeta_B^{(n)} \circ\mu_{0,B}^{(n)} \circ\alpha_0 \\
  &\stackrel{\eqref{eq:newmap1}}{=}&
  \Th_B \circ\tfrac{1}{n} \rho_B \circ\alpha_0 \\
  &\stackrel{(KT_u)}=&
  \Th_B \circ\tfrac{1}{n} \gamma \circ\rho_A.
\end{array}
\end{equation}
  Since $\gamma$ is $\mathbb R$-linear, it follows that $\tfrac{1}{n} \gamma = \gamma \circ \tfrac{1}{n}$. 
Hence, continuing the
computations, we get
\begin{equation}
\begin{array}{rcl}
  \Th_B\circ \tfrac{1}{n} \gamma \circ \rho_A
  &=& \Th_B \circ \gamma \circ \tfrac{1}{n} \rho_A \\
  &\stackrel{\eqref{Extension:E1}}=&
  \beta' \circ \Th_A \circ \tfrac{1}{n} \rho_A \\
  &\stackrel{\eqref{eq:newmap1}}{=}&
  \beta'\circ \zeta_A^{(n)} \circ \mu^{(n)}_{0,A}.
\end{array}
\end{equation}
Hence the map $\beta' \circ \zeta^{(n)}_A - \zeta^{(n)}_B \circ \alpha^{(n)}_0$
vanishes on the image of $\mu^{(n)}_{0,A}$. By exactness of \eqref{eq:bockstein-2}, this image is $\ker\nu^{(n)}_{0,A}$, and so the
map factors through $\mathrm{Tor}(K_1(A), \Zn{n})$ via the map $\overline \nu^{(n)}_{0,A} \colon
K_0(A;\Zn{n}) \twoheadrightarrow \mathrm{Tor}(K_1(A),\Zn{n})$ from
\eqref{eq:bockstein-new}.

Similarly, we find that
\begin{equation}
\begin{array}{rcl}
  \minusa_B \circ \zeta^{(n)}_B \circ \alpha_0^{(n)}
  &\stackrel{\eqref{eq:newmap1}}{=}&
  \nu_{0,B}^{(n)} \circ \alpha_0^{(n)}\\
  &
  \stackrel{(\Lambda)}=&
  \alpha_1 \circ \nu^{(n)}_{0,A} \\
  &\stackrel{\eqref{eq:newmap1}}{=}&
  \alpha_1 \circ \minusa_A \circ \zeta^{(n)}_A \\
  &\stackrel{\eqref{Extension:E1}}=&
  \minusa_B \circ \beta' \circ \zeta^{(n)}_A
\end{array}
\end{equation}
and so the range of $\beta' \circ \zeta^{(n)}_A - \zeta^{(n)}_B\circ
\alpha^{(n)}_0$ is contained in $\ker \minusa_B \subseteq
\Ka(B)$.
It follows that there is a unique homomorphism
\begin{equation}
  s_n \colon \mathrm{Tor}(K_1(A), \Zn{n}) \longrightarrow \ker \minusa_B \subseteq \Ka(B)
\end{equation}
such that
\begin{equation}\label{eq:rndef}
  s_n \circ \overline \nu^{(n)}_{0,A}
  =
  \beta' \circ\zeta^{(n)}_A - \zeta^{(n)}_B \circ\alpha^{(n)}_0.
\end{equation}

Allowing $n$ to vary, we identify the torsion subgroup
$\mathrm{Tor}(K_1(A))$ of $K_1(A)$ with the union $\bigcup_{n\geq 2}
\mathrm{Tor}(K_1(A), \Zn{n})$. Let $s_{nm,n}$ denote the restriction
of $s_{nm}$ to $\mathrm{Tor}(K_1(A), \Zn{n})$. We will show that
$s_{nm,n} = s_n$ for all $n,m\geq 2$, so that these maps induce a
homomorphism $s\colon \mathrm{Tor}(K_1(A)) \to \ker \minusa_B$.
By uniqueness of $s_n$ in \eqref{eq:rndef}, it suffices to show that
\begin{equation}\label{eq:Thrnmnufinal}
  s_{nm,n} \circ \overline \nu^{(n)}_{0,A}
  =
  \beta' \circ \zeta^{(n)}_A - \zeta^{(n)}_B \circ \alpha^{(n)}_0.
\end{equation}
As noted in \eqref{eq:bockstein-commute}, $\nu^{(nm)}_{0,A}
\circ \kappa^{(nm,n)}_{0,A} = \nu^{(n)}_{0,A} \colon K_0(A;\Zn{n}) \to
K_1(A)$, so it follows that $\overline \nu^{(n)}_{0,A}$ is the
corestriction of $\overline \nu^{(nm)}_{0,A} \circ \kappa^{(nm,n)}_A$ to
$\mathrm{Tor}(K_1(A), \Zn{n})$.  Hence
\begin{equation}\label{eq:Thrnmnu}
  \begin{aligned}
    s_{nm,n} \circ \overline \nu^{(n)}_{0,A} &\stackrel{\phantom{\eqref{eq:rndef}}}{=} s_{nm} \circ \overline \nu^{(nm)}_{0,A} \circ \kappa^{(nm,n)}_A\\
    &\stackrel{\eqref{eq:rndef}}{=} (\beta' \circ \zeta^{(nm)}_A -
    \zeta^{(nm)}_B \circ \alpha^{(nm)}_0)\circ \kappa^{(nm,n)}_A.
  \end{aligned}
\end{equation}
By commutativity of \eqref{eq:newmap2}, we have $\beta'\circ \zeta^{(nm)}_A \circ \kappa^{(nm,n)}_A = \beta'\circ \zeta^{(n)}_A$. Also
\begin{equation}
\zeta^{(nm)}_B \circ \alpha_0^{(nm)} \circ \kappa^{(nm,n)}_{0,A} \stackrel{(\Lambda)}= \zeta^{(nm)}_B \circ \kappa^{(nm,n)}_{0,B} \circ \alpha_0^{(n)} \stackrel{\eqref{eq:newmap2}}= \zeta^{(n)}_B \circ \alpha^{(n)}_0.
\end{equation}
Combining these last two equations with \eqref{eq:Thrnmnu}, we obtain
\eqref{eq:Thrnmnufinal}, as desired. In conclusion, the maps $s_n$ for
$n\geq 2$ induce a homomorphism
\begin{equation}
  s\colon \mathrm{Tor}(K_1(A)) \longrightarrow \ker \minusa_B.
\end{equation}

Since $\Aff T(B) /\overline{\im\rho_B}$ is divisible, and
hence injective (see \cite[Corollary 2.3.2]{Weibel94}), so is $\ker
\minusa_B$, because the two groups are isomorphic via $\overline{\Th}_B$.
Accordingly, $s$ can be extended to a homomorphism
\begin{equation}
\label{eq:ExtensionDefhats}
  \hat s \colon K_1(A) \longrightarrow \ker \minusa_B \subseteq \Ka(B).
\end{equation}
We now use this to correct the definition of $\beta'$ by setting
\begin{equation}\label{Extension:E3}
  \beta \coloneqq \beta' - \hat s \circ \minusa_A \colon \Ka(A) \to \Ka(B).
\end{equation}

It remains to prove that $(\underline \alpha, \beta, \gamma)$ is a morphism
of $\inv$.  We already know that $(\alpha_\ast , \gamma)$ is a
morphism of $KT_u$, and that $\underline \alpha$ is a
$\Lambda$-morphism, so we only need to show that the last two
squares of \eqref{eq:compatibility} commute and that
\eqref{eq:compatibility2} commutes.  Since $\minusa_A\circ \overline{\Th}_A
= 0$, we have
\begin{equation}
 \beta\circ \Th_A\stackrel{\eqref{Extension:E3}}{=}\beta'\circ \Th_A \stackrel{\eqref{Extension:E1}}{=} \Th_A \circ \gamma.
\end{equation}
Moreover,
\begin{equation}
\minusa_B\circ \beta\stackrel{\eqref{eq:ExtensionDefhats}}=\minusa_B\circ \beta' \stackrel{\eqref{Extension:E1}}{=} \alpha_1 \circ \minusa_A.
\end{equation}
Therefore, the last two squares of \eqref{eq:compatibility} commute.

For commutativity of \eqref{eq:compatibility2}, fix $n\geq 2$.
Recall that $\hat s$ is an extension of $s_n$ and
that $\overline \nu^{(n)}_{0,A}$ is the corestriction of
$\nu^{(n)}_{0,A}$ to $\mathrm{Tor}(K_1(A),\Zn{n})$, which is the
domain of $s_n$.  It follows that $\hat s \circ \nu^{(n)}_{0,A} = s_n
\circ \overline \nu^{(n)}_{0,A}$. Hence
\begin{equation}\label{Extension:E4}
  \hat s \circ \minusa_A \circ \zeta^{(n)}_A
  \stackrel{\eqref{eq:newmap1}}=
  s_n \circ \overline \nu^{(n)}_{0,A}
  \stackrel{\eqref{eq:rndef}}{=}
  \beta'\circ \zeta^{(n)}_A - \zeta^{(n)}_B \circ \alpha^{(n)}_0,
\end{equation}
and so
\begin{equation}
\begin{aligned}
\beta \circ \zeta^{(n)}_A 
&\stackrel{\eqref{Extension:E3}}= 
(\beta' -\hat s\circ \minusa_A)\circ \zeta^{(n)}_A \\
&\stackrel{\eqref{Extension:E4}}{=} \beta' \circ \zeta^{(n)}_A - (\beta' \circ \zeta^{(n)}_A - \zeta^{(n)}_B \circ \alpha^{(n)}_0) = \zeta^{(n)}_B \circ \alpha^{(n)}_0,
\end{aligned}
\end{equation}
establishing commutativity of \eqref{eq:compatibility2}.  Therefore
$(\underline \alpha, \beta, \gamma)$ defines a well-defined morphism
$\inv(A) \to \inv(B)$.

Finally, Propositions~\ref{PropExtendAlgK1} and~\ref{TotalKExtend}
ensure that $\underline \alpha$ and $\beta$ are isomorphisms if
$\alpha_\ast$ and $\gamma$ are, which implies that $(\underline
\alpha, \beta, \gamma)$ is an isomorphism whenever $(\alpha_\ast,
\gamma)$ is.
\end{proof}

\begin{remark}\label{ThmExtension:Rem}
  Note that the proof of Theorem~\ref{ThmExtension} shows that given a $KT_u$-morphism $(\alpha_*,\gamma):KT_u(A) \to KT_u(B)$ and any $\underline{\alpha}:\underline{K}(A) \to \underline{K}(B)$ extending $\alpha_*$, there is an $\inv$-morphism of the form $(\underline{\alpha},\beta,\gamma)$.
\end{remark}

The following lemma is used in the computation of $KL(A,J_B)$ in Theorem~\ref{intro:calcKL} to show that every $\inv$-morphism in part \ref{intro:calcKL.2} comes from an appropriate $KL$-class.
Its proof illustrates the role played by the maps $\zeta^{(n)}$.

\begin{lemma}\label{Compatibility.Lem}
  Let $A$ and $B$ be unital $C^*$-algebras. If
  \begin{equation}
  (\underline \alpha, \beta, \gamma), (\underline \alpha' , \beta' ,
  \gamma') \colon \inv(A) \longrightarrow \inv(B)
  \end{equation}
  are morphisms such that
  $\underline \alpha = \underline \alpha'$ and $\gamma =
  \gamma'$, then there is a unique homomorphism 
  \begin{equation}
  r\colon K_1(A)/\mathrm{Tor}(K_1(A))\longrightarrow\ker \minusa_B
\end{equation}
such that $\beta-\beta'\colon \Ka(A)\to\Ka(B)$ is the composition
\begin{equation}
\begin{tikzcd}[column sep = 3.8ex]
    \Ka(A) \ar[r, "\minusa_A"] & K_1(A) \ar[r, two heads] & K_1(A) / \mathrm{Tor}(K_1(A)) \ar[r, "r"] & \ker \minusa_B \ar[r,tail] & \Ka(B).
\end{tikzcd}
\end{equation}
\end{lemma}
\begin{proof}
  We have $(\beta - \beta') \circ \Th_A = \Th_B \circ (\gamma - \gamma') = 0$.  It
  follows that $\beta- \beta'$ vanishes on $\im \Th_A = \ker
  \minusa_A$ and thus factors uniquely as $\beta-\beta'=\hat r \circ \minusa_A$, for some map $\hat r \colon
  K_1(A) \to \Ka(B)$. It suffices to show that $\hat r$ vanishes on
  $\mathrm{Tor}(K_1(A))$ and takes values in $\ker \minusa_B$. Since
  for any $n\geq 2$ we have
  \begin{equation}
\begin{array}{rcl}
    \hat r \circ \nu_{0,A}^{(n)}
    &\stackrel{\eqref{eq:newmap1}}{=}&
    \hat r \circ \minusa_A \circ \zeta_A^{(n)}
    = (\beta-\beta') \circ \zeta_A^{(n)} \\
    &\stackrel{\eqref{eq:compatibility2}}{=}&
    \zeta_B^{(n)} \circ (\alpha_0^{(n)}-\alpha_0'^{(n)})
    = 0,
\end{array}
  \end{equation}
  it follows that $\hat r$ vanishes on the image of $\nu_{0,A}^{(n)}$.
By exactness of \eqref{eq:bockstein-2}, this image is $\mathrm{Tor}(K_1(A), \Zn{n})$, and as such, $\hat r$ vanishes on
$\mathrm{Tor}(K_1(A))$. Notice that
\begin{equation}
  \minusa_B \circ \hat r \circ \minusa_A
  = \minusa_B \circ (\beta - \beta')
  \stackrel{\eqref{Extension:E1}}{=}
  (\alpha_1 - \alpha_1') \circ \minusa_A
  = 0.
\end{equation}
Since $\minusa_A$ is surjective, it follows that $\hat r$
takes values in $\ker \minusa_B$.
\end{proof}

Although we will not use it in this paper, we note that 
compatibility with the maps $\zeta^{(n)}$ is automatic when $K_1(A)$ is torsion-free. Gong, Lin, and Niu made a similar observation, by means of a different approach in their framework (see \cite[Corollary 5.13 with condition (4)]{Gong-Lin-etal23}).

\begin{proposition}
\label{prop:NewMapsTorsion}
  Suppose that $A$ and $B$ are unital $C^*$-algebras with $K_1(A)$
  torsion-free. Then the condition of commutativity of \eqref{eq:compatibility2} 
  is automatic in the definition of a
  $\inv$-morphism.
\end{proposition}
\begin{proof}
  Let $(\underline{\alpha},\beta,\gamma)$ be as in the definition of a
  $\inv$-morphism, except we only assume that \eqref{eq:compatibility}
  commutes. We will show that commutativity of \eqref{eq:compatibility2} 
  follows. Fix $n\geq 2$. We have
  \begin{equation}
    \beta \circ \zeta_A^{(n)} \circ \mu_{0,A}^{(n)}
    \stackrel{\eqref{eq:newmap1}}{=}
    \beta \circ \Th_A \circ \tfrac{1}{n} \rho_A
    \stackrel{\eqref{eq:compatibility}}{=}
    \Th_B \circ \gamma \circ \tfrac{1}{n}  \rho_A .
  \end{equation}
  Since $\gamma$ is $\mathbb R$-linear we have $\gamma \circ \tfrac{1}{n} =
  \tfrac{1}{n} \gamma$. 
Therefore, continuing the computation above,
  we get
  \begin{equation}
\begin{array}{rcl}
    \Th_B \circ \tfrac{1}{n}  \gamma  \circ \rho_A
    &\stackrel{\eqref{eq:compatibility}}{=}&
    \Th_B \circ \tfrac{1}{n}\rho_B  \circ \alpha_0 \\
    &\stackrel{\eqref{eq:newmap1}}{=}&
    \zeta_B^{(n)}\circ \mu_{0,B}^{(n)}  \circ \alpha_0
    \stackrel{(\Lambda)}{=}
    \zeta_B^{(n)} \circ \alpha_0^{(n)} \circ \mu_{0,A}^{(n)}.
\end{array}
  \end{equation}
  Multiplication by $n$ is an injective endomorphism on $K_1(A)$
  because $K_1(A)$ is torsion-free, and so $\mu_{0,A}^{(n)} \colon K_0(A)
  \to K_0(A; \Zn{n})$ is surjective by exactness of \eqref{eq:bockstein-2}. 
  Therefore, $\beta \circ \zeta_A^{(n)} =\zeta_B^{(n)} \circ \alpha_0^{(n)}$ as desired.
\end{proof}

\subsection{The sequence algebra}\label{sec:sequence-algebra}

In this short section we define the sequence algebra $B_\infty$ of a
$C^*$-algebra, record Kirchberg's $\epsilon$-test for later use, and observe that $B\hookrightarrow
B_\infty$ is injective on $\inv$.  

\begin{definition}
Let $B$ be a $C^*$-algebra.  The \emph{sequence algebra} of $B$, denoted $B_\infty$, is the quotient $\ell^\infty(B)/c_0(B)$ of bounded sequences modulo those tending to zero.  There is a canonical embedding
\begin{equation}
\label{eq:iota}
\iota_B\colon B\to B_\infty
\end{equation}
induced by embedding $B$ into $\ell^\infty(B)$ as constant sequences.
\end{definition}

As is standard in the literature, we will typically use representative
sequences in $\ell^\infty(B)$ to denote elements of $B_\infty$.  One
important feature of the sequence algebra is that one can often use
reindexing arguments to convert approximate properties into exact ones.
Our preference is to formulate this using a version of
Kirchberg's $\epsilon$-test from \cite{Kirchberg06} suited to sequence
algebras. The form we quote below can be proven by an easy modification of the proof of 
\cite[Lemma A.1]{Kirchberg06} or \cite[Lemma 3.1]{Kirchberg-Rordam14}.

\begin{lemma}[Kirchberg's $\epsilon$-test]
  \label{lem:EpsTest}
  Let $(X_n)_{n=1}^\infty$ be a sequence of non-empty sets and write
  $X\coloneqq \prod_{n=1}^\infty X_n$. Given functions $f_n^{(k)}\colon X_n\to
  [0,\infty]$ for $k,n\in\mathbb N$, define functions $f^{(k)}\colon
  X\to [0,\infty]$ by $f^{(k)}((x_n)_{n=1}^\infty)\coloneqq
  \limsup_nf^{(k)}_n(x_n)$.  Suppose that for all $\epsilon>0$ and
  $k_0\in\mathbb N$, there exists $x\in X$ with $f^{(k)}(x)<\epsilon$
  for $k=1,\dots,k_0$.  Then there exists $x\in X$ with $f^{(k)}(x)=0$
  for all $k\in\mathbb N$.
\end{lemma}

The following application of the $\epsilon$-test is completely standard. (See the proof of \cite[Lemma 1.17]{Bosa-Brown-etal15}, for example.)

\begin{lemma}
\label{lem:AUEImpliesUE}
Let $A$ and $B$ be $C^*$-algebras with $A$ separable, and let $\phi,\psi\colon A\to B_\infty$ be $^*$-homomorphisms.  Then $\phi$ and $\psi$ are approximately unitarily equivalent if and only if they are unitarily equivalent.
\end{lemma}

We end by checking that the inclusion $\iota_B:B\hookrightarrow B_\infty$ is injective on the total invariant.

\begin{lemma}
  \label{lem:inv-iota-injective}
  If $B$ is a unital $C^*$-algebra, then
  \begin{equation}
    \inv(\iota_B) \colon \inv(B) \longrightarrow \inv(B_\infty)
  \end{equation}
  is a monomorphism.
\end{lemma}

\begin{proof}
  It suffices to prove that each of $\underline{K}(\iota_B)$,
  $\Ka(\iota_B)$, and $\Aff T(\iota_B)$ is injective.

For $\Aff T(\iota_B)$, when $T(B) = \emptyset$, the result is vacuous, so assume $T(B) \neq \emptyset$. Let $\tau\in T(B)$, let $\omega$ be a free
  ultrafilter on $\mathbb{N}$ and define $\tau_\omega\colon B_\infty
  \to \mathbb{C}$ by $\tau_\omega((b_n)_{n=1}^\infty) \coloneqq \lim_{n \to
    \omega} \tau(b_n)$.  Then $\tau_\omega \in T(B_\infty)$ and
  $\tau_\omega \circ \iota_B = \tau$, proving that $T(\iota_B)$ is
  surjective, and hence $\Aff T(\iota_B)$ is injective by duality.

For $\Ka(\iota_B)$, take a unitary $u\in M_r(B)$ with $\ka{\iota_B(u)} = 0$ in
  $\Ka(B_\infty)$.  Given $\epsilon > 0$, there are positive
  integers $k$ and $s$ and unitaries $v_1, \dots, v_k, w_1, \dots,
  w_k$ in $M_{r + s}(B_\infty)$ such that
  \begin{equation}
    \| \iota_B(u) \oplus 1_{M_s(B_\infty)} - v_1w_1v_1^*w_1^*
    \cdots v_kw_kv_k^*w_k^* \| < \epsilon.
  \end{equation}
  For $i \in \{1, \ldots, k\}$, let $(v_{i, n})_{n=1}^\infty$ and
  $(w_{i, n})_{n=1}^\infty$ be sequences of unitaries in $M_{r+s}(B)$ lifting $v_i$ and $w_i$, respectively.  Then, for
  sufficiently large $n$, we have
  \begin{equation}
    \| u \oplus 1_{M_s(B)} - v_{1, n}w_{1, n}v_{1,n}^*w_{1,n}^* \cdots v_{k, n}w_{k, n}v_{k,n}^*w_{k,n}^* \| < \epsilon.
  \end{equation}
  Thus, $\ka{u} = 0$ in $\Ka(B)$, proving that
  $\Ka(\iota_B)$ is injective.

  For injectivity of $\underline{K}(\iota_B)$, we start by establishing injectivity of $K_1(\iota_B)$.  Suppose that $r$ and $s$ are
  positive integers, that $u \in U_r(B)$, and that $\iota_B(u) \oplus
  1_{ M_s(B_\infty)}$ is homotopic to the identity in $U_{r +
   s}(B_\infty)$. Then there is a finite sequence $v_1,\dots,v_k$ in
 $U_{r+s}(B_\infty)$ with $v_1=u\oplus 1_{M_s(B_\infty)}$ and $v_k=1_{
   M_{r+s}(B_\infty)}$ such that $\|v_{j+1}-v_j\|<2$ for each
 $j=1,\dots,k-1$.  Lifting $v_2,\dots,v_{k-1}$ to sequences
 $(v_{2,n})_{n=1}^\infty,\dots,(v_{k-1,n})_{n=1}^\infty$ in
 $U_{r+s}(B)$, and setting $v_{k,n}\coloneqq1_{ M_{r+s}(B)}$,
 $v_{1,n}\coloneqq u\oplus 1_{M_s(B)}$, we can find some $n\in\mathbb N$, such that $\|v_{j+1,n}-v_{j,n}\|<2$ for all $j=1,\dots,k-1$, so that $u\oplus 1_{M_s(B)}$ is homotopic to the identity in $U_{r+s}(B)$, and hence $[u]_1=0$ in $K_1(B)$.  Hence $K_1(\iota_B)$ is injective.  A very similar argument shows that $K_0(\iota_B)$ is injective.

As we noted immediately after \eqref{eq:I_n}, we could equally well have defined total $K$-theory in terms of $K_i(A\otimes\mathcal O_{n+1})$: the functors $K_i(\,\cdot\, ; \Zn{n})$ and  $K_i(\,\cdot\, \otimes \mathcal O_{n+1})$ are naturally isomorphic by \cite[Theorem 6.4]{Schochet84}. 
Since $\mathcal O_{n+1}$ is nuclear, the canonical maps from $B_\infty$ and $\mathcal O_{n+1}$ to $(B\otimes \mathcal O_{n+1})_\infty$ with commuting ranges combine to yield a $^*$-homomorphism $\theta\colon B_\infty\otimes\mathcal O_{n+1}\rightarrow (B\otimes\mathcal O_{n+1})_\infty$.  Then we have the factorization $\iota_{B\otimes \mathcal O_{n+1}}=\theta\circ(\iota_B\otimes\mathrm{id}_{\mathcal O_{n+1}})$. Since $\mathcal O_{n+1}$ is unital (in contrast to $\mathbb I_n$), the above argument implies that $K_i(\iota_{B\otimes\mathcal O_{n+1}})$ is injective, 
and thus so is $K_i(\iota_B;\Zn{n})$.
\end{proof}

\section{$\Z$-stability and Cuntz semigroup techniques}\label{Sect:Z}

The key regularity hypothesis in the classification of embeddings theorem (Theorem~\ref{Main2}) is $\Z$-stability of the codomain $B$, i.e., $B\cong B\otimes\mathcal Z$, where $\Z$ is the Jiang--Su algebra from \cite{Jiang-Su99}.  We give a brief review of the Jiang--Su algebra in Section~\ref{SSZ} and then turn to $\Z$-stability, setting out in Sections~\ref{SSZ}--\ref{SSZStableSeq} the various consequences that we use elsewhere in the paper.   Section~\ref{SSK1Inject} is devoted to a modern proof of the unpublished result of Jiang (\cite{Jiang97}) that $\Z$-stable $C^*$-algebras are $K_1$-injective, for use in Section~\ref{sec:KK-classification}.

The road from $\Z$-stability to absorption passes through the Cuntz semigroup and the corona factorization property.  We provide the relevant background in Section~\ref{SSCuntz}. Finally, Section~\ref{SSZStableSeq} collates the various structural properties of sequence algebras of $\Z$-stable $C^*$-algebras that we need from the literature.

\subsection{The Jiang--Su algebra}\label{SSZ}
The Jiang--Su algebra $\mathcal Z$ was originally constructed in \cite{Jiang-Su99} as an inductive limit of prime dimension drop algebras (certain subhomogeneous $C^*$-algebras over the interval whose $K$-theory is $(\mathbb Z,0)$) with unital connecting maps chosen carefully to ensure simplicity and uniqueness of trace. In this way $KT_u(\mathcal Z)\cong KT_u(\mathbb C)$. See \cite[Chapter 15]{Strung21} for an expository account of the original construction of $\Z$. 

The only place where we need a somewhat explicit description of $\Z$ rather than its properties is in Section \ref{SSK1Inject}. For this we use the approach of R\o{}rdam and Winter from \cite{Rordam-Winter10} (which also powers Winter's localization technique from \cite{Winter14}).

Let $p,q\geq 2$ be coprime, and consider the UHF algebras $M_{p^\infty}$ and $M_{q^\infty}$.  Define the generalized dimension drop algebra
\begin{equation}\label{DefZpq}
  \mathcal{Z}_{p^\infty, q^\infty} \subseteq C([0, 1], M_{p^\infty}
  \otimes M_{q^\infty})
\end{equation}
to be the subalgebra of functions $f \in C([0, 1], M_{p^\infty}
  \otimes M_{q^\infty})$ satisfying
\begin{equation}
  f(0) \in M_{p^\infty} \otimes \mathbb{C}1_{M_{q^\infty}} \quad \text{and} \quad f(1)
  \in \mathbb{C}1_{M_{p^\infty}} \otimes M_{q^\infty}.
\end{equation}
For a unital endomorphism $\theta\colon \mathcal{Z}_{p^\infty, q^\infty}\to \mathcal{Z}_{p^\infty, q^\infty}$ which is trace-collapsing (i.e., $\tau_1\circ\theta=\tau_2\circ\theta$ for all $\tau_1,\tau_2\in Z_{p^\infty,q^\infty}$), the Jiang--Su algebra $\Z$ arises as the stationary inductive limit:
\begin{equation}\label{Zstationaryinductivelimit}
  \mathcal{Z} \cong \underset{\longrightarrow}{\lim} \, \big( \mathcal{Z}_{p^\infty, q^\infty}
\xrightarrow{\theta} \mathcal{Z}_{p^\infty, q^\infty} \xrightarrow \theta \mathcal{Z}_{p^\infty, q^\infty}
\xrightarrow{\theta} \cdots \big).
\end{equation}
Such a trace-collapsing embedding $\theta$ was constructed in a self-contained fashion in \cite[Section 4]{Schemaitat19}  (in their work, R\o{}rdam and Winter used Jiang and Su's construction of $\Z$ and its subsequent properties to build a trace-collapsing $\theta$). 

One of the most important facts about the Jiang--Su algebra --- and a major black box in this paper --- is that it is \emph{strongly self-absorbing} (a notion formalized by Toms and Winter in \cite{Toms-Winter07}).
\begin{theorem}[{Jiang--Su, \cite{Jiang-Su99}}]\label{ZisSSA}
The Jiang--Su algebra $\mathcal Z$ is strongly self-absorbing; i.e., there is an isomorphism $\mathcal Z\to\mathcal Z\otimes\mathcal Z$ that is approximately unitarily equivalent to $\id_{\mathcal Z}\otimes 1_{\mathcal Z}\colon \Z\to \Z\otimes\Z$.
\end{theorem}

Jiang and Su use their original explicit inductive limit model together with their classification results for limits of dimension drop algebras to show that $\Z$ has an approximately inner half flip (\cite[Proposition 8.3]{Jiang-Su99}), there is a unital embedding $\Z\otimes\Z\to\Z$ (\cite[Proposition 8.5]{Jiang-Su99}), and $\Z$ embeds approximately centrally into itself (\cite[Corollary 8.6]{Jiang-Su99}). From these ingredients they deduced what is now called strong self-absorption of $\Z$. Subsequently, Toms and Winter showed these three conditions give strong self-absorption in an abstract setting (\cite[Proposition 1.10]{Toms-Winter07}).

Following a direction outlined in Winter's 2011 CBMS lectures,  Schemaitat (\cite{Schemaitat19}) directly established strong self-absorption from the R\o{}rdam--Winter picture of $\Z$ using more elementary classification results (asymptotic classification of maps between UHF algebras). With hindsight, it is possible to take the R\o{}rdam--Winter picture as the definition of $\Z$ and use Winter's work (\cite{Winter11}) to show that such an algebra embeds unitally in all strongly self-absorbing algebras (from which it follows that it must be isomorphic to the original construction). One can also obtain all the other properties of $\Z$ needed in the classification theorem from this viewpoint. In a different direction, Ghasemi (\cite{Ghasemi21}, building on \cite{EFHKKL16,Masumoto17}) shows how to obtain strong self-absorption of $\Z$ from the approximately inner flip (obtained by Jiang and Su) via the model-theoretic notion of Fra\"iss\'e limits.

\smallskip Now we turn to $\Z$-stability. First, the definition.
\begin{definition}\label{DefZStable}
A $C^*$-algebra $B$ is \emph{$\Z$-stable} (or \emph{$\mathcal Z$-absorbing}) if $B\cong B\otimes\mathcal Z$.
\end{definition}

One immediate consequence of (strong) self-absorption is that $\Z$ is itself $\Z$-stable. Moreover, the strong form of the isomorphism $\Z\cong \Z\otimes\Z$ immediately strengthens the isomorphism in Definition~\ref{DefZStable} as follows (see\ \cite[Theorem 2.2]{Toms-Winter07}).

\begin{proposition}\label{PropSSA}
Let $B$ be a unital $C^*$-algebra.  The following are equivalent:
\begin{enumerate}[(i)]
\item $B$ is $\Z$-stable.
\item\label{PropSSA.2} There exists an isomorphism $B\to B\otimes\mathcal Z$ that is approximately unitarily equivalent to the first-factor embedding $\id_B\otimes 1_\Z\colon B\to B\otimes\mathcal Z$.
\end{enumerate}
\end{proposition}

 It is using this (or similar formulations, such as \cite[Lemma 4.4]{Rordam04}) that strong self-absorption of $\Z$ is generally used to obtain structural properties of $\Z$-stable $C^*$-algebras, such as Proposition \ref{p:ZStableUnperforated} below (due to R\o{}rdam as \cite[Theorem 4.5]{Rordam04}). This form of the isomorphism is also a key ingredient in the proof of the $\Z$-stable $KK$-uniqueness theorem in Section~\ref{sec:KK-Uniqueness}. 

The first use of $\Z$-stability in the unital classification theorem is through the dichotomy it provides between purely infinite and stably finite simple $C^*$-algebras, as shown in  \cite[Theorem 3]{Gong-Jiang-etal00}. This is a special case of a more general dichotomy result of Kirchberg for simple $C^*$-algebras that can be written non-trivially as a tensor product (\cite[Corollary 3.9(ii)]{Blanchard-Kirchberg04}; see also \cite[Theorem 4.1.10(ii)]{Rordam02}). 

\begin{theorem}\label{T:KDich}
Let $B$ be a simple exact $\Z$-stable $C^*$-algebra. Then $B$ is either purely infinite or stably finite.
\end{theorem}

We end the subsection by recording that tensoring with $\Z$ leaves the invariant $KT_u(\cdot)$ unchanged. The $KK$-equivalence fact below was observed by Jiang and Su in \cite[Lemma 2.11]{Jiang-Su99}. In the separable case, the isomorphism of invariants follows from this and uniqueness of trace on $\Z$; the non-separable case follows by continuity.

\begin{proposition}\label{prop:ZKKequiv}
    For any unital $C^*$-algebra $D$, the first-factor embedding induces both an isomorphism $KT_u(D)\cong KT_u(D\otimes \Z)$, and (when $D$ is separable) a $KK$-equivalence.
\end{proposition}

\subsection{$K_1$-injectivity}\label{SSK1Inject}
In this section we give a short proof of an unpublished result of Jiang that we will use in the $\Z$-stable $KK$-uniqueness theorem in Section~\ref{sec:KK-Uniqueness}.
 Recall that a unital $C^*$-algebra $D$ is \emph{$K_1$-injective} if whenever $u\in U_1(D)$ vanishes in $K_1$, then $u\in U^{(0)}_1(D)$.  That is, no matrix amplification is needed to construct a homotopy from $u$ to $1_D$.

\begin{theorem}[Jiang, \cite{Jiang97}]
  \label{prop:z-stable-K1-inj}
  If $D$ is a unital $C^*$-algebra, then $D\otimes \Z$ is $K_1$-injective.
\end{theorem}

Before proving the theorem, it is helpful to isolate the following
fact regarding the non-stable $K$-theory of $C^*$-algebras that absorb
a UHF algebra of infinite type. 

\begin{lemma}
  \label{lemma:uhf-stable-K1-inj}
  Let $D$ be a unital $C^*$-algebra and $n\in\mathbb N$ with $n\geq
  2$. Then $D \otimes M_{n^\infty}$ is $K_1$-injective and $K_0(D
  \otimes M_{n^\infty})$ is generated by 
 \begin{equation}
 	\{ [p]_0 : p\text{ is a projection in } D \otimes M_{n^\infty}\}.
 \end{equation}
\end{lemma}

\begin{proof}
  Write $B\coloneqq D\otimes M_{n^\infty}$ and suppose $u \in U_1(B)$ satisfies $[u]_1=0$.  Then
  there is an $m \geq 1$ such that $u \oplus 1^{\oplus (m-1)}_B$ is
  homotopic to $1^{\oplus m}_B$ in $U_m(B)$.
Enlarging $m$ if necessary, we may assume $m = n^k$ for some $k \geq 1$.  
We have
\begin{equation} u^{\oplus m} = (u \oplus 1_B^{\oplus (m-1)})(1_B \oplus u^{\oplus (m-1)}) \sim_h 1_B \oplus u^{\oplus (m-1)}, \end{equation}
and by repeating this, it follows that $u^{\oplus m}$ is homotopic to the identity in $U_{m}(B)$. Put another way, this says that
  $u\otimes 1_{M_{n^k}}$ is homotopic to the identity in $U_1(B\otimes
  M_{n^k})$.  Therefore, $u\otimes 1_{M_{n^\infty}}$ is homotopic to
  the identity in the unitary group of $B\otimes M_{n^\infty}$. 
  
  Use strong self-absorption of $M_{n^\infty}$ (which is straightforward for $M_{n^\infty}$) to fix an isomorphism $
    \theta\colon B \rightarrow B \otimes M_{n^\infty}$
  that is approximately unitarily equivalent to the embedding $x \mapsto x
  \otimes 1_{M_{n^\infty}}$. Then take a unitary $v\in B\otimes M_{n^\infty}$ with 
  \begin{equation} 
    \|\theta(u)-v(u\otimes 1_{M_{n^\infty}})v^*\|<1
  \end{equation} 
  so that $\theta(u)$ is homotopic to $v(u\otimes 1_{M_{n^\infty}})v^*$ in $U_1(B\otimes M_{n^\infty})$, which in turn, is homotopic to $1_{B\otimes M_{n^\infty}}$.    Applying $\theta^{-1}$ shows that $u$ is homotopic to $1_B$ in $U_1(B)$.

  The statement regarding $K_0$ follows in a similar fashion.  If
  $p\in B \otimes M_m$ is a projection, we may again assume $m
  = n^k$ for some $k$.  This time, take an isomorphism
$\theta\colon B \rightarrow B \otimes M_{n^k}$
  that is approximately unitarily equivalent to the embedding $x \mapsto x
  \otimes 1_{M_{n^k}}$ (such a $\theta$ exists as $B \cong B\otimes M_{n^\infty} \cong B \otimes M_{n^k} \otimes M_{n^\infty} \cong B \otimes M_{n^k}$, where the first and last isomorphisms are approximately unitarily equivalent to the first-factor embeddings).
Set $q \coloneqq \theta^{-1}(p) \in B$.  As $p=\theta(q)$ is approximately unitarily equivalent to $q\otimes 1_{M_{n^k}}$ in $B\otimes M_{n^k}$, 
  $[p]_0 = n^k[q]_0$, and so $[p]_0$ is in the group generated by $K_0$-classes of
  projections in $B$.
\end{proof}

\begin{proof}[Proof of Theorem~\ref{prop:z-stable-K1-inj}]
  $D \otimes \mathcal{Z}$ is a stationary inductive limit of the
  algebra $D \otimes \mathcal{Z}_{2^\infty, 3^\infty}$; see
  \eqref{DefZpq}.  It is therefore enough to show $D \otimes
  \mathcal{Z}_{2^\infty, 3^\infty}$ is $K_1$-injective.  To this end,
  choose $u \in U_1(D \otimes \mathcal{Z}_{2^\infty, 3^\infty})$ with
  $[u]_1 = 0$ in $K_1(D \otimes \mathcal{Z}_{2^\infty, 3^\infty})$.

  Let $SM_{6^\infty}$ be the suspension $C_0((0,1), M_{6^\infty}) \subseteq \mathcal Z_{2^\infty,3^\infty}$. As each of
    $SM_{6^\infty}$, $\mathcal Z_{2^\infty,3^\infty}$, and
    $M_{2^\infty}\oplus M_{3^\infty}$ is nuclear, we get an exact sequence
  \begin{equation}
   0 \rightarrow D \otimes SM_{6^\infty} \rightarrow D \otimes
    \mathcal{Z}_{2^\infty, 3^\infty} \xrightarrow{\sigma} D \otimes
    (M_{2^\infty} \oplus M_{3^\infty}) \rightarrow 0.
  \end{equation}
  Note that $[\sigma(u)]_1 = 0$ in $K_1(D \otimes (M_{2^\infty} \oplus
  M_{3^\infty}))$.  As $D \otimes (M_{2^\infty} \oplus M_{3^\infty}) \cong D \otimes M_{2^\infty} \oplus D\otimes M_{3^\infty}$, and Lemma~\ref{lemma:uhf-stable-K1-inj} applies to both $D \otimes M_{2^\infty}$ and $D \otimes M_{3^\infty}$, 
  it follows that $\sigma(u)$ is in the path
  component of the identity in the unitary group of this algebra.
  Hence, $\sigma(u)$ lifts to a unitary $v$ in the path component of the identity
  in $U_1(D\otimes\Z_{2^\infty,3^\infty})$ (i.e.,  $\sigma(v)=\sigma(u)$).  After
  replacing $u$ with $uv^*$, we may assume $\sigma(u) =
  1_{D\otimes(M_{2^\infty}\oplus M_{3^\infty})}$ so that $u$ is a
  unitary in $(D \otimes SM_{6^\infty})^\dagger$.

  Since $[u]_1 = 0$ in $K_1(D \otimes \mathcal{Z}_{2^\infty,
    3^\infty})$, the 6-term exact sequence gives an element $x \in
  K_0(D \otimes (M_{2^\infty} \oplus M_{3^\infty}))$ such that
  $\mathrm{exp}(x) = [u]_1$ in $K_1(D \otimes SM_{6^\infty})$.  As above, 
  another application of Lemma~\ref{lemma:uhf-stable-K1-inj} shows there are projections $p_1,
  \ldots, p_n, q_1, \ldots, q_n$ in $D \otimes (M_{2^\infty} \oplus
  M_{3^\infty})$ with $x = \sum_j ([p_j]_0 - [q_j]_0)$.  Then, in
  $K_1(D \otimes SM_{6^\infty})$, we have
  \begin{equation}
    [u]_1 = [e^{2\pi i h_1} \cdots e^{2\pi i h_n} e^{- 2 \pi i k_1}
    \cdots e^{- 2 \pi i k_n}]_1
  \end{equation}
  where the $h_j$, $k_j$ are positive contractions in $D \otimes
  \mathcal{Z}_{2^\infty, 3^\infty}$ lifting $p_j, q_j$, respectively.
  Replacing $u$ with
  \begin{equation}
    u e^{- 2\pi i h_1} \cdots e^{- 2\pi i h_n} e^{ 2 \pi i k_1} \cdots
    e^{ 2 \pi i k_n},
  \end{equation}
  if necessary, we may assume that $u \in (D \otimes
  SM_{6^\infty})^\dagger$ with $[u]_1 = 0$ in $K_1(D \otimes
  SM_{6^\infty})$ and with scalar part equal to $1$.

  With these reductions in place, we will now show $u$ is in the path
  component of the identity in the unitary group of $(D \otimes SM_{6^\infty})^\dagger$.
  We identify $D \otimes SM_{6^\infty}$ with the subalgebra of
  $C(\mathbb{T}, D \otimes M_{6^\infty})$ consisting of functions
  vanishing at $1$.  Then $u$ defines an element in $C(\mathbb{T}, D
  \otimes M_{6^\infty})$ with trivial $K_1$-class.  By a third, and
  final, application of Lemma~\ref{lemma:uhf-stable-K1-inj}, $u$ is in
  the path component of the identity of $U_1(C(\mathbb{T}, D \otimes
  M_{6^\infty}))$.  Let $(v_t)_{t\in[0,1]}$ be a path of unitaries in
  $C(\mathbb{T}, D \otimes M_{6^\infty})$ with $v_0 = 1_{D\otimes M_{6^\infty}}$ and $v_1 =
  u$.  Define $u_t = v_t(1)^* v_t \in C(\mathbb{T}, D \otimes
  M_{6^\infty})$.  Then $u_t(1) = 1_{D\otimes M_{6^\infty}}$ for all $t \in [0, 1]$, so
  $(u_t)_{t\in[0,1]}$ is a path of unitaries in $(D \otimes
  SM_{6^\infty})^\dagger$.  Moreover, $u_0 = 1_{(D\otimes SM_{6^\infty})^\dagger}$ and $u_1 = u$, as
  required.
\end{proof}

\subsection{The Cuntz semigroup and the corona factorization property}\label{SSCuntz}
Considerable information about a $C^*$-algebra is carried in an algebraic object --- the Cuntz semigroup --- constructed from its positive elements. The ideas for this semigroup originated in \cite{Cuntz78} and were subsequently developed in \cite{Blackadar-Handelman82,Rordam92,Perera97,Kirchberg-Rordam02,Rordam04,Brown-Perera-etal08,Coward-Elliott-etal08} and other works; see the survey \cite{Ara-Perera-Toms11}. 

Given a $C^*$-algebra $D$ and positive elements $a,b\in D\otimes\mathcal K$, write $a\precsim b$ if there exists a sequence $(x_n)_{n=1}^\infty$ in $D\otimes\mathcal K$ with $x_n^*bx_n\to a$.  Say $a$ and $b$ are \emph{Cuntz equivalent}, written $a\sim b$, if $a\precsim b$ and $b\precsim a$.  The \emph{Cuntz semigroup}, $\Cu(D)$, is defined by $(D\otimes\mathcal K)_+/{\sim}$.  This is an abelian semigroup (with addition given by identifying $\mathcal K\cong M_2(\mathcal K)$ and then adding representatives diagonally) that inherits an order $\leq$ from $\precsim$.  When consulting the literature, it is important to distinguish between the complete Cuntz semigroup $\Cu(D)$ introduced in \cite{Coward-Elliott-etal08} and the earlier incomplete version $W(D)$, which is defined in a similar fashion as above using $\bigcup_{n=1}^\infty M_n(D)_+$ in place of $(D\otimes K)_+$.  Note that $\Cu(D)\cong W(D\otimes\mathcal K)$. In order to accurately cite the results we need, we use both constructions.

Given $\tau\in T(D)$ and $a\in (D\otimes\mathcal K)_+$ defining 
\begin{equation}
d_\tau(a)\coloneqq \lim_{r\to\infty}\tau(a^{1/r}),
\end{equation}
where $\tau$ is extended to a densely defined lower semicontinuous trace on $D\otimes\mathcal K$, gives \emph{states} $d_\tau\colon W(D)\to[0,\infty)$ and $d_\tau\colon\Cu(D)\to[0,\infty]$,
i.e., order-preserving semigroup homomorphisms which map $[1_D]$ to $1$ (this notion of a state can also be extended to the non-unital case). (See \cite[Proposition~4.2]{Elliott-Robert-etal11},  which shows, using work going back to \cite{Cuntz78,Blackadar-Handelman82}, that all states on $\Cu(D)$ are of this form for suitable lower semicontinuous extended quasitraces on $D$.) A particularly useful concept, strict comparison, allows  Cuntz comparison to be recovered from tracial data. This idea goes back to \cite[Section 6]{Blackadar88}, which gives an account of comparison conditions for simple $C^*$-algebras in the spirit of Murray and von Neumann's comparison theory for von Neumann algebra factors.  
There are several variations of ``strict comparison'' in the literature.
We use the following form, which appears in \cite[Definition 1.5]{Bosa-Brown-etal15}.

\begin{definition}\label{DefStrictComp}
Let $D$ be a $C^*$-algebra. Say that $D$ has \emph{strict comparison with respect to bounded traces} when for all $n\in\mathbb N$ and $a,b\in M_n(D)_+$, we have
\begin{equation}
d_\tau(a)<d_\tau(b) \text{ for all }\tau\in T(D)\implies a\precsim b.
\end{equation}
\end{definition}

An important well-known application of strict comparison is to catalyze checking the fullness of $^*$-homomorphisms. We recall both the definition of fullness, and the relevant result below. The first statement of Lemma \ref{p:fullcomp} is trivial, and the second is relatively standard; see \cite[Lemma 2.2]{TWW},     for example.

\begin{definition}
Let $B,D$ be $C^*$-algebras. A $^*$-homomorphism $\phi\colon B \to D$ is \emph{full} if $\phi(b)$ generates $D$ as an ideal, for every non-zero $b \in B$.
\end{definition}

\begin{lemma}\label{p:fullcomp}
Let $\theta \colon B\to D$ be a $^*$-homomorphism between $C^*$-algebras. If $\theta$ is full, then $\tau\circ\theta$ is faithful for every $\tau\in T(D)$. The converse holds if $D$ is unital and has strict comparison with respect to bounded traces. \end{lemma}

While the order structure of the Cuntz semigroup of a general $C^*$-algebra --- even a simple one --- can be quite badly behaved (see \cite{Villadsen98,Toms08}), R\o{}rdam showed the situation is much better for a $\Z$-stable $C^*$-algebra. In the second part of the following proposition, R\o{}rdam uses Haagerup's theorem (\cite{Haagerup14}) that quasitraces are traces for exact $C^*$-algebras. This is the place where this fundamental result appears in our proof of the unital classification theorem. 

\begin{proposition}[{R\o{}rdam, \cite{Rordam04}}]\label{p:ZStableUnperforated}
Let $D$ be a $\Z$-stable $C^*$-algebra.  Then
\begin{enumerate}[(i)]
\item $\Cu(D)$ is almost unperforated: whenever $x,y\in \Cu(D)$
  satisfy $nx\leq my$ for some $n>m$ in $\mathbb N$, we have $x\leq
  y$.  \label{p:ZStableUnperforated.1}
\item If, in addition, $D$ is unital, simple, and exact with
  $T(D)\neq\emptyset$, then $D$ has strict comparison of positive
  elements with respect to bounded
  traces.
\label{p:ZStableUnperforated.2}
\end{enumerate}
\end{proposition}
\begin{proof}
Theorem 4.5 of \cite{Rordam04} shows that the incomplete Cuntz semigroup $W(D)$ is almost unperforated when $D$ is $\Z$-stable.  Part~\ref{p:ZStableUnperforated.1} follows by applying this to $D\otimes\mathcal K$, since $\Cu(D)\cong W(D\otimes \mathcal K)$.  Part~\ref{p:ZStableUnperforated.2} follows by applying \cite[Corollary 4.6]{Rordam04} to each $M_n(D)$.
\end{proof}

Now we turn to the corona factorization property of Kucerovsky and Ng from \cite{Kucerovsky-Ng06}, which we use to verify absorption in Section~\ref{sec:KK-classification} and beyond.  A $\sigma$-unital $C^*$-algebra $D$ has the \emph{corona factorization property} if every (norm-)full projection in the multiplier algebra  of $D\otimes\mathcal K$ is properly infinite. This is a relatively mild regularity hypothesis, and is implied by $\Z$-stability, although a tight reference for this seems difficult to come by.  The route through the literature that we give below uses Ortega, Perera, and R\o{}rdam's characterization of the corona factorization property in terms of the complete Cuntz semigroup (\cite[Theorem 5.11]{Ortega-Perera-etal12}).

\begin{proposition}\label{prop:Zcfp} 
Let $D$ be a $\sigma$-unital $\Z$-stable $C^*$-algebra.  Then $D$ has the corona factorization property.
\end{proposition}
\begin{proof}
By Proposition~\ref{p:ZStableUnperforated}\ref{p:ZStableUnperforated.1}, $\Cu(D)$ is almost unperforated. By \cite[Remark 2.4]{Ortega-Perera-etal12}, this is equivalent to $0$-comparison for $\Cu(D)$ in the sense of \cite[Definition 2.8]{Ortega-Perera-etal12}.  In particular, $\Cu(D)$ has the $\omega$-comparison property of \cite[Definition 2.11]{Ortega-Perera-etal12} (see the sentence immediately following that definition). Then we can combine \cite[Proposition 2.17]{Ortega-Perera-etal12} and \cite[Theorem 5.11]{Ortega-Perera-etal12} to conclude that $D$ has the corona factorization property.
\end{proof}	

\subsection{Sequence algebras of $\Z$-stable $C^*$-algebras}\label{SSZStableSeq}
We end this section by recording some structural properties that a sequence algebra $B_\infty$ enjoys when the underlying $C^*$-algebra $B$ is $\Z$-stable.  Such a sequence algebra cannot be $\Z$-stable (by \cite{Ghasemi15}, where the SAW$^*$ hypothesis needed to use this result is obtained for $B_\infty$ by a standard application of Kirchberg $\epsilon$-test). Accordingly, we work with the following version of $\Z$-stability in terms of separable subalgebras, in the spirit of Blackadar's notion of separable inheritability (\cite[Section II.8.5]{Blackadar06}).

\begin{definition}[{cf.\ \cite[Definition 1.4]{Schafhauser18}}]
  A $C^*$-algebra $D$ is 
  \emph{separably $\mathcal Z$-stable} if, for every separable
  $C^*$-subalgebra $D_0$ of $D$, there exists a separable $\Z$-stable
  $C^*$-subalgebra $D_1\subseteq D$ containing $D_0$.
\end{definition}

When $B$ is separable and unital, an intertwining argument which seems to be originally
due to Elliott (see \cite[Theorem 8.2]{Kirchberg-Rordam02}, \cite[Page 34]{Rordam94}, and
 \cite[Section 2]{Toms-Winter07}) shows that $\Z$-stability of
$B$ can be characterized in terms of an embedding of $\Z$ into the
central sequence algebra $B_\infty\cap B'$.
This is in the spirit of McDuff's
characterization of separably acting II$_1$ factors absorbing the
hyperfinite II$_1$ factor tensorially.  It gives rise to a local
characterization of $\Z$-stability for unital separable
$C^*$-algebras, that extends to characterize separable
$\Z$-stability; this is why separable $\Z$-stability is well-suited to
working with non-separable $C^*$-algebras (see \cite[Lemma
1.11]{Schafhauser18}).  As a consequence, one has the following.

\begin{proposition}[{\cite[Proposition 1.12]{Schafhauser18}}]
  \label{prop:B_inftyProperties_1}
  If $B$ is a unital separable $\mathcal Z$-stable $C^*$-algebra, then
  $B_\infty$ is separably $\mathcal Z$-stable.
\end{proposition}

Now we turn to traces on sequence algebras, with the aim of using Kirchberg's $\epsilon$-test to observe that the image of $K_0(B_\infty)$ is closed in $\Aff T(B_\infty)$ when $B$ satisfies the codomain assumptions in the classification of embeddings theorem.
 Once this is done, Thomsen's work described in Section~\ref{SectAlgK1} (in particular, Properties~\ref{ThomsenProp}) will give an exact sequence
\begin{equation}\label{ThomsenExactSequenceAlg}
K_0(B_\infty)\xrightarrow{\rho_{B_\infty}} \Aff T(B_\infty)\xrightarrow{\Th_{B_\infty}} \Ka(B_\infty)\xrightarrow{\minusa_{B_\infty}} K_1(B_\infty)\to 0
\end{equation}
which we use in the computation of $KL(A,J_B)$ in Section~\ref{sec:uct-rot}. 
First, we recall the notion of limit traces; these are the traces on $B_\infty$ that are computationally tractable.
\begin{definition}\label{DefLimitTraces}
  The set of \emph{limit traces} on $B_\infty$, written $T_\infty(B)$,
  consists of all $\tau \in T(B_\infty)$ of the form
  $\tau((b_n)_{n=1}^\infty) = \lim_{n\to\omega} \tau_n(b_n)$, where
  $(\tau_n)_{n=1}^\infty$ is a sequence in $T(B)$ and $\omega$ is a
  free ultrafilter on $\mathbb N$.
\end{definition}

\begin{proposition}[{Ozawa, cf.\ \cite[Theorem 1.2]{Ng-Robert16}}]\label{NoSillyTraces}
If $B$ is an exact $\Z$-stable $C^*$-algebra, then the convex hull of $T_\infty(B)$ is weak$^*$-dense in $T(B_\infty)$. 
\end{proposition}
\begin{proof}
Ozawa's \cite[Theorem 8]{Ozawa13} shows this for ultraproducts of $\Z$-stable separable exact $C^*$-algebras, and the case of sequence algebras is essentially identical.  One follows the proof of \cite[Theorem 8]{Ozawa13} verbatim with $A\coloneqq B_\infty$ and $\Sigma\coloneqq\mathrm{co}\,T_\infty(B)$ and changing ``Let $I\in\mathcal U$ (or $I=\mathbb N$ in case $A=\prod A_n$)'' to ``Let $I$ be cofinite''.
\end{proof}

The following proposition collects additional properties of the sequence
algebra that will be useful later.

\begin{proposition}
  \label{prop:B_inftyProperties_2}
  Let $B$ be a unital simple separable $\mathcal Z$-stable
  $C^*$-algebra with $T(B)\neq\emptyset$.  Then
  \begin{enumerate}[(i)]
  \item $B_\infty$ has stable rank one,
    and \label{prop:B_inftyProperties.2}
  \item when $B$ is additionally exact, $B_\infty$ has strict comparison of positive elements with
    respect to bounded traces. \label{prop:B_inftyProperties.3}
  \end{enumerate}
\end{proposition}

\begin{proof}
  \ref{prop:B_inftyProperties.2} follows directly from \cite[Lemma
  19.2.2]{Loring97a}, since $B$ has stable rank one by \cite[Theorem
  6.7]{Rordam04}.

  For~\ref{prop:B_inftyProperties.3}, note that $B$ has strict
  comparison of positive elements by bounded traces by R\o{}rdam's
  result recalled in
  Proposition~\ref{p:ZStableUnperforated}\ref{p:ZStableUnperforated.2}. Using
  this, the version of~\ref{prop:B_inftyProperties.3} for the
  ultrapower $B_\omega$ has been observed in \cite[Lemma
  1.23]{Bosa-Brown-etal15}. The result for
  sequence algebras is obtained by following this short proof
  verbatim, with the following changes: replace $B_\omega$ by
  $B_\infty$, $T_\omega(B_\omega)$ by $T_\infty(B)$, and $I \in
  \omega$ in \cite[(1.38)]{Bosa-Brown-etal15} by ``$I$ is cofinite.''
\end{proof}

Now we can show exactness of (\ref{ThomsenExactSequenceAlg}).

\begin{proposition}
  \label{prop:B_inftyK0ClosedInAff}
  Let $B$ be a unital simple separable exact $\mathcal Z$-stable
  $C^*$-algebra with $T(B)\neq \emptyset$.  Then
  $\rho_{B_\infty}(K_0(B_\infty))$ is closed in $\Aff T(B_\infty)$, and so \eqref{ThomsenExactSequenceAlg}
is exact.
\end{proposition}
\begin{proof}
  Let $g \in \overline{\rho_{B_\infty}(K_0(B_\infty))}$. By  adding a suitably large $n=\rho_{B_\infty}([1_{M_n(B_\infty)}]_0)$ to $g$, we can assume that $0<g<m$ for some $m\in\mathbb N$.   As a consequence of Proposition~\ref{prop:cuntz-pedersen}, there exists
a self-adjoint  $c=(c_n)_{n=1}^\infty \in B_\infty$ such that $\tau(c)=g(\tau)$ for
  all $\tau \in T(B_\infty)$.

  Let $X$ be the unit ball of $M_m(B)$ and define $f_n\colon X \to [0,\infty)$ by
  \begin{equation}
    f_n(b) \coloneqq \|b-b^*\|+\|b^2-b\|+\sup_{\tau \in T(B)}
    |\tau(b-c_n)|,
  \end{equation}
  and then define $f\colon \prod_{n=1}^\infty X \to [0,\infty)$ by
  \begin{equation}
    f((b_n)_{n=1}^\infty) \coloneqq \limsup_{n\to\infty} f_n(b_n).
  \end{equation}
  For a projection $p=(p_n)_{n=1}^\infty \in M_m(B_\infty)$, one computes
  \begin{equation}
    f\left((p_n)_{n=1}^\infty\right) = \sup_{\tau \in T_\infty(B)} |\tau(p)-g(\tau)|\,
    \stackrel{\text{Prop.~\ref{NoSillyTraces}}}=
    \sup_{\tau \in T(B_\infty)} |\tau(p)-g(\tau)|.
  \end{equation}

  By hypothesis, for any $\epsilon>0$, there exists $x \in
  K_0(B_\infty)$ such that $\|\rho_{B_\infty}(x)-g\|_\infty <
  \epsilon$.  Write $x=[r]_0-[s]_0$ for projections $r,s$ in matrices over $B_\infty$. We may assume $\epsilon$ is small enough that $0<\rho_{B_\infty}(x)<m$. As Cuntz subequivalent projections are Murray--von Neumann subequivalent (\cite[Proposition 2.1]{Rordam92}), we can use strict comparison by bounded traces (Proposition~\ref{prop:B_inftyProperties_2}\ref{prop:B_inftyProperties.3}) first to see that $s$ is subequivalent to $r$ and hence $x$ is the class of some projection, and then again to compare this projection with $1_{M_m(B_\infty)}$. In this way, $x$ is the class of a projection $p\in M_m(B_\infty)$ and this $p$ has $f(p)< \epsilon$.

  By Kirchberg's $\epsilon$-test (Lemma~\ref{lem:EpsTest}), there
  exists $b \in \prod_{n=1}^\infty X$ such that $f(b)=0$.  This represents a
  projection in $M_m(B_\infty)$ such that $\tau(b)=g(\tau)$ for all $\tau \in
  T(B_\infty)$. Therefore, \eqref{ThomsenExactSequenceAlg} is exact at $\Aff T(B)$, and hence exact by Properties~\ref{ThomsenProp}.
\end{proof}


\section{Absorption and $KK$-existence and uniqueness theorems}\label{sec:KK-classification}

The main goal of this section is to prove the $\Z$-stable $KK$- and $KL$-uniqueness theorems (collected together as Theorem~\ref{thm:KK-Uniqueness}) in Section~\ref{sec:KK-Uniqueness}. This proves Theorem~\ref{intro-KK-unique} from the introduction. We start by setting out the Cuntz--Thomsen picture of $KK$-theory in Section~\ref{subsec:elem-of-kk}, leaving certain proofs for the non-separable setting to Appendix~\ref{sec:kkappendix}. We follow this by recalling salient facts about absorption in Section~\ref{sec:absorption} and record the $KK$-existence theorem in Section~\ref{sec:KK-Existence}. 

\subsection{The Cuntz--Thomsen picture of Kasparov's
  $KK$-theory}\label{subsec:elem-of-kk} We start by collecting those
aspects of $KK$- and $KL$-theory that we need.  As was the case in Kasparov's original definition of $KK$-theory (\cite{Kasparov80}) using Fredholm modules, $KK(A,I)$ is typically defined for separable $A$ and $\sigma$-unital
$I$.
However, we need
to allow the non-$\sigma$-unital trace-kernel ideal $J_B$ in the second variable. 
As suggested by
Skandalis in \cite[Section 3]{Skandalis85} (see also
\cite[Definition~5.5]{Skandalis88}), this can be achieved by taking
inductive limits over separable $C^*$-subalgebras. We set out how to
do this here. 
All definitions and results in this subsection are
standard when the second variable is $\sigma$-unital; we provide
 proofs of how to extend these to the general case in
Appendix~\ref{sec:kkappendix}.

Of the many equivalent ways of defining $KK$-theory (see \cite[Chapter
17]{Blackadar98}, or \cite{Jensen-Thomsen91}), the Cuntz--Thomsen picture
is particularly well-suited to classification problems (as
demonstrated, for example, by Dadarlat and Eilers in
\cite{Dadarlat-Eilers01,Dadarlat-Eilers02}). Originally, this approach
was developed by means of reductions to Kasparov's Fredholm module
construction.  This is implicit in \cite[Section 5]{Cuntz83a}
  and set out in \cite[17.6.2]{Blackadar98} --- see also \cite[Section
  3.1]{Dadarlat-Eilers01}.  Thomsen later gave a self-contained
treatment in \cite{Thomsen90}.

\begin{definition}
  Let $A$ and $I$ be $C^*$-algebras with $A$ separable.  An
  \emph{$(A,I)$-Cuntz pair} is a pair $(\phi,\psi)$ of
  $^*$-homomorphisms $A\to E$ for some $C^*$-algebra $E$ containing
  $I$ as an ideal, such that $\phi(a)-\psi(a)\in I$ for all $a\in
  A$. We will often write $(\phi,\psi)\colon A\rightrightarrows E\rhd
  I$ for such a Cuntz pair.  We call $(\phi,\psi)$ \emph{unital} when all of $A$, $E$, $\phi$, and $\psi$ are unital. 
\end{definition}

\begin{definition}
\label{defn:KK}
  Let $A$ and $I$ be $C^*$-algebras with $A$ separable and $I$ $\sigma$-unital. Write
  $\mathcal M(I\otimes\mathcal K)$ for the multiplier algebra of the
  stabilization $I\otimes\mathcal K$. Let $C_{\sigma}\big( [0,1],\mathcal M(I\otimes\mathcal K)
    \big)$ denote the algebra of strictly continuous functions $[0,1]\to
    \mathcal{M}(I\otimes\mathcal K)$ (these are automatically bounded
    by the principle of uniform boundedness). By
    \cite[Corollary~3.4]{Akemann-Pedersen-etal73}, one has
    \begin{equation}
    C_{\sigma}\big( [0,1],\mathcal M(I\otimes\mathcal K)
    \big)\cong \mathcal{M}\big( C([0,1],I\otimes\mathcal K) \big).
    \end{equation}
    Two Cuntz pairs
  \begin{equation}
    (\phi_i,\psi_i)\colon A\rightrightarrows \mathcal M(I\otimes\mathcal K)\rhd I\otimes\mathcal K,\quad i=0,1,
  \end{equation}
  are said to be \emph{homotopic} if there is a Cuntz
  pair
\begin{equation}
(\phi,\psi)\colon A \rightrightarrows C_{\sigma}\big( [0,1],\mathcal
M(I\otimes\mathcal K) \big)\rhd C([0,1],I\otimes\mathcal
K)
\end{equation}
such that evaluation at $i$ induces the  Cuntz pairs $(\phi_i,\psi_i)$ for $i=0,1$. 
The Cuntz--Thomsen picture of $KK$-theory is then
\begin{equation}
  KK(A,I) \coloneqq \{
  (\phi,\psi)\colon A \rightrightarrows\mathcal{M}(I\otimes \mathcal{K})
  \rhd I\otimes \mathcal{K}
  \} /
  \text{homotopy}. 
\end{equation}
\end{definition}

We write $[\phi,\psi]_{KK(A,I)}$ for the class in $KK(A,I)$ of
$(\phi,\psi)\colon A\rightrightarrows \mathcal M(I\otimes\mathcal
K)\rhd I\otimes\mathcal K$. More generally, given a Cuntz pair $(\phi,\psi)\colon
A\rightrightarrows E\rhd I$, there is a canonical map $\mu \colon E \otimes \mathcal K \rightarrow \mathcal M(I \otimes \mathcal K)$.  Writing $e$ for a rank one projection in $\mathcal K$ gives a Cuntz pair
\begin{equation}
    (\mu\circ (\phi\otimes e),\mu\circ (\psi\otimes e))\colon
  A\rightrightarrows \mathcal M(I\otimes\mathcal K)\rhd
  I\otimes\mathcal K.
\end{equation}
In this way, $(\phi, \psi)$ defines an element
\begin{equation}
\label{eq:cuntz-pair-convention}
[\phi, \psi]_{KK(A,I)} \coloneqq [\mu\circ (\phi\otimes e),\mu\circ (\psi\otimes e)]_{KK(A,I)} \in KK(A, I).
\end{equation}
An important special case is when $\phi\colon A\to I$ is a $^*$-homomorphism, so that $(\phi,0)$ is an $(A,I)$-Cuntz
pair. We write $[\phi]_{KK(A,I)}$ for the induced class in $KK(A,I)$.

Addition in $KK(A,I)$ is defined by means of an orthogonal direct sum. There is a copy of $\mathcal B(\mathcal H)$ in $\mathcal{M}(I\otimes \mathcal K)$, so one can find $s_1,s_2\in\mathcal M(I\otimes\mathcal K)$ satisfying $s_i^*s_i =s_1s_1^* + s_2s_2^* = 1_{\mathcal{M}(I\otimes \mathcal K)}$. Given $\phi_1,\phi_2\colon A\to \mathcal M(I\otimes\mathcal K)$, write%
\begin{equation}\label{KK-directsumeq}
  (\phi_1\oplus\phi_2)(a) \coloneqq s_1\phi_1(a)s_1^*+s_2\phi_2(a)s_2^*,\quad a\in
  A.
\end{equation}
As a map, $\phi_1\oplus\phi_2$ depends on the choice of $s_1$ and $s_2$.
  However, this sum is well-defined up to unitary equivalence, and hence up to homotopy in the point-strict topology, since the unitary group of $\mathcal M(I\otimes \mathcal K)$ is path-connected in the strict topology by \cite[Proposition~12.2.2]{Blackadar98}, for example. Note that $\oplus$ is associative up to unitary equivalence.
In this way, the sum in $KK(A,I)$ is well-defined by
\begin{equation}
[\phi_1,\psi_1]_{KK(A,I)}+[\phi_2,\psi_2]_{KK(A,I)}\coloneqq[\phi_1\oplus \phi_2,\psi_1\oplus\psi_2]_{KK(A,I)}
\end{equation}
(using the same choice of Cuntz isometries $s_1$ and $s_2$  to define both direct sums).
This turns $KK(A,I)$ into an abelian group.  The zero element is  $[\phi,\phi]_{KK(A,I)}$ where $\phi\colon A\to \mathcal M(I\otimes
\mathcal K)$ is  any $^*$-homomorphism (\cite[Lemma 4.1.5]{Jensen-Thomsen91}). 

\begin{remark}\label{rem:internal-vs-external}
    If $I$ is a  $C^*$-algebra and $s_1, s_2 \in \mathcal M(I \otimes \mathcal K)$ are isometries with $s_1 s_1^* + s_2 s_2^* = 1_{\mathcal M(I \otimes \mathcal K)}$, then there is an isomorphism
    \begin{equation} I \otimes \mathcal K \overset\cong\longrightarrow M_2(I \otimes \mathcal K) \colon x \mapsto \begin{pmatrix} s_1^* \\ s_2^* \end{pmatrix} x \begin{pmatrix} s_1 & s_2\end{pmatrix},
\end{equation}
which extends to an isomorphism $\mathcal M(I \otimes \mathcal K) \rightarrow M_2(\mathcal M(I \otimes \mathcal K))$.  Further, for $^*$-homomorphisms $\phi, \psi \colon A \rightarrow \mathcal M(I \otimes \mathcal K)$,  this isomorphism conjugates the Cuntz sum 
\begin{equation}
    s_1 \phi(\,\cdot\,)s_1^* + s_2 \psi(\,\cdot\,)s_2^* \colon A \rightarrow \mathcal M(I \otimes \mathcal K)
\end{equation} 
to the usual diagonal direct sum
\begin{equation}
\begin{pmatrix} 
    \phi & 0 \\ 0 & \psi
\end{pmatrix} 
\colon A \rightarrow M_2(\mathcal M(I \otimes \mathcal K)).
\end{equation}
For this reason, we use the notation $\phi \oplus \psi$ for both maps.
\end{remark}

A key feature is that $KK(\,\cdot\, ,\,\cdot\,)$ is a bifunctor.
Contravariance in the first variable is straightforward: a $^*$-homomorphism $\theta\colon A_1\to A_2$ induces a homomorphism $KK(\theta,I)\colon KK(A_2,I)\to KK(A_1,I)$ by precomposing Cuntz pairs (and homotopies thereof) by $\theta$.
A direct definition of $KK(A,\theta)$ (for $\theta\colon I_1 \to I_2)$ is possible, though subtle, in the Cuntz--Thomsen picture (see \cite{Thomsen90} or \cite[Section 4.1]{Jensen-Thomsen91}), though is straightforward in the pictures introduced by Cuntz (\cite{Cuntz83a,Cuntz87}).  All of our computations involving functoriality in the second variable will use the statement in Proposition~\ref{prop:KK-facts}\ref{prop:KK-facts.5} below.

We now (re)define $KK(A,I)$ to cover the case that $I$ is not $\sigma$-unital.

\begin{definition}
\label{KK-inductivelimit}
Let $A$ and $I$ be $C^*$-algebras with $A$ separable. Then
\begin{equation}\label{KK-inductivelimit-equation} KK(A,I) \coloneqq \varinjlim_{I_0\text{ sep.}} KK(A,I_0), \end{equation}
where $I_0$ ranges over all separable $C^*$-subalgebras of $I$, ordered by inclusion, and with connecting maps $KK(A,\iota_{I_0\subseteq I_1})$ for $I_0\subseteq I_1$.

Given a Cuntz pair $(\phi,\psi)\colon A \rightrightarrows E \rhd I$, we define $[\phi,\psi]_{KK(A,I)}$ to be the image of $[\phi|^{E_0},\psi|^{E_0}]_{KK(A,I_0)}$ in the right-hand side of \eqref{KK-inductivelimit-equation}, where $E_0\subseteq E$ and $I_0\subseteq I$ are separable $C^*$-subalgebras such that $I_0\lhd E_0$, $\phi(A)\cup\psi(A) \subseteq E_0$, and $(\phi-\psi)(A)\subseteq I_0$. We will check this is well-defined in Proposition~\ref{prop:KKwelldefined}.
\end{definition}

Definition~\ref{KK-inductivelimit} allows the functoriality in the
  second variable to be
  extended to non-$\sigma$-unital $I$, giving a bifunctor.  
We note that when $I$ is $\sigma$-unital but not separable, a priori Definitions~\ref{defn:KK} and \ref{KK-inductivelimit} give two competing definitions of $KK(A,I)$.
In Proposition~\ref{prop:KKequalsKKc-nonsep}, we show that these two definitions naturally agree.

The following facts are standard when $I$ is separable.  The proof that they extend to general $I$ is deferred to Proposition~\ref{prop:KasparovProdAppendix}.

\begin{proposition}
  \label{prop:KK-facts}
Let $A$ and $I$ be $C^*$-algebras with $A$ separable, let $E$ be a $C^*$-algebra containing $I$ as an ideal, and let $(\phi,\psi)\colon A \rightrightarrows E \rhd I$ be an $(A,I)$-Cuntz pair.
  \begin{enumerate}
  \item \label{prop:KK-facts.1}
    Let $\iota^{(2)}_I\colon I \to M_2(I)$ be the top-left
    corner inclusion. Then the induced map $KK(A,\iota^{(2)}_I)\colon KK(A,I) \to
    KK(A,M_2(I))$ is an isomorphism taking $[\phi,\psi]_{KK(A,I)}$ to $[\iota^{(2)}_{E}\circ\phi,\iota^{(2)}_E\circ\psi]_{KK(A,M_2(I))}$.
  \item \label{prop:KK-facts.2}
    Given a unitary $u \in I^\dagger$,
    \begin{equation}
      [\phi,\psi]_{KK(A,I)}=[\Ad u\circ\phi,\psi]_{KK(A,I)}.
    \end{equation}
\item \label{prop:KK-facts.4}
$[\phi,\psi]_{KK(A,I)}+[\psi,\phi]_{KK(A,I)}=0$.
\item \label{prop:KK-facts.5}
If $J \lhd F$ and $\theta\colon I \to J$ is a $^*$-homomorphism that extends to a $^*$-ho\-mo\-mor\-phism $\bar\theta\colon E \to F$, then $KK(A,\theta)([\phi,\psi]_{KK(A,I)}) = [\bar\theta\circ \phi,\bar\theta\circ \psi]_{KK(A,J)}$.
In particular,
$KK(A, \iota_{I\subseteq E})([\phi, \psi]_{KK(A,I)}) = [\phi]_{KK(A,E)} - [\psi]_{KK(A,E)}$ and if $\phi\colon A \to I$ and $\theta\colon I \to J$ are $^*$-homomorphisms, then 
\begin{equation}
    KK(A,\theta)([\phi]_{KK(A,I)})=[\theta\circ\phi]_{KK(A,J)}.
\end{equation}
  \item \label{prop:KK-facts.3}
    Let $A$ be nuclear, and suppose
    \begin{equation}
      \mathsf{e}\colon
      0\longrightarrow I \stackrel{j_{\mathsf{e}}}\longrightarrow
      E\stackrel{q_{\mathsf{e}}}\longrightarrow D\longrightarrow 0
    \end{equation}
is an extension of $C^*$-algebras.
Then the sequence
    \begin{equation}
      KK(A,I) \xrightarrow{KK(A,j_{\mathsf{e}})}
      KK(A,E)\xrightarrow{KK(A,q_{\mathsf{e}})} KK(A,D) 
    \end{equation}
    is exact.  The same holds when $A$ is $KK$-equivalent to a nuclear $C^*$-algebra.
  \end{enumerate}
\end{proposition}

There is a well-known action of $KK$ on $K$-theory,
\begin{equation}
  \label{KKtoHom}
  \Gamma^{(A,I)} \colon KK(A,I) \to \Hom(K_*(A), K_*(I))
\end{equation}
such that if $\phi \colon A \rightarrow I$ is a
$^*$-homomorphism, then
\begin{equation}
  \label{eq:Gamma-KK-K}
  \Gamma^{(A,I)}_i([\phi]_{KK(A, I)}) = K_i(\phi).
\end{equation}
This map is defined using the Kasparov product in the case where $I$ is $\sigma$-unital, and it is extended to the non-separable case by taking the limit; this is laid out (and \eqref{eq:Gamma-KK-K} is justified) in  \eqref{eq:AppendixKKtoHom}.

Next, we recall the $KL$-groups, first conceived to help classify
$^*$-ho\-mo\-mo\-rphisms up to approximate unitary equivalence.  R\o{}rdam's
original definition (\cite[Section 5]{Rordam95}) required a UCT
assumption. This was relaxed by Dadarlat in \cite{Dadarlat05}, who
defined $KL(A,I)$ to be the quotient of $KK(A,I)$ by the closure of
$\{0\}$ in various equivalent topologies on $KK(A,I)$.
We use a direct description of this closure, also from
\cite{Dadarlat05}, which we give without any separability assumptions
on $I$.

\begin{definition}[cf. {\cite[Theorem~3.5, Theorem~4.1, and Section~5]{Dadarlat05}}]
\label{def:KL}
  Let $A$ and $I$ be $C^*$-algebras with $A$ separable. Write $\bcN
  \coloneqq \mathbb N \cup \{\infty\}$ for the one-point
  compactification of $\mathbb N$ and $\ev_n$ for evaluation at a
  point $n\in\bcN$.  Define
\begin{equation}
\label{eq:KLZ}
  Z_{KK(A,I)}\coloneqq
  \left\{
    KK(A,\ev_\infty)(\kappa)\colon \!
    \begin{array}{l}
       \kappa \in KK(A, C(\bcN, I)) \text{ and }\\
       KK(A,\ev_n)(\kappa) = 0\, \forall n \in \mathbb N
    \end{array}\hspace{-1ex}
    \right\}
\end{equation}
and
\begin{equation}
  \label{eq:def-kl}
  KL(A,I)\coloneqq KK(A,I)/Z_{KK(A,I)}.
\end{equation}
\end{definition}

With this definition, the bifunctor structure descends from $KK(\,\cdot\, , \cdot\,)$ to $KL(\,\cdot\,,\,\cdot\,)$, just as it does when $I$ is separable.  Likewise $\Gamma^{(A,I)}$ descends to a map \begin{equation}\label{def:tildegamma}
  \widetilde{\Gamma}\colon KL(A,I)\to\mathrm{Hom}(K_*(A),K_*(I)),  
\end{equation}
(see the more general statement in the first part of Theorem \ref{thm:the-umct}, proved for non-separable $I$ in Appendix \ref{appendix.corestrict}).
We can also express $KL$ as an inductive limit over separable subalgebras, analogous to how we defined $KK$ in Definition~\ref{KK-inductivelimit}; we prove the following in Appendix~\ref{sec:kkappendix}, immediately following Lemma~\ref{lem:ZKK-limit}.

\begin{proposition}\label{KL-inductivelimit}
Let $A$ and $I$ be $C^*$-algebras with $A$ separable. Then
\begin{equation}\label{KL-inductivelimit-equation} KL(A,I) \cong \varinjlim_{I_0\text{\emph{ sep.}}} KL(A,I_0) \end{equation}
where $I_0$ ranges over all separable $C^*$-subalgebras of $I$, ordered by inclusion, and the isomorphism is induced by the inclusions $I_0 \rightarrow I$.
\end{proposition}

Bootstrapping standard facts for $KK$ (from Proposition~\ref{prop:KK-facts}), we arrive at the following $KL$ version (see Proposition~\ref{prop:KasparovProdAppendix}).

\begin{proposition}
\label{prop:KL-facts}
Let $A$ and $I$ be $C^*$-algebras with $A$ separable, let $E$ be a $C^*$-algebra containing $I$ as an ideal, and let $(\phi,\psi)\colon A \rightrightarrows E \rhd I$ be an $(A,I)$-Cuntz pair. Then:
  \begin{enumerate}
  \item \label{prop:KL-facts.1}
    Let $\iota^{(2)}_I\colon I \to M_2(I)$ be the top-left
    corner inclusion. Then the induced map $KL(A,\iota^{(2)}_I)\colon KL(A,I) \to
    KL(A,M_2(I))$ is an isomorphism taking $[\phi,\psi]_{KL(A,I)}$ to $[\iota^{(2)}_E \circ\phi,\iota^{(2)}_E\circ\psi]_{KL(A,M_2(I))}$.
  \item \label{prop:KL-facts.2}
    Given a unitary in $I^\dagger$,
    \begin{equation}
      [\phi,\psi]_{KL(A,I)}=[\Ad u\circ\phi,\psi]_{KL(A,I)}.
    \end{equation}
\item \label{prop:KL-facts.4}
$[\phi,\psi]_{KL(A,I)}+[\psi,\phi]_{KL(A,I)}=0$.
\item \label{prop:KL-facts.5}
If $J \lhd F$ and $\theta\colon I \to J$ is a $^*$-homomorphism that extends to a $^*$-ho\-mo\-morphism $\bar\theta\colon E \to F$, then $KL(A,\theta)([\phi,\psi]_{KL(A,I)}) = [\bar\theta\circ \phi,\bar\theta\circ \psi]_{KL(A,J)}$.
In particular, 
$KL(A, \iota_{I\subseteq E})([\phi, \psi]_{KL(A,I)}) = [\phi]_{KL(A,E)} - [\psi]_{KL(A,E)}$, and if $\phi\colon A \to I$ and $\theta\colon I \to J$ are $^*$-homomorphisms, then
\begin{equation}
    KL(A,\theta)([\phi]_{KL(A,I)})=[\theta\circ\phi]_{KL(A,J)}.
\end{equation}
  \end{enumerate}
\end{proposition}

\subsection{Absorption}\label{sec:absorption}

Now we turn to absorption. Loosely speaking, an
absorbing $^*$-homomorphism is one for which the conclusion of
Voiculescu's non-commutative Weyl--von Neumann Theorem holds.  The
relevance of absorption can be seen in the early developments of the
theory: in \cite{Pimsner-Popa-etal79}, for example, Pimsner, Popa, and
Voiculescu prove an absorption theorem for extensions of $C(X)\otimes
\mathcal{K}$ for $X$ finite dimensional, in part to generalize Brown--Douglas--Fillmore theory;
Kasparov's generalization of Voiculescu's theorem in \cite{Kasparov80a}
provides (nuclearly) absorbing $^*$-ho\-mo\-mor\-phisms. 

Whereas we set out $KK$-theory without
restriction on the second variable, absorption results are another
matter. For these, it is vital that one has a countable approximate unit.

Throughout the rest of the section, $I$ will be $\sigma$-unital and
stable. We write $\mathcal Q(I)$ for the corona algebra $\mathcal
M(I)/I$, and $q_I\colon \mathcal M(I)\to \mathcal Q(I)$ for the
quotient map. We can form the direct sum of two maps
$\phi_1,\phi_2\colon A\to \mathcal M(I)$ as in \eqref{KK-directsumeq}.
Likewise, given maps $\psi_1,\psi_2\colon A\to \mathcal Q(I)$, we can
form the direct sum $\psi_1\oplus \psi_2\colon A\to \mathcal Q(I)$
(using isometries coming from $\mathcal M(I)$). Both of these are well
defined up to unitary equivalence (in the latter case by unitaries
from $\mathcal M(I)$). Therefore, the notions of absorption in Definition \ref{defn:absorption} below do not depend on the isometries chosen. The second part of the definition goes back to Kasparov
  (phrased in the language of extensions) as \cite[p. 560, Definition
  2]{Kasparov80}.  To our knowledge, absorbing
  $^*$-ho\-mo\-mor\-phisms into general multiplier algebras (as in
  part \ref{absorption.1} of the definition) started to appear explicitly around the
  turn of the millennium, and are formalized in \cite{Thomsen01}.

\begin{definition}\label{defn:absorption}
  Let $A$ and $I$ $C^*$-algebras with $A$ separable and $I$ $\sigma$-unital and stable.
  \begin{enumerate}
  \item \label{absorption.1}
A $^*$-ho\-mo\-mor\-phism $\phi \colon A \to \M(I)$ is
    \emph{absorbing} if for every $^*$-ho\-mo\-mor\-phism $\psi \colon
    A \to \M(I)$, there is a sequence of unitaries $(u_n)_{n=1}^\infty
    \subseteq \M(I)$ such that, for each $a\in A$, we have
    \begin{equation}
      \big(n \mapsto  u_n (\phi(a) \oplus \psi(a))u_n^* - \phi(a)
      \big) \in C_0(\mathbb N, I).
    \end{equation}
  \item \label{absorption.2}
A $^*$-ho\-mo\-mor\-phism $\theta \colon A \to \mathcal Q(I)$
    is \emph{absorbing} if for every $^*$-ho\-mo\-mor\-phism $\psi
    \colon A \to \M(I)$ there is a unitary $u \in \M(I)$ such that
    $\Ad(q_I(u)) \circ (\theta \oplus (q_I\circ\psi)) = \theta$.
\end{enumerate}
\end{definition}

The next three results are not new, but they are not entirely
straightforward to extract from the literature.  The proof below
follows the computations in \cite{Dadarlat-Eilers01} and
\cite{Dadarlat-Eilers02}, which in turn uses ideas from
\cite{Kasparov80a}.  The ``asymptotic absorption'' condition in \ref{absorption.cond3}
below originates in \cite{Dadarlat-Eilers01}, and Corollary
\ref{cor:absorption-uniqueness} below is implicit there.

\begin{proposition}\label{prop:absorption}
Let $A$ and $I$ be $C^*$-algebras with $A$ separable and $I$ $\sigma$-unital and
stable, and let $\phi \colon A \to \M(I)$ be a $^*$-ho\-mo\-mor\-phism. The following are equivalent:
\begin{enumerate}
  \item \label{absorption.cond1}
$\phi$ is absorbing;
  \item \label{absorption.cond2}
for all $^*$-ho\-mo\-mor\-phisms $\psi \colon A \to \M(I)$, there is a unitary $u \in \M(I)$ such that
  \begin{equation}
     u(\phi(a) \oplus \psi(a)) u^* - \phi(a) \in I,\quad a\in A;  \end{equation}
  \item \label{absorption.cond3}
for all $^*$-ho\-mo\-mor\-phisms $\psi \colon A \to \M(I)$, there is a norm-continuous path $(u_t)_{t \geq 0} \subseteq \M(I)$ of unitaries such that
\begin{equation}
  \big(t \mapsto u_t (\phi(a) \oplus \psi(a)) u_t^* - \phi(a) \big) \in C_0([0, \infty), I),\quad a\in A.
\end{equation}
\end{enumerate}
\end{proposition}

\begin{proof}
  The implications \ref{absorption.cond3}$\Rightarrow$\ref{absorption.cond1}$\Rightarrow$\ref{absorption.cond2} are
  clear. Now assume that \ref{absorption.2} holds and fix a $^*$-ho\-mo\-mor\-phism
  $\psi \colon A \to \M(I)$.  Because $I$ is stable, one has $1_{\M(I)} \otimes \mathcal
    B(\mathcal H) \subseteq \M(I\otimes \mathcal K)$, so we may find a
  sequence $(v_n)_{n=1}^\infty \subseteq \M(I)$ of isometries such
  that $\sum_{n=1}^\infty v_n v_n^* = 1_{\M(I)}$ with convergence in the
  strict topology. In
  particular,
\begin{equation}\label{prop:absorption.3}
  v_n^*xv_n\to 0,\quad x\in I.
\end{equation}

Define $^*$-ho\-mo\-mor\-phisms $\psi_\infty,(\psi_\infty)_\infty\colon A\to \M(I)$ by
\begin{equation}\label{eq:psi-infty}
  \psi_\infty(a) \coloneqq \sum_{n=1}^\infty v_n \psi(a) v_n^*,\quad a\in A,\end{equation}
and
\begin{equation}\label{eq:psi-infty-infty}
  (\psi_{\infty})_\infty(a)\coloneqq\sum_{n=1}^\infty v_n\psi_\infty(a)v_n^*=\sum_{m,n=1}^\infty v_nv_m\psi(a)v_m^*v_n^*,\quad a\in A.
\end{equation}
Note that, for each $n\in\mathbb N$ and $a\in A$,
\begin{equation}\label{prop:absorption.2}
v_n^*(\psi_{\infty})_\infty(a)v_n=\psi_\infty(a).
\end{equation}
Now fix isometries $s_1,s_2 \in \M(I)$ with $s_1s_1^* + s_2s_2^* =
1_{\M(I)}$ so that
\begin{equation}
  \big( \phi \oplus (\psi_\infty)_\infty \big)(a) = s_1\phi(a)s_1^* +
  s_2(\psi_\infty)_\infty(a)s_2^*,\quad a\in A. 
\end{equation}
Using \ref{absorption.cond2}, let $u \in \M(I)$ be a unitary with
\begin{equation}
  u \big( \phi(a) \oplus (\psi_\infty)_\infty(a) \big) u^* - \phi(a)
  \in I,\quad a \in A,
\end{equation}
and note that this implies
\begin{equation}\label{prop:absorption.1}
  s_2^*u^*\phi(a) u s_2 - (\psi_\infty)_\infty(a)\in I, \quad a\in A.
\end{equation}
Then $w_n \coloneqq u s_2v_n$ is an isometry for each $n \geq 1$ and,
for $n\neq m$, we have $w_n^*w_m=v_n^*v_m=0$.  Furthermore, for each
$a\in A$,
\begin{equation}
n\mapsto w_n^* \phi(a) w_n - \psi_\infty(a) \stackrel{\eqref{prop:absorption.2}}= v_n^* \big( s_2^* u^*
\phi(a) u s_2^* - ( \psi_\infty )_\infty(a) \big) v_n
\end{equation}
lies in $C_0(\mathbb N,I)$ by \eqref{prop:absorption.1}  and \eqref{prop:absorption.3}. Then, for all $a\in A$, we have
\begin{equation}
  \begin{aligned}
    \Big( n \mapsto \big(w_n&\psi_\infty (a) - \phi(a)w_n\big)^*\big(w_n\psi_\infty(a)-\phi(a)w_n\big) \Big) \\
     &= \Big( n \mapsto \psi_\infty(a^*)\big(\psi_\infty(a)-w_n^*\phi(a)w_n\big) \\
     &\qquad+ \big(\psi_\infty(a^*)-w_n^*\phi(a^*)w_n\big)\psi_\infty(a)
    \\ &\qquad+\big(w_n^*\phi(a^*a)w_n-\psi_\infty(a^*a)\big)\Big) \in
    C_0(\mathbb N, I).
  \end{aligned}
\end{equation}
Therefore,
\begin{equation}
  \big(n \mapsto w_n \psi_\infty(a) - \phi(a) w_n \big) \in C_0(\mathbb N, I)
\end{equation}
by the $C^*$-identity.  Condition \ref{absorption.cond3} now follows from  \cite[Lemma~2.3]{Dadarlat-Eilers01} (which uses \cite[Lemma~2.16]{Dadarlat-Eilers02}).
\end{proof}

The two notions of absorption in Definition~\ref{defn:absorption} are closely related.

\begin{corollary}\label{cor:multipler-vs-corona}
Suppose $A$ and $I$ are $C^*$-algebras with $A$ separable and $I$ $\sigma$-unital and stable.  Then a $^*$-ho\-mo\-mor\-phism $\phi \colon A \to \M(I)$ is absorbing if and only if $q_I\circ \phi \colon A \to \Q(I)$ is absorbing.
\end{corollary}

\begin{proof}
  The forward direction is clear.  The backward direction follows from
  \ref{absorption.cond2}$\Rightarrow$\ref{absorption.cond1} of Proposition~\ref{prop:absorption}.
\end{proof}

The asymptotic absorption in Proposition~\ref{prop:absorption}\ref{absorption.cond3} leads to the following asymptotic uniqueness result for absorbing $^*$-ho\-mo\-mor\-phisms, which we use in the $\Z$-stable $KK$- and $KL$-uniqueness theorems.

\begin{corollary}\label{cor:absorption-uniqueness}
Suppose $A$ and $I$ are $C^*$-algebras with $A$ separable and $I$ $\sigma$-unital and stable.  If $\phi, \psi \colon A \to \M(I)$ are absorbing representations, then there is a norm-continuous path $(u_t)_{t \geq 0} \subseteq \M(I)$ of unitaries such that
\begin{equation}\label{cor:absorption-uniqueness.1}
  \big(t \mapsto u_t \phi(a) u_t^* - \psi(a) \big) \in C_0([0, \infty), I).
\end{equation}
\end{corollary}

\begin{proof}
By Proposition~\ref{prop:absorption}\ref{absorption.cond3} applied to both $\phi$ and $\psi$, there are norm-continuous paths $(v_t)_{t\geq 0}$ and $(w_t)_{t\geq0}$ of unitaries in $\mathcal M(I)$ such that for $a\in A$,
\begin{equation}
  \begin{aligned}
    &\big(t\mapsto v_t(\phi(a)\oplus \psi(a))v_t^*-\phi(a)\big)\in C_0([0,\infty),I),\quad\text{ and}\\
    &\big(t\mapsto w_t(\phi(a)\oplus \psi(a))w_t^*-\psi(a)\big)\in
    C_0([0,\infty),I).
  \end{aligned}
\end{equation}
The result follows with $u_t\coloneqq w_tv_t^*$.
\end{proof}

We record the standard fact that the relative commutant of an absorbing map in the corona algebra is properly infinite, for later reference.
\begin{proposition}\label{absorb:propinfinite}
Let $A$ and $I$ be $C^*$-algebras with $A$ separable and $I$ $\sigma$-unital and stable, and let $\phi\colon A\to\mathcal M(I)$ be an absorbing $^*$-homomorphism. Then the relative commutant $\mathcal Q(I)\cap q_I(\phi(A))'$ is properly infinite.
\end{proposition}
\begin{proof}
     Since $\phi$ is absorbing, $\phi$ is unitarily equivalent to $\phi \oplus \phi$ modulo $I$ by Corollary~\ref{cor:absorption-uniqueness}, and if $u$ is a unitary implementing this equivalence and $s_1$ and $s_2$ are the Cuntz isometries defining the direct sum, then $q_I(us_1)$ and $q_I(us_2)$ are isometries in $\mathcal Q(I) \cap q_I(\phi(A))'$ with orthogonal range projections.
\end{proof}

Next we turn to tools for verifying absorption of maps $A\to\mathcal
Q(I)$ when $A$ is separable and $I$ is $\sigma$-unital and stable. These date
back to Voiculescu's theorem (\cite{Voiculescu76}) for the case
$I=\mathcal K$, which was then generalized by Kasparov to produce
absorbing maps when at least one of $A$ or $I$ is nuclear
(\cite[Theorem 6]{Kasparov80a}).  The same result when $I$ is separable but without the
  nuclearity assumption was later obtained by Thomsen
  (\cite{Thomsen01}).  Kirchberg first characterized absorption for
$^*$-ho\-mo\-mor\-phisms under certain pure infiniteness criteria
(\cite[Theorem 6]{Kirchberg95b}; see also \cite[Theorem
8.3.1]{Rordam02}).  In the presence of nuclearity, Elliott and
Kucerovsky used Kirchberg's work (see \cite[Lemma
11]{Elliott-Kucerovsky01}) to give a beautiful characterization
(\cite[Theorem 6]{Elliott-Kucerovsky01}) of unitally absorbing
$^*$-ho\-mo\-mor\-phisms (defined by requiring all maps in
Definition~\ref{defn:absorption} to be unital) by a condition they
call \emph{pure largeness.} We will explore pure largeness in more
detail in our subsequent work, but, for the maps appearing in this
paper, Kucerovsky and Ng's corona factorization property
(\cite{Kucerovsky-Ng06}) allows a clean characterization of absorption
in terms of unitizably full maps. Recall that $A^\dag$ denotes the forced unitization of $A$, i.e., with a new unit added when $A$ is unital.

\begin{definition} A $^*$-ho\-mo\-mor\-phism $\phi \colon A \to D$
  between $C^*$-algebras with $D$ unital is said to be
  \emph{unitizably full} if the unitized map $\phi^\dag \colon A^\dag
  \to D$ is full.
\end{definition}

With this in place, the characterization we use is as follows.  The
short proof is deceiving as the real work is outsourced to
references, especially \cite[Theorem~6]{Elliott-Kucerovsky01}, which characterizes unitally nuclearly absorbing maps; the non-unital version \cite[Theorem~2.6]{Gabe16} used here is an immediate corollary of the unital case.
It is fundamental to our abstract approach to $C^*$-algebra
classification in just the same way that a corresponding version for
$C^*$-algebras that absorb the universal UHF algebra was vital in
\cite{Schafhauser17,Schafhauser18}.

\begin{theorem}\label{thm:absorbing-z-stable}
If $A$ is a separable nuclear $C^*$-algebra and $I$ is a $\sigma$-unital stable $\mathcal Z$-stable $C^*$-algebra, then a $^*$-ho\-mo\-mor\-phism $A \to \M(I)$ or $A \to \mathcal Q(I)$ is absorbing if and only if it is unitizably full.
\end{theorem}

\begin{proof}
By Proposition~\ref{prop:Zcfp}, $I$ has the corona factorization property.
The statement for a $^*$-ho\-mo\-mor\-phism $\phi\colon A\to\Q(I)$ is \cite[Theorem~2.6]{Gabe16} (which is phrased in the language of extensions) after noting that, since $A$ is nuclear, absorption and nuclear absorption of $\phi$ are equivalent. The multiplier version follows from Corollary~\ref{cor:multipler-vs-corona} and the standard fact that (by stability of $I$) an element $x\in \M(I)$ is full if and only if its image in $\mathcal Q(I)$ is full. To see this fact, take $x\in \mathcal M(I)$ with full image in $\mathcal Q(I)$, so there are $y_1,\dots, y_n, z_1,\dots, z_n \in \M(I)$ such that $c \coloneqq 1_{\mathcal M(I)} - \sum_{i=1}^n y_i x z_i \in I$. As $I$ is stable we may pick an isometry $v\in \M(I)$ such that $\|v^* c v \| < 1$. Hence $\| 1_{\mathcal M(I)} - \sum_{i=1}^n v^* y_i x z_i v\| =\| v^* c v \|< 1$ which implies that $x$ is full in $\mathcal M(I)$.
\end{proof}

\subsection{$KK$-existence}\label{sec:KK-Existence}

The $KK$-existence theorem we use comes in two forms, both of which are essentially folklore to experts.  The first shows that all $KK$-classes can be realized by Cuntz pairs with a fixed absorbing $^*$-ho\-mo\-mor\-phism in the second variable, while the second uses a result of Dadarlat to rephrase a standard fact from extension theory in terms of $KK$-theory.

\begin{theorem}[$KK$-existence]\label{thm:KK-Existence}
  Suppose $A$ is a separable $C^*$-algebra and $I$ is a
  $\sigma$-unital stable $C^*$-algebra.
  \begin{enumerate}
  \item \label{KKexistence.1}
If $\psi \colon A \to \M(I)$ is an absorbing
    $^*$-ho\-mo\-mor\-phism and $\kappa \in KK(A, I)$, then there is
    an absorbing $^*$-ho\-mo\-mor\-phism $\phi \colon A \to \M(I)$
    such that $(\phi, \psi)$ is an $(A, I)$-Cuntz pair and $[\phi,
    \psi]_{KK(A,I)} = \kappa$.    \label{thm:KK-Existence.C1}
  \item\label{KKexistence.2}
    Suppose $A$ is nuclear.  If $\theta \colon A \to \mathcal
    Q(I)$ is an absorbing $^*$-ho\-mo\-mor\-phism, then there is a
    (necessarily) absorbing $^*$-ho\-mo\-mor\-phism $\phi \colon A \to
    \M(I)$ lifting $\theta$ if and only if $[\theta]_{KK(A,\Q(I))}=0$.
    \label{thm:KK-Existence.C2}
  \end{enumerate}
\end{theorem}

Part \ref{KKexistence.1} is standard --- a version for weakly nuclear, nuclearly absorbing $^*$-ho\-mo\-mor\-phisms is given in \cite[Proposition~2.6]{Schafhauser18}, for example, and the same proof works here.  The result has its origins in Higson's Paschke duality result (\cite{Higson95}) for $K$-homology, showing that if $\kappa \in KK(A, \mathcal K)$ and $\psi \colon A \to \M(\mathcal K)$ is absorbing, then there is a unitary $u \in \M(\mathcal K)$ commuting with $\psi(A)$ modulo $\mathcal K$ such that $\kappa = [\psi, \psi, u]_{KK(A, \mathcal K)}$ in the Fredholm picture of $KK$-theory.
The result for general $I$ was obtained by Thomsen in \cite{Thomsen01} in the proof of his Paschke duality result (particularly, the proof of surjectivity in \cite[Theorem~3.2]{Thomsen01}).

Part \ref{KKexistence.2} is a consequence of Dadarlat's isomorphism between $\mathrm{Ext}(A,I)$ and $KK(A,\Q(I))$ for $A$ separable and nuclear and $I$ stable and $\sigma$-unital from \cite{Dadarlat00b}.  In the proof below, we identify extensions with their Busby invariants. Recall that, under the hypotheses of Theorem~\ref{thm:KK-Existence}, an extension of $A$ by $I$ with Busby invariant $\theta\colon A\to\Q(I)$ splits if and only if $\theta$ lifts to $\phi\colon A\to \M(I)$. Moreover, the class of $\theta$ in $\Ext(A,I)$ vanishes precisely when $\theta\oplus\psi$ splits for some split extension $\psi$.  See \cite[Chapter~15]{Blackadar98} for details.

\begin{proof}[Proof of Theorem~\ref{thm:KK-Existence}\ref{KKexistence.2}]
If $\theta$ lifts to a $^*$-ho\-mo\-mor\-phism $\phi \colon A \to
\M(I)$, then $[\theta]_{KK(A,\Q(I))} = 0$ since $KK(A, \M(I)) = 0$ by
\cite[Proposition~4.1]{Dadarlat00b}. For the converse, if
$[\theta]_{KK(A,\Q(I))}$ vanishes, then the extension with Busby
invariant $\theta$ has the trivial class in $\Ext(A,I)$, by applying the isomorphism between $KK(A,\Q(I))$ and $\Ext(A,I)$ from \cite[Proposition~4.2]{Dadarlat00b}.  Thus there is a split extension, say with Busby invariant $q_I\circ \psi\colon A\to\Q(I)$ for some $^*$-ho\-mo\-mor\-phism $\psi\colon A\to\M(I)$, such that $\theta\oplus (q_I\circ \psi)$ splits, and so has a lift to $\M(I)$. But since $\theta$ is absorbing, there is a unitary $u\in \M(I)$ with $q_I(u)(\theta\oplus (q_I\circ \psi))q_I(u)^*=\theta$.  Thus $\theta$ has a lift $\phi\colon A\to \M(I)$, which is necessarily absorbing by
Corollary~\ref{cor:multipler-vs-corona}, since $\theta$ is absorbing.
\end{proof}

\subsection{$KK$-uniqueness}\label{sec:KK-Uniqueness}

Now we turn to our $\Z$-stable $KK$- and $KL$-uniqueness theorems.  The first part of the following result is \ref{kk-unique1}$\Rightarrow$\ref{kk-unique2} of Theorem~\ref{intro-KK-unique} from the introduction.  The other implication will be proved at the end of this subsection.

\begin{theorem}[{$\mathcal Z$-stable $KK$- and $KL$-uniqueness}]\label{thm:KK-Uniqueness}\mbox{}
Let $A$ and $I$ be $C^*$-algebras with $A$ separable and $I$ $\sigma$-unital and stable. 
Let $(\phi,\psi)\colon A\rightrightarrows \mathcal M(I) \rhd I$ be a Cuntz pair with $\phi$ and $\psi$ absorbing.
\begin{enumerate}
  \item \label{KK-Uniqueness}
If $[\phi, \psi]_{KK(A,I)} = 0$, then there is a norm-continuous path $(u_t)_{t \geq 0}$ of unitaries in $(I \otimes \mathcal Z)^\dag$ such that
      \begin{equation}
         \| u_t (\phi(a) \otimes 1_{\mathcal Z}) u_t^* - \psi(a)\otimes 1_\Z \| \to 0, \quad a \in A.
      \end{equation}
\item \label{KL-Uniqueness}
If $[\phi, \psi]_{KL(A,I)} = 0$, then there is a sequence $(u_n)_{n=1}^\infty$ of unitaries in $(I \otimes \mathcal Z)^\dag$ such that
      \begin{equation}
         \| u_n (\phi(a) \otimes 1_{\mathcal Z}) u_n^* - \psi(a)\otimes 1_\Z \| \to 0, \quad a \in A.
      \end{equation}
\end{enumerate}
\end{theorem}

Our proof of \ref{KK-Uniqueness} follows the same strategy as Dadarlat and Eilers' original $KK$-uniqueness theorem for $I=\mathcal K$ and no $\mathcal Z$-stabilization (\cite[Theorem~3.12]{Dadarlat-Eilers01}), making heavy use of Paschke duality (\cite{Paschke81}, and generalizations \cite[Theorem 3.2]{Thomsen01}, \cite[Lemma~3.5]{Dadarlat-Eilers01}).  More precisely, given a Cuntz pair $(\phi,\psi)\colon A\rightrightarrows \mathcal M(I)\rhd I$ of absorbing representations, take a unitary path $(u_t)_{t\geq 0}$ as in Corollary~\ref{cor:absorption-uniqueness}. Then $q_I(u_0)$ commutes with $q_I(\phi(A))$. Paschke duality ensures that $[\phi,\psi]_{KK(A,I)}=0$ if and only if the class of $q_I(u_0)$ is trivial in $K_1(\mathcal Q(I)\cap q_I(\phi(A))')$.  
If $q_I(u_0)$ is in the path component of the identity in $\Q(I) \cap q_I(\phi(A))'$, then $\phi$ and $\psi$ are properly asymptotically unitarily equivalent; i.e., there is a continuous path $(u_t)_{t\geq 0}$ of unitaries in $I^\dag$ with $u_t\phi(a)u_t^*\rightarrow \psi(a)$ for all $a\in A$. This is extracted from the proof of \cite[Theorem~3.12]{Dadarlat-Eilers01} as Lemma~\ref{lemma:AsympEquiv}\ref{asympequiv.1} below. Accordingly, this strategy gives $KK$-uniqueness whenever $\mathcal Q(I) \cap q_I(\phi(A))'$ is $K_1$-injective.

Dadarlat and Eilers' stable $KK$-uniqueness theorem (recalled in the remarks after Theorem~\ref{intro-KK-unique}) avoids the $K_1$-injectivity problem by using an extra summand.  As $\mathcal Q(I) \cap q_I(\phi(A))'$ is properly infinite (noted in Proposition \ref{absorb:propinfinite}), $[q_I(u_0)]_1=0$ implies that $q_I(u_0) \oplus 1_{\mathcal Q(I)}$ is in the path component of the identity in $U_2(\mathcal Q(I)\cap q_I(\phi(A))')$ (see \cite[Exercise 8.11(ii)]{Rordam-Larsen-etal00}). The requirement for the direct summand of $1_{\mathcal Q(I)}$ in this argument leads to the summand $\eta$ in condition \ref{kk-unique3} following Theorem \ref{intro-KK-unique}, which describes the stable uniqueness theorem (in the case when $\phi$ and $\psi$ are absorbing).


For our $\Z$-stable $KK$-uniqueness theorem, we bypass the problem of whether $\mathcal Q(I)\cap q_I(\phi(A))'$ is $K_1$-injective, by working with its $\Z$-stabilization --- which is $K_1$-injective by Jiang's result (Theorem~\ref{prop:z-stable-K1-inj}), 
or alternatively the later unpublished result from Rohde's PhD thesis 
\cite{Rohde09} 
that properly infinite $\Z$-stable $C^*$-algebras are $K_1$-injective. 
This allows us to conclude that $u_0 \otimes 1_\mathcal Z$ is in the path component of the identity in the $\Z$-stabilization of $\mathcal Q(I)\cap q_I(\phi(A))'$.
So in spirit, our $\Z$-stable $KK$-uniqueness theorem is obtained by tensoring on a copy of $1_\Z$, whereas the Dadarlat--Eilers theorem adjoins a direct summand of the identity.

The $KL$-uniqueness theorem (Theorem~\ref{thm:KK-Uniqueness}\ref{KL-Uniqueness}) follows from the $KK$-uniqueness theorem (part \ref{KK-Uniqueness}). 
The strategy, which goes back to Lin (\cite{Lin02,Lin05}) and Dadarlat (\cite{Dadarlat05}), is to relate vanishing in $KL$ with approximate vanishing in $KK$. This is exemplified in  \cite[Theorem 5.1]{Dadarlat05}, where Dadarlat uses a reduction to $KK$-classification to reprove that $KL$ detects approximate uniqueness of morphisms between Kirchberg algebras; our argument for \ref{KL-Uniqueness} is in the same spirit.

We now give an abstract version of Dadarlat and Eilers' argument.  The strategy has its origins in the large body of work on automorphisms and derivations on operator algebras undertaken in the 1960s (\cite{Kadison66,Sakai66,KR67,OlesenPedersen74}; see also \cite[Sections 8.6 and 8.7]{Pedersen79}).

\begin{lemma}\label{lemma:AsympEquiv}
Let $A$ be a unital separable $C^*$-algebra, and let
\begin{equation}
\begin{tikzcd}
    0\ar[r]& I\ar[r] &E \ar[r,"q"] &D \ar[r] &0
\end{tikzcd}
\end{equation}
be an extension of $C^*$-algebras with $E$ unital.  Suppose $\phi, \psi \colon A \to E$ are unital $^*$-ho\-mo\-mor\-phisms such that $q \circ \phi = q \circ\psi$ and that this composition is injective.  Suppose further that there is a continuous path $(v_t)_{t \geq 0}$ of unitaries in $E$ such that
\begin{equation}\label{eq:asymp-equiv-lemma}
  \big( t \mapsto v_t \phi(a) v_t^* - \psi(a) \big) \in C_0([0, \infty), I), \quad a\in A.
\end{equation}
\begin{enumerate}[(i)]
  \item \label{asympequiv.1}
If the path $(v_t)_{t \geq 0}$ satisfying \eqref{eq:asymp-equiv-lemma} can be chosen such that $q(v_0)$ lies in the path component of the identity in the unitary group of $D\cap q(\phi(A))'$, then it can be chosen with $v_t \in I^\dag \subseteq E$ for all $t \geq 0$. \label{lemma:AsympEquiv.C1}
  \item \label{asympequiv.2}
If the path $(v_t)_{t \geq 0}$ satisfying \eqref{eq:asymp-equiv-lemma} can be chosen with $[q(v_0)]_1 = 0$ in $K_1(D \cap q(\phi(A))')$, then there is a path $(u_t)_{t \geq 0}$ of unitaries in $(I \otimes \mathcal Z)^\dag$ such that for all $a \in A$,
  \begin{equation}\label{eq:asymp-equiv-lemma2}
    \big( t \mapsto u_t (\phi (a) \otimes 1_{\mathcal Z}) u_t^* - \psi(a) \otimes 1_{\mathcal Z} \big) \in C_0([0, \infty), I \otimes \mathcal Z).
  \end{equation}
\end{enumerate}
\end{lemma}

\begin{proof}
\ref{asympequiv.1}: A path $(v_t)_{t\geq 0}$ of unitaries  satisfying \eqref{eq:asymp-equiv-lemma} such that $q(v_0)$ is in the path component of the identity in $D \cap q(\phi(A))'$ can be adjusted so that $v_0\in I^\dagger$.  Indeed, the hypothesis gives self-adjoint elements $\bar h_1, \ldots, \bar h_n \in D \cap q(\phi(A))'$ such that
$q(v_0) = e^{i \bar h_1} \cdots e^{i \bar h_n}$. Lifting these to self-adjoints $h_1, \ldots, h_n \in E$ with $q(h_i) = \bar h_i$ for all $i = 1, \ldots, n$, define
$v_t\coloneqq v_0 e^{i th_n} \cdots e^{i t h_1}$ for $-1 \leq t < 0$. Then $(v_t)_{t \geq -1}$ is a continuous path of unitaries in $E$ with $v_{-1} \in I^\dag$ and
\begin{equation*}
  \big( t \mapsto v_t \phi(a) v_t^* - \psi(a) \big) \in C_0([-1, \infty), I).
\end{equation*}
Hence by shifting the index, we assume the given path $(v_t)_{t \geq 0}$ satisfies $v_0 \in I^\dag$.

As
\begin{equation}
\begin{aligned}
v_0^*v_t\phi(a)v_t^*v_0-\phi(a)&=v_0^*(v_t\phi(a)v_t^*-\psi(a))v_0\\&\quad+(v_0^*\psi(a)v_0-\phi(a))\in I,\quad a\in A,
\end{aligned}
\end{equation}
 $\phi(A) + I$ is a $C^*$-subalgebra of $\mathcal M(I)$ which is invariant under $\Ad(v_0^* v_t)$ for all $t \geq 0$.  Then, since $A$ is separable, we can find a separable $C^*$-subalgebra $C \subseteq \phi(A) + I$ containing $\phi(A)$ such that $C$ is invariant under $\Ad(v_0^* v_t)$ for all $t \geq 0$. To see this, let $(t_n)_{n=1}^\infty$ be an enumeration of $[0,\infty)\cap\mathbb Q$. Starting with $C_0=\phi(A)$, choose separable $C^*$-subalgebras $C_0\subseteq C_1\subseteq\dots$ of $\phi(A)+I$ such that for each $n\in\mathbb N$, $C_n$ is invariant under $\Ad(v_0^*v_{t_m})$ for $m=1,\dots,n$.  The limit $C\coloneqq \overline{\bigcup_{n=1}^\infty C_n}$ provides the required separable invariant subalgebra.
 Then $(\Ad(v_0^* v_t))_{t \geq 0}$ is a norm-continuous path of automorphisms of $C$ starting at $\mathrm{id}_C$.  By \cite[Proposition~2.15]{Dadarlat-Eilers01}, there is a continuous path $(w_t')_{t \geq 0}$ in $C \subseteq \phi(A) + I$ of unitaries such that
\begin{equation}
 \| \Ad(w_t')(c) - \Ad(v_0^*v_t)(c) \| \to 0,\quad c\in C.
\end{equation}
Set $w_t\coloneqq v_0 w_t'$.  Then $w_t$ is a unitary in $\phi(A) + I$ for all $t \geq 0$ and
\begin{equation}\label{lemma:AsympEquiv.E1}
 \| w_t \phi(a) w_t^* - \psi(a) \| \to 0, \quad a\in A.
\end{equation}

Write $w_t = \phi(a_t) + x_t$, where $a_t \in A$ and $x_t \in I$, so that $q(\phi(a_t)) = q(w_t)$ is a unitary. Since $q\circ\phi$ is injective, $a_t$ (and hence $x_t$) are uniquely determined by $w_t$, and $(a_t)_{\geq 0}$ is a continuous path of unitaries in $A$.  Also, since $q\circ\phi = q\circ\psi$, by \eqref{lemma:AsympEquiv.E1} we have
\begin{equation}
 \| q(\phi(a_t)) q(\phi(a)) q(\phi(a_t))^* - q(\phi(a)) \| \to 0,\quad a\in A,
\end{equation}
so another application of injectivity of $q\circ\phi$ gives
\begin{equation}\label{lemma:AsympEquiv.E2}
\|a_taa_t^*-a\|\to 0,\quad a\in A.
\end{equation}
Now, define
\begin{equation}
  u_t \coloneqq w_t \phi(a_t)^* = (1_{I^\dag} + x_t \phi(a_t)^*)\in I^\dagger.
\end{equation}
Then $(u_t)_{t \geq 0}$ is a continuous path of unitaries in $I^\dag$, and combining \eqref{lemma:AsympEquiv.E1} and \eqref{lemma:AsympEquiv.E2}, we have, for all $a \in A$,
\begin{equation}
\begin{aligned}
\|u_t\phi(a)u_t^*-\psi(a)\|&\leq \|w_t(\phi(a_t^*)\phi(a)\phi(a_t))w_t^*-w_t\phi(a)w_t^*\|\\
&\quad +\|w_t\phi(a)w_t^*-\psi(a)\|\to 0,
\end{aligned}
\end{equation}
as required.

\ref{asympequiv.2}: Let $(v_t)_{t \geq 0}$ be a path of unitaries as in \ref{asympequiv.2}.  By nuclearity of $\Z$, we have a short exact sequence
\begin{equation}
\begin{tikzcd}    
0\ar[r]& I \otimes \mathcal Z \ar[r]& E \otimes \mathcal Z \ar[r,"q \otimes \id_{\mathcal Z}"] &D \otimes \mathcal Z \ar[r]& 0.
\end{tikzcd}
\end{equation}
We will complete the proof by applying \ref{asympequiv.1} to this extension, the $^*$-ho\-mo\-mor\-phisms $\phi \otimes 1_{\mathcal Z}$ and $\psi \otimes 1_{\mathcal Z}$, and the path $(v_t\otimes 1_{\mathcal Z})_{t\geq 0}$.  

One has a containment of relative commutants
\begin{equation}\label{eq:commutant-tensor}
\theta\colon(D\cap q(\phi(A))')\otimes \mathcal Z\hookrightarrow (D\otimes\mathcal Z)\cap (q\otimes\id_{\mathcal Z})\big((\phi\otimes 1_{\mathcal Z})(A)\big)'.
\end{equation}
Since $q(v_0)$ has trivial class in $K_1(D\cap q(\phi(A))')$, its image $q(v_0)\otimes 1_{\mathcal Z}$ is also trivial in $K_1((D\cap q(\phi(A))')\otimes\mathcal Z)$.   Jiang's $K_1$-injectivity of $\mathcal Z$-stable $C^*$-algebras (Theorem~\ref{prop:z-stable-K1-inj}) implies that $q(v_0) \otimes 1_{\mathcal Z}$ is in the path component of the identity in $(D\cap q(\phi(A))')\otimes\mathcal Z$. Applying $\theta$, it follows that $(q\otimes \id_{\mathcal Z})(v_0\otimes 1_\mathcal Z)$ is in the path component of the identity in $(D\otimes\mathcal Z)\cap (q\otimes\id_{\mathcal Z})(\phi\otimes 1_{\mathcal Z})(A)'$.  The result now follows from \ref{asympequiv.1}.
\end{proof}

\begin{remark}
Using the completely positive approximation property for $\mathcal Z$, one can check that the map $\theta$ in \eqref{eq:commutant-tensor} is surjective and hence an isomorphism, though we do not need this.  
\end{remark}

The following lemma is a restatement of \cite[Lemma~3.5]{Dadarlat-Eilers01}.

\begin{lemma}\label{lem:DE3.5}
    Let $A$ and $I$ be $C^*$-algebras with $A$ separable and $I$ $\sigma$-unital and stable, and let $\phi \colon A \rightarrow \mathcal M(I)$ be a unital $^*$-homomorphism.  Suppose $u \in \mathcal M(I)$ is a unitary such that $q_I(u) \in \mathcal Q(I) \cap q_I(\phi(A))'$, and note that $(\phi, \mathrm{Ad}(u)\circ \phi) \colon A \rightrightarrows \mathcal M(I) \rhd I$ is a Cuntz pair.  If $[\phi, \mathrm{Ad}(u) \circ \phi]_{KK(A, I)} = 0$, then there is a unital $^*$-homomorphism $\theta \colon A \rightarrow \mathcal M(I)$ such that $1_{\mathcal Q(I)} \oplus 1_{\mathcal Q(I)}$ and $q_I(u) \oplus 1_{\mathcal Q(I)}$ are homotopic in the unitary group of $M_2(\mathcal Q(I)) \cap q_{M_2(I)}((\phi\oplus \theta)(A))'$. 
\end{lemma}

\begin{proof}
    In order to correctly cite \cite[Lemma~3.5]{Dadarlat-Eilers01}, we need the Fredholm picture of $KK$-theory.  We briefly explain how to translate between this and the Cuntz--Thomsen picture used here following the discussion in \cite[Section 3.1]{Dadarlat-Eilers01}.  A Cuntz pair $(\psi_1,\psi_2)$ gives the Fredholm triple  $(\psi_1,\psi_2,1_{\mathcal M(I)})$, and a Fredholm triple $(\psi_1,\psi_2, v)$ with $v$ a unitary in $\M(I)$ corresponds to the Cuntz pair $(\Ad(v)\circ\psi_1, \psi_2)$. Therefore, the Cuntz pair $(\phi, \Ad(u)\circ\phi)$
  corresponds to the cycle $(\phi, \phi, u^*)$ and, since  $[\phi, \Ad(u)\circ \phi]_{KK(A, I)} = 0 = [\phi,
  \phi]_{KK(A, I )}$, we can apply
  \cite[Lemma~3.5]{Dadarlat-Eilers01} to the Fredholm triples $(\phi,
  \phi, u)$ and $(\phi, \phi, 1_{\mathcal M(I)})$. This gives the desired $\theta$ and homotopy.
\end{proof}

We now prove the $\Z$-stable $KK$- and $KL$-uniqueness theorems.

\begin{proof}[Proof of Theorem~\ref{thm:KK-Uniqueness}]
\ref{KK-Uniqueness}:
In this proof, we use $\oplus$ to denote the usual direct sum (using a matrix amplification), rather than the definition in \eqref{KK-directsumeq}; see Remark~\ref{rem:internal-vs-external}.
Since $\phi$ and $\psi$ are absorbing, Corollary~\ref{cor:absorption-uniqueness} provides a norm-continuous path $(u_t)_{t \geq 0} \subseteq \M(I)$ of unitaries such that
\begin{equation}
  \big( t \mapsto u_t \phi(a) u_t^* - \psi(a) \big) \in C_0([0, \infty), I).
\end{equation}
By Lemma~\ref{lemma:AsympEquiv}\ref{asympequiv.2}, it suffices to show $[q_I(u_0)]_1 = 0$ in $K_1(\mathcal Q(I) \cap q_I(\phi^\dag(A^\dag))')$.

Now, $(\phi,\Ad(u_0)\circ\phi)$ is a Cuntz pair, and it is homotopic (via $(\phi,\mathrm{Ad}(u_t)\circ\phi)$) to $(\phi,\psi)$ (this is \cite[Lemma~3.1]{Dadarlat-Eilers01}), so that $[\phi, \Ad(u_0)\circ \phi]_{KK(A,I)} = [\phi, \psi]_{KK(A,I)} = 0$.  By considering the split exact sequence
\begin{equation}
\begin{tikzcd}
0 \arrow{r} & KK(\mathbb C, I) \arrow[shift right=.5ex]{r} & KK(A^\dag, I) \arrow{r} \arrow[shift right=.5ex]{l} & KK(A, I) \arrow{r} & 0,
\end{tikzcd}
\end{equation}
it follows that $[\phi^\dag, \Ad(u_0)\circ\phi^\dag]_{KK(A^\dagger,I)} =
0$.
  This follows, as the maps $KK(A^\dag, I)\to KK(A, I)$ and $KK(A^\dag,
  I)\to KK(\mathbb C, I)$ arise from the inclusions of $A$ and
  $\mathbb C$ into $A^\dag$, and so are given by restricting Cuntz
  pairs on $A^\dag$ to $A$ and $\mathbb C$ respectively (by
  definition). Thus $[\phi^\dag, \Ad(u_0)\circ \phi^\dag ]_{KK(A^\dag, I
    )}$ has image $[\phi, \Ad(u_0)\circ \phi]_{KK(A, I)} = 0$ in $KK(A, I)$
  and image $[1_{\M(I)}, 1_{\M(I) }]_{KK(\mathbb C, I)} = 0$ in
  $KK(\mathbb C, I)$. Hence $[\phi^\dag,
  \Ad(u_0)\circ\phi^\dag]_{KK(A^\dag, I)} = 0$.
  
Lemma~\ref{lem:DE3.5} gives rise to  a norm-continuous path $(v_t)_{0 \leq t \leq 1} \subseteq M_2(\M(I))$
with $v_0 = 1_{\M(I)} \oplus 1_{\M(I)}$ and $v_1 =u_0 \oplus
1_{\M(I)}$ such that each $q_{M_2(I)}(v_t)$ is a unitary in $M_2(\mathcal Q(I))$
commuting with $q_{M_2(I)}((\phi^\dag \oplus \theta)(A))$.
  Thus $q_I(u_0)\oplus 1_{\mathcal Q(I)}$ is homotopic to
$1_{\mathcal Q(I)}\oplus 1_{\mathcal Q(I)}$ in $M_2(\mathcal Q(I))\cap
q_{M_2(I)}((\phi^\dag\oplus\theta)(A))'$.  Taking the direct sum of this homotopy
with $1_{\mathcal Q(I)}$, gives a homotopy between $q_I(u_0)\oplus
1_{\mathcal Q(I)}\oplus 1_{\mathcal Q(I)}$ and $1_{\mathcal
  Q(I)}\oplus 1_{\mathcal Q(I)}\oplus 1_{\mathcal Q(I)}$ in
$M_3(\mathcal Q(I))\cap q_{M_3(I)}((\phi^\dag\oplus\theta\oplus\phi^\dag)(A))'$.

Since $\phi$ is absorbing and $I$ is stable, $\theta|_A\oplus\phi$ and $\phi$ are unitarily equivalent modulo $I$,
and hence so too are $\theta\oplus \phi^\dag$ and $\phi^\dag$.  It follows that $q_I(u_0) \oplus 1_{\mathcal Q(I)}$ is homotopic to $1_{\mathcal Q(I)} \oplus 1_{\mathcal Q(I)}$ in the unitary group of
\begin{equation}
M_2(\mathcal Q(I)) \cap ((q_I\circ\phi^\dag) \oplus (q_I\circ\phi^\dag))(A^\dag)'= M_2(\mathcal Q(I)\cap q_I(\phi^\dag(A^\dag))').
\end{equation}
Accordingly, $[q_I(u_0)]_1 = 0$ in $K_1(\mathcal Q(I)\cap q_I(\phi^\dag(A^\dag))')$ as required.

\ref{KL-Uniqueness}: Suppose $[\phi, \psi]_{KL(A,I)} = 0$. Then there exists $\kappa \in KK(A, C(\bcN, I))$ such that $KK(A,\ev_n)(\kappa) = 0$ for $n \in \mathbb N$ and $KK(A,\ev_\infty)(\kappa) = [\phi, \psi]_{KK(A,I)}$ (see Definition~\ref{def:KL}, and recall that $\bcN = \mathbb N \cup \{\infty\}$). By \cite[Corollary~3.4]{Akemann-Pedersen-etal73},
\begin{equation}
  \M(C(\bcN,I))\cong C_\sigma(\bcN,\M(I)),
\end{equation}
where the right hand side is the $C^*$-algebra of (norm-bounded) strictly continuous functions from $\bcN$ to $\M(I)$ (see Definition~\ref{defn:KK}).
Let $\Psi \colon A \to C_\sigma(\bcN, \M(I))$ be given by $\Psi(a)(n)\coloneqq \psi(a)$ for all $n\in\bcN$. It is absorbing by \cite[Proposition~3.2]{Dadarlat05} (noting that the proof only requires that $I$ is $\sigma$-unital instead of separable.)

The $KK$-existence theorem (Theorem~\ref{thm:KK-Existence}), gives an absorbing $^*$-ho\-mo\-mor\-phism $\Phi \colon A \to \M(C(\bcN, I))$ with $[\Phi, \Psi]_{KK(A,C(\bcN,I))} = \kappa$.  Viewing $\Phi$ as taking values in $C_\sigma(\bcN,\M(I))$, let $\phi_n\coloneqq \ev_n\circ \Phi \colon A \to \M(I)$.  Then $(\phi_n, \psi)$ is a Cuntz pair for all $n\in\bcN$.  Since $\psi$ is absorbing, Corollary~\ref{cor:multipler-vs-corona} (applied twice) implies that $\phi_n$ is also absorbing for all $n\in\bcN$.

Fix a finite subset $\mathcal F$ of $A$ and $\epsilon>0$. As $A$ is separable, it suffices to find a unitary $u\in (I\otimes\Z)^\dag$ such that $\| u(\phi(a)\otimes 1_\Z)u^*- \psi(a)\otimes 1_\Z\| < \epsilon$ for all $a\in \mathcal F$. A priori, for each $a\in A$, $\phi_n(a)\to\phi_\infty(a)$ strictly in $\M(I)$. However, since $\Phi(a)-\Psi(a)\in C(\bcN,I)$, we have
\begin{equation}
\phi_n(a)-\phi_\infty(a)=(\phi_n(a)-\psi(a))+(\psi(a)-\phi_\infty(a))\in I,\quad a\in A,
\end{equation} so the convergence $\phi_n(a)\to\phi_\infty(a)$ actually happens in norm.    Then we can find $n\in\mathbb N$ such that $\|\phi_n(a)-\phi_\infty(a)\|<\epsilon/3$ for $a\in \mathcal F$.  Since $[\phi_n, \psi]_{KK(A,I)} = KK(A,\ev_n)(\kappa) = 0$ (by Proposition~\ref{prop:KK-facts}\ref{prop:KK-facts.5}), we can use part \ref{KK-Uniqueness} to find a unitary $v\in (I\otimes\mathcal Z)^\dag$ with 
\begin{equation}
  \| v(\phi_n(a)\otimes1_\Z)v^*-\psi(a)\otimes 1_{\Z} \| < \epsilon /3, \quad a\in \mathcal F.
\end{equation}
Also 
\begin{equation} 
  [\phi_\infty, \psi]_{KK(A,I)} = KK(A,\ev_\infty)(\kappa) = [\phi, \psi]_{KK(A,I)}.
\end{equation}
Accordingly, $[\phi_\infty, \phi]_{KK(A,I)} = 0$ (by Proposition~\ref{prop:KK-facts}\ref{prop:KK-facts.4}) and so another application of part \ref{KK-Uniqueness} gives a unitary $w\in (I\otimes\mathcal Z)^\dag$ with 
\begin{equation} 
\| w(\phi(a)\otimes 1_\Z)w^*-\phi_\infty(a)\otimes 1_\Z\| < \epsilon /3, \quad a\in \mathcal F.
\end{equation}
 Set $u\coloneqq vw \in (I\otimes \Z)^\dag$.  Then
 \begin{equation}
\| u(\phi(a)\otimes 1_\Z)u^*-\psi(a)\otimes 1_\Z\| < \epsilon,  \quad a\in \mathcal F.  \qedhere
 \end{equation}
\end{proof}

We do not know if either the tensor factor $1_\mathcal Z$ in 
Theorem~\ref{intro-KK-unique}\ref{kk-unique2} or the direct summand $\eta$ in condition \ref{kk-unique3} following Theorem~\ref{intro-KK-unique} are necessary. This leads to the \emph{$KK$-uniqueness problem}.

\begin{question}[The $KK$-uniqueness problem] \label{q:kk-uniqueness}
Let $A$ and $I$ be $C^*$-algebras with $A$ separable and $I$ $\sigma$-unital and stable, and let $(\phi, \psi)\colon A\rightrightarrows \mathcal M(I) \rhd I$ be a Cuntz pair with $\phi$ and $\psi$ absorbing and $[\phi,\psi]_{KK(A, I)} = 0$.  Must there be a norm-continuous path $(u_t)_{t \geq 0}$ of unitaries in $I^\dag$ such that
\begin{equation}
    \|u_t \phi(a)u_t^* - \psi(a)\| \rightarrow 0
\end{equation}
for all $a\in A$?
\end{question}

The first result in this direction is Dadarlat and Eilers' theorem (\cite[Theorem~3.12]{Dadarlat-Eilers01}), giving a positive solution when $I = \mathcal K$.  We point out the modifications needed in the treatment above to obtain this result, but we emphasize that the proof is essentially the same as that given in \cite{Dadarlat-Eilers01}.   Paschke proved in \cite{Paschke81} that the $C^*$-algebra $\mathcal Q(\mathcal K)\cap q_{\mathcal K}(\phi(A))'$ is $K_1$-injective.  Apply this in the final line of the proof of Theorem~\ref{thm:KK-Uniqueness}\ref{KK-Uniqueness} to conclude that $u_0$ is in the path component of the identify in $U(\mathcal Q(\mathcal K)\cap q_{\mathcal K}(\phi(A))')$, and then apply part \ref{asympequiv.1} of Lemma~\ref{lemma:AsympEquiv} in place of \ref{asympequiv.2}.

As noted above, Dadarlat and Eiler's strategy gives a positive answer to the $KK$-uniqueness problem for $(A,I)$, whenever $\mathcal Q(I) \cap q_I(\phi(A))'$ is $K_1$-injective for some (equivalently, any) absorbing $^*$-homomorphism $\phi \colon A \rightarrow \M(I)$ (cf. \cite[Theorem~2.11]{Lee11}).  Using this strategy, a partial result was obtained in \cite{Loreaux-Ng20}. We would like to reiterate the following question from \cite{BRR08}. As the relative commutants in question are properly infinite (Proposition \ref{absorb:propinfinite}), a positive answer would imply a positive answer to the $KK$-uniqueness problem.

\begin{question}[{\cite[Question~2.9]{BRR08}}] \label{q:propinfK1inj}
    Is every properly infinite $C^*$-algebra $K_1$-injective?
\end{question}

We end this subsection by recording the easy implication of Theorem~\ref{intro-KK-unique} --- the hard direction follows from the $KK$-uniqueness theorem.

\begin{proof}[Proof of Theorem~\ref{intro-KK-unique}]
The implication \ref{kk-unique1}$\Rightarrow$\ref{kk-unique2} is Theorem~\ref{thm:KK-Uniqueness}\ref{KK-Uniqueness}.
For \ref{kk-unique2}$\Rightarrow$\ref{kk-unique1}, note that \ref{kk-unique2} implies that $[\phi\otimes 1_{\Z},\psi\otimes 1_{\Z}]_{KK(A,I\otimes\Z)}=0$.
The inclusion $\iota:I \to I \otimes \Z$ is a $KK$-equivalence by Proposition~\ref{prop:ZKKequiv} and $[\phi\otimes 1_\Z,\psi\otimes 1_\Z]_{KK(A,I\otimes\Z)}=KK(A,\iota)([\phi,\psi]_{KK(A,I)})$, by Proposition~\ref{prop:KK-facts}\ref{prop:KK-facts.5}, so it follows that $[\phi,\psi]_{KK(A,I)}=0$.
\end{proof}

\begin{remark}
Since the submission of this paper, Szab\'o has given a positive answer to Question \ref{q:kk-uniqueness}  (\cite{Szabo26}) by establishing $K_1$-injectivity of suitable relative commutants. Question \ref{q:propinfK1inj} remains open.  See Addendum~\ref{addendum}.
\end{remark}

\section{The trace-kernel extension}
\label{sec:trace-kern-ext}

We recall the trace-kernel extension in Section~\ref{SSTraceKernel} and describe the background behind Theorem~\ref{intro:classtraces} in Section~\ref{SSTKQuotient}. In Section~\ref{SSTKKThy}, we use this to compute the $K$-theory of the trace-kernel quotient (under the codomain hypotheses of Theorem~\ref{Main}) and prove Theorem~\ref{intro:calcKJB}. We end with Section~\ref{SSTKSepStab}, in which we show that the
trace-kernel ideal of a simple exact $\Z$-stable $C^*$-algebra is
separably stable.

\subsection{The trace-kernel extension}\label{SSTraceKernel}
Matui and Sato's breakthrough work (\cite{Matui-Sato12,Matui-Sato14}) on
the Toms--Winter conjecture involved a new way of using von Neumann algebras to solve $C^*$-algebraic problems, which was later pushed further with extension and $KK$-theoretic arguments in  \cite{Schafhauser17,Schafhauser18}.
Matui and Sato introduced what we now call the trace-kernel
ideal $J_B$ in the sequence algebra $B_\infty$ as the intersection of all the ideals associated to limit traces on $B_\infty$ (recall Definition~\ref{DefLimitTraces}). It appears implicitly in \cite{Matui-Sato12} and explicitly (in an ultrapower formulation) in \cite{Kirchberg-Rordam14} and \cite{Matui-Sato14}.

\begin{definition} Let $B$ be a unital separable $C^*$-algebra satisfying $T(B)\neq\emptyset$. 
  The \emph{trace-kernel ideal} 
  of $B$ is
  \begin{equation}
    J_B \coloneqq \{ b \in B_\infty : \tau(b^*b) = 0\text{ for all }\tau \in
    T_\infty(B) \},
  \end{equation}
and the \emph{trace-kernel extension} is the short exact sequence
  \begin{equation}
  \label{eq:trace-kernel-ext}
    \begin{tikzcd}
	0 \arrow{r} & J_B \arrow{r}{j_B} & B_\infty \arrow{r}{q_B} & B^\infty \arrow{r} & 0,
\end{tikzcd}      
\end{equation}
  where $j_B$ is the canonical inclusion, $B^\infty \coloneqq B_\infty
  / J_B$, called the
  \emph{trace-kernel quotient}, and $q_B$ is the quotient map.
\end{definition}

Note that given $x= (x_n)_{n=1}^\infty \in B_\infty$, we have $x\in
J_B$ if and only if
\begin{equation}\label{eq.DefJB}
\lim_{n\to \infty}\sup_{\tau\in T(B)}\tau(x_n^*x_n)=0.
\end{equation}
As with $B_\infty$, we denote elements of $B^\infty$
by representatives from $\ell^\infty(B)$.

Note too that the limit traces descend from $B_\infty$
to give traces on $B^\infty$; we will use the notation $T_\infty(B)$
in both cases.  For a positive element $x\in B^\infty$, if $\tau(x)=0$ for all $\tau\in T_\infty(B)$, then $x=0$.

\begin{remark}
A unital $^*$-homomorphism $\phi\colon B\to C$ induces a map $\phi_\infty\colon B_\infty\to C_\infty$. As $\phi$ is unital, any trace $\tau\in T_\infty(C)\subseteq T(C_\infty)$ induces a trace $\phi_\infty^*(\tau)\in T_\infty(B)$. Therefore,  $\phi_\infty(J_B)\subseteq J_C$, and hence $\phi$ also induces a map $\phi^\infty\colon B^\infty\to C^\infty$.  In this way,  the assignment of the trace-kernel extension $\mathsf{e}_B$ to $B$ is functorial for unital $C^*$-algebras with unital $^*$-homomorphisms. 
\end{remark}

As recorded in Proposition~\ref{NoSillyTraces}, convex combinations of limit traces are dense in $T(B_\infty)$ when $B$ is $\Z$-stable
and exact. When this holds, it descends to the trace-kernel quotient
$B^\infty$. While we prove this using Proposition~\ref{NoSillyTraces} here, it is possible for the convex combinations
  of the limit traces to be dense in $T(B^\infty)$
  even when they are not dense in $T(B_\infty)$; see \cite[Proposition 2.5]{CETW21}.

\begin{proposition}\label{NoSillyTracesB^infty}
  Let $B$ be a unital simple separable exact $\Z$-stable $C^*$-algebra
  with $T(B)\neq\emptyset$.
\begin{enumerate}[(i)]
\item The map $T(q_B)\colon T(B^\infty)\to T(B_\infty)$ induced by the quotient map $q_B\colon B_\infty\to B^\infty$ is an affine homeomorphism. In particular, the convex hull of the limit traces $T_\infty(B)$ is weak$^*$-dense in $T(B^\infty)$.  \label{NoSillyTracesB^infty.1}
\item $B^\infty$ has strict comparison of positive elements by bounded traces.\label{NoSillyTracesB^infty.2}
\item Let $A$ be a $C^*$-algebra. A $^*$-homomorphism $\theta:A\to B^\infty$ is full if and only if $\tau\circ\theta$ is faithful for all traces $\tau\in T(B^\infty)$.\label{NoSillyTracesB^infty.3}
\end{enumerate}
\end{proposition}
\begin{proof}
  For \ref{NoSillyTracesB^infty.1}, the affine map $T(q_B)$ is a
  continuous injection (because $q_B$ is surjective).  It is
  surjective as well  ---  and therefore a homeomorphism  --- as its image is
  compact, convex, and contains the limit traces on $B_\infty$, which have dense convex hull
  in $T(B_\infty)$ by Proposition~\ref{NoSillyTraces}.

  Part \ref{NoSillyTracesB^infty.2} follows since if $a,b\in M_k(B^\infty)$ are positive elements with $d_\tau(a)<d_\tau(b)$ for all $\tau\in T(B^\infty)$, then taking positive lifts $\tilde{a}$ and $\tilde{b}$ of $a$ and $b$ in $M_k(B_\infty)$, the previous part gives $d_\tau(\tilde{a})<d_\tau(\tilde{b})$ for all $\tau\in T(B_\infty)$. As $B_\infty$ has strict comparison of positive elements by bounded traces (Proposition~\ref{prop:B_inftyProperties_2}\ref{prop:B_inftyProperties.3}), $\tilde{a}\precsim \tilde{b}$, and hence $a\precsim b$.
 
Part \ref{NoSillyTracesB^infty.3} is an immediate consequence of \ref{NoSillyTracesB^infty.2} and Lemma~\ref{p:fullcomp}.
\end{proof}

We will need to work with matrix amplifications of the trace-kernel extension in Sections~\ref{SSTKKThy} and \ref{sect:unital-lifts}.  The following observation is implicitly used (in the setting of
ultrapowers) in \cite[Section~4]{Schafhauser18}. It boils down to the fact that $\tau\mapsto \tau_{M_k} \otimes \tau$ gives an affine homeomorphism $T(B)\to T(M_k \otimes B)$.

\begin{remark}\label{matrix}
  For a  $C^*$-algebra $D$, write $\iota^{(k)}_D\colon D\to M_k(D)$ for the top-left corner embedding of $D$ into $M_k(D)$. If $B$ is a unital $C^*$-algebra and $T(B)\neq\emptyset$, then the following diagram commutes, with natural maps forming the isomorphism of extensions between the bottom two rows:
\begin{equation}\label{matrix.diag}
  \begin{tikzcd}
    0 \ar[r] & J_B\ar[r, "{j_B}"] \ar[d, "{\iota^{(k)}_{J_B}}"] & B_\infty \ar[d, "{\iota^{(k)}_{B_{\infty}}}"] \ar[r, "{q_B}"] & B^\infty \ar[d, "{\iota^{(k)}_{B^\infty}}"] \ar[r] & 0\phantom{.}\\
    0 \ar[r] & M_k(J_B) \ar[r] \ar[d, "{\cong}"] & M_k(B_\infty) \ar[d, "\cong"] \ar[r] & M_k(B^\infty) \ar[d, "\cong"] \ar[r] & 0\phantom{.}\\
    0 \ar[r] & J_{M_k(B)} \ar[r, "{j_{M_k(B)}}"] & M_k(B)_\infty \ar[r, "{q_{M_k(B)}}"] & M_k(B)^\infty \ar[r] & 0.
  \end{tikzcd}
  \end{equation}
  Note too that if $B$ satisfies our standing codomain hypotheses (unital, simple, separable, nuclear and $\Z$-stable), then so too do all matrix amplifications $M_k(B)$.
\end{remark}

\subsection{Classification into the trace-kernel quotient}\label{SSTKQuotient}
As mentioned in the introduction, the intuition behind the classification of maps into the trace-kernel quotient is most readily seen with the ultrapower version of the trace-kernel extension $0\to J_{B,\omega}\to B_\omega\to B^\omega\to 0$, for some free ultrafilter $\omega$ on $\mathbb N$, and when $B$ has unique trace, so that $B^\omega$ is the tracial von Neumann ultrapower $(\pi_\tau(B)'',\tau)^\omega$, where $\pi_\tau$ is the GNS-representation associated to $\tau$ (see \cite{Kirchberg-Rordam14}, for example). Here, Connes' theorem implies a classification of maps from separable nuclear $C^*$-algebras
into type II$_1$ von Neumann algebras  (see \cite[Corollary 10 and Theorem 5]{Ding-Hadwin05}, or more generally, \cite[Proposition
  2.1]{Ciuperca-Giordano-etal13}).

In general, $B^\omega$ will not be a von Neumann algebra, but under the right conditions, classification of maps into $B^\omega$ (or $B^\infty$) is possible by a tracial gluing partition of unity technique.
This technique was developed in stages. When $T(B)$ is a Bauer simplex (i.e., $\partial_eT(B)$ is compact), this gluing can be achieved through Ozawa's concept of $W^*$-bundles and his trivialization theorem (\cite[Theorem 15]{Ozawa13}; see \cite[Proposition 3.23]{Bosa-Brown-etal15} for the relevant partition of unity argument).  Outside the Bauer simplex setting, the situation is more complicated, and a general technique using nuclearity and $\Z$-stability of $B$ was developed in \cite{Castillejos-Evington-etal21}, for use in the Toms--Winter conjecture.  See the second half of the introduction to \cite{Castillejos-Evington-etal21} for more details. The use of \cite{Castillejos-Evington-etal21} for the results in this and the next subsection is the main reason for requiring our codomains to be nuclear (we also need that quasitraces are traces on the codomain).

The exact form of the classification of maps by traces that we need is in \cite{CETW21} (a short sequel to \cite{Castillejos-Evington-etal21}).
Note that \cite{CETW21} contains a small error that is corrected in \cite{CETW-corrigendum}.

\begin{theorem}[{Classification of maps into $B^\infty$ by traces; 
    \cite[Theorem~A]{CETW21}}]
  \label{thm:B^inftyClassification}
  Let $A$ be a unital separable nuclear $C^*$-algebra and let $B$ be a unital separable
  nuclear $\mathcal Z$-stable $C^*$-algebra with $T(B)\neq
  \emptyset$.
  For every positive unital linear map $\gamma \colon \Aff T(A) \to \Aff T(B^\infty)$,
  there is a unital $^*$-ho\-mo\-mor\-phism $\theta \colon A \to B^\infty$ satisfying
  $\Aff T(\theta) = \gamma$.  Moreover, $\theta$ is unique up to unitary equivalence.
\end{theorem}

\begin{proof}
  By Kadison duality (see the second paragraph of
  Section~\ref{sec:Elliott-invariant}), there is a natural bijection
  between positive unital linear functions $\Aff T(A)\to\Aff T(B^\infty)$ and
  continuous affine functions $T(B^\infty) \rightarrow T(A)$.  The
  existence and uniqueness of a not-necessarily-unital map $\theta$
  follow from \cite[Theorem~A]{CETW21}.  Further, if $\theta \colon
  A \rightarrow B^\infty$ is a $^*$-homomorphism with $\tau
  \circ \theta \in T(A)$ for all $\tau \in T(B^\infty)$, (i.e., $\tau\circ\theta$ is a tracial state not just a tracial functional) then
  $1_{B^\infty} - \theta(1_A)$ is a positive element of $B^\infty$
  which vanishes on all traces on $B^\infty$. 
  Since a general positive element of $B^\infty$ that vanishes on all traces is automatically zero, this implies $\theta$ is unital.
\end{proof}

\subsection{$K$-theory of the trace-kernel ideal and quotient}\label{SSTKKThy}

We now turn to the computation of $K_1(J_B)$ when $B$ is an allowed codomain in Theorem~\ref{Main2}. Firstly, the gluing procedures used to obtain the classification of maps into $B^\infty$ also show that every unitary in $B^\infty$ is an exponential.

\begin{proposition}[{\cite[Proposition 2.1]{CETW21}}]\label{prop:B^inftyExponentials}
  Let $B$ be a unital separable nuclear $\mathcal Z$-stable
  $C^*$-algebra with $T(B)\neq \emptyset$.  If $u \in B^\infty$ is a
  unitary, then $u=e^{ih}$ for some self-adjoint $h \in B^\infty$ of
  norm at most $\pi$. In particular, every unitary in $B^\infty$ is
  the image of a unitary in $B_\infty$ under $q_B$.
\end{proposition}

Combining Proposition~\ref{prop:B^inftyExponentials} with the classification of maps into $B^\infty$ by traces gives the following $K$-theory calculation, analogous to the fact that the $K$-theory of a II$_1$ factor is $(\mathbb R,0)$.

\begin{proposition}
  \label{prop:B^inftyKtheory}
  Let $B$ be a unital separable nuclear $\mathcal Z$-stable $C^*$-algebra
  with $T(B)\neq \emptyset$.  Then $(K_0(B^\infty),K_1(B^\infty)) \cong
  (\Aff T(B^\infty),0)$, with the pairing map $\rho_{B^\infty}$ inducing the isomorphism
  in the first entry, and $K_*(B^\infty;\Zn{n})=0$ for all $n\geq 2$.
\end{proposition}

\begin{proof}
  The computation
  of $K_0(B^\infty)$ is a consequence of
  Theorem~\ref{thm:B^inftyClassification} with $A\coloneqq \bC^2$,
  as follows.  Given $f \in \Aff T(B^\infty)_+$, fix $n\in\bN$
  such that $f\leq n$, which exists by the compactness of $T(B^\infty)$. 
  Let $\tilde f \in \Aff T(M_n \otimes B^\infty)$ be given by $\tilde f (\tau_{M_n} \otimes \tau) \coloneqq f(\tau)$ for $\tau\in T(B^\infty)$.
  Define $\gamma \colon \mathbb R\oplus \mathbb R \to \Aff T(M_n \otimes B^\infty)$ by $\gamma(s,t) \coloneqq \tfrac{s}{n} \tilde f + t(1-\tfrac{1}{n} \tilde f)$.
  Then existence
  in Theorem~\ref{thm:B^inftyClassification} (working with the allowed codomain $M_n\otimes B$ in
  place of $B$, as noted at the end of Remark~\ref{matrix}) gives us a unital $^*$-homomorphism $\phi \colon \mathbb C\oplus \mathbb C \to M_n\otimes B^\infty$ realizing $\gamma$, and thus 
  a projection $p \coloneqq \phi(1,0) \in (M_n\otimes B)^\infty\cong M_n(B^\infty)$ (see Remark \ref{matrix}) such that $(\tau_{M_n}
  \otimes \tau)(p)=\tfrac{1}{n}f(\tau)$ for all $\tau \in
  T(B^\infty)$.
  This means that $\rho_{B^\infty}([p]_0) = f$.  Moreover, uniqueness in Theorem~\ref{thm:B^inftyClassification} tells us that any projection $q \in
  M_n\otimes B^\infty$ with the same property is unitarily equivalent
  to $p$, and thus $[p]_0=[q]_0$ in $K_0(B^\infty)$. Taking linear combinations shows that the pairing map induces an
  isomorphism $K_0(B^\infty) \cong \Aff T(B^\infty)$.

  Proposition~\ref{prop:B^inftyExponentials}, applied to $M_n\otimes B^\infty \cong (M_n
  \otimes B)^\infty$ (via Remark~\ref{matrix}) implies $K_1(B^\infty)=0$.
  For the last statement, the map $K_0(B^\infty) \rightarrow K_0(B^\infty)$ of multiplication by $n$ is an isomorphism for all $n \geq 2$ (as $K_0(B^\infty)$ is a real vector space).  Using this, and $K_1(B^\infty)=0$, in the exact sequence \eqref{eq:bockstein-2} implies $K_*(B^\infty ; \mathbb Z/n) = 0$ for all $n\geq 2$.
\end{proof}

Now we turn both to the proof of Theorem~\ref{intro:calcKJB} and the setup for Theorem~\ref{intro:calcKL}.
 Note that when $B$ is unital, simple, separable, exact, and $\Z$-stable, we have that $\Aff T(q_B)\colon\Aff T(B_\infty)\to \Aff T(B^\infty)$ is an isomorphism by Proposition~\ref{NoSillyTracesB^infty}\ref{NoSillyTracesB^infty.1}. In the following diagram, the first row is an extract of the six-term exact sequence for the trace-kernel extension, while the second row is exact as it is Thomsen's extension for $B_\infty$ (Proposition~\ref{prop:B_inftyK0ClosedInAff}). Commutativity of the left-hand square is naturality of the pairing map.

\begin{equation}\label{Thm:K1JDiag}
    \begin{tikzcd}[row sep = 7ex]
    K_0(B_\infty)\ar[r,"K_0(q_B)"]\ar[equal,d]&K_0(B^\infty)\ar[r,"\partial"]\arrow{d}[description]{(\Aff T(q_B))^{-1}\circ \rho_{B^\infty}}&K_1(J_B)\ar[r,"K_1(j_B)"]\ar[d,dashed,"\omega_B"]&K_1(B_\infty)\ar[r,"K_1(q_B)"]\ar[d,equal]&K_1(B^\infty)\ar[d]\\
        K_0(B_\infty)\ar[r,"\rho_{B_\infty}"]&
        \Aff T(B_\infty)\ar[r,"\Th_{B_\infty}"]&\Ka(B_\infty)
        \ar[r,"\minusa_{B_\infty}"]&K_1(B_\infty)\ar[r]&0
    \end{tikzcd}
    \end{equation}
A diagram chase gives a natural map 
\begin{equation}\label{Thm:K1J.3}
\omega'_B\colon\im\partial\to \im\Th_{B_\infty}=\ker\minusa_{B_\infty}\colon \partial([p]_0)\mapsto\ka{e^{2\pi ia}},
\end{equation}
where $a\in M_k(B_\infty)$ is a self-adjoint lifting the projection $p\in M_k(B^\infty)$.  For our calculation of $KL(A,J_B)$ in Theorem~\ref{thm:calcKL} (the precise version of Theorem~\ref{intro:calcKL}), we need to know that $\omega_B'$ is an isomorphism and describe it explicitly in terms of elements of $\ker K_1(j_B)=\im\partial$. The explicit description we give extends to a natural map $\omega_B\colon K_1(J_B)\to\Ka(B_\infty)$. Moreover, both this and $\omega_B'$ are isomorphisms when $B^\infty$ has the $K$-theory provided by Proposition~\ref{prop:B^inftyKtheory}.

\begin{theorem}[Calculating $K_1(J_B)$]\label{Thm:K1J}
Let $B$ be a unital separable exact $\Z$-stable $C^*$-algebra with $T(B)\neq\emptyset$. 
\begin{enumerate}
\item There is a natural map $\omega_B\colon K_1(J_B)\to \Ka(B_\infty)$ such that (\ref{Thm:K1JDiag}) commutes, given explicitly by
\begin{equation}\label{Thm:K1J:Formula}\omega_B([u]_1)\coloneqq \ka{u}-\ka{s(u)}, \quad u\in U_\infty(J_B^\dagger),
\end{equation} where $s\colon J_B^\dagger\to\mathbb C1_{B_\infty} \subseteq B_\infty$ is the canonical character.  This map extends the natural map $\omega_B'$ from \eqref{Thm:K1J.3}.\label{Thm:K1J.1}

\item When $B$ is also simple and nuclear, $\omega_B$ and $\omega_B'$ are isomorphisms.\label{Thm:K1J.2}
\end{enumerate}
\end{theorem}

\begin{proof}
\ref{Thm:K1J.1}: The map $\tilde{\omega}_B\colon U_\infty(J_B^\dagger)\to \Ka(B_\infty)$ given by $\tilde{\omega}_B(u)\coloneqq \ka{u}-\ka{s(u)}$ for $u\in U_\infty(J_B^\dagger)$ is a group homomorphism.  To show that this induces a well-defined map $\omega_B$ on $K_1(J_B)$, it suffices to check that $\tilde{\omega}_B(e^{2\pi i h})=0$ in $\Ka(B_\infty)$ for any self-adjoint $h\in M_n(J_B^\dagger)$. To this end, write $h=h_0+s(h)$, where $h_0=h_0^*\in M_n(J_B)$, so that $\tilde{\omega}_B(e^{2\pi i h})=\ka{e^{2\pi i h_0}}\in\ker\minusa_{B_\infty}$. Using exactness and $\Z$-stability of $B$, all traces on $B_\infty$ vanish on $J_B$ (Proposition~\ref{NoSillyTraces}), and so $\hat{h}_0=0$ in $\Aff T(B_\infty)$.  Applying the Thomsen map from \eqref{DefThomsen}, $\tilde{\omega}_B(e^{2\pi i h})=\Th_{B_\infty}(\hat{h}_0)=0$, as required.  The explicit formula (\ref{Thm:K1J:Formula}) confirms that $\omega_B$ is natural in $B$.

By construction, the induced map $\omega_B$ makes the third square of (\ref{Thm:K1JDiag}) commute. For the second square, consider a projection $p\in M_n(B^\infty)$, which we lift to a positive contraction $a\in M_n(B_\infty)$. This satisfies $\Aff T(q_B)(\hat{a})=\hat{p}$, $e^{2\pi ia}\in U_n(J_B^\dagger)$ and $\partial([p]_0)=[e^{2\pi i a}]_1$. Noting that $s(e^{2\pi i a})=1_{M_n(B_\infty)}$, we have $(\omega_B\circ\partial)([p]_0)=\ka{e^{2\pi ia}}=\Th_{B_\infty}(\hat{a})$ (from the definition of $\omega_B$ and that of the Thomsen map in (\ref{DefThomsen})).  As $(\Aff T(q_B)^{-1}\circ\rho_{B_\infty})([p]_0)=\Aff T(q_B)^{-1}(\hat{p})=\hat{a}$, the commutativity follows.  Note that commutativity of the second square shows that $\omega_B$ extends $\omega_B'$ from \eqref{Thm:K1J.3}.

\ref{Thm:K1J.2}: Proposition~\ref{NoSillyTracesB^infty}\ref{NoSillyTracesB^infty.1} ensures that $\Aff T(q_B)$ is an isomorphism and, using the nuclearity hypothesis, the $K$-theory computation of Proposition~\ref{prop:B^inftyKtheory} shows that $\rho_{B^\infty}$ is an isomorphism and $K_1(B^\infty)=0$.  The second row of \eqref{Thm:K1JDiag} is exact by Proposition~\ref{prop:B_inftyK0ClosedInAff}, and as the four outermost vertical arrows of (\ref{Thm:K1JDiag}) are isomorphisms, the result follows from the five lemma.
\end{proof}

\subsection{Separable stability of the trace-kernel ideal $J_B$}\label{SSTKSepStab}
In order to work with $KK(A,J_B)$ in
Sections~\ref{sec:class-lifts-trace} and~\ref{sec:main-results}, it
would be most convenient if the trace-kernel ideal $J_B$ was stable.
However, provided $T(B)\neq\emptyset$ and $J_B \neq 0$, this is never
the case as a consequence of Ghasemi's result (\cite{Ghasemi15}) 
since $J_B$ is an $SAW^*$-algebra (which can be proven by an $\epsilon$-test argument similar to \cite[Proposition 4.6]{Kirchberg-Rordam14}). This is similar to $B_\infty$ not being $\Z$-stable (see Section~\ref{SSZStableSeq}), and the solution is the same: to separabilize (pass to separable
subalgebras with the desired property).

\begin{definition}[{cf.\ \cite[Definition 1.4]{Schafhauser18}}]
A $C^*$-algebra $D$ is \emph{separably stable} if for every separable $C^*$-subalgebra $D_0\subseteq D$, there exists a stable separable $C^*$-subalgebra $E\subseteq D$ with $D_0\subseteq E$.
\end{definition}

Our objective is to show in Lemma~\ref{lem:J_BsepStable} that $J_B$ is
separably stable whenever $B$ satisfies the codomain assumptions in
the classification of embeddings. A version of this lemma was obtained
in \cite[Propositions~3.2 and~3.3]{Schafhauser18} under much more restrictive hypotheses, which ensure
that $B_\omega$ is real rank zero.
In our setting, we must use Cuntz semigroup techniques rather than
projections.  To do this, we note that Hjelmborg and R\o{}rdam's
characterization of stability for $\sigma$-unital $C^*$-algebras via
positive elements extends to characterize separable stability.  We say
that a $C^*$-algebra $D$ satisfies the \emph{Hjelmborg--R\o{}rdam
  criterion} if, for every $a\in D_+$ and $\epsilon>0$, there exists
$x\in D$ satisfying $\|a-xx^*\|<\epsilon$ and
$\|(x^*x)(xx^*)\|<\epsilon$ (this is equivalent to
  condition \cite[Proposition 2.2(b)]{Hjelmborg-Rordam98} since the
  set $F(D)$ of \cite{Hjelmborg-Rordam98} is dense in $D_+$ as noted
  on \cite[Page 154]{Hjelmborg-Rordam98}).
A $\sigma$-unital $C^*$-algebra $D$ is stable if and only if it satisfies the Hjelmborg--R\o{}rdam criterion by  \cite[Theorem 2.1 and Proposition~2.2]{Hjelmborg-Rordam98}.

\begin{proposition}
  \label{prop:HR-stability}
  Let $D$ be a $C^*$-algebra.  Then $D$ is separably stable if and
  only if it satisfies the Hjelmborg--R\o{}rdam criterion.
  \end{proposition}

\begin{proof}
It is immediate that separable stability implies the  Hjelmborg--R\o{}rdam criterion.  To see the converse, it suffices to show that if $D$ is a $C^*$-algebra
satisfying the Hjelmborg--R\o{}rdam criterion and $D_0$ is a
separable subalgebra of $D$, then there exists a separable subalgebra
$E\subseteq D$ satisfying the Hjelmborg--R\o{}rdam criterion with
$D_0\subseteq E$.  Fix a dense sequence $(d_l)_{l=1}^\infty$ in
$(D_0)_+$.  By the Hjelmborg--R\o{}rdam criterion, for each
$k,l\in\mathbb N$, there exists $x_{k,l}\in D$ such that $\|d_l -
x_{k,l}^*x_{k,l}\| < 1/k$ and
$\|(x_{k,l}^*x_{k,l})(x_{k,l}x_{k,l}^*)\| < 1/k$.  Let $D_1\coloneqq
C^*(D_0,\{x_{k,l}:k,l\in\mathbb N\})$.  This ensures that for any
$a\in (D_0)_+$, and $\epsilon>0$, there exists $x\in D_1$ with
$\|a-x^*x\|<\epsilon$ and $\|(x^*x)(xx^*)\|<\epsilon$.  Continuing in
this way, we obtain a nested sequence $D_0 \subseteq D_1 \subseteq
\cdots $ of separable $C^*$-subalgebras of $D$ so that the
Hjelmborg--R\o{}rdam criterion for positive elements $a$ in $D_n$ can
be witnessed by elements $x\in D_{n+1}$.  Then $E\coloneqq
\overline{\bigcup_{n=1}^\infty D_n}$ will satisfy the
Hjelmborg--R\o{}rdam criterion.
\end{proof}

With the Hjelmborg--R\o{}rdam criterion in place, we now deduce
separable stability of the relevant trace-kernel
ideals.
While we only need the following result for nuclear $B$, we
  state it for $B$ exact --- the exactness only being used to apply
  Haagerup's result that quasitraces are traces (\cite{Haagerup14}).

\begin{lemma}
  \label{lem:J_BsepStable}
  Let $B$ be a unital simple separable exact $\mathcal Z$-stable
  $C^*$-algebra with $T(B)\neq \emptyset$.  Then $J_B$ is separably
  stable.
\end{lemma}

\begin{proof}
  We check that $J_B$ satisfies the Hjelmborg--R\o{}rdam criterion of
  Proposition~\ref{prop:HR-stability}.  Let $a \in (J_B)_+$ and
  $\epsilon>0$.  By perturbing $a$ slightly and using functional
  calculus, we may assume that there is a contraction $e \in (J_B)_+$
  such that $ea=a$.  Set $b\coloneqq 1_{B_\infty}-e \in (B_\infty)_+$,
  and note that $ab=0$.  Since $e \in J_B$, we have $\tau(b)=1$ for all $\tau \in T_\infty(B)$, and therefore, for all
  $\tau \in T(B_\infty)$ by density of the convex hull of the limit
  traces in $T(B_\infty)$ (Proposition~\ref{NoSillyTraces}). Since $b \leq 1$, $d_\tau(b)\geq \tau(b)= 1$ for all $\tau \in T(B_\infty)$. On the
  other hand (again using Proposition~\ref{NoSillyTraces}),
  $d_\tau(a)=\lim_{r\to\infty} \tau(a^{1/r}) = 0$ for all $\tau \in
  T(B_\infty)$, since $a^{1/r} \in J_B$ for all $r$.  As $B_\infty$
  has strict comparison of positive elements by bounded traces
  (Proposition~\ref{prop:B_inftyProperties_2}\ref{prop:B_inftyProperties.3}), it
  follows that $a \precsim b$.  By
  \cite[Proposition~2.7(iii)]{Kirchberg-Rordam00} there exists $x\in
  B_\infty$ such that $x^*x = (a-\epsilon)_+$ and $xx^* \in
  \overline{b B_\infty b}$.
  Because $x^*x=(a-\epsilon)_+ \in J_B$, we get $x \in J_B$ and
  $\|x^*x-a\|<\epsilon$.  Also, since $xx^* \in \overline{bB_\infty b}$
  and $ab=0$, we have $xx^*a=0$, so $(xx^*)(x^*x) = 0$.
\end{proof}

\begin{remark}\label{rem6.12}
  As mentioned in the introduction, it is natural to consider the
  extension associated with the finite part of the bidual $B^{**}_{\text{fin}}$.  However,
  when $B$ has infinitely many extreme traces, the kernel of the
  surjection $B_\omega \to (B^{**}_{\text{fin}})^\omega$ may not
  be separably stable.  
  For example, if $T(B)$ is a Bauer simplex with $K\coloneqq \partial_e T(B)$ infinite, we may find a sequence of strictly positive functions $f_n \in C_{\mathbb R}(K)=\Aff T(B)$ of norm one that converge pointwise to zero (by taking positive functions with mutually disjoint supports then increasing them slightly). Using Proposition~\ref{prop:cuntz-pedersen}, we may then obtain a bounded sequence $(b_n)_{n=1}^\infty$ of positive elements such that $\hat b_n = f_n$. Then we have $\|b_n\|_{2,\tau}
  \to 0$ for all $\tau$ yet $\sup_{\tau \in T(B)} \|b_n\|_{2,\tau} \to
  1$ (i.e., the convergence is pointwise but not uniform).  This means
  that the image of $b\coloneqq (b_n)_{n=1}^\infty$ in $(B^{**}_{\text{fin}})^\omega$ is zero, but
  the trace condition makes it impossible to find $x^*x \in
  B_\omega$ close to $b$ with $\|(xx^*)(x^*x)\|$ small.  There is no reason to believe that Bauer simplices are special here; a similar argument works whenever there is an order embedding of $c_0$ into $\Aff T(B)$, as happens say if $\partial_eT(B)$ is countable and infinite, or when $T(B)$ is the Poulsen simplex. In the latter, case for any metrisable Choquet simplex $S$ there is a continuous affine surjection $T(B)\to S$ by \cite[Theorem 2.5]{LOS}. Taking $S$ to be the Bauer simplex whose boundary is the one-point compactification of $\mathbb N$, the dual of this map gives the required embedding.  We expect, but have not shown, that such a sequence $(b_n)$ exists whenever $T(B)$ is not finite dimensional.
\end{remark}

\section{Classifying lifts along the trace-kernel extension}
\label{sec:class-lifts-trace}

The main objective of this section is to classify lifts from the trace-kernel quotient in terms of $KL(A,J_B)$.  There are two versions of
this result. We first establish Theorem~\ref{thm:ClassifyingLifts},
which classifies lifts of unitizably full maps $\theta\colon A\to
B^\infty$, and then we follow a ``de-unitization'' procedure to convert
this result into the classification of unital lifts (Theorem
\ref{thm:ClassifyingUnitalLifts}, which is the precise formulation of Theorem~\ref{intro-lifts}).

The classification of lifts has a long history, dating back to the
theory of extensions of $C^*$-algebras developed by Brown, Douglas,
and Fillmore (\cite{Brown-Douglas-etal77}). 
Our results build on more recent developments by CS (\cite{Schafhauser17,Schafhauser18}); see Remarks~\ref{r:qdlift} and \ref{rmk:LiftsContext}.

\subsection{Classifying lifts of unitizably full maps}

In the following theorem, we follow the definitions of $KK$- and
$KL$-groups for non-$\sigma$-unital codomain $C^*$-algebras (such as
$J_B$) from Section~\ref{subsec:elem-of-kk}.  We defer the proof to Section~\ref{subsec:proof-of-lifts}, allowing us to first address preliminaries.

\begin{theorem}[Classification of lifts]
\label{thm:ClassifyingLifts}
Let $A$ be a separable nuclear $C^*$-algebra, let $B$ be a unital simple separable nuclear $\mathcal Z$-stable $C^*$-algebra with $T(B)\neq\emptyset$, and suppose
$\theta\colon A \to B^\infty$ is a unitizably full $^*$-ho\-mo\-mor\-phism.
\begin{enumerate}[(i)]
\item For any $\kappa \in \KK(A,B_\infty)$ with $KK(A, q_B)(\kappa) = [\theta]_{KK(A,B^\infty)}$, there exists a $^*$-ho\-mo\-mor\-phism $\phi\colon A \to B_\infty$ that lifts $\theta$ and such that $[\phi]_{KK(A,B_\infty)}=\kappa$.\label{thm:ClassifyingLifts.C1}
\item Given a $^*$-homomorphism $\psi\colon A\to B_\infty$ that
  lifts $\theta$ and any $\lambda \in \KL(A,J_B)$, there exists a
  $^*$-ho\-mo\-mor\-phism $\phi\colon A \to B_\infty$ that also lifts
  $\theta$ and such that
  $[\phi,\psi]_{KL(A,J_B)}=\lambda$.\label{thm:ClassifyingLifts.C2}
\item If $\phi_1$ and $\phi_2$ are $^*$-homomorphisms that lift $\theta$ and satisfy $[\phi_1,\phi_2]_{KL(A,J_B)} = 0$, then they are unitarily equivalent.
  In particular, the map $\phi$ in \ref{thm:ClassifyingLifts.C2} is unique up to unitary equivalence.\label{thm:ClassifyingLifts.C3}
\end{enumerate}
\end{theorem}

The data from Theorem~\ref{thm:ClassifyingLifts} is summarized in the
following diagram: given $\theta$ and a $KK$-class $\kappa$ we can
lift $\theta$ to a map $A\to B_\infty$.  Further, for a fixed lift
$\psi$ of $\theta$, the lifts $\phi$ are classified up to unitary
equivalence by the class $[\phi, \psi]_{KL(A, J_B)}$.
\begin{equation}\label{ClassificationLifts.Diag1}
  \begin{tikzcd}
    & & & A \ar[dl, bend right, dash dot, "{\kappa}" description] \ar[dl, bend left,
    "{\phi,\psi}" description, end anchor={[yshift=+.5ex]}] \ar[d, "\theta\text{ unitizably
      full}"] \ar[dll, bend right, dash dot,
    "{\lambda=[\phi,\psi]_{KL}}" swap]\\
    0\ar[r] & \ar[r] J_B\ar[r] & B_\infty\ar[r] & B^\infty\ar[r] & 0
  \end{tikzcd}
\end{equation}

\begin{remark}
  \label{rem:KK-lifts-are-needed}
It is important in Theorem
\ref{thm:ClassifyingLifts}\ref{thm:ClassifyingLifts.C1} that we are
given a $KK$-lift $\kappa$ of $\theta$. This is a genuine assumption
--- for example, it certainly fails when $A$ has a non-trivial
projection and $B$ does not.  When $A$ satisfies the UCT, the
existence of a $KK$-lift of $\theta$ is equivalent to the existence of
$K_0$-lift of $\theta$ --- that is, a map $\alpha_0 \colon K_0(A)
\rightarrow K_0(B_\infty)$ such that $K_0(q_B) \circ \alpha_0 = K_0(\theta)$ (by Proposition~\ref{Prop:KK-trace-kernel-quotient} and surjectivity of $\Gamma^{(A,B_\infty)}\colon KK(A,B_\infty) \to \Hom(K_*(A),K_*(B))$).
When we apply Theorem~\ref{thm:ClassifyingLifts}\ref{thm:ClassifyingLifts.C1} in the classification of full approximate embeddings, we will use the UCT to produce a $KK$-lift of $\theta$ compatible with a specified morphism in total $K$-theory.
\end{remark}

\begin{remark}\label{r:qdlift}
  CS's proof in \cite{Schafhauser17} of the quasidiagonality theorem of \cite{TWW}, that every faithful trace on a separable nuclear $C^*$-algebra satisfying the UCT is quasidiagonal, can be viewed as a
  special case of Theorem~\ref{thm:ClassifyingLifts}\ref{thm:ClassifyingLifts.C1} (in the setting of ultrapowers, which are more natural in the unique trace setting of \cite{Schafhauser17}).  Indeed,
  a standard computation (see Section~3 of \cite{Schafhauser17}, for
  example) shows that $K_0(q_{\mathcal Q}) \colon K_0({\mathcal Q}_\infty) \rightarrow
  K_0({\mathcal Q}^\infty)$ is a surjective linear map of rational vector spaces
  and hence splits.  Now, if $A$ is a nuclear $C^*$-algebra with a
  faithful trace, one obtains a unitizably full
  $^*$-ho\-mo\-mor\-phism $\theta\colon A \to {\mathcal Q}^\infty$ from the
  amenability of this trace (a naive construction yields a full but not unitizably full map; embedding into the top-left entry of $M_2(\mathcal Q)^\infty \cong \mathcal Q^\infty$ provides enough space so that the composition is unitizably full).  This necessarily admits a $K_0$-lift
  since $K_0(q_{\mathcal Q})$ is a split surjection, so when $A$ satisfies the
  UCT, $\theta$ admits a $KK$-lift as well by
  Remark~\ref{rem:KK-lifts-are-needed}.  Now,
  Theorem~\ref{thm:ClassifyingLifts}\ref{thm:ClassifyingLifts.C1}
  implies there is an embedding $A \hookrightarrow {\mathcal Q}_\infty$,
  which allows access to Voiculescu's local characterization of
  quasidiagonality (\cite[Theorem~1]{Voiculescu91}) (since $A$ is
  nuclear, and so this map has a c.p.c.\ lift to $\ell^\infty({\mathcal Q})$ by the
  Choi--Effros lifting theorem).
\end{remark}

\begin{remark}\label{rmk:LiftsContext}
CS implicitly establishes versions of Theorem~\ref{thm:ClassifyingLifts}\ref{thm:ClassifyingLifts.C2} and \ref{thm:ClassifyingLifts.C3} in the proofs of \cite[Propositions~4.2 and~4.3]{Schafhauser18}, in the case that $B$ has a unique (quasi)-trace and is ${\mathcal Q}$-stable (but is not necessarily nuclear). These results, and CS's approach to the quasidiagonality theorem in \cite{Schafhauser17}, allow for exact domains with appropriate amenability conditions on the trace induced by $\theta$ and on $\phi$ and $\psi$ (in \cite{Schafhauser17}, this recaptures the more general version of the quasidiagonality theorem from \cite{Gabe17}).  In comparison, Theorem~\ref{thm:ClassifyingLifts} relaxes $\mathcal Q$-stability to $\Z$-stability and allows for general trace simplices but only covers nuclear domains and codomains.  The additional machinery to allow for exact domains and general $\Z$-stable codomains (whose quasitraces are traces) will be found in the second paper of this series.
\end{remark}

\begin{remark}
  The uniqueness in Theorem~\ref{thm:ClassifyingLifts}\ref{thm:ClassifyingLifts.C3} means that, given
  $\psi$ as in \ref{thm:ClassifyingLifts.C2}, any two lifts
  $\phi_1,\phi_2\colon A\to B_\infty$ of $\theta$ with
  $[\phi_1,\psi]_{KL(A,J_B)}=[\phi_2,\psi]_{KL(A,J_B)}$ are unitarily
  equivalent. In order to view this as a classification of lifts up to
  unitary equivalence by $KL(A,J_B)$, we should check that $[\,\cdot\,
  , \psi]_{KK(A,J_B)}$ (and hence $[\,\cdot\, ,\psi]_{KL(A,J_B)}$) is
  invariant under unitary equivalence: a priori, it is only invariant
  under conjugation by unitaries in $J_B^\dag$ (see
  Proposition~\ref{prop:KK-facts}).  Given $\phi_1\colon A\to
  B_\infty$ lifting $\theta$ and a unitary $u\in B_\infty$, set
  $\phi_2\coloneqq \Ad(u)\circ\phi_1$. By \cite[Proposition~2.1]{CETW21},
  the unitary group of $B^\infty \cap q_B(\phi_1(A))'$ is
  path-connected, and hence
  Lemma~\ref{lemma:AsympEquiv}\ref{lemma:AsympEquiv.C1} implies that
  there is a continuous path of unitaries $(v_t)_{t \geq 0} \subseteq
  J_B^\dag$ such that $\mathrm{Ad}(v_t) \circ \phi_1 \rightarrow
  \phi_2$ point-norm (taking the continuous path $(u_t)_{t \geq 0}$ of
  the lemma to be the constant path $u$).  Then a standard application
  of Kirchberg's $\epsilon$-test (Lemma~\ref{lem:EpsTest}) provides a
  unitary $v \in J_B^\dag$ with $\Ad(v) \circ \phi_1 = \phi_2$.  This
  implies $[\phi_1, \psi]_{KK(A,J_B)} = [\phi_2, \psi]_{KK(A,J_B)}$, as
  claimed.
\end{remark}

Pullback arguments relating extensions to the standard multiplier-corona extension are repeatedly used in the proof of Theorem~\ref{thm:ClassifyingLifts}. We set out our notation for this below, which will be used freely throughout the rest of the section.

\begin{notation}
  Given an extension $\mathsf{e}\colon 0 \to I
  \xrightarrow{j_{\mathsf{e}}} E \xrightarrow{q_{\mathsf{e}}} D \to
  0$, we let $\mu_{\mathsf{e}}\colon E \to \mathcal M(I)$ be the
  canonical map of $E$ into the multiplier algebra of $I$, and write
  $\bar\mu_{\mathsf{e}}\colon D \to \mathcal Q(I)$ for the Busby map
  of $\mathsf{e}$.  Recall that these maps fit into the following
  commuting diagram, in which the right-hand square is a pullback:
  \begin{equation}
    \label{eq:ClassifyingLiftsPullback}
    \begin{tikzcd}
      \mathsf{e}\colon \, 0 \ar{r} & I \ar{r}{j_\mathsf{e}} \ar[equals]{d} & E \ar{r}{q_{\mathsf{e}}} \ar{d}{\mu_{\mathsf{e}}} & D \ar{r} \ar{d}{\bar{\mu}_{\mathsf{e}}} & 0\hphantom{.} \\
      \bar{\mathsf{e}}\colon \, 0 \ar{r} & I \ar{r}{\bar{\jmath}_\mathsf{e}} & \mathcal M(I) \ar{r}{\bar{q}_\mathsf{e}} & \mathcal{Q}(I) \ar{r} & 0.
    \end{tikzcd}
  \end{equation}
  That is,
  $^*$-ho\-mo\-mor\-phisms $\theta\colon A\to D$ and $\bar{\phi}\colon
  A\to\mathcal M(I)$ with $\bar q_{\mathsf e} \circ \bar\phi = \bar
  \mu_{\mathsf e} \circ \theta$ induce a unique
  $^*$-ho\-mo\-mor\-phism $\phi\colon A\to E$ with $\mu_{\mathsf{e}}
  \circ \phi = \bar{\phi}$ and $q_{\mathsf{e}} \circ \phi = \theta$.
  More concretely, $E \cong \{(x,d)\in \mathcal M(I)\oplus D :
  \bar{q}_{\mathsf{e}}(x) = \bar{\mu}_{\mathsf{e}}(d)\}$ via the
  isomorphism $e\mapsto (\mu_\mathsf{e}(e), q_\mathsf{e}(e))$.
\end{notation}

\begin{lemma}\label{ClassifyingLifts.AbsorptionCondition}
Let $A$ be a separable nuclear $C^*$-algebra and let
\begin{equation}
\begin{tikzcd}
\mathsf{e}\colon 0\ar[r]& I\ar[r]& E\ar[r]& D\ar[r]& 0    
\end{tikzcd}
\end{equation}
be a unital separable extension such that $I$ is stable and $\Z$-stable. If $\theta\colon
A\rightarrow D$ is a unitizably full $^*$-ho\-mo\-mor\-phism, then 
$\bar{\mu}_{\mathsf{e}} \circ \theta\colon A\to \mathcal Q(I)$ is an absorbing $^*$-ho\-mo\-mor\-phism, and if
$\phi\colon A\to E$ is a $^*$-homomorphism that lifts $\theta$, then  $\mu_{\mathsf{e}} \circ
\phi \colon A\to\mathcal M(I)$ is an absorbing $^*$-ho\-mo\-mor\-phism.
\end{lemma}
\begin{proof}
  Since $\bar{\mu}_{\mathsf{e}}\circ\theta$ is the composition of a
  unitizably full $^*$-ho\-mo\-mor\-phism followed by a unital
  $^*$-ho\-mo\-mor\-phism, it is unitizably full, so absorbing by
  Theorem~\ref{thm:absorbing-z-stable}.  The second part follows from
  the first and Corollary~\ref{cor:multipler-vs-corona}.
\end{proof}

\subsection{Separabilization}
Theorem~\ref{thm:ClassifyingLifts} will be proved using the
$KK$-existence and uniqueness results of
Section~\ref{sec:KK-classification}.  As such, it is vital to
separabilize --- pass to a separable subextension --- so that one can
obtain the absorption needed to apply these results.  We do this in
Lemma~\ref{lem:ClassifyingLiftsSeparablize}.  The idea is to use the
separable stability of $J_B$ and separable $\Z$-stability of
$B_\infty$ to factorize the data in the diagram
(\ref{ClassificationLifts.Diag1}) through a unital separable
subextension.
This is illustrated schematically (the actual factorizations are slightly different) below, where $E$ is
$\mathcal Z$-stable and $I$ is stable:

\begin{equation}
  \begin{tikzcd}[column sep=large]
    & & & A \ar[dl, bend right, dash dot, "{\kappa'}" description] \ar[dl, bend left,
     "{\phi|^E,\psi|^E}" description, pos=0.4, end anchor={[yshift=+0.75ex]}] \ar[d,
    "{\theta|^D}"] \ar[dll, bend right, dash dot, "{\lambda' =
      [\phi|^E,\psi|^E]_{KL}}", swap]\\
    0\ar[r] & \ar[r] I\ar[d, hookrightarrow] \ar[r] & E \ar[d,
    hookrightarrow] \ar[r] & D \ar[d, hookrightarrow] \ar[r] & 0\phantom{.}\\
    0\ar[r] & \ar[r] J_B \ar[r] & B_\infty \ar[r] & B^\infty\ar[r] & 0.
  \end{tikzcd}
\end{equation}
 The proof of \cite[Proposition~1.9]{Schafhauser18}, used below, employs separably inheritable properties (in the sense of Blackadar, \cite[II.8.5]{Blackadar06}), and the fact that they are closed under countable intersections.

\begin{lemma}
  \label{lem:ClassifyingLiftsSeparablize}
  Let $A$, $B$, and $\theta$ be as in Theorem~\ref{thm:ClassifyingLifts}.  There is a unital separable
  subextension
  \begin{equation}
    \label{eq:ClassifyingLiftsSeparablize}
    \begin{tikzcd}
    \mathsf{e}\colon  0 \arrow{r} &  I \arrow{r}{j_{\mathsf{e}}} & E \arrow{r}{q_{\mathsf{e}}} & D  \arrow{r}& 0
    \end{tikzcd}
  \end{equation}
  of the trace-kernel extension $\mathsf{e}_B$ such that $\theta$
  corestricts to a unitizably full map $\theta|^D\colon A \to D$, $E$
  is unital and $\mathcal Z$-stable (and therefore so is $I$), and $I$ is stable.  Furthermore:
  \begin{enumerate}[(i)]
  \item Given $\kappa \in \KK(A,B_\infty)$ such that $KK(A,
    q_B)(\kappa)=[\theta]_{KK(A,B^\infty)}$, $\mathsf{e}$ can be
    chosen so that $\kappa$ factorizes as $KK(A, \iota_{E\subseteq
      B_\infty})(\kappa')$ with $\kappa' \in \KK(A,E)$ and $KK(A,
    q_{\mathsf{e}})(\kappa')=[\theta|^D]_{KK(A,
      D)}$.\label{lem:ClassifyingLiftsSeparablize.C1}
  \item Given $\psi\colon A\to B_\infty$ and $\kappa \in \KK(A,J_B)$,
    $\mathsf{e}$ can be chosen so that both $\psi(A)\subseteq E$ and
    $\kappa$ factorizes as $KK(A, \iota_{I \subseteq J_B})(\kappa')$
    for some $\kappa' \in KK(A,
    I)$.\label{lem:ClassifyingLiftsSeparablize.C2}
  \item Given two maps $\phi_1,\phi_2\colon A \to B_\infty$,
    $\mathsf{e}$ can be chosen so that
    $\phi_1(A)\cup\phi_2(A)\subseteq E$. Moreover, if each of $\phi_1$
    and $\phi_2$ lift $\theta$ and $[\phi_1,\phi_2]_{KL(A,J_B)}=0$ and
    $\phi_1|^E$ and $\phi_2|^E$ denote the corestrictions of $\phi_1$
    and $\phi_2$ to $E$, then $\mathsf{e}$ can be chosen so that
    $\big[\phi_1|^E,\phi_2|^E\big]_{KL(A, I)} =
    0$.\label{lem:ClassifyingLiftsSeparablize.C3}
  \end{enumerate}
\end{lemma}

\begin{proof}
Note that if $\mathsf e$ is a unital separable subextension of $\mathsf e_B$ satisfying \ref{lem:ClassifyingLiftsSeparablize.C1}, \ref{lem:ClassifyingLiftsSeparablize.C2}, or \ref{lem:ClassifyingLiftsSeparablize.C3}, then the same holds for any larger subextension of $\mathsf e_B$. Similarly, if $\mathsf e$ is such that $\theta|^{D}$ is unitizably full, then the same property is satisfied for any larger subextension of $\mathsf e_B$ since $1_{B^\infty} \in D$.
We start with any separable subextension $\mathsf{e}_0\colon 0 \to I_0 \to E_0 \to D_0 \to 0$ such that $E_0$ contains $1_{B_\infty}$ and then enlarge the algebras while keeping them separable as follows. 

Applying \cite[Proposition 1.9]{Schafhauser18} to the forced unitization $\theta^\dagger\colon A^\dagger\to B^\infty$, $D_0$ may be enlarged to arrange that $\theta$ corestricts to a unitizably full map $A \to D_0$. We make the corresponding enlargement of $E_0$ so that it contains lifts of a countable dense subset of $D_0$, and enlarge $I_0$ to $J_B\cap E_0$.

For conditions \ref{lem:ClassifyingLiftsSeparablize.C1},
\ref{lem:ClassifyingLiftsSeparablize.C2}, and
\ref{lem:ClassifyingLiftsSeparablize.C3}, we use the characterizations
of $\KK(A, C)$ and $\KL(A,C)$ as inductive limits of $\KK(A, C_0)$ and
$\KL(A,C_0)$ over separable subalgebras $C_0$ of $C$ (this is Definition~\ref{KK-inductivelimit} for $KK$ and Proposition~\ref{KL-inductivelimit} for $KL$).
In particular, given $\kappa$ as in
\ref{lem:ClassifyingLiftsSeparablize.C1}, considering the identity
$KK(A,q_B)(\kappa)=[\theta]_{KK(A,B^\infty)}$, we can find suitable
separable enlargements of $E_0$ and $D_0$ and $\kappa'\in KK(A,E_0)$
so that $q_B(E_0)=D_0$, $\kappa=KK(A, \iota_{E_0\subseteq
  B_\infty})(\kappa')$, and $KK(A,q_B|^{D_0}_{E_0})(\kappa') =
[\theta|^{D_0}]_{KK(A,D_0)}$.

Given $\psi$ and $\kappa$ as in
\ref{lem:ClassifyingLiftsSeparablize.C2}, $I_0$ may be enlarged so
that $\psi(A)\subseteq E_0$ and $\kappa$ factorizes through $I_0$;
this forces corresponding enlargements to $E_0$ and $D_0$.
Likewise, given $\phi_1$ and $\phi_2$ as in
\ref{lem:ClassifyingLiftsSeparablize.C3}, we can enlarge $E_0$ (and
hence also $I_0$ and $D_0$) so that $\phi_1(A)\cup\phi_2(A)\subseteq
E_0$. As $[\phi_1,\phi_2]_{KL(A,J_B)}=0$, by the characterization of $KL$ as an inductive limit, a suitable separable enlargement of
$I_0$ and $E_0$ satisfies $[\phi_1|^{E_0},\phi_2|^{E_0}]_{KL(A, I_0)}=
0$.  We enlarge $D_0$ to the image $q_B(E_0)$.

Finally, the $\mathcal Z$-stability hypothesis on $B$ ensures that
$B_\infty$ is separably $\mathcal Z$-stable by Proposition
\ref{prop:B_inftyProperties_2}\ref{prop:B_inftyProperties.2}, and,
together with nuclearity, this hypothesis also gives that $J_B$ is
separably stable by Lemma~\ref{lem:J_BsepStable}.  As both stability
and $\Z$-stability are preserved under sequential inductive limits of
separable $C^*$-algebras with injective connecting maps (\cite[Corollary~4.1]{Hjelmborg-Rordam98} and
\cite[Corollary~3.4]{Toms-Winter07}), we can enlarge $0\rightarrow
I_0\rightarrow E_0\rightarrow D_0\rightarrow 0$ to $\mathsf{e}$ as in
\eqref{eq:ClassifyingLiftsSeparablize} in such a way that $I$ is
stable and $E$ is $\mathcal Z$-stable by \cite[Proposition
1.6]{Schafhauser18}.
As $\Z$-stability passes to ideals, the $\Z$-stability of $I$ is automatic.
\end{proof}

\subsection{Proof of Theorem~\ref{thm:ClassifyingLifts}}
\label{subsec:proof-of-lifts}

We start the proof of Theorem~\ref{thm:ClassifyingLifts} with the
following strengthening of part \ref{thm:ClassifyingLifts.C2} of that theorem, which
modifies lifts to realize classes in $KK(A,J_B)$ (rather than just
classes in $KL(A,J_B)$); we will also use this to correct the
$KK$-class of a lift in the proof of part
\ref{thm:ClassifyingLifts.C1}.

\begin{proposition}\label{prop:KKExistence}
In the notation of Theorem~\ref{thm:ClassifyingLifts}, if $\psi \colon A \rightarrow B_\infty$ is a lift of $\theta$ and $\kappa\in \KK(A,J_B)$, then there exists a $^*$-ho\-mo\-mor\-phism $\phi\colon A \to B_\infty$ that also lifts $\theta$ and such that $[\phi,\psi]_{KK(A, J_B)}=\kappa$.
\end{proposition}

\begin{proof}
  Use
  Lemma~\ref{lem:ClassifyingLiftsSeparablize}\ref{lem:ClassifyingLiftsSeparablize.C2}
  to find a unital separable subextension $\mathsf{e}$ of the
  trace-kernel extension such that $E$ is $\Z$-stable, $I$ is stable,
  $\psi(A)\subseteq E$, $\theta|^D$ is unitizably full, and
  $\kappa=KK(A,\iota_{I\subseteq J_B})(\kappa')$ for some
  $\kappa'\in KK(A,I)$.  By Lemma~\ref{ClassifyingLifts.AbsorptionCondition}, $\mu_{\mathsf{e}} \circ
  \psi$ is absorbing.  Therefore, by $KK$-existence
  (Theorem~\ref{thm:KK-Existence}\ref{thm:KK-Existence.C1}),
  $\kappa'$ can be realized as $[\bar{\phi},{\mu}_{\mathsf{e}} \circ
  \psi]_{KK(A,I)}$ for some absorbing $^*$-ho\-mo\-mor\-phism
  $\bar{\phi}\colon A\to \mathcal M(I)$ so that
  $(\bar{\phi},\mu_{\mathsf{e}} \circ \psi)$ forms a Cuntz pair.  In
  particular, $\bar{\phi}$ lifts $\bar\mu_{\mathsf{e}} \circ
  \theta|^D$, and so the pullback square in
  \eqref{eq:ClassifyingLiftsPullback} induces a
  $^*$-ho\-mo\-mor\-phism $\phi\colon A\to E$ lifting $\theta|^D$ and
  satisfying $\bar{\phi} = \mu_{\mathsf{e}} \circ \phi$.  Then, by
  definition (see also \eqref{eq:cuntz-pair-convention}),
\begin{equation}
  \begin{aligned}
    [\phi, \psi]_{KK(A,I)} &=
    [\mu_{\mathsf{e}} \circ \phi, \mu_{\mathsf{e}} \circ \psi]_{KK(A,I)} \\
    &= [\bar\phi,\mu_{\mathsf{e}} \circ \psi]_{KK(A,I)} \\ &= \kappa'.
  \end{aligned}
\end{equation}
Therefore, viewing $\phi$ as a map into $B_\infty$,
$[\phi,\psi]_{KK(A,J_B)}=\kappa$.
\end{proof}

\begin{proof}[Proof of Theorem~\ref{thm:ClassifyingLifts}]
\ref{thm:ClassifyingLifts.C1}:
  Use
  Lemma~\ref{lem:ClassifyingLiftsSeparablize}\ref{lem:ClassifyingLiftsSeparablize.C1}
  to produce a unital separable subextension $\mathsf{e}$ of the
  trace-kernel extension $\mathsf e_B$ such that $I$ is stable, $E$ is
  $\Z$-stable, $\theta(A)\subseteq D$, $\theta|^D$ is unitizably full, and so that $\kappa$
  factorizes as $\kappa=KK(A,\iota_{E\subseteq B_\infty})(\kappa')$
  for some $\kappa'\in KK(A,E)$ with $KK(A,q_\mathsf{e})(\kappa')=[\theta|^D]_{KK(A,D)}$. Since $\kappa'$ is a $KK$-lift of $\theta|^D$ and the
  pullback square in \eqref{eq:ClassifyingLiftsPullback} commutes, it
  follows that $[\bar{\mu}_{\mathsf{e}} \circ
  \theta|^D]_{KK(A,\mathcal Q(I))}$ factorizes through $KK(A,\mathcal
  M(I))$.  As $I$ is stable, $\KK( A, \mathcal M(I) ) = 0$ by
  \cite[Proposition~4.1]{Dadarlat00b}, so $[\bar{\mu}_{\mathsf{e}}
  \circ \theta|^D]_{KK(A,\mathcal Q(I))}=0$.  By the $KK$-existence
  theorem (Theorem~\ref{thm:KK-Existence}\ref{thm:KK-Existence.C2}), $\bar{\mu}_{\mathsf{e}}
  \circ \theta|^D\colon A\to \mathcal{Q}(I)$ lifts to a
  $^*$-ho\-mo\-mor\-phism $\bar{\psi} \colon A \to \mathcal M(I)$.  By
  the pullback \eqref{eq:ClassifyingLiftsPullback},
  $(\bar\psi,\theta|^D)$ induces a $^*$-ho\-mo\-mor\-phism $\psi\colon
  A \to E$. Regarding $\psi$ as taking values in $B_\infty$, it
  follows that $\psi$ lifts $\theta$.

  There is no reason to expect that $\psi$ provides the $\KK$-class we
  need.  This must be corrected.  By construction, in $KK(A,
  B^\infty)$, we have
  \begin{equation}
    \begin{aligned}
      KK(A,q_B)(\kappa) &= [\theta]_{KK(A, B^\infty)}  \\ &= [q_B \circ\psi]_{KK(A, B^\infty)}\\
      &= KK(A,q_{B})([\psi]_{KK(A,B_\infty)}),
    \end{aligned}
  \end{equation}
  (the last equality is recorded in Proposition~\ref{prop:KK-facts}\ref{prop:KK-facts.5}).
  Accordingly, $\kappa- [\psi]_{KK(A,B_\infty)}$ lies in the kernel of
  $KK(A, q_B)$. By half-exactness of $KK(A, \,\cdot\,)$ when $A$ is
  nuclear (see Proposition~\ref{prop:KK-facts}\ref{prop:KK-facts.3}), there exists $\kappa'\in
  KK(A,J_B)$ with
  \begin{equation} \label{eq:ClassifyingLiftsEx3} KK(A,
    j_B)(\kappa')=\kappa-[\psi]_{KK(A,B_\infty)}. \end{equation} By
  Proposition~\ref{prop:KKExistence}, there exists $\phi\colon A \to
  B_\infty$ also lifting $\theta$ such that
  $[\phi,\psi]_{\KK(A,J_B)}=\kappa'$.  Using Proposition~\ref{prop:KK-facts}\ref{prop:KK-facts.5}, we have
  \begin{equation}
    \begin{aligned}
      [\phi]_{KK(A,B_\infty)} - [\psi]_{KK(A,B_\infty)} &= KK(A,
      j_B)(\kappa')\\
      & \overset{\mathclap{\eqref{eq:ClassifyingLiftsEx3}}}= \hspace{1ex}\kappa - [\psi]_{KK(A,B_\infty)},
    \end{aligned}
  \end{equation}
  and so $[\phi]_{KK(A,B_\infty)} = \kappa$.

\ref{thm:ClassifyingLifts.C2}:
  This is an immediate consequence
  of Proposition~\ref{prop:KKExistence}, since $KL(A, J_B)$ is a quotient of $KK(A, J_B)$.

\ref{thm:ClassifyingLifts.C3}:
  Fix a subextension $\mathsf{e}$ as in Lemma~\ref{lem:ClassifyingLiftsSeparablize}\ref{lem:ClassifyingLiftsSeparablize.C3}.
  Then both the maps
  \begin{equation}
      \mu_{\mathsf{e}} \circ \phi_1|^E, \mu_{\mathsf{e}}
  \circ \phi_2|^E\colon A \to \mathcal M(I)
  \end{equation}are absorbing by Lemma~\ref{ClassifyingLifts.AbsorptionCondition}.  The extension
  $\mathsf{e}$ is constructed so that the Cuntz pair
  $(\mu_{\mathsf{e}} \circ \phi_1|^E,\mu_{\mathsf{e}} \circ
  \phi_2|^E)$ represents the trivial class in $KL(A,I)$.  Therefore,
  the $\mathcal Z$-stable $KL$-uniqueness theorem
  (Theorem~\ref{thm:KK-Uniqueness}\ref{KL-Uniqueness}) gives a sequence
  $(u_n)_{n=1}^\infty$ of unitaries in $(I \otimes \mathcal Z)^\dag
  \subseteq E\otimes\mathcal Z$ such that
  \begin{equation}
    \|\mathrm{Ad}(u_n)\big((\mu_\mathsf{e}\otimes \id_{\mathcal Z})(\phi_1(a)\otimes 1_\mathcal{Z})\big) - (\mu_\mathsf{e}\otimes \id_{\mathcal Z})(\phi_2(a)\otimes 1_\mathcal{Z})\| \to 0
 \end{equation}
 for all $a \in A$.
  Since $\Ad(u_n)$ commutes with $\mu_\mathsf{e}\otimes \id_{\Z}$, and $(\mu_\mathsf{e}\otimes\id_{\Z})|_{I\otimes \Z}$ is injective, it follows that $\mathrm{Ad}(u_n)(\phi_1(a)\otimes 1_{\Z}) \to \phi_2(a)\otimes 1_{\Z}$.

  We now have that $\phi_1|^E \otimes 1_\mathcal Z$ and $\phi_2|^E \otimes 1_\mathcal Z$ are approximately unitarily equivalent.  Since $E$ is $\mathcal Z$-stable, Proposition~\ref{PropSSA} implies $\phi_1$ and $\phi_2$ are approximately unitarily equivalent.  Therefore, $\phi_1$ and $\phi_2$ are unitarily equivalent by Lemma~\ref{lem:AUEImpliesUE}.
\end{proof}

\subsection{Classification of unital lifts}\label{sect:unital-lifts}
We end this section by proving Theorem~\ref{intro-lifts}, a unital version of
Theorem~\ref{thm:ClassifyingLifts} which is needed in the proof of the
classification of full approximate embeddings (Theorem
\ref{Main3}).  The map $\Gamma^{(A,B_\infty)}_0\colon KK(A,B_\infty)\to\Hom(K_0(A),K_0(B_\infty)\big)$ below was defined in \eqref{KKtoHom}.

\begin{theorem}[Classification of unital  lifts]
  \label{thm:ClassifyingUnitalLifts}
  Let $A$ be a unital separable nuclear $C^*$-algebra, let $B$ be a
  unital simple separable nuclear $\mathcal Z$-stable $C^*$-algebra
  with $T(B)\neq\emptyset$, and suppose $\theta\colon A \to B^\infty$
  is a unital $^*$-ho\-mo\-mor\-phism such that $\tau \circ \theta$ is
  a faithful trace on $A$ for all $\tau\in T(B^\infty)$.  
  \begin{enumerate}[(i)]
  \item For any $\kappa \in \KK(A,B_\infty)$ satisfying $KK(A, q_B)(\kappa)
    = [\theta]_{KK(A,B^\infty)}$ and
    $\Gamma^{(A,B_\infty)}_0(\kappa)([1_A]_0)=[1_{B_\infty}]_0$, there exists a unital
    $^*$-ho\-mo\-mor\-phism $\phi\colon A \to B_\infty$ that lifts
    $\theta$ and satisfies
    $[\phi]_{KK(A,B_\infty)}=\kappa$.\label{thm:ClassifyingUnitalLifts.C1}
  \item Given any unital $^*$-homomorphism $\psi\colon A\to B_\infty$
    that lifts $\theta$ and any $\lambda \in \ker\KL(A,j_B)$, there exists a
    unital $^*$-ho\-mo\-mor\-phism $\phi\colon A \to B_\infty$ that
    also lifts $\theta$ and satisfies
    $[\phi,\psi]_{KL(A,J_B)}=\lambda$.\label{thm:ClassifyingUnitalLifts.C2}
  \item Any two unital $^*$-homomorphisms $\phi_1$ and $\phi_2$ that both lift $\theta$ and satisfy $[\phi_1,\phi_2]_{KL(A,J_B)}=0$ are unitarily
    equivalent.\label{thm:ClassifyingUnitalLifts.C3}
\end{enumerate}
\end{theorem}

A unital $^*$-homomorphism can never be unitizably full, so Theorem
\ref{thm:ClassifyingLifts} cannot be applied directly to such
maps. Instead, we use a standard ``de-unitization'' trick (seen in \cite{Schafhauser17} and \cite{Schafhauser18}, for example), working in a
$2\times 2$ matrix amplification to provide additional room for maps
to become unitizably full. Recall that $\iota^{(2)}_D\colon D \to M_2(D)$ denotes the top left-hand corner embedding (see Remark~\ref{matrix}).

\begin{proposition}\label{Prop.EasyUnitizablyFull}
  Let $A$ and $D$ be unital $C^*$-algebras.  If $\theta\colon A\to D$
  is a unital and full $^*$-homomorphism, then $\iota^{(2)}_D \circ \theta$
  is unitizably full.
\end{proposition}
\begin{proof}
  An element $x\in A^\dagger$ can be written as $x=a+\lambda
  (1_{A^\dagger}-1_A)$ for some $a\in A$ and $\lambda\in \mathbb C$.
  Thus
  \begin{equation}
    (\iota^{(2)}_D \circ \theta)^\dagger(x)=\begin{pmatrix}\theta(a)&0\\0&\lambda
      1_D
    \end{pmatrix}.
  \end{equation}
  If $a\neq 0$ or $\lambda\neq 0$, then this generates $M_2(D)$ as an ideal since $\theta$ is full.
\end{proof}

Given a unital map $\theta\colon A\to B^\infty$, we will apply
Theorem~\ref{thm:ClassifyingLifts} to $\iota^{(2)}_{B^\infty} \circ \theta$.  Recall that $M_2(B)$ is also unital, simple, separable, nuclear, and $\Z$-stable, so we can apply our earlier results with $M_2(B)$ in place of $B$ (see Remark~\ref{matrix.diag}).

\begin{proof}[Proof of Theorem~\ref{thm:ClassifyingUnitalLifts}] \ref{thm:ClassifyingUnitalLifts.C1}:
  The hypothesis that $\tau \circ \theta$ is faithful for every
  $\tau\in T(B^\infty)$ ensures that $\theta$ is full by
  Lemma~\ref{p:fullcomp}.  Accordingly,
  $\iota^{(2)}_{B^\infty}\circ\theta\colon A\to M_2(B^\infty)$ is unitizably
  full by Proposition~\ref{Prop.EasyUnitizablyFull} (which uses that $B^\infty$ has strict comparison of positive elements by bounded traces; see Proposition~\ref{NoSillyTracesB^infty}\ref{NoSillyTracesB^infty.2}).  Applying
  commutativity of the top-right square of \eqref{matrix.diag} to
  $KK(A,q_B)(\kappa)=[\theta]_{KK(A,B^\infty)}$ gives
  \begin{equation}
    KK(A,q_{M_2(B)})\big( KK(A,\iota^{(2)}_{B_\infty})
    (\kappa) \big)
    =
    [\iota_{B^\infty}^{(2)} \circ \theta]_{KK(A,M_2(B)^\infty)};
  \end{equation}
  i.e., $KK(A,\iota^{(2)}_{B_\infty})(\kappa)$ is a $KK$-lift of
  $\iota^{(2)}_{B^\infty} \circ \theta$.  Therefore, Theorem~\ref{thm:ClassifyingLifts}\ref{thm:ClassifyingLifts.C1}, applied to
  $M_2(B)$ in place of $B$, gives
  a map $\phi^{(2)}\colon A\to M_2(B_\infty)\cong M_2(B)_\infty$ (see Remark~\ref{matrix})
  lifting $\iota^{(2)}_{B^\infty} \circ \theta$ with
  \begin{equation}\label{ClassifyingUnitalLifts.E1}
    [\phi^{(2)}]_{KK(A,M_2(B_\infty))} =
    KK(A,\iota^{(2)}_{B_\infty})(\kappa).
  \end{equation}
  As $\Gamma^{(A,B_\infty)}_0(\kappa)([1_A]_0)=[1_{B_\infty}]_0$, it follows from naturality of $\Gamma^{(A,\ \cdot\ )}_0$ and \eqref{eq:Gamma-KK-K} that
  $[\phi^{(2)}(1_A)]_0=[\iota_{B_\infty}^{(2)}(1_{B_\infty})]_0$.
  Applying Proposition
  \ref{prop:B_inftyProperties_2}\ref{prop:B_inftyProperties.2} to 
  $M_2(B)$ in place of $B$, we
  see that $M_2(B_\infty)\cong M_2(B)_\infty$ has stable rank one, which implies it has cancellation and so projections which agree in $K_0$ are unitarily equivalent (\cite[Proposition 6.5.1]{Blackadar98}, for example).
  Hence $\phi^{(2)}(1_A)$ is unitarily equivalent to
  $\iota^{(2)}_{B_\infty}(1_{B_\infty})$ and so $\phi^{(2)}$ is unitarily equivalent to
  a map of the form $\iota^{(2)}_{B_\infty} \circ \tilde\phi$ for a unital
  $^*$-homomorphism $\tilde{\phi}\colon A\to B_\infty$.

  Fix a trace $\tau \in T(B^\infty)$. 
  Recall that $\tau_2$ is the canonical non-normalized tracial
  functional on $M_2(B^\infty)$ extending $\tau$ (so,
  $\tau_2(1_{M_2(B^\infty)}) = 2)$.  Then
  \begin{equation}
  \begin{array}{rcl}
      \tau \circ q_B  \circ \tilde \phi
      &=& \tau_2 \circ \iota^{(2)}_{B^\infty} \circ  q_B  \circ \tilde\phi\\
      &\overset{\text{Rem. \ref{matrix}}}=&
      \tau_2 \circ q_{M_2(B)} \circ \iota_{B_\infty}^{(2)} \circ \tilde\phi\\
      &=& \tau_2 \circ q_{M_2(B)} \circ \phi^{(2)} \\
      &=& \tau_2 \circ 
      \iota^{(2)}_{B_\infty}  \circ \theta \\
      &=& \tau \circ \theta,
    \end{array}
    \end{equation}
  where the third equality follows from unitary invariance of the
  trace.  As this holds for all $\tau \in T(B^\infty)$, we have $\Aff
  T(q_B \circ \tilde{\phi})=\Aff T(\theta)$. Therefore, by the
  classification of maps into $B^\infty$ (Theorem~\ref{thm:B^inftyClassification}), $q_B \circ \tilde{\phi}$ is
  unitarily equivalent to $\theta$.  The von Neumann-like behavior of
  $B^\infty$ ensures that unitaries in $B^\infty$ lift to unitaries in
  $B_\infty$ (Proposition~\ref{prop:B^inftyExponentials}), so we can
  find a unitary $u\in B_\infty$ such that $\Ad(q_B(u)) \circ q_B
  \circ \tilde{\phi}=\theta$.  Define $\phi\coloneqq \Ad(u) \circ
  \tilde{\phi}$.  Since $\phi^{(2)}$ is unitarily equivalent to
  $\iota^{(2)}_{B_\infty} \circ \phi$,
  \eqref{ClassifyingUnitalLifts.E1} gives
  \begin{equation}
    KK(A,\iota^{(2)}_{B_\infty})([\phi]_{KK(A,B_\infty)}) =
    KK(A,\iota^{(2)}_{B_\infty})(\kappa).
  \end{equation}
  But $KK(A,\iota^{(2)}_{B_\infty})$ is an isomorphism
  (Proposition~\ref{prop:KK-facts}\ref{prop:KK-facts.1}), and so
  $[\phi]_{KK(A,B_\infty)}=\kappa$.

\ref{thm:ClassifyingUnitalLifts.C2}: 
  Using the commutativity of the top-right square of \eqref{matrix.diag} and
  making the identification $M_2(B^\infty)\cong M_2(B)^\infty$ (see Remark~\ref{matrix}) it
  follows that $\iota^{(2)}_{B_\infty} \circ \psi$ lifts
  $\iota^{(2)}_{B^\infty} \circ \theta$, and the latter map is
  unitizably full, as noted in the proof of part
  \ref{thm:ClassifyingUnitalLifts.C1}.

  Applying Theorem~\ref{thm:ClassifyingLifts}\ref{thm:ClassifyingLifts.C2} to $M_2(B)$,
  $\iota^{(2)}_{B_\infty} \circ \psi$, $\iota^{(2)}_{B^\infty} \circ
  \theta$, and $KL(A,\iota^{(2)}_{J_B}) (\lambda)$, there is a
  $^*$-homomorphism $\phi^{(2)}\colon A\to M_2(B_\infty)\cong
  M_2(B)_\infty$ lifting $\iota^{(2)}_{B^\infty} \circ \theta$ and
  with $[\phi^{(2)}, \iota^{(2)}_{B_\infty} \circ
  \psi]_{KL(A,M_2(J_B))} = KL(A,\iota^{(2)}_{J_B})(\lambda)$. Then
  using Proposition~\ref{prop:KK-facts}\ref{prop:KK-facts.5} in the first step,
  and commutativity of the top-left square of \eqref{matrix.diag} at
  the third step, we have
  \begin{equation}
    \begin{aligned}
      [\phi^{(2)}]_{KL(A,M_2(B_\infty))} &- [\iota^{(2)}_{B_\infty}
      \circ \psi]_{KL(A,M_2(B_\infty))} \\
      &{}\hspace{-3ex}= KL(A,j_{M_2(B)})\big( [\phi^{(2)}, \iota^{(2)}_{B_\infty} \circ
      \psi]_{KL(A,M_2(J_B))}\big)\\
      &{}\hspace{-3ex}= KL(A,j_{M_2(B)} \circ \iota^{(2)}_{J_B})(
      \lambda )\\
      &{}\hspace{-3ex}= KL(A, \iota^{(2)}_{B_\infty}) \circ
      KL(A,j_B)(\lambda)\\
      &{}\hspace{-3ex}=0.
    \end{aligned}
  \end{equation}
  Therefore, the induced map on $K_0$ vanishes, so
  $[\phi^{(2)}(1_A)]_0 = [(\iota^{(2)}_{B_\infty} \circ \psi)(1_A)]_0 =
  [\iota^{(2)}_{B_\infty}(1_{B_\infty})]_0$.

  Using that $M_2(B_\infty)\cong M_2(B)_\infty$ has stable rank one
  (by Proposition~\ref{prop:B_inftyProperties_2}\ref{prop:B_inftyProperties.2}), there exists a unitary $u\in M_2(B_\infty)$ such
  that $\mathrm{Ad}(u) \circ \phi^{(2)}(1_A) =
  \iota_{B_\infty}^{(2)}(1_{B_\infty})$.  As $\phi^{(2)}$ lifts $\iota^{(2)}_{B^\infty}\circ\theta$, we have
  \begin{equation}
    \big(
    \mathrm{Ad}(q_{M_2(B)}(u)) \circ
    \iota^{(2)}_{B^\infty}
    \big)
    (1_{B^\infty}) =
    \iota^{(2)}_{B^\infty}(1_{B^\infty}),
  \end{equation} 
  so we can write
  \begin{equation}
    q_{M_2(B)}(u)=\overline v\oplus\overline
    w\coloneqq \begin{pmatrix}\overline{v}&0\\0&\overline{w}\end{pmatrix} 
  \end{equation}
  for some unitaries $\overline{v},\overline{w}\in B^\infty$.  These
  lift to unitaries $v,w\in B_\infty$ respectively by
  Proposition~\ref{prop:B^inftyExponentials}.  Since $\big( \Ad
  ((v\oplus w)^*u) \circ \phi^{(2)}\big)(1_A) =
  \iota^{(2)}_{B_\infty}(1_{B_\infty})$, we can write $\Ad ((v\oplus
  w)^*u) \circ \phi^{(2)}= \iota^{(2)}_{B_\infty} \circ \phi$ for some
  unital $^*$-homomorphism $\phi\colon A\to B_\infty$.  By
  construction, $q_{M_2(B)}((v\oplus w)^*u)=1_{M_2(B^\infty)}$, so
  $\phi$ lifts $\theta$.  As $(v\oplus w)^*u\in J_{M_2(B)}^\dagger$,
  we have
  \begin{equation}
    [\iota^{(2)}_{B_\infty} \circ \phi, \iota^{(2)}_{B_\infty} \circ
    \psi]_{KL(A,M_2(J_B))} = [\phi^{(2)},\iota^{(2)}_{B_\infty} \circ
    \psi]_{KL(A,M_2(J_B))}
  \end{equation}
  (by Proposition~\ref{prop:KK-facts}\ref{prop:KK-facts.2}).  Therefore,
  \begin{equation}
    \begin{aligned}
      KL(A,\iota^{(2)}_{J_B})( [\phi,\psi]_{KL(A,J_B)} ) &=
      [\iota^{(2)}_{B_\infty} \circ \phi, \iota^{(2)}_{B_\infty}
      \circ \psi]_{KL(A,M_2(J_B))}\\
      &= KL(A,\iota^{(2)}_{J_B})(\lambda).
    \end{aligned}
  \end{equation}
  Since $KL(A,\iota^{(2)}_{J_B})$ is an isomorphism
  (Proposition~\ref{prop:KL-facts}\ref{prop:KL-facts.1}), $[\phi,\psi]_{KL(A,J_B)} = \lambda$.

\ref{thm:ClassifyingUnitalLifts.C3}:
  Given
  two unital lifts $\phi_1,\phi_2\colon A\to B_\infty$ of $\theta$
  such that $[\phi_1,\phi_2]_{KL(A,J_B)}=0$, Proposition~\ref{prop:KL-facts}\ref{prop:KL-facts.5} yields
  \begin{equation}
     [\iota_{B_\infty}^{(2)} \circ \phi_1, \iota_{B_\infty}^{(2)}
    \circ \phi_2]_{KL(A,M_2(J_B))} 
    = KL(A,\iota_{J_B}^{(2)})( [\phi_1,\phi_2]_{KL(A,J_B)} ) = 0. 
  \end{equation}
  Just as at the beginning of \ref{thm:ClassifyingUnitalLifts.C1},
  $\iota^{(2)}_{B^\infty} \circ \theta$ is unitizably full, so we can
  apply
  Theorem~\ref{thm:ClassifyingLifts}\ref{thm:ClassifyingLifts.C3} with
  $M_2(B)$ in place of $B$, $\iota^{(2)}_{B^\infty} \circ \theta$ in
  place of $\theta$, $\iota^{(2)}_{B_\infty} \circ \phi_1$ in place of $\phi_1$, $\iota^{(2)}_{B_\infty} \circ\phi_2$ in place of
  $\phi_2$, 
  and $\lambda = 0$ to learn that $\iota^{(2)}_{B_\infty}\circ
  \phi_1$ and $\iota^{(2)}_{B_\infty}\circ \phi_2$ are unitarily
  equivalent.  Since $\phi_1$ and $\phi_2$ are unital, a unitary $u\in
  M_2(B_\infty)$ with $\Ad(u)\circ \iota^{(2)}_{B_\infty} \circ \phi_1
  = \iota^{(2)}_{B_\infty} \circ \phi_2$ must be of the form $u
  = \begin{pmatrix}v & 0\\ 0 & w\end{pmatrix}$ for unitaries $v,w\in
  B_\infty$. Hence $\Ad(v) \circ \phi_1=\phi_2$.
\end{proof}

We end the section by establishing Theorem~\ref{intro-lifts}\ref{it:intro-lifts.2}. Recall that $\widetilde{\Gamma}_0^{(A,J_B)}$ is the $K_0$-component of the map from \eqref{def:tildegamma}.

\begin{theorem}\label{thm:ClassifyingUnitalLiftsIntroRepeat}
    Let $A$ be a unital separable nuclear $C^*$-algebra and $B$ be a unital simple separable nuclear $\mathcal Z$-stable $C^*$-algebra with $T(B) \neq \emptyset$. Suppose that $\psi \colon A \rightarrow B_\infty$ is a unital full embedding and set $\theta\coloneqq q_B \circ \psi\colon A \rightarrow B^\infty$.
    Then $\phi \mapsto [\phi,\psi]_{KL(A,J_B)}$ induces a bijection between unitary equivalence classes of unital full lifts $\phi$ of $\theta$ and elements $\lambda$ of $KL(A,J_B)$ satisfying $\widetilde\Gamma^{(A,J_B)}_0(\lambda)([1_A]_0)=0$.
\end{theorem}

\begin{proof}
Certainly the map $\phi\mapsto [\phi,\psi]_{KL(A,J_B)}$ is well defined on unital lifts $\phi$ of $\theta$, and invariant under unitary conjugation. Moreover, for any such lift $\phi$, we have
$\widetilde{\Gamma}_0^{(A,J_B)}([\phi,\psi]_{KL(A,J_B)})([1_A]_0) = [\phi(1_A)]_0-[\psi(1_A)]_0 = 0$ by Proposition \ref{prop:KL-facts} and \eqref{eq:Gamma-KK-K}.
 
Theorem~\ref{thm:ClassifyingUnitalLifts}\ref{thm:ClassifyingUnitalLifts.C3} shows that the induced map is injective.  To show it is surjective, consider $\lambda \in KL(A,J_B)$ satisfying $\widetilde\Gamma^{(A,J_B)}_0(\lambda)([1_A]_0) = 0$.
Lift $\lambda$ to an element $\kappa \in KK(A,J_B)$.
Then consider $[\psi]_{KK(A,B_\infty)}+KK(A,j_B)(\kappa) \in KK(A,B_\infty)$.
This element satisfies
\begin{equation}
    KK(A,q_B)\Big([\psi]_{KK(A,B_\infty)}+KK(A,j_B)(\kappa)\Big) = [\theta]_{KK(A,B^\infty)}
\end{equation}
by Proposition \ref{prop:KK-facts}\ref{prop:KK-facts.5}, and 
\begin{equation}    \Gamma_0^{(A,B_\infty)}\Big([\psi]_{KK(A,B_\infty)}+KK(A,j_B)(\kappa)\Big)([1_A]_0) = [1_{B_\infty}]_0,
\end{equation}
by naturality of $\Gamma_0^{(A,-)}$.

By Theorem~\ref{thm:ClassifyingUnitalLifts}\ref{thm:ClassifyingUnitalLifts.C1}, there exists a unital full lift $\phi_0$ of $\theta$ such that
\begin{equation} [\phi_0]_{KK(A,B_\infty)}=[\psi]_{KK(A,B_\infty)}+KK(A,j_B)(\kappa). \end{equation}
Descending to $KL$ and using Proposition \ref{prop:KL-facts}\ref{prop:KL-facts.4}, we have 
\begin{equation} \lambda-[\phi_0,\psi]_{KL(A,J_B)} \in \ker KL(A,j_B). \end{equation}
Therefore by Theorem~\ref{thm:ClassifyingUnitalLifts}\ref{thm:ClassifyingUnitalLifts.C2} (with $\phi_0$ in place of $\psi$), there exists a unital lift $\phi$ of $\theta$ satisfying
\begin{equation} [\phi,\phi_0]_{KL(A,J_B)} = \lambda-[\phi_0,\psi]_{KL(A,J_B)}. \end{equation}
Therefore,
\begin{equation} [\phi,\psi]_{KL(A,J_B)} = [\phi,\phi_0]_{KL(A,J_B)}+[\phi_0,\psi]_{KL(A,J_B)}=\lambda, \end{equation}
as required.
\end{proof}


\section{The universal (multi)coefficient theorem and computing $KL(A,J_B)$}
\label{sec:uct-rot}

Our main goal in this section is to prove Theorem~\ref{intro:calcKL}, computing $KL(A,J_B)$ in terms of the invariant $\inv$. This is the point in the argument at which the universal (multi)coefficient theorem becomes crucial.  We review the UCT and UMCT in Sections~\ref{subsec:the-uct} and \ref{subsec:the-umct} and turn to the calculation of $KL(A,J_B)$ in Section~\ref{Sect:RotationMap}.

\subsection{The universal coefficient theorem}
\label{subsec:the-uct}

Inspired by Brown's earlier UCT for Brown--Douglas--Fillmore theory (\cite{Brown84}), Rosenberg and Schochet's UCT (\cite{Rosenberg-Schochet87}) establishes a class of $C^*$-algebras $A$ for which $KK(A,I)$ can be computed using $K$-theory. 
\begin{definition}\label{DefUCT}
  A separable $C^*$-algebra $A$ is said to \emph{satisfy the universal coefficient theorem (UCT)}, or \emph{belong to the UCT class} if, for every
  $\sigma$-unital $C^*$-algebra $I$, the map $\Gamma^{(A,I)}$ from \eqref{KKtoHom} is surjective
  and a natural map
  \begin{equation}
  \label{eq:UCT-kernel}
  \ker\Gamma^{(A, I)} \rightarrow \Ext(K_*(A), K_{1-*}(I)).
\end{equation}
is bijective (see \cite[Section~23.1]{Blackadar98} for a definition of this map).
\end{definition}

Accordingly, when $A$ satisfies the UCT, there is a short exact
sequence
\begin{equation}
  \label{eq:the-UCT}
  0 \rightarrow \Ext\big( K_*(A), K_{1-*}(I) \big) \rightarrow
  KK(A,I)
  \xrightarrow{\Gamma^{(A,I)}}
  \Hom\big( K_*(A), K_*(I) \big) \rightarrow 0
\end{equation}
where the first map is the inverse of the map in \eqref{eq:UCT-kernel}. It is through this sequence that $KK(A,I)$ can be calculated using $K$-theory. 

As discussed in the introduction, the fundamental question is which $C^*$-algebras satisfy the UCT. Shortly after \cite{Rosenberg-Schochet87}, Skandalis gave examples of $C^*$-algebras outside the class (\cite[Th\'eor\`eme 4.1 and Corollaire 4.2]{Skandalis88}), but it remains a major open problem whether all separable nuclear $C^*$-algebras satisfy the UCT. Building on Higson and Kasparov's deep work on the Baum--Connes conjecture (\cite{Higson-Kasparov97}), Tu established the UCT for $C^*$-algebras associated to second countable locally compact Hausdorff amenable groupoids (\cite[Proposition 10.7 and Lemme 3.5]{Tu99}). This was used to show that nuclear $C^*$-algebras containing a Cartan masa satisfy the UCT in \cite{Barlak-Li17}. Many permanence properties for the UCT date back to \cite{Rosenberg-Schochet87}, but the state of the art for crossed products is more recent.

\begin{remark}\label{rmk:UCTcrossedproducts}
    The problem of whether the crossed product $A\rtimes G$ satisfies the UCT has a very different status depending on whether $G$ has torsion.
    \begin{enumerate}
        \item 
    Let $A$ be a separable $C^*$-algebra which satisfies the UCT and let $G$ be a countable discrete amenable group acting on $A$.  
    When $G$ is torsion-free, it follows from \cite[Corollary~9.4]{Meyer-Nest06} that $A\rtimes G$ satisfies the UCT.  In fact, the same result implies $A\rtimes G$ satisfies the UCT if $A\rtimes H$ does for every finite subgroup $H$ of $G$. 
    To see this, note that the notation $\langle \star\rangle$ in \cite{Meyer-Nest06} denotes the class of separable $C^*$-algebras which satisfy the UCT (since it is closed under $KK$-equivalence and contains the bootstrap class described above, and hence all separable abelian $C^*$-algebras).
    The dual Dirac morphism hypothesis of \cite[Corollary~9.4]{Meyer-Nest06} follows from \cite{Higson-Kasparov97}, and when $G$ is discrete, all compact subgroups of $G$ are finite.
        \item
    For finite groups, the UCT problem (whether every separable nuclear $C^*$-algebra satisfies the UCT) is equivalent to whether the UCT is preserved by crossed products of nuclear $C^*$-algebras by finite groups.
    More precisely, \cite[Theorem~4.17]{Barlak-Szabo17} shows that the UCT problem is equivalent to whether every crossed product of $\mathcal O_2$ by $\Zn{2}$ or $\Zn{3}$ satisfies the UCT.
    This is related to results announced by Kirchberg (cf.\ \cite[Exercise 23.15.12]{Blackadar98}).
    \end{enumerate}   
\end{remark}

We record the following explicit computation of $\Gamma^{(A,I)}_1$ for use in the calculation of $\ker KL(A,j_B)$ in Section~\ref{Sect:RotationMap}. This is well-known to experts, but isn't easily pinned down in the literature; the approach we take borrows ideas from \cite{Higson87}.

\begin{proposition}
  \label{prop:KK-K1-computation}
  Suppose $A$ is a unital separable $C^*$-algebra, $I$ is a $C^*$-algebra, and $(\phi, \psi) \colon A \rightrightarrows E\rhd I$ is a unital 
 Cuntz pair (i.e., $E$, $\phi$, and $\psi$ are all unital). Then $\Gamma_1^{(A,I)}([\phi,\psi]_{KK(A,I)})\colon K_1(A) \to K_1(I)$ is the map
\begin{equation}
    \label{eq:UCT-map-formula}
	[u]_{K_1(A)} \mapsto [\phi(u)\psi(u)^*]_{K_1(I)},
\end{equation}
for all unitaries $u \in U_n(A)$ and $n \in \mathbb N$.
\end{proposition}

Note that the right side of \eqref{eq:UCT-map-formula} is indeed in $K_1(I)$, since $\phi(u)\psi(u)^*\in U_n(I^\dagger)$ by the hypothesis that $\phi$ and $\psi$ are unital.

\begin{proof}
  Consider the pullback diagram
  \begin{equation}
    \label{eq:UCT-pullback}
    \begin{tikzcd}
      0 \arrow{r} & I \arrow{r}{\tilde{\jmath}} \arrow[equals]{d}
        & C \arrow{r}{\tilde{q}} \arrow{d}{\pi}
        & A \arrow{r} \arrow{d}{q\circ\phi=q\circ\psi} & 0\phantom{.} \\
      0 \arrow{r} & I \arrow{r}{j} & E \arrow{r}{q} & E/I \arrow{r} & 0.
    \end{tikzcd}
  \end{equation}
By the pullback property, there are unital $^*$-homomorphisms $\tilde{\phi}, \tilde{\psi}
  \colon A \rightarrow C$ so that 
 \begin{equation}\label{prop:KK-K1-computation.eq.new4}
 \pi\circ\tilde{\phi} = \phi,\ \pi\circ
  \tilde{\psi} = \psi,\text{ and }\tilde{q} \circ\tilde{\phi} =
  \tilde{q}\circ\tilde{\psi} = \mathrm{id}_A.
  \end{equation}
  By Proposition~\ref{prop:KK-facts}\ref{prop:KK-facts.5} (with $J\coloneqq I$, $F\coloneqq C$, $\theta\coloneqq \id_I$, and $\bar\theta\coloneqq \pi$), $[\phi, \psi]_{KK(A,I)} = [\tilde \phi , \tilde \psi]_{KK(A,I)}$. Using Proposition~\ref{prop:KK-facts}\ref{prop:KK-facts.5} again, we have
  \begin{equation}\label{prop:KK-K1-computation.eq.new5}
  KK(A,\tilde{\jmath})([\phi, \psi]_{KK(A,I)}) = [\tilde \phi]_{KK(A,C)} - [\tilde \psi]_{KK(A,C)}.
  \end{equation} 
  Using \eqref{prop:KK-K1-computation.eq.new5} and naturality of $\Gamma_1^{(A,\ \cdot\ )}$, 
  we obtain
  \begin{equation}\label{prop:KK-K1-computation.eq.new2}
  \begin{array}{rcl}
  K_1(\tilde \jmath) \circ \Gamma_1^{(A,I)}([\phi, \psi]_{KK(A,I)}) &=&  \Gamma_1^{(A,C)}( [\tilde \phi]_{KK(A,C)} - [\tilde \psi]_{KK(A,C)}) \\
  &\stackrel{\eqref{eq:Gamma-KK-K}}=& K_1(\tilde \phi) - K_1(\tilde \psi).
  \end{array}
  \end{equation}
  For $n\in \mathbb N$, and a unitary $u\in U_n(A)$, $\pi(\tilde{\phi}(u)\tilde{\psi}(u)^*)=\phi(u)\psi(u)^*$, and so in $U_n(I^\dagger)$,
  \begin{equation}\label{prop:KK-K1-computation.eq.new}
\tilde{\phi}(u)\tilde{\psi}(u)^*=\phi(u)\psi(u)^*.
  \end{equation}
  Consequently, we get
  \begin{equation}\label{prop:KK-K1-computation.eq.new3}
  (K_1(\tilde \phi) - K_1(\tilde \psi))([u]_{K_1(A)}) = [\tilde \phi(u)\tilde \psi(u)^*]_{K_1(C)} = K_1(\tilde \jmath)([\phi(u)\psi(u)^*]_{K_1(I)}).
  \end{equation}
  As the top row of \eqref{eq:UCT-pullback} splits by \eqref{prop:KK-K1-computation.eq.new4}, it follows that $K_1(\tilde \jmath)$ is injective. Combining this with \eqref{prop:KK-K1-computation.eq.new2} and \eqref{prop:KK-K1-computation.eq.new3} we get
  \begin{equation}
  \Gamma_1^{(A,I)}([\phi, \psi]_{KK(A,I)})([u]_{K_1(A)}) = [\phi(u)\psi(u)^*]_{K_1(I)}. \qedhere
  \end{equation}  
  \end{proof}

\subsection{The universal multicoefficient theorem}
\label{subsec:the-umct} We now turn to the universal multicoefficient theorem obtained by Dadarlat and Loring in \cite{Dadarlat-Loring96}. From a modern viewpoint, this computes $KL$ as homomorphisms of total $K$-theory. 

At this point we need to change pictures of total $K$-theory from the definition $K_i(D;\Zn{n})\coloneqq K_{1-i}(D\otimes \mathbb I_n)$ we gave in Section~\ref{ss:totalKtheory} to the definition $K_0(D;\Zn{n})\coloneqq KK(\mathbb I_n,D)$ and $K_i(D;\Zn{n})\coloneqq KK^i(\mathbb I_n,D)$ used by Dadarlat and Loring ($KK^0$ is just the usual $KK$-group; $KK^1$ is defined by suspending in one of the variables). 
As noted in \cite{Dadarlat-Loring96}, it is not so difficult to provide a natural isomorphism between $K_{1-i}(D\otimes\mathbb I_n)$ and $KK^i(\mathbb I_n,D)$; we do so for completeness in Proposition~\ref{p:Zntwodef} (the only proof we know of in the literature is somewhat indirect, given in \cite{Kaminker-Schochet19} as a consequence of Spanier--Whitehead duality).  However, as Dadarlat and Loring also observe, it would be awkward to explicitly transfer the Bockstein operations through such an isomorphism, so they construct a family of Bockstein operations directly in their picture using the Kasparov product.  As experts are undoubtedly aware, any two such families generate the same ring and so give rise to the same morphisms in total $K$-theory (see Lemma~\ref{lem:BocksteinYoneda}\ref{BocksteinYoneda.1}). We formalize this in the following proposition (proved in Appendix~\ref{appendix.a.2}), the upshot of which is that it is legitimate for us to apply the universal multicoefficient theorem with either definition of total $K$-theory.

\begin{proposition}\label{prop:TotalKtheoryAgrees}
    The definition of total $K$-theory from \cite{Dadarlat-Loring96} is naturally isomorphic to the definition in Section~\ref{ss:totalKtheory}. Precisely, given any natural isomorphism of the functors $KK^i(\mathbb I_n,\,\cdot \,)$ with $K_i(\,\cdot\, ; \Zn{n})$, the resulting morphisms of total $K$-theory coincide.
\end{proposition}

The maps $\Gamma^{(A,I)}\colon KK(A,I)\to\mathrm{Hom}(K_*(A),K_*(I))$ from (\ref{KKtoHom}) extend to actions of $KK(A,I)$ on total $K$-theory. For separable $A$, working in the $KK^i(\mathbb I_n,D)$ picture of total $K$-theory, the Kasparov product provides a map
\begin{equation}\label{KKtoHomLambda}
\Gamma^{(A,I)}_\Lambda\colon KK(A,I)\to \Hom_\Lambda(\totK(A),\totK(I))
\end{equation}
that is natural in both variables (associativity of the Kasparov product shows that the induced maps do intertwine the Bockstein operations). As with $\Gamma^{(A,I)}$, the map $\Gamma^{(A,I)}_\Lambda$ is defined in the non-separable setting as a limit over separable subalgebras; see \eqref{eq:AppendixKKtoHom2}. Moreover, as the Kasparov product descends by continuity to $KL$, $\Gamma^{(A,I)}_\Lambda$ factors through $KL(A,I)$ (for the non-separable case, this is shown in the proof of Theorem~\ref{thm:the-umct} in Appendix~\ref{sec:kkappendix}).

Assuming that $A$ satisfies the UCT (and assuming $I$ is separable), Dadarlat and Loring's universal multicoefficient theorem identifies the kernel of the map in \eqref{KKtoHomLambda} with the subgroup  $\mathrm{PExt}(K_*(A), K_{1-*}(I))$ of $\mathrm{Ext}(K_*(A), K_{1-*}(I))$ consisting of \emph{pure extensions} (i.e., group extensions where every element of finite order in the quotient lifts to an element of the same order). 
Using his work on topologies on $KK$-theory, Dadarlat  subsequently
 identified $\mathrm{PExt}(K_*(A), K_{1-*}(I))$ naturally with $Z_{KK(A,I)}\subseteq KK(A,I)$ from \eqref{eq:KLZ} when $A$ satisfies the UCT and $I$ is $\sigma$-unital (\cite{Dadarlat05}, following work by Schochet in the nuclear setting, \cite{Schochet02}).  Combining these results gives the modern viewpoint on the Dadarlat--Loring universal multicoefficient theorem (which is equivalent to the UCT for separable nuclear $C^*$-algebras; this is \cite[Theorem~5.5]{Dadarlat05}, noting that Lin's approximate UCT from \cite{Lin05} is equivalent to the UMCT, as the topology used in \cite{Lin05} is equivalent to that in Definition \ref{def:KL}; see \cite[Section 5]{Dadarlat05}).  In Appendix \ref{sec:nonsepUMCT} we note how to extend the universal multicoefficient theorem to general $I$.

\begin{theorem}[The universal multicoefficient theorem]
  \label{thm:the-umct}
  Let $A$ and $I$ be $C^*$-algebras with $A$ separable.  Then the map $\Gamma^{(A,I)}_\Lambda$
  in \eqref{KKtoHomLambda} induces a map
  \begin{equation}\label{MUCT.E1}
   \widetilde{\Gamma}_{\Lambda}^{(A,I)}\colon KL(A,I)\to\Hom_\Lambda(\totK(A),\totK(I)).
  \end{equation}
This map is natural in $A$ and $I$ and is
  an isomorphism when $A$ satisfies the UCT. For a $^*$-homomorphism $\phi:A \to I$, the map satisfies
  \begin{equation} \label{MUCT.E2} \widetilde\Gamma_\Lambda^{(A,I)}([\phi]_{KL(A,I)})=\totK(\phi).
  \end{equation}
\end{theorem}

 In the special case when the codomain is the trace-kernel quotient of one of our allowed codomains, the $K$-theory computations of Section~\ref{SSTKKThy} combine with the universal (multi)coefficient theorem to show that both $\Gamma_0^{(A,B^\infty)}$ and $\Gamma_\Lambda^{(A,B^\infty)}$ compute
 $KK(A,B^\infty)$ (and not just $KL(A,B^\infty)$).

\begin{proposition}\label{Prop:KK-trace-kernel-quotient}
  Let $A$ be a separable $C^*$-algebra satisfying the UCT and
  let $B$ be a unital simple separable nuclear finite $\Z$-stable $C^*$-algebra.  Then the canonical maps
\begin{equation}\label{eq:KK-trace-kernel-quotient}
\begin{split}
\Gamma^{(A, B^\infty)}_0& \colon KK(A,B^\infty)\to \Hom (K_0(A),K_0(B^\infty)),\text{ and} \\
\Gamma^{(A, B^\infty)}_\Lambda& \colon KK(A,B^\infty)\to \Hom_{\Lambda}(\totK(A),\totK(B^\infty))
\end{split}
\end{equation}
are isomorphisms.
\end{proposition}
\begin{proof}
The $K$-theory of $B^\infty$ was computed in Proposition~\ref{prop:B^inftyKtheory}: we have
  $K_0(B^\infty)\cong\Aff T(B^\infty)$ and $K_1(B^\infty) = 0$. In
  particular, $K_0(B^\infty)$ is divisible.  It follows  that $\Ext (K_*(A),K_{1-*}(B^\infty))=0$ (by \cite[Corollary~2.3.2 and Exercise~2.5.1]{Weibel94}, for example). Therefore, as $A$ satisfies the
  UCT, $\Gamma_0^{(A,B^\infty)}$ is an isomorphism by the exactness of \eqref{eq:the-UCT}.
  
  As $\Gamma_0^{(A,B^\infty)}$ factors through $\Gamma_\Lambda^{(A,B^\infty)}$, the latter is necessarily injective.  Combining this with the surjectivity of $\Gamma_\Lambda^{(A,B^\infty)}$ from the universal multicoefficient theorem gives the result.
\end{proof}

\subsection{Computing $KL(A,J_B)$}\label{Sect:RotationMap}

Theorem~\ref{thm:algK1} below relates the $KL$-class of a unital $(A, J_B)$-Cuntz pair $(\phi, \psi)\colon A\rightrightarrows B_\infty\rhd J_B$ satisfying $[\phi]_{KK(A,B_\infty)}=[\psi]_{KK(A,B_\infty)}$ and the homomorphisms they induce on $\Ka(A)$.
In this case, the map $\Ka(\phi)-\Ka(\psi)$ yields a homomorphism
\begin{equation}\label{neweq:8.15}
	K_1(A)/\mathrm{Tor}(K_1(A)) \to \ker \minusa_{B_\infty}
        \subseteq \Ka(B_\infty).
\end{equation}
The following says that, under the UCT, the map in \eqref{neweq:8.15} encodes the class $[\phi,\psi]_{KL(A, J_B)}$.
 The connection between the map $R_{A, B}$ in the following theorem and the ``rotation maps'' $R_{\phi, \psi}$ appearing in the Gong--Lin--Niu classification (\cite{GLN-part1, GLN-part2}) is discussed in Remark~\ref{rem:rotation}.

\begin{theorem}
  \label{thm:algK1}
  Let $A$ be a unital separable $C^*$-algebra, and let $B$ be a unital
  simple separable nuclear finite $\mathcal{Z}$-stable $C^*$-algebra.  Then
  there is a
   natural
  group homomorphism
  \begin{equation}
    R_{A,B}\colon
    \ker KL(A,j_B)
    \to
    \Hom \big(
      K_1(A)\big/ \mathrm{Tor}(K_1(A)), \ker \minusa_{B_\infty}
    \big)
  \end{equation} with the following property.  For a unital $(A,J_B)$-Cuntz pair $(\phi,\psi)\colon A\rightrightarrows 
   B_\infty\rhd J_B$ for which $[\phi]_{KL(A, B_\infty)} = [\psi]_{KL(A,
    B_\infty)}$, the diagram
  \begin{equation}
\label{eq:algK1Diag}
  \begin{tikzcd}
  \Ka(A)\ar[rrrr,"\Ka(\phi)-\Ka(\psi)"]\ar[d, two heads, "\minusa_A"] &&&& \Ka(B_\infty) \\
  K_1(A)\ar[r,"t_A"]&K_1(A)/\mathrm{Tor}(K_1(A))\ar[rrr,"R_{A,B}({[\phi,\psi]}_{KL(A,J_B)})"] &&& \ker\minusa_{B_\infty}\ar[u, tail]  \\
\end{tikzcd}
  \end{equation}
commutes, where $t_A\colon K_1(A) \to K_1(A)/\mathrm{Tor}(K_1(A))$ is the
  quotient map. If, in addition, $A$ satisfies the UCT, then
  $R_{A,B}$ is an isomorphism.
\end{theorem}

\begin{proof}
Applying naturality of the maps $\widetilde{\Gamma}^{(A,\,\cdot\,)}_\Lambda$ in the universal multicoefficient theorem (Theorem~\ref{thm:the-umct}) to $j_B$ induces $\xi\colon \ker KL(A, j_B) \to\ker \Hom_\Lambda(\underline{K}(A), \underline{K}(j_B))$ as in the following commutative diagram:

\begin{equation}\label{Rotation.Diag}
    \begin{tikzcd}
     \ker KL(A, j_B) \ar[r,"\xi"]\ar[d,tail]& \ker \Hom_\Lambda(\underline{K}(A), \underline{K}(j_B))\ar[d,tail]\\
      KL(A, J_B) \ar[r,"\widetilde{\Gamma}_{\Lambda}^{(A,J_B)}"] \ar[d, "{KL(A, j_B)}", swap] &
      \Hom_\Lambda\big( \underline{K}(A), \underline{K}(J_B) \big)
      \ar[d, "{\Hom_\Lambda(\underline{K}(A), \underline{K}(j_B))}"]\\
      KL(A, B_\infty) \ar[r,"\widetilde{\Gamma}_{\Lambda}^{(A,B_\infty)}",swap] & \Hom_\Lambda\big(\underline{K}(A),
      \underline{K}(B_\infty)\big)
    \end{tikzcd}
  \end{equation}
The universal multicoefficient theorem tells us that the horizontal maps $\widetilde{\Gamma}_\Lambda^{(A,J_B)}$ and $\widetilde{\Gamma}_\Lambda^{(A,B_\infty)}$ are isomorphisms when $A$ satisfies the UCT. Therefore, $\xi$ is also an isomorphism (by the five lemma, after adding two rows of zeros to the top of the diagram) in this case.
  
  As $B$ is nuclear and $\Z$-stable, the $K$-theory computation of Proposition~\ref{prop:B^inftyKtheory} ensures that $K_1(B^\infty)=0$ and $K_0(B^\infty)\cong \Aff T(B^\infty)$, so that $K_0(B^\infty)$ is uniquely divisible.  Thus Lemma~\ref{lem:hom-Lambda-computation} applies to the trace-kernel extension 
  and gives a natural isomorphism  \begin{equation}
    \label{eq:KTheoryIsoagain}
  \begin{aligned}
  \ker\Hom_\Lambda\big(
    \underline{K}(A),\underline{K}(j_B)
  \big)
  &\stackrel\cong\longrightarrow
  \Hom \big(
    K_1(A)/\mathrm{Tor}\big(K_1(A)\big), \ker K_1(j_B)
  \big)
  \end{aligned}
  \end{equation}
sending $\underline{\alpha}$ to the map $\alpha_1'\colon K_1(A)/\mathrm{Tor}(K_1(A))\to \ker K_1(j_B)$ induced by $\alpha_1$.  Theorem~\ref{Thm:K1J} provides a natural isomorphism $\omega_B'\colon \ker K_1(j_B)\to \ker\minusa_{B_\infty}$ which is given explicitly by \eqref{Thm:K1J:Formula}.  We define $R_{A,B}$ to be the composition of these three natural maps, as follows:
\begin{equation}
\begin{tikzcd}
\ker KL(A,j_B)\ar[r,"\xi"]\ar[ddr,bend right=20,swap,"R_{A,B}"]&\ker \Hom_\Lambda\big(
    \underline{K}(A),\underline{K}(j_B)\big)\ar[d,"\underline{\alpha}\mapsto \alpha_1'"]\\&\Hom \big(
    K_1(A)/\mathrm{Tor}\big(K_1(A)\big), \ker K_1(j_B)\big)\ar[d,"{\Hom (K_1(A)/\mathrm{Tor}(K_1(A)),\omega_B')}"]\\
    &\Hom (K_1(A)/\mathrm{Tor}(K_1(A)), \ker\minusa_{B_\infty}).
\end{tikzcd}
\end{equation}
As the vertical maps above are isomorphisms, $R_{A,B}$ is an isomorphism when $\xi$ is, and so $R_{A,B}$ is an isomorphism when $A$ satisfies the UCT.

Now let $(\phi,\psi)\colon A\rightrightarrows B_\infty\rhd J_B$ be a unital Cuntz pair such that $[\phi, \psi]_{KK(A, J_B)}$ belongs to $\ker KK(A, j_B)$.  The computation of $R_{A,B}([\phi,\psi]_{KL(A,J_B)})\circ t_A\circ \minusa_A$ is summarized in the following commutative diagram (commutativity of the middle square is the definition of $R_{A,B})$:
  \begin{equation}
  \begin{tikzcd}
    \Ka(A) \ar[r, "t_A\circ \minusa_A"] \ar[dr, swap, "\minusa_A"] &
    K_1(A)/\mathrm{Tor}(K_1(A))
    \ar[rr, "R_{A,B}({[\phi,\psi]}_{KL(A,J_B)})"]
    &&[2em] \ker\minusa_{B_\infty} &[-3em] \subseteq
    &[-2.75em] \Ka(B_\infty)\\
    & K_1(A) \ar[u, "t_A"] \ar[rr, swap, "\tilde\Gamma_1^{(A,J_B)}({[\phi,\psi]}_{KL(A,J_B)})"] && \ker
    K_1(j_B) \ar[u, "\omega_B'"] &[-3em] \subseteq
    &[-2.75em] K_1(J_B) \ar[u,"\omega_B"].
  \end{tikzcd}
\end{equation}
It remains to check that 
\begin{equation}
\omega_B\circ \tilde\Gamma_1^{(A,J_B)}\big([\phi,\psi]_{KL(A,J_B)}\big)\circ \minusa_A=\Ka(\phi)-\Ka(\psi).
\end{equation}
Let $u \in U_n(A)$.
By Proposition~\ref{prop:KK-K1-computation}, 
\begin{equation}
\tilde\Gamma_1^{(A,J_B)}([\phi,\psi]_{KL(A,J_B)})([u]_1)=[\phi(u)\psi(u)^*]_1.    
\end{equation}
Letting $s:J_B^\dagger \to \mathbb C1_{B_\infty} \subseteq B_\infty$ be the canonical scalar map, so that $s(\phi(u)\psi(u)^*)=1_{M_n(B_\infty)}$ since $(\phi,\psi)$ is a unital Cuntz pair.
Therefore, the explicit formula for $\omega_B$ from Theorem~\ref{Thm:K1J}\ref{Thm:K1J.1} gives 
\begin{equation}
\begin{split}
\omega_B\big(\tilde\Gamma_1^{(A,J_B)}([\phi,\psi]_{KL(A,J_B)})([u]_1)\big)&= \ka{\phi(u)\psi(u)^*} - \ka{1_{B_\infty}}\\
&= \ka{\phi(u)}-\ka{\psi(u)}\\
&=(\Ka(\phi)-\Ka(\psi))(\ka{u}),
\end{split}
\end{equation}
as required. \end{proof}

The map in Theorem~\ref{thm:algK1} can be directly related to algebraic $K_1$.
\begin{remark}
Let $A$ be a unital separable $C^*$-algebra and $B$ a unital separable exact
$\mathcal Z$-stable $C^*$-algebra with $T(B) \neq \emptyset$. Then there is a natural homomorphism
\begin{equation}\label{eq:KLtoalg}
    KL(A, J_B) \to \Hom(\Ka(A), \Ka(B_\infty)), \quad \lambda \mapsto \omega_B \circ \tilde \Gamma_1^{(A,J_B)}(\lambda) \circ \minusa_A,
\end{equation}
    where $\omega_B \colon K_1(J_B) \to \Ka(B_\infty)$ is the map from Theorem~\ref{Thm:K1J}. The same computation as in the end of the proof of Theorem~\ref{thm:algK1} shows that if  $(\phi,\psi)\colon A\rightrightarrows 
   B_\infty\rhd J_B$ is a unital $(A,J_B)$-Cuntz pair, then 
   \begin{equation}
       [\phi, \psi]_{KL(A,J_B)} \mapsto \Ka(\phi) - \Ka(\psi)
   \end{equation} 
   is the map in \eqref{eq:KLtoalg}.
\end{remark}

We conclude this section by proving Theorem~\ref{intro:calcKL}, explicitly describing $KL(A,J_B)$ in terms of $\inv(\,\cdot\,)$.

\begin{theorem}
\label{thm:calcKL}
    Let $A$ be a unital separable $C^*$-algebra satisfying the UCT and let $B$ be a unital simple separable nuclear finite $\mathcal Z$-stable $C^*$-algebra.  
    The map 
    \begin{equation}
\Theta\coloneqq \tilde \Gamma_\Lambda^{(A, B_\infty)} \circ KL(A, j_B)\colon KL(A,J_B) \to \Hom_\Lambda(\totK(A),\totK(B_\infty))
    \end{equation}
    fits into the exact sequence
    \begin{equation}\label{eq:KKJ-again}
    \begin{tikzcd}
	   0 \arrow{rr} &[-3ex] &[-28ex] \ker \Hom_\Lambda\big(\totK(A),\totK(j_B)\big) \arrow{r} \arrow[phantom]{dl}[coordinate, name = Z, near end]{} &[-7ex] KL(A, J_B) \arrow[rounded corners, to path={-- ([xshift=3ex]\tikztostart.east) |- (Z) [near end]\tikztonodes -| ([xshift=-3ex]\tikztotarget.west) -- (\tikztotarget)}]{dll}[swap, pos = .9]{\Theta } \\[1ex] & \Hom_\Lambda \big(\totK(A), \totK(B_\infty)\big) \arrow{rr} & & \Hom_\Lambda \big(\totK(A), \totK(B^\infty) \big),
	\end{tikzcd}
	\end{equation}
    where the second arrow is the restriction of $(\tilde\Gamma_\Lambda^{(A, J_B)})^{-1}$ (an isomorphism by Theorem~\ref{thm:the-umct}), 
    and the final arrow is given by $\Hom_\Lambda\big(\totK(A), \totK(q_B)\big)$.  Moreover:
\begin{enumerate}
\item
\label{calcKL.1}
The range of $\Theta$ consists of all $\underline\alpha$ for which $\rho_{B_\infty} \circ \alpha_0 = 0$ and the kernel of $\Theta$ is  $\ker KL(A,j_B)$.
\item
\label{calcKL.2}
Let $\psi\colon A \to B_\infty$ be a unital full $^*$-homomorphism.  There is a bijection $\Omega$ taking $\ker \Theta \subseteq KL(A,J_B)$ to the set of $\beta\colon \Ka(A) \to \Ka(B_\infty)$ such that $(\totK(\psi),\beta,\Aff T(\psi))$ is a $\inv(\,\cdot\,)$-morphism, given by
\begin{equation}
\label{eq:calcKLOmegaDef}
\Omega(\lambda)\coloneqq \Ka(\psi) + R_{A,B}(\lambda) \circ t_A \circ \minusa_A,\qquad \lambda \in \ker\Theta.
\end{equation}
\end{enumerate}
\end{theorem}

\begin{proof}
The exactness at the first two entries of \eqref{eq:KKJ-again} follows from the naturality of $\tilde\Gamma_\Lambda^{(A, \,\cdot\,)}$.  Indeed, after using $\tilde\Gamma_\Lambda^{(A, J_B)}$ to identify $KL(A, J_B)$ with $\Hom_\Lambda\big(\totK(A), \totK(J_B)\big)$ via the universal multicoefficient theorem, the second arrow becomes the inclusion, $\Theta$ becomes $\Hom_\Lambda\big( \totK(A), \totK(j_B) \big)$ and \eqref{eq:KKJ-again} is certainly exact at this entry.  For exactness at $\Hom_\Lambda(\totK(A),\totK(B_\infty))$, it suffices to show that $\im \Theta$ is contained in $\ker \Hom_\Lambda\big(\totK(A), \totK(q_B)\big)$ since the reverse inclusion follows from $q_B \circ j_B = 0$. The issue is that, in general, neither $KL(A, \,\cdot\,)$ nor $\Hom_\Lambda\big(\totK(A), 
    \totK(\,\cdot\,)\big)$ is half-exact, so we instead use half-exactness of $KK(A,\,\cdot\,)$ together with the control on $KK(A,B^\infty)$ given by the von Neumann-like behavior of $B^\infty$.
    
    Consider the commutative diagram
    \begin{equation}\label{eq:KKJ2}
    \begin{tikzcd}[column sep = 2.7ex, row sep = 8ex]
        KK(A, J_B) \arrow{r}\arrow{d}{\Gamma_\Lambda^{(A, J_B)}} & KK(A, B_\infty) \arrow{r} \arrow{d}  {\Gamma_\Lambda^{(A, B_\infty)}} & KK(A, B^\infty) \arrow{d}{\Gamma_\Lambda^{(A, B^\infty)}} \hskip -2ex \\
        \Hom_\Lambda\big(\totK(A), \totK(J_B)\big) \arrow{r} & \Hom_\Lambda\big(\totK(A), \totK(B_\infty)\big) \arrow{r} & \Hom_\Lambda\big(\totK(A), \totK(B^\infty)\big) \hskip -2ex
    \end{tikzcd}
    \end{equation}
    associated to the trace-kernel extension. Since $A$ satisfies the UCT, $A$ is $KK$-equivalent to an abelian $C^*$-algebra (see \cite[Theorem 23.10.5]{Blackadar98}). Then half-exactness of $KK(A,\,\cdot\,)$ (recorded as Proposition~\ref{prop:KK-facts}\ref{prop:KK-facts.3}) gives exactness of the top row of \eqref{eq:KKJ2}.
    A diagram chase using the surjectivity of $\Gamma_\Lambda^{(A, B_\infty)}$ (from the UMCT) and the injectivity of $\Gamma_\Lambda^{(A, B^\infty)}$ (from Proposition~\ref{Prop:KK-trace-kernel-quotient}) proves that $\ker \Hom_\Lambda\big(\totK(A), \totK(q_B)\big)$ is contained in $\im \Hom_\Lambda\big(\totK(A), \totK(j_B)\big)$. Indeed, any $\underline{\alpha}\in \ker\Hom_\Lambda(\totK(A),\totK(q_B))$ is induced by some $\kappa\in KK(A,B_\infty)$, which is in the kernel of $KK(A,q_B)$ by injectivity of $\Gamma_\Lambda^{(A,B^\infty)}$. Half-exactness shows that $\kappa=KK(A,j_B)(\kappa')$ for some $\kappa'\in KK(A,J_B)$, and commutativity of the left square of \eqref{eq:KKJ2} shows that $\Gamma_\Lambda^{(A,J_B)}(\kappa')$ is mapped onto $\underline{\alpha}$ via $\Hom(\totK(A), \totK(j_B))$.
    
    By naturality of $\tilde{\Gamma}_\Lambda^{(A,\,\cdot\,)}$,
    \begin{equation}
    \Theta=\Hom_\Lambda(\totK(A),\totK(j_B))\circ\tilde{\Gamma}_\Lambda^{(A,J_B)},
    \end{equation}
    so that (using that $\tilde{\Gamma}_\Lambda^{(A,J_B)}$ is an isomorphism by the universal multicoefficient theorem) $\im\Theta=\im\Hom_\Lambda(\totK(A),\totK(j_B))$. Therefore, \eqref{eq:KKJ-again} is exact.

\ref{calcKL.1}:
The exactness of \eqref{eq:KKJ-again} implies that the image of $\Theta$ consists of the set of all $\underline\alpha\colon \totK(A)\to \totK(B_\infty)$ for which $\underline K(q_B) \circ \underline\alpha = 0$.
By Proposition~\ref{prop:B^inftyKtheory}, this condition is equivalent to $\rho_{B_\infty} \circ \alpha_0 = 0$.  The kernel of $\Theta$ is $\ker\,KL(A,j_B)$ as $\tilde{\Gamma}_\Lambda^{(A,B_\infty)}$ is an isomorphism.

\ref{calcKL.2}:
Fix $\lambda \in \ker \Theta$ and define $\beta\coloneqq \Omega(\lambda)$. We first check that $(\totK(\psi),\beta,\Aff T(\psi))$ is a morphism $\inv(A) \to \inv(B_\infty)$. This boils down to the fact that $\inv(\psi)$ is a morphism (Proposition~\ref{Inv:AFunctor}) and the correcting term $R_{A,B}(\lambda)\circ t_A\circ\minusa_A$ in \eqref{eq:calcKLOmegaDef} is annihilated in the $\inv(\,\cdot\,)$-compatibility relations. 

Precisely, the first square of \eqref{eq:compatibility} commutes for free since it does not involve the map $\beta$.
For the second square, since $\minusa_A \circ \Th_A = 0$, we have 
\begin{equation}
\beta \circ \Th_A \stackrel{\eqref{eq:calcKLOmegaDef}}= \Ka(\psi) \circ \Th_A \stackrel{\text{Prop.\ \ref{Inv:AFunctor}}}= \Th_{B_\infty} \circ \Aff T(\psi),
\end{equation}
For the third square of \eqref{eq:compatibility}, since $\im R_{A,B}(\lambda) \subseteq \ker\minusa_{B_\infty}$, we similarly have 
\begin{equation}
\minusa_{B_\infty} \circ \beta \stackrel{\eqref{eq:calcKLOmegaDef}}= \minusa_{B_\infty} \circ \Ka(\psi) \stackrel{\text{Prop.\ \ref{Inv:AFunctor}}}= K_1(\psi) \circ \minusa_A.
\end{equation}
Finally, for commutativity of \eqref{eq:compatibility2}, note that by \eqref{eq:newmap1}, we have $\minusa_A \circ \zeta_A^{(n)} = \nu_{0,A}^{(n)}$, so the range of this composition is contained in $\mathrm{Tor}(K_1(A),\Zn{n})$.
Consequently, $t_A \circ \minusa_A \circ \zeta_A^{(n)} = 0$, so similar to the previous calculations, we have
\begin{equation}
\beta \circ \zeta_A^{(n)} \stackrel{\eqref{eq:calcKLOmegaDef}}= \Ka(\psi) \circ \zeta_A^{(n)} \stackrel{\text{Prop.\ \ref{Inv:AFunctor}}}= \zeta_B^{(n)} \circ K_0(\psi;\Zn{n}).
\end{equation}
This shows that $(\totK(\psi),\beta,\Aff T(\psi))$ is a $\inv(\,\cdot\,)$-morphism.

Suppose now that $\lambda,\lambda' \in \ker\Theta$ satisfy $\Omega(\lambda)=\Omega(\lambda')$.
By \eqref{eq:calcKLOmegaDef}, we get $R_{A,B}(\lambda-\lambda')\circ t_A \circ \minusa_A =0$.
Since $t_A \circ \minusa_A$ is surjective, and $R_{A,B}$ is injective (from Theorem~\ref{thm:algK1}), it follows that $\lambda=\lambda'$.
This shows that $\Omega$ is injective.

Finally, let us show that $\Omega$ is surjective; that is, if $\beta:\Ka(A)\to\Ka(B_\infty)$ is such that $(\totK(\psi),\beta,\Aff T(\psi))$ is a morphism $\inv(A) \to \inv(B_\infty)$, then there exists $\lambda \in \ker\Theta$ such that $\Omega(\lambda)=\beta$.
Using Lemma~\ref{Compatibility.Lem} with $\beta'\coloneqq \Ka(\psi)$, there is a homomorphism $r\colon K_1(A)/\mathrm{Tor}(K_1(A))\to\ker\minusa_{B_\infty}$ such that
  \begin{equation}
\label{eq:calcKLDiagram}
  \begin{tikzcd}
  \Ka(A)\ar[rr,"\beta-\Ka(\psi)"]\ar[d,"\minusa_A", two heads] && \Ka(B_\infty) \\
  K_1(A)\ar[r,"t_A", two heads]&K_1(A)/\mathrm{Tor}(K_1(A))\ar[r,"r"] & \ker\minusa_{B_\infty}\ar[u,tail]  \\
\end{tikzcd}
\end{equation}
commutes.
By surjectivity of $R_{A,B}$ (from Theorem~\ref{thm:algK1}), it follows that there exists $\lambda \in \ker KL(A,j_B)=\ker \Theta$ such that $r=R_{A,B}(\lambda)$.
Hence,
\begin{equation}
\begin{array}{rcl}
\Omega(\lambda)
&\stackrel{\eqref{eq:calcKLOmegaDef}}=& \Ka(\psi)+R_{A,B}(\lambda)\circ t_A \circ \minusa_A \\
&=& \Ka(\psi)+r \circ t_A \circ \minusa_A \\
&\stackrel{\eqref{eq:calcKLDiagram}}=& \Ka(\psi)+(\beta-\Ka(\psi)) \\
&=& \beta,
\end{array}
\end{equation}
as required.
\end{proof}

\begin{remark}\label{rem:rotation}
The map $R_{A, B}$ in Theorem~\ref{thm:algK1} is  related to the rotation maps appearing in early classification results of $^*$-homomorphism up to asymptotic unitary equivalence, dating back to \cite{Kishimoto-Kumjian01}.  For the sake of comparison, we recall the version of the rotation map from \cite[Definition~2.21]{GLN-part1}.

Suppose $A$ and $B$ are unital $C^*$-algebras with $A$ separable and $\phi, \psi \colon A \rightarrow B$ are $^*$-homomorphisms with $[\phi]_{KK(A, B)} = [\psi]_{KK(A, B)}$ and $\Aff T(\phi) = \Aff T(\psi).$ Define the \emph{mapping torus} of $\phi$ and $\psi$ by
\begin{equation}
    M_{\phi,\psi} \coloneqq \{(f,a) \in  C([0, 1], B) \oplus A : \phi(a) = f(0) \text{ and } \psi(a) = f(1) \}.
\end{equation}
There is a corresponding extension
\begin{equation}\label{8.38}
\begin{tikzcd}
    0 \arrow{r} & SB \arrow{r} & M_{\phi, \psi} \arrow{r} & A \arrow{r} & 0
\end{tikzcd}
\end{equation}
with the first map being the inclusion into the first factor and the second map being the projection onto the second factor. Since the boundary map in the associated six-term exact sequence of $KK$-groups is naturally identified with $KK(A, \phi) - KK(A, \psi)=0$ (as $[\phi]_{KK(A, B)} = [\psi]_{KK(A, B)}$), any preimage of $[\id_A]_{KK(A,A)}$ in $KK(A,M_{\phi,\psi})$ is a $KK$-splitting $\kappa \in KK(A, M_{\phi, \psi})$ of the extension \eqref{8.38}. Hence there is an induced split exact sequence
\begin{equation}
\begin{tikzcd}
    0 \arrow{r} & K_0(B) \arrow{r} & K_1(M_{\phi, \psi}) \arrow[shift right = .5ex]{r} & K_1(A) \arrow{r} \arrow[shift right = .5ex]{l}[swap]{\kappa_1} \arrow{r} & 0,
\end{tikzcd}
\end{equation}
where $\kappa_1 \coloneqq \Gamma_1^{(A, M_{\phi, \psi})}(\kappa)$.

Using that $\Aff T(\phi) = \Aff T(\psi)$, the de la Harpe--Skandalis determinant produces a well-defined map $R_{\phi, \psi} \colon  K_1(M_{\phi, \psi}) \rightarrow \Aff T(B)$ by
\begin{equation}
\label{eq:rotationR}
    R_{\phi, \psi}([(u, v)]_1) \coloneqq \widetilde\Delta_B(u), \qquad (u, v) \in U_\infty(M_{\phi, \psi}),
\end{equation}
called the \emph{rotation map} corresponding to $\phi$ and $\psi$.  The composition
\begin{equation}
\label{eq:rotationTildeR}
\begin{tikzcd}
    K_1(A) \arrow{r}{\kappa_1} & K_1(M_{\phi, \psi}) \arrow{r}{R_{\phi, \psi}} & \Aff T(B) \arrow[two heads]{r} & \Aff T(B)/\overline{\im \rho_B},
\end{tikzcd}
\end{equation}
which is denoted $\tilde R_{\phi, \psi}$, is independent of the choice of splitting $\kappa$ (see \cite[Lemma~3.3]{Lin08}).

If $(\phi, \psi)\colon A\rightrightarrows B_\infty \rhd J_B$ is a Cuntz pair as in Theorem~\ref{thm:algK1}, and, in addition, $[\phi]_{KK(A, B_\infty)} = [\psi]_{KK(A, B_\infty)}$, so that both $\tilde R_{\phi, \psi}$ and $R_{A, B}([\phi, \psi]_{KL(A, J_B)})$ are defined, then they are related by the equation
\begin{equation}\label{equation:comparingrotationmaps}
    \tilde R_{\phi, \psi} = -\overline{\det}_{B_\infty} \circ R_{A, B}([\phi, \psi]_{KL(A, J_B)}) \circ t_A.
\end{equation}

To see this, by \eqref{eq:detInverse} and \eqref{eq:rotationTildeR}, it suffices to show that \begin{equation}
    \Th_{B_\infty} \circ R_{\phi,\psi}\circ \kappa_1=-R_{A,B}([\phi,\psi]_{KL(A,J_B)})\circ t_A.
\end{equation}
For $v \in U_\infty(A)$, write $\kappa_1([v]_1)=[(u,v')]_1$, so that $[v']_1=[v]_1$ and $u$ is a path in $U_\infty(B_\infty)$ from $\phi(v')$ to $\psi(v')$.
Then
\begin{equation}
    \begin{array}{rcl}
        \Th_{B_\infty}\left(R_{\phi,\psi}(\kappa_1([v]_1))\right)
        &=& \Th_{B_\infty}(R_{\phi,\psi}([(u,v')]_1)) \\
        &\stackrel{\eqref{eq:rotationR}}=& \Th_{B_\infty}(\widetilde\Delta_{B_\infty}(u)) \\
        &\stackrel{\eqref{DefDeterminant},\eqref{DefThomsen}}=& [u(1)u(0)^*]_{\mathrm{alg}} \\
        &=& [\psi(v')]_{\mathrm{alg}}-[\phi(v')]_{\mathrm{alg}} \\
        &=& (\Ka(\psi)-\Ka(\phi))([v']_{\mathrm{alg}}) \\
        &\stackrel{\eqref{eq:algK1Diag}}=& -R_{A,B}([\phi,\psi]_{KL(A,J_B)})(t_A([v']_1)) \\
        &=& -R_{A,B}([\phi,\psi]_{KL(A,J_B)})(t_A([v]_1)),
    \end{array}
\end{equation}
establishing \eqref{equation:comparingrotationmaps}.
\end{remark}

\section{Classification of unital $^*$-homomorphisms}
\label{sec:main-results}

We now have all the pieces in place to establish the main classification theorems: Theorems~\ref{Main}, \ref{Main2}, and~\ref{Main3}. First we prove the  classification of unital full approximate embeddings, and then we use intertwining arguments to deduce Theorem~\ref{Main2} in Section~\ref{sect:classembed} and Theorem~\ref{Main} in Section~\ref{sec:unitalclass}.

\subsection{Classification of unital full approximate embeddings}
\label{subsec:approx-classif-embeddings}

We start with the precise statement of Theorem~\ref{Main3}.

\begin{theorem}[Classification of unital full approximate embeddings]
  \label{approximate-classification}
  Let $A$ and $B$ be unital separable nuclear $C^*$-algebras such that $A$ satisfies the UCT and $B$ is simple and $\mathcal Z$-stable.  If
  $(\underline{\alpha},\beta,\gamma) \colon \inv(A)\to\inv(B_\infty)$ is a faithful morphism, then there exists a unital full
  $^*$-homomorphism $\phi \colon A\to B_\infty$ such that $\inv(\phi)
  = (\underline{\alpha},\beta,\gamma)$, and this $\phi$ is unique up
  to unitary equivalence.
\end{theorem}

As mentioned in the introduction, 
by means of Kirchberg's dichotomy (Theorem~\ref{T:KDich}), $B$ is either purely infinite or stably finite.
While the stably finite case is our main result, let us first discuss the purely infinite case, which is a refined version of the Kirchberg--Phillips Theorem.  This could be proved using Kirchberg's (unpublished) techniques from the 1990s (\cite{Kirchberg}); we will deduce the result from the purely infinite classification theorem in \cite{Gabe-Preprint}.

\begin{proof}[Proof of Theorem~\ref{approximate-classification} when $T(B)=\emptyset$]
	$B$ is purely infinite and hence is $\mathcal O_\infty$-stable by \cite[Theorem 3.15]{Kirchberg-Phillips00}.
As noted in Remark~\ref{rmk:TotInvPI}, $\inv$-morphisms from $\inv(A)$ to $\inv(B_\infty)$ (all of which are vacuously faithful) correspond to $\totK$-morphisms from $\totK(A)$ to $\totK(B_\infty)$ sending $[1_A]_0$ to $[1_{B_\infty}]_0$.
Hence by \cite[Theorem 8.12]{Gabe-Preprint}, these correspond bijectively to full $\mathcal O_\infty$-stable (in the sense of \cite[Definition 4.1]{Gabe-Preprint}) unital $^*$-homomorphisms $A \to B_\infty$, up to approximate unitary equivalence.

By \cite[Proposition~1.12]{Schafhauser18}, $B_\infty$ is separably $\mathcal O_\infty$-stable, and consequently any $^*$-homomorphism $A \to B_\infty$ factorizes through a separable $\mathcal O_\infty$-stable subalgebra of $B_\infty$, and is therefore $\mathcal O_\infty$-stable.
Note also that approximate unitary equivalence is the same as unitary equivalence (as recorded in Lemma~\ref{lem:AUEImpliesUE}).
This proves the theorem in this case.
\end{proof}

Now we return to the stably finite case, beginning with the existence portion of Theorem~\ref{approximate-classification}. The strategy is to first realize $\gamma$ by a map into $B^\infty$, obtained from the classification of maps into $B^\infty$, and lift this to a map into $B_\infty$ realizing $\underline{\alpha}$ using the classification of unital lifts.  Finally, we adjust this lift by combining the second part of the classification of unital lifts with the $KL$ computations in Section~\ref{sec:uct-rot}.

\begin{proof}[Proof of existence in Theorem~\ref{approximate-classification} when $T(B)\neq \emptyset$]
  Suppose $(\underline{\alpha}, \beta, \gamma) \colon \inv(A) \to \inv(B_\infty)$
  is a faithful morphism.  By the classification of maps into
  $B^\infty$ by traces (Theorem~\ref{thm:B^inftyClassification}),
  there is a unital $^*$-homomorphism $\theta \colon A \to B^\infty$
  such that
  \begin{equation}
    \label{Existence.Eq.T.Theta}
    \Aff T(\theta)
    = \Aff T(q_B) \circ \gamma \colon \Aff T(A) \to \Aff T(B^\infty).
  \end{equation}
  Since $(\underline{\alpha}, \beta, \gamma)$ is faithful, every trace
  $\tau\in T(B^\infty)$ gives rise to a faithful trace $\tau\circ \theta = \gamma^*(\tau\circ q_B)$ on
  $A$.

  Since $A$ satisfies the UCT, Dadarlat and Loring's universal multicoefficient theorem
  (Theorem~\ref{thm:the-umct}) provides $\kappa \in KK(A, B_\infty)$
  inducing $\underline{\alpha}$. We
  claim that $\kappa$ is a $KK$-lift of $\theta$, i.e., that
  \begin{equation}\label{eq:KK-lift}
    KK(A, q_B)(\kappa) = [\theta]_{KK(A,B^\infty)}.
  \end{equation}
To see this, first note that $\Gamma_0^{(A,B^\infty)}(KK(A,q_B)(\kappa))=K_0(q_B)\circ\alpha_0$ using \eqref{eq:Gamma-KK-K}. Using naturality of the pairing map $\rho$ in
  the first and last line, and the compatibility of $\alpha_0$ and
  $\gamma$ (i.e., commutativity of the first square in
  \eqref{eq:compatibility} from the conditions on $\inv$-morphisms) in
  the second, we have
  \begin{equation}
    \begin{array}{rcl}
      \rho_{B^\infty} \circ K_0(q_B) \circ \alpha_0
      &=& \ \Aff T(q_B) \circ \rho_{B_\infty} \circ \alpha_0 \\
      &=& \ \Aff T(q_B) \circ \gamma \circ \rho_{A} \\
      &\stackrel{\mathclap{\eqref{Existence.Eq.T.Theta}}}{=}& \ \Aff T(\theta) \circ \rho_{A}\\
      &=& \ \rho_{B^\infty} \circ K_0(\theta).
    \end{array}
  \end{equation}
  Since $\rho_{B^\infty}$ is an isomorphism 
  (Proposition~\ref{prop:B^inftyKtheory}), $K_0(\theta)=K_0(q_B)\circ
  \alpha_0 $, and so $\Gamma_0^{(A,B^\infty)}(KK(A,q_B)(\kappa))=\Gamma_0^{(A,B^\infty)}([\theta]_{KK(A,B^\infty)})$. The claim in \eqref{eq:KK-lift} now follows as $\Gamma_0^{(A,B^\infty)}$ is an isomorphism by Proposition~\ref{Prop:KK-trace-kernel-quotient} (using that $A$ satisfies the UCT).

  Note also that $\Gamma_0^{(A,B_\infty)}(\kappa)([1_A]_0)=\alpha_0([1_A]_0)=[1_{B_\infty}]_0$, as $(\underline{\alpha},\beta,\gamma)$ is a $\inv$-morphism.
  Therefore, the classification of unital lifts (Theorem~\ref{thm:ClassifyingUnitalLifts}\ref{thm:ClassifyingUnitalLifts.C1}) implies the existence of a unital $^*$-homomorphism $\psi\colon A\to B_\infty$ lifting $\theta$ and satisfying $[\psi]_{KK(A,B_\infty)}=\kappa$. Since $\psi$ lifts $\theta$ and $\Aff T(q_B)$ is an isomorphism (by Proposition~\ref{NoSillyTracesB^infty}\ref{NoSillyTracesB^infty.1}), we have
  \begin{equation}
    \label{approximate.classification.e.3}
    \Aff T(\psi)=\gamma.
  \end{equation}
  Likewise, since $\kappa$ induces $\underline{\alpha}$, we have
  $\underline{K}(\psi)=\underline{\alpha}$. However, there is no
  reason why $\Ka(\psi)$ should be equal to $\beta$.  

By Theorem~\ref{thm:calcKL}\ref{calcKL.2}, whose proof uses the compatibility of $\beta$ with the maps $\zeta^{(n)}$ (via Lemma~\ref{Compatibility.Lem}), 
there exists $\lambda\in \ker KL(A,j_B) \subseteq KL(A,J_B)$ such that
\begin{equation}
\label{approximate.classification.e.4}
\beta = \Ka(\psi)+R_{A,B}(\lambda) \circ t_A \circ \minusa_A.
\end{equation}
By the classification of unital lifts (Theorem~\ref{thm:ClassifyingUnitalLifts}\ref{thm:ClassifyingUnitalLifts.C2}), there exists a unital $^*$-homomorphism $\phi\colon A\to B_\infty$ lifting $\theta$ with $[\phi,\psi]_{KL(A,J_B)}=\lambda$.   Since this class lies in
  $\ker KL(A, j_B)$, we have 
  $[\phi]_{KL(A,B_\infty)} = [\psi]_{KL(A,B_\infty)}$ by Proposition~\ref{prop:KL-facts}\ref{prop:KL-facts.5}.
Therefore, by applying $\tilde\Gamma_\Lambda^{(A,B_\infty)}$, we have
  \begin{equation}
    \totK(\phi) = \totK(\psi) = \underline{\alpha}.
  \end{equation}
  By Theorem~\ref{thm:algK1}, we have
  \begin{equation}
\begin{array}{rcl}
    \Ka(\phi)
    &=& \Ka(\psi) + R_{A, B}([\phi, \psi]_{KL(A, J_B)}) \circ t_A \circ \minusa_A \\
    &=& \Ka(\psi) + R_{A, B}(\lambda) \circ t_A \circ \minusa_A \\
    &\stackrel{\eqref{approximate.classification.e.4}}=& \beta.
\end{array}
  \end{equation}
  
  Finally, since $\phi$ lifts $\theta$ and $\Aff T(q_B)$ is an isomorphism, we have $\Aff T(\phi)=\gamma$.
  The three previous equations combine to show that $\inv(\phi) =
  (\underline{\alpha}, \beta, \gamma)$, completing the proof.
\end{proof}

Now we turn to uniqueness. The pairing maps $\zeta^{(n)}$ of
Section~\ref{sec:new-map} are not needed for this part of the
theorem.

\begin{proof}[Proof of uniqueness in Theorem~\ref{approximate-classification} when $T(B)\neq \emptyset$]
  Let $\phi, \psi\colon A\to B_\infty$ be two unital full
  $^*$-homomorphisms such that $\inv(\phi) = \inv(\psi)$.  Then, in particular, $\Aff
  T(\phi) = \Aff T(\psi)$, and we have
  \begin{equation}
    \Aff T(q_B \circ \phi)
    = \Aff T(q_B \circ \psi) \colon \Aff T(A) \to \Aff T(B^\infty).
  \end{equation}
  Hence $q_B\circ \phi$ and $q_B\circ \psi$ are unitarily equivalent by the
  classification of maps into $B^\infty$ by traces
  (Theorem~\ref{thm:B^inftyClassification}).  Since every unitary in
  $B^\infty$ can be lifted to a unitary in $B_\infty$ by
  Proposition~\ref{prop:B^inftyExponentials}, we have that $q_B\circ\phi = q_B
  \circ \Ad(u) \circ \psi$ for some unitary $u\in B_\infty$.  By replacing $\psi$ with $\Ad(u)\circ\psi$ (which leaves $\inv(\,\cdot\,)$ unchanged --- see Proposition~\ref{Inv:AFunctor}), we may assume that
  $q_B\circ \phi = q_B\circ \psi$.
  Therefore, $(\phi,\psi)$ forms an $(A,J_B)$-Cuntz pair, and defines
  an element $[\phi,\psi]_{KL(A,J_B)}$ in $KL(A, J_B)$.

  Letting $\Theta$ be as in Theorem~\ref{thm:calcKL}\ref{calcKL.1}, we have
\begin{equation}
\begin{array}{rcl}
\Theta([\phi,\psi]_{KL(A,J_B)})
&\stackrel{\text{Thm.\ \ref{thm:calcKL}\ref{calcKL.1}}}{=}&\big(\tilde\Gamma_\Lambda^{(A,B_\infty)}\circ KL(A,j_B)\big)([\phi,\psi]_{KL(A,J_B)}) \\
&\stackrel{\text{Prop.\ \ref{prop:KL-facts}\ref{prop:KL-facts.5}}}=& \tilde\Gamma_\Lambda^{(A,B_\infty)}([\phi]_{KL(A,B_\infty)}-[\psi]_{KL(A,B_\infty)})\\
&\stackrel{\eqref{MUCT.E2}}=& \totK(\phi)-\totK(\psi)\\
&=& 0.
\end{array}
\end{equation}
This shows that $[\phi,\psi]_{KL(A,J_B)}\in \ker \Theta =\ker KL(A,j_B) $.

Since $\Ka(\phi)=\Ka(\psi)$ and both $t_A$ and $\minusa_A$ are surjective, it follows by Theorem~\ref{thm:calcKL} that $R_{A,B}([\phi,\psi]_{KL(A,J_B)})=0$.
Since $R_{A,B}$ is injective (by Theorem~\ref{thm:calcKL}), we conclude that $[\phi,\psi]_{KL(A,J_B)}=0$.
  
Hence by the classification of unital lifts (Theorem~\ref{thm:ClassifyingUnitalLifts}\ref{thm:ClassifyingUnitalLifts.C3}), it follows that $\phi$ and $\psi$ are unitarily equivalent.
\end{proof}

\subsection{Classification of embeddings}\label{sect:classembed}

Next we use the classification of unital full approximate embeddings to prove the classification of embeddings theorem (Theorem~\ref{Main2}).  Whereas uniqueness up to approximate unitary equivalence in Theorem~\ref{Main2} is an immediate consequence of the uniqueness up to unitary equivalence in Theorem~\ref{approximate-classification}, an intertwining argument is needed for existence.
In the literature (such as \cite{Lin01a,Lin07,GLN-part1,GLN-part2}), this intertwining argument is often applied in an approximate form, by reformulating the approximate classification theorem in terms of approximately multiplicative maps, requiring the use of approximately defined maps on the invariant.
In contrast, we opt for an approach that uses intertwining by reparameterization to identify those maps $A\to B_\infty$ that are unitarily equivalent to maps factoring through $B$.  To the best of our knowledge, the idea first appears (albeit in a continuous form) in Phillips' approach to the Kirchberg--Phillips theorem in \cite[Proposition~1.3.7]{Phillips00}, where it is attributed to Kirchberg.

To set this up, given a function $r \colon \mathbb{N} \to \mathbb{N}$ define a reparameterization map $r^*\colon{}\ell^\infty(B)\to \ell^\infty(B)$ by $r^*((b_n)_{n=1}^\infty)\coloneqq (b_{r(n)})_{n=1}^\infty$.  When $\lim_{n\to\infty}r(n)=\infty$, this also induces a map $B_\infty\to B_\infty$, which we continue to denote by $r^*$. 
Recall from \eqref{eq:iota} that $\iota_B\colon B\to B_\infty$ denotes the canonical embedding.

  \begin{proposition}[Intertwining by reparameterization, {\cite[Theorem 4.3]{Gabe20}}]\label{t:Gabereindex}
  Let $A$ and $B$ be $C^*$-algebras, with $A$ separable and $B$ unital. Let $\phi_\infty \colon A \to B_\infty$ be a $^*$-homomorphism. There exists a $^*$-homomorphism $\phi\colon A \to B$ such that $\phi_\infty$ and $\iota_B\circ \phi$ are unitarily equivalent if and only if for every $r\colon \mathbb N\to \mathbb N$ with $\lim_{n\to \infty}r(n) = \infty$, the maps $\phi_\infty$ and $r^*\circ \phi_\infty$ are unitarily equivalent.
  \end{proposition}

With this additional ingredient in place we can prove the classification of embeddings (Theorem~\ref{Main2}).

\begin{theorem}[Classification of embeddings]
  \label{one-sided-classification}
	Let $A$ and $B$ be unital separable nuclear $C^*$-algebras such that $A$ satisfies the UCT and $B$ is simple and $\mathcal Z$-stable.  If
	$(\underline{\alpha},\beta,\gamma) \colon \inv(A)\to\inv(B)$ is a faithful morphism, then there exists a unital injective
	$^*$-homomorphism $\phi \colon A\to B$ such that $\inv(\phi)
	= 	(\underline{\alpha},\beta,\gamma)$, and this $\phi$ is unique up to approximate unitary equivalence.
\end{theorem}

\begin{proof}
Uniqueness is essentially an immediate consequence of Theorem~\ref{approximate-classification}. Given unital injective $^*$-homomorphisms $\phi, \psi \colon A \to B$, the compositions $\iota_B \circ \phi,\iota_B \circ \psi\colon A\to B_\infty$ are full as $B$ is simple. Accordingly, if $\inv(\phi)=\inv(\psi)$, then $\inv(\iota_B \circ \phi) = \inv(\iota_B \circ \psi)$ and so $\iota_B \circ \phi$ and
  $\iota_B \circ \psi$ are unitarily equivalent by Theorem~\ref{approximate-classification}. As unitaries in $B_\infty$ can be lifted to unitaries in $\ell^\infty(B)$, $\phi$ and $\psi$
  are approximately unitarily equivalent.

  For existence, consider a faithful $\inv$-morphism
\begin{equation}
(\underline{\alpha}, \beta, \gamma) \colon \inv(A)
  \to \inv(B).
\end{equation}
Then $\inv(\iota_B)\circ(\underline{\alpha},\beta,\gamma)$ is also faithful, so  
the existence part of Theorem~\ref{approximate-classification} gives a unital full
  $^*$-homomorphism $\phi_\infty \colon A \to B_\infty$ such that
  \begin{equation}
    \inv(\phi_\infty) = \inv(\iota_B)\circ (\underline{\alpha}, \beta, \gamma).
  \end{equation}
  Fix $r \colon \mathbb{N} \to \mathbb{N}$ with $\lim_{n \to \infty} r(n) =
  \infty$. Since 
  $r^* \circ \iota_B = \iota_B$, it follows that
  \begin{equation}\inv(r^* \circ \phi_\infty) = \big( \inv(r^*) \circ \inv(\iota_B) \big) (\underline \alpha, \beta, \gamma) =  \inv(\phi_\infty).
  \end{equation}
  Moreover, $r^*\circ\phi_\infty$ is full as $r^*$ is unital and $\phi_\infty$ is full. Therefore,  $r^* \circ\phi_\infty$ and
  $\phi_\infty$ are unitarily equivalent by the uniqueness portion of Theorem~\ref{approximate-classification}.  Then intertwining by reparameterizations (Proposition~\ref{t:Gabereindex}) gives a $^*$-homomorphism $\phi \colon A \to
  B$ such that $\iota_B \circ \phi$ and $\phi_\infty$ are unitarily
  equivalent.  As $\phi_\infty$ is unital and faithful, so is $\phi$.
  Note that
  \begin{equation}
\inv(\iota_B)\circ \inv(\phi)= \inv(\iota_B \circ \phi) \stackrel{\text{Prop. }\ref{Inv:AFunctor}}= \inv(\phi_\infty) = \inv(\iota_B) \circ (\underline{\alpha}, \beta, \gamma).
  \end{equation}
  By Lemma~\ref{lem:inv-iota-injective}, $\inv(\iota_B)$ is a monomorphism, so $\inv(\phi) = (\underline{\alpha}, \beta, \gamma)$.
\end{proof}

\begin{remark}
    We use sequence algebras $B_\infty$ rather that ultrapowers $B_\omega$ in our classification of full approximate embeddings (Theorem~\ref{Main3}) in order to access the intertwining by reparameterization technique; reparameterizations do not make sense for ultrapowers. Prompted by the discrepancy between sequence algebras and ultrapowers, Farah proved that an existence result into the ultrapower is equivalent to an existence result into the sequence algebra (\cite[Theorem A]{Farah}).
\end{remark}

Corollary~\ref{MainCor} is an immediate consequence of the classification of embeddings.

\begin{corollary}
Let $A$ and $B$ be unital separable nuclear $C^*$-algebras with non-empty tracial state spaces, and suppose that $A$ satisfies the UCT and that $B$ is simple and $\mathcal Z$-stable.  Given a faithful unital compatible pair
\begin{equation}
(\alpha_*,\gamma)\colon{}(K_*(A),\Aff T(A))\to (K_*(B),\Aff T(B)),
\end{equation}
there is a unital embedding $\phi\colon{}A\to B$ such that $K_*(\phi)=\alpha_*$ and $\Aff T(\phi)=\gamma$.
\end{corollary}
\begin{proof}
By Proposition~\ref{PropExtendAlgK1}, we can extend the pair $(\alpha_*,\gamma)$ to a compatible triple $(\underline{\alpha},\beta,\gamma)\colon{}\inv(A)\to\inv(B)$.  The classification of embeddings theorem (Theorem~\ref{one-sided-classification}) then gives a unital embedding $\phi\colon{}A\to B$ with $\inv(\phi)=(\underline{\alpha},\beta,\gamma)$. Hence $K_*(\phi)=\alpha_*$ and $\Aff T(\phi)=\gamma$.
\end{proof}

Note that in the previous corollary the embedding $\phi$ is not necessarily unique up to approximate unitary equivalence.  In general, uniqueness requires the total invariant.  See Example \ref{ZwreathZ} below.

\begin{remark}
\label{rmk:MatuiClassification} Let us note how Theorem~\ref{one-sided-classification} relates to other classification results for embeddings. In \cite[Corollary 6.8 and Theorem 7.1]{Matui11}, Matui proves uniqueness results for embeddings $A\to B$, where $A$ is either a unital AH algebra (including $C(X)$) or a unital simple separable nuclear $\mathcal Z$-stable $C^*$-algebra that is rationally TAF and satisfies the UCT.  Matui allows unital simple separable stably finite $\mathcal Z$-stable codomains $B$ that are either rationally TAF or exact with finitely many extreme traces that are separated by projections.  For comparing maps $\phi, \psi \colon A \rightarrow B$, Matui uses an invariant consisting of $\underline K(\, \cdot\,)$, $\Aff T(\, \cdot\,)$ and (vanishing of) an induced homomorphism $\Theta_{\phi, \psi} \colon K_1(A) \to \Aff T(B)/ \overline{\im\rho_B}$ (defined when $K_1(\phi)=K_1(\psi)$ and $T(\phi)=T(\psi)$ --- see \cite[Lemma~3.1]{Matui11}) given by
\begin{equation}
	\Theta_{\phi, \psi}([u]_1) = \mathrm{det}_B(\phi(u)^* \psi(u)), \qquad u\in U_\infty(A).
\end{equation}
Using this, one can compute that $\overline{\Th}_B \circ \Theta_{\phi, \psi} \circ \minusa_A = \Ka(\psi) - \Ka(\phi)$ so (since $\minusa_A$ is surjective and $\overline{\Th}_B$ is injective) $\Theta_{\phi, \psi}$ vanishes if and only if $\Ka(\phi) = \Ka(\psi)$. Note that this map $\Theta$ coincides with the map $\tilde R_{\phi, \psi}$ of Remark~\ref{rem:rotation} in the case $[\phi]_{KK(A, B)} = [\psi]_{KK(A, B)}$ (so that $\tilde R_{\phi, \psi}$ is defined).  Accordingly, in the case when the codomains are additionally assumed to be nuclear, Matui's uniqueness theorems are encompassed within Theorem~\ref{one-sided-classification} above.

While this paper was being prepared, Gong, Lin, and Niu independently obtained a classification of embeddings (\cite[Theorems 4.3 and 6.5]{Gong-Lin-etal23}), building on prior work of Lin and Niu in the rationally tracial rank one setting (\cite{Lin-Niu14}).  Compared with Theorem \ref{one-sided-classification}, the main difference in their result is a simplicity hypothesis on the domain algebra.

Whereas our approach establishes Theorem~\ref{one-sided-classification} as an ingredient towards classification of $C^*$-algebras, Gong, Lin, and Niu use their earlier classification theorems (\cite{GLN-part1,GLN-part2}) and the additional simplicity hypothesis on $A$ to obtain internal tracial approximation structure on its UHF-stabilizations.
They also use a variation of Winter's localization technique.
\end{remark}

\subsection{The unital classification theorem}\label{sec:unitalclass}

The stably finite half of the unital classification theorem is now obtained by the standard strategy of symmetrizing the assumptions on the domain and codomain in the classification of embeddings and using the following form of Elliott's two-sided intertwining argument (see \cite[Corollary~2.3.4]{Rordam02}, for example).
\begin{proposition}[Elliott's two-sided intertwining]\label{two-sided-intertwining}
Let $A$ and $B$ be unital separable $C^*$-algebras, and let $\phi_0\colon A\to B$ and $\psi_0\colon B\to A$ be $^*$-homomorphisms. Suppose that $\psi_0\circ\phi_0$ is approximately unitarily equivalent to $\mathrm{id}_A$ and $\phi_0\circ\psi_0$ is approximately unitarily equivalent to $\mathrm{id}_B$.  Then $\phi_0$ is approximately unitarily equivalent to an isomorphism between $A$ and $B$.
\end{proposition}

When $A$ and $B$ are as in the unital classification theorem (Theorem~\ref{Main}), an isomorphism of the total invariant can be lifted to an isomorphism of the algebra.

\begin{theorem}\label{algebra-classification-total}
  Let $A$ and $B$ be unital simple separable nuclear
  $\mathcal{Z}$-stable $C^*$-algebras satisfying the UCT.  Suppose
\begin{equation}
(\underline \alpha, \beta, \gamma)\colon{} \inv(A) \to \inv(B)
\end{equation}
is a $\inv$-isomorphism. Then there exists a $^*$-isomorphism $\phi\colon A\to B$, unique up to approximate unitary equivalence, such that $\inv(\phi) = (\underline \alpha, \beta, \gamma)$.
\end{theorem}

\begin{proof}
Uniqueness follows from the uniqueness portion of the classification of embeddings theorem (Theorem~\ref{one-sided-classification}), noting that faithfulness of the invariant is automatic since all traces on $A$ are faithful.

The existence portion of the classification of embeddings theorem (Theorem~\ref{one-sided-classification}) gives unital $^*$-homomorphisms $\phi_0\colon A\to B$ and $\psi_0\colon A\to B$ such that $\inv(\phi_0)=(\underline{\alpha},\beta,\gamma)$ and $\inv(\psi_0)=\inv(\phi_0)^{-1}$. By the uniqueness aspect of the classification of embeddings theorem (Theorem~\ref{one-sided-classification}), $\phi_0 \circ \psi_0$ is approximately unitarily equivalent to $\id_B$ and $\psi_0 \circ \phi_0$ is approximately unitarily equivalent to $\id_A$. Therefore, Elliott's two-sided intertwining argument (Proposition~\ref{two-sided-intertwining}) provides a $^*$-isomorphism $\phi \colon A \to B$ which is approximately unitarily equivalent to $\phi_0$. Hence $\inv(\phi) = \inv(\phi_0) = (\underline \alpha, \beta, \gamma)$ since $\inv$ is invariant under approximate unitary equivalence (Proposition~\ref{Inv:AFunctor}).
\end{proof}

We now establish Theorem~\ref{Main} as a consequence.

\begin{theorem}\label{algebra-classification}
  Let $A$ and $B$ be unital simple separable nuclear
  $\mathcal{Z}$-stable $C^*$-algebras satisfying the UCT.  Suppose
\begin{equation}
(\alpha_*,\gamma)\colon{}(K_*(A),\Aff T(A))\to (K_*(B),\Aff T(B))
\end{equation}
is a $KT_u$-isomorphism. Then there exists a $^*$-isomorphism $\phi\colon A\to B$ such that $K_*(\phi) = \alpha_*$ and $\Aff T(\phi)=\gamma$.
\end{theorem}
\begin{proof}
By Proposition~\ref{PropExtendAlgK1}, there is a unital faithful invertible compatible triple $(\underline{\alpha},\beta,\gamma)\colon{}\inv(A)\to\inv(B)$ extending $(\alpha_*,\gamma)$. By Proposition~\ref{algebra-classification-total}, there is an isomorphism $\phi \colon A \to B$ with $\inv(\phi) = (\underline \alpha, \beta, \gamma)$. Hence $K_*(\phi)=\alpha_*$ and $\Aff T(\phi)=\gamma$.
 \end{proof}

The classification of automorphisms modulo approximate unitary equivalence has been well-studied in operator algebras.
For example, Connes' theorem builds on his prior work on this problem for II$_1$ factors (e.g., \cite{Connes75a,Connes75b}).
Given a unital $C^*$-algebra $A$, write $\mathrm{Aut}(A)$ for the automorphism group of $A$, and $\overline{\mathrm{Inn}(A)}$ for the subgroup of approximately inner automorphisms.  When $A$ is a unital Kirchberg $C^*$-algebra satisfying the UCT, the quotient group $\mathrm{Aut}(A)/\overline{\mathrm{Inn}(A)}$ is isomorphic to the group of unital automorphisms of $\underline{K}(A)$ (cf.~\cite[Theorem 5.6]{Elliott-Rordam95}). For unital finite classifiable $C^*$-algebras, unital automorphisms of the total invariant encode this information.

\begin{corollary}
\label{cor:AutClassification}
Let $A$ be a unital simple separable nuclear $\Z$-stable $C^*$-algebra satisfying the universal coefficient theorem.  Then the quotient $\mathrm{Aut}(A)/\overline{\mathrm{Inn}(A)}$ is isomorphic to the group of automorphisms of $\inv(A)$.
\end{corollary}
\begin{proof}
This follows immediately from Theorem~\ref{algebra-classification-total}.
\end{proof}

We found the following Bernoulli shift example, where the $\Ka$-class of an automorphism is particularly discernible, informative during our work on this paper. 

\begin{example}[Automorphisms of $\mathcal Z \wr \mathbb Z$]\label{ZwreathZ}
	Consider the infinite tensor product $\mathcal Z^{\otimes \mathbb Z}$, let $\mathbb Z$ act on this algebra by shifting, and let $A \coloneqq  \mathcal Z^{\otimes \mathbb Z} \rtimes \mathbb Z$ denote the crossed product.  As we review below, it is well-known that such a construction gives rise to a $C^*$-algebra covered by Theorem \ref{Main}.
 
It is immediate that $A$ is separable and unital.  Since $\Z^{\otimes \mathbb Z}$ is isomorphic to $\mathcal Z$ (\cite[Corollary~8.8]{Jiang-Su99}), it is nuclear and satisfies the UCT.  Both properties are preserved by crossed products by $\mathbb Z$, and therefore $A$ is also nuclear and satisfies the UCT.
The shift action on $\mathcal R^{\bar\otimes \mathbb Z}$ is outer (see \cite[Lemma~2.5]{Patchell}), which means that the shift action on $\Z^{\otimes \mathbb Z}$ is strongly outer.
Therefore, $A$ is simple by \cite[Theorem 3.1]{Kishimoto81}, $\mathcal Z$-stable by \cite[Corollary 4.10]{Matui-Sato12a}, and has unique trace by \cite[Theorem~4.3(4)$\Rightarrow$(1)]{Thomsen95} (cf.\ \cite[Theorem 1.11]{Ursu21}).
	
	Now we turn to the invariant of $A$. Using $K_0(\mathcal Z) \cong \mathbb Z$ with unit corresponding to $1 \in \mathbb Z$ and $K_1(\mathcal Z) = 0$, the Pimsner--Voiculescu sequence (\cite{Pimsner-Voiculescu}) implies $K_0(A) \cong \mathbb Z$ with the unit corresponding to $1 \in \mathbb Z$ and $K_1(A) \cong \mathbb Z$.
Combining this with the fact that $A$ has unique trace, we may therefore (by \eqref{eq:K1algSplitting}) identify
\begin{equation}
\label{eq:ZwreathZKa}
\Ka(A) \cong \mathbb T \oplus \mathbb Z,
\end{equation}
and so (as explained below) Corollary \ref{cor:AutClassification} gives
	\begin{equation}
		\mathrm{Aut}(A) / \overline{\mathrm{Inn}(A)} \cong \mathbb T \rtimes \Zn{2}.
	\end{equation}

To do this concretely, let $u \in A$ denote the canonical unitary implementing the group action. Then $[u]_{\mathrm{alg}}$ is a generator of the copy of $\mathbb Z$.
We get a group monomorphism from $\mathbb T$ to $\Ka(A)$ given by
\begin{equation} z \mapsto [z1_A]_{\mathrm{alg}}, \end{equation}
giving a complementary copy of $\mathbb T$.

Since $K_*(A)$ is torsion-free, we have, by \eqref{eq:bockstein-new}, that $K_i(A;\Zn{n})\cong K_i(A)\otimes \Zn{n} = \Zn{n}$ for all $n$.
We also see that for an automorphism $\phi$ of $A$, its action on $K$-theory determines its action on total $K$-theory; since $\phi$ also fixes $[1_A]_0$, $\inv(\phi)$ is entirely determined by $\Ka(\phi)$.

An automorphism of $\Ka(\phi)$ is compatible with the Thomsen map $\Th_A$ if and only if it acts as the identity on $\mathbb T$.
Hence, by Corollary~\ref{cor:AutClassification}, $\mathrm{Aut}(A)/\overline{\mathrm{Inn}(A)}$ identifies with the group of automorphisms of $\Ka(A)\cong \mathbb T \oplus \mathbb Z$ that fix $\mathbb T$.
Each such automorphism is determined by its action on the generator $[u]_{\mathrm{alg}}$ of the copy of $\mathbb Z$.
This group of automorphisms is $\mathbb T \rtimes \Zn{2}$, where an element $z\in\mathbb T$ corresponds to the automorphism $\alpha_z$ taking $[u]_{\mathrm{alg}}$ to $[zu]_{\mathrm{alg}}$, and the generator of $\Zn{2}$ corresponds to the automorphism $\omega$ taking $[u]_{\mathrm{alg}}$ to $-[u]_{\mathrm{alg}}$.

We can find explicit representatives for $\alpha_z$ and $\omega$ as follows:
\begin{enumerate}
\item
For $z \in \mathbb T$, let $\phi_z\colon A \to A$ denote the gauge action.
In other words, $\phi_z$ fixes $\Z^{\otimes \mathbb Z}$ and
\begin{equation} \phi_z(u)\coloneqq  zu.
\end{equation}
Then $\phi_z$ induces the automorphism $\alpha_z$.
\item Let $\psi\colon A \to A$ be the $^*$-homomorphism determined by
\begin{equation} \begin{split}
\psi(u)&\coloneqq u^*,\\ \psi(\cdots \otimes a_{-1} \otimes a_0 \otimes a_1 \otimes \cdots) &\coloneqq \cdots \otimes a_1 \otimes a_0 \otimes a_{-1} \otimes \cdots, \end{split}
\end{equation}
where $\cdots \otimes a_{-1} \otimes a_0 \otimes a_1 \otimes \cdots$ denotes an elementary tensor in $\Z^{\otimes \mathbb Z}$.
Then we can see that $[\psi(u)]_{\mathrm{alg}} = [u^*]_{\mathrm{alg}} = -[u]_{\mathrm{alg}}$, so $\psi$ induces the automorphism $\omega$.
\end{enumerate}
Since the gauge actions act as the identity on $KT_u(A)$, their non-innerness is not detected by that invariant.

Interestingly, the automorphisms $\phi_z$ and $\psi$ just described generate a copy of $\mathbb T \rtimes \Zn{2}$ inside $\mathrm{Aut}(A)$ (that is, without taking the quotient by $\overline{\mathrm{Inn}(A)}$).
Consequently, 
\begin{equation}
     \mathrm{Aut}(A) \cong \overline{\mathrm{Inn}(A)} \rtimes ( \mathbb T \rtimes \Zn{2}) \cong \overline{\mathrm{Inn}(A)} \rtimes \mathrm{Aut}(\inv(A)).
\end{equation}
In particular, if $G$ is a discrete group acting on $\inv(A)$, then the action lifts to an action on $A$.
\end{example}

Connes' deep analysis of automorphisms of factors from \cite{Connes75b,Connes75a,Connes77} was instrumental in his equivalence of injectivity and hyperfiniteness and subsequent classification results (\cite{Connes76,Connes75}).  Among other results, Connes showed that there is a unique outer action of a finite cyclic group on $\mathcal R$ up to conjugacy \cite[Theorem 5.1]{Connes77}. This was subsequently generalized by Jones to all finite groups (\cite{Jones80}) and by Ocneanu to all discrete amenable groups (\cite{Ocneanu85}). These results instigated a vast body of work examining actions of amenable groups on classifiable $C^*$-algebras (see Izumi's survey \cite{Izumi-10}). A recent capstone result is JG and Szab\'o's dynamical Kirchberg-Phillips theorem (\cite{Gabe-Szabo23}), determining when outer actions of discrete amenable groups on Kirchberg algebras are cocycle conjugate in terms of equivariant $KK$-theory.

Note that any countable discrete group $G$ acts by outer automorphisms on any classifiable $C^*$-algebra $A$: identify $A$ with $A\otimes\Z^{\otimes G}$ and consider $\id_A\otimes\sigma$, where $\sigma\colon G\curvearrowright \Z^{\otimes G}$ is the Bernouli shift action obtained by permuting the tensor factors.  However, it is far from clear what possible actions are induced at the level of invariants.  In view of the classification of automorphisms modulo approximately inner automorphisms (Corollary~\ref{cor:AutClassification}), this problem takes the following form (see the introduction to \cite{Barlak-Szabo17}).
    
\begin{question}\label{q:groupactions}
    Let $A$ be a classifiable $C^*$-algebra (that is, one satisfying the hypotheses of Theorem~\ref{Main}) and let $G$ be a discrete amenable group.
    When does a group action $G\curvearrowright \inv(A)$ lift to an action $G \curvearrowright A$?
\end{question}

State-of-the-art results for Kirchberg algebras can be found in \cite{Katsura08} (extending \cite{BKP,Spielberg07}). In particular, Katsura shows that actions of $\Zn{n}$ on $KT_u(A)$ lift to actions on $A$ when $A$ is a unital Kirchberg algebra satisfying the UCT.  However Question \ref{q:groupactions} is open for a general discrete amenable groups acting on Kirchberg algebras, and even for arbitrary finite groups.  Likewise, general results are lacking in the stably finite setting. Indeed, even for a simple unital AF-algebra $A$, it appears to be open whether any action $\Zn{n}\curvearrowright \mathrm{Aut}(K_0(A),K_0(A)_+,[1_A]_0)$ lifts to an action on $A$ (see \cite[Exercise 10.11.3]{Blackadar98}, and \cite{Zhang18}). Some constructions for simple A$\mathbb T$ algebras (akin to \cite{BKP} in the Kirchberg setting) have recently been given in \cite{Zhang23}.

For a finite group $G$, a strategy was given by Barlak and Szab\'o in \cite{Barlak-Szabo17}. In the presence of enough UHF-absorption, they are able to lift a homomorphism $G\rightarrow\mathrm{Aut}(A) / \overline{\mathrm{Inn}(A)}$ to a group action by means of a model action of $G$ on the UHF-algebra $M_{|G|^\infty}$.   Their method produces actions with the \emph{Rokhlin property}: a strong $C^*$-algebraic condition inspired by the non-commutative Rokhlin lemmas used by Connes and Ocneanu which goes back to \cite{Kishimoto77,Herman-Jones82} and was extensively developed by Izumi (\cite{Izumi-Duke-04,Izumi-Advances04}). 
\begin{definition}[{\cite[Definition 3.1]{Izumi-Duke-04}}]
    For a finite group $G$, an action $G \curvearrowright A$ has the Rokhlin property if there is an equivariant $^*$-homomorphism $C(G) \to A_\infty \cap A'$
\end{definition}

Barlak and Szab\'o used their technique to answer Question \ref{q:groupactions} for those $M_{|G|^\infty}$-stable $C^*$-algebras where there is a classification of automorphisms up to approximate unitary equivalence (\cite[Corollaries 2.8 and 2.13]{Barlak-Szabo17}). Combining their work with Theorem \ref{algebra-classification-total} yields the following result.

\begin{corollary}\label{cor:groupactions}
    If $A$ is a unital simple separable nuclear $C^*$-algebra satisfying the UCT, $G$ is a finite group, and $A \otimes M_{|G|^\infty} \cong A$, then every group action $G \curvearrowright \inv(A)$ lifts to an action $G \curvearrowright A$ with the Rokhlin property.
    
    In particular, Question~\ref{q:groupactions} has a positive answer when $A$ is $Q$-stable and $G$ is finite.
\end{corollary}
\begin{proof}
Fix an action $\Phi\colon G\curvearrowright \inv(A)$. By Theorem \ref{algebra-classification-total}, there exist automorphisms $(\theta_g)_{g\in G}$ of $A$ such that $\inv(\theta_g)=\Phi_g$ for each $g\in G$.  By the uniqueness in Theorem \ref{algebra-classification-total},  $\theta_g\circ \theta_h$ is approximately unitarily equivalent to $\theta_{gh}$, and so \cite[Theorem 2.3]{Barlak-Szabo17} provides the required Rokhlin action.
\end{proof}

The Rokhlin property is a strong condition giving rise to a lot of rigidity. Indeed, in \cite[Section 3]{Izumi-Advances04}, Izumi analyzes  $K$-theoretic restrictions imposed when an action has the Rokhlin property. Moreover he shows how to use the Rokhlin property to run an Evans--Kishimoto type intertwining argument (\cite{Evans-Kishimoto97}) to obtain uniqueness theorems whenever one has a classification of morphisms (\cite[Theorem 3.5]{Izumi-Duke-04}).  Combining Izumi's result with Theorem \ref{algebra-classification-total} gives the uniqueness counterpart to Corollary \ref{cor:groupactions}.  We thank Gabor Szab\'o for bringing this to our attention, through seminar talks in 2020.

\begin{corollary}
    Let $A$ be a unital simple separable nuclear $\mathcal Z$-stable $C^*$-algebra satisfying the UCT and let $G$ be a finite group.  Two actions $G \curvearrowright A$ with the Rokhlin property are conjugate via an automorphism of $A$ if and only if the induced actions $G \curvearrowright \inv(A)$ are conjugate via an automorphism of $\inv(A)$.
\end{corollary}
\begin{proof}
    Suppose that $\alpha,\beta\colon G\curvearrowright A$ are Rokhlin actions such that there is an automorphism $\Theta$ of $\inv(A)$ with $\Theta\circ\inv(\alpha_g)\circ\Theta^{-1}=\inv(\beta_g)$ for all $g\in G$.  By Theorem \ref{algebra-classification-total}, there exists an automorphism $\theta$ of $A$ inducing $\Theta$.  Another application of Theorem \ref{algebra-classification-total} gives that $\theta\circ\alpha_g\circ\theta^{-1}\approx_u\beta_g$ for all $g\in G$. The result then follows from \cite[Theorem 3.5]{Izumi-Duke-04}.    
\end{proof}

\appendix

\section{Total $K$-theory}\label{appendix:totalappendix}

In this appendix we collect some technical facts regarding total $K$-theory that are not easily available in the literature.

\subsection{Two definitions of $\Zn{n}$-coefficients}

Our first objective is to give a direct proof of the following proposition which relates the two pictures of total $K$-theory we use in this paper.   To the best of our knowledge, the only proof in the literature is found in the recent paper \cite[Theorem 8.1]{Kaminker-Schochet19}, where it is derived from a Spanier--Whitehead duality theory developed therein.

\begin{proposition}\label{p:Zntwodef}
There exists a natural isomorphism between the two functors $K_0(\,\cdot\, ; \Zn{n}) \coloneqq K_1(\mathbb I_n \otimes \, \cdot \, )$ and $KK(\mathbb I_n,\, \cdot \,)$ (defined on the category of separable $C^*$-algebras).  
\end{proposition}

Before proving the proposition, we isolate some notation. Here, and throughout this section, all tensor products are spatial.  For separable $C^*$-algebras $C$, $D$, and $E$ there is a natural map
\begin{equation}
    \mathcal T^{(C, D)}_E \colon KK(C, D) \rightarrow KK(E \otimes C , E \otimes D),
\end{equation}
defined by
\begin{equation}\label{tensorD}
    \mathcal T^{(C, D)}_E\big([\phi, \psi]_{KK(C, D)}\big) = [\mathrm{id}_E \otimes \phi, \mathrm{id}_E \otimes \psi]_{KK(E \otimes C, E \otimes D)}
\end{equation}
for a Cuntz pair $(\phi, \psi) \colon C \rightrightarrows \mathcal M(D \otimes \mathcal K) \rhd  D \otimes \mathcal K$. The following standard lemma is a $KK$-theoretic analogue of the fact that that $\theta\otimes \id_C$ and $\id_E\otimes \phi$ commute for $^*$-homomorphisms $\theta\colon E\to F$ and $\phi\colon C\to D$.  To see the analogy, consider the effect of the diagram \eqref{lem:app-tensor-diag} on $[\phi]_{KK(C,D)}$. This can be found in the Fredholm module picture as \cite[Proposition 17.8.6]{Blackadar98}, though we note that there is a typo in that statement (the equality there should be of maps $KK(A,B)\to KK(A\hat\otimes D_1,B\hat\otimes D_2)$).

\begin{lemma}\label{lem:app-tensor}
    Let $C$, $D$, $E$ and $F$ be separable $C^*$-algebras and let $\theta \colon E \rightarrow F$ be a $^*$-homomorphism.  Then the diagram
    \begin{equation}\label{lem:app-tensor-diag}
    \begin{tikzcd}[row sep = 8ex, column sep = 18ex]
        KK(C, D) \arrow{r}{\mathcal T_E^{(C, D)}} \arrow{d}{\mathcal T_F^{(C, D)}} & KK(E \otimes C, E \otimes D) \arrow{d}{KK(E \otimes C, \theta \otimes \mathrm{id}_D)} \\
        KK(F \otimes C, F \otimes D) \arrow{r}{KK(\theta \otimes \mathrm{id}_C, F \otimes D)} & KK(E \otimes C, F \otimes D)
    \end{tikzcd}
    \end{equation}
    commutes.
\end{lemma}

\begin{proof}[Proof of Proposition~\ref{p:Zntwodef}]
It suffices to show $K_0(\mathbb I_n \otimes \, . \,) \cong KK^1(\mathbb I_n, \,\cdot \,)$ as the result follows by taking suspensions.  By Bott periodicity and the stability of $KK$-theory, it is enough to show 
\begin{equation}
    KK(\mathbb C, \mathbb I_n \otimes \,\cdot \,) \cong KK(\mathbb I_n, SM_n \otimes \,\cdot \,).
\end{equation}
This is what we will show.

Recall that for $n\geq 2$,
\begin{equation}
\mathbb I_n  \coloneqq \{ f \in C([0,1], M_n) : f(0) \in \mathbb C1_{M_n}, f(1) = 0\}.
\end{equation}
Let $\mu_n \colon SM_n \to \mathbb I_n$ and $\nu_n \colon \mathbb I_n \to \mathbb C$ denote the canonical $^*$-homo\-morphisms given by the inclusion and evaluation at 0, respectively.  Define 
\begin{equation}
\theta \colon \mathbb I_n \otimes \mathbb I_n \longrightarrow SM_n
\end{equation}
to be the $^*$-homomorphism given on elementary tensors by
\begin{equation}
\theta(f\otimes g) (t) \coloneqq \begin{cases}
f(1-2t) \nu_n(g),  &  t\in [0,\tfrac{1}{2}]; \\
 g(2t-1)\nu_n(f), &  t\in [\tfrac{1}{2}, 1]
\end{cases}
\end{equation}
for $f,g\in \mathbb I_n$ and $t\in [0,1]$.
Define a natural transformation 
\begin{equation}
    \Omega \colon KK(\mathbb C, \mathbb I_n \otimes \,\cdot\,) \rightarrow KK(\mathbb I_n, SM_n \otimes \,\cdot\,)
\end{equation}
by 
\begin{equation}\label{omega}
    \Omega_D \coloneqq KK(\mathbb I_n, \theta \otimes \mathrm{id}_D) \circ \mathcal T^{(\mathbb C, \mathbb I_n \otimes D)}_{\mathbb I_n}
\end{equation}
for a separable $C^*$-algebra $D$.  We will show $\Omega_D$ is an isomorphism.

The composition $\theta \circ (\id_{\mathbb I_n} \otimes \mu_n) \colon \mathbb I_n \otimes SM_n \to SM_n$ is given on elementary tensors by
\begin{equation}
 \theta \big((\id_{\mathbb I_n} \otimes \mu_n) (f \otimes g)\big)(t) = 
 \begin{cases}
 0, & t\in [0,\tfrac{1}{2}]; \\
 g(2t-1)  \nu_n(f), & t\in [\tfrac{1}{2}, 1]
\end{cases} 
\end{equation}
for $f \in \mathbb I_n$, $g\in SM_n$, and $t\in [0,1]$. It follows that $\theta \circ (\id_{\mathbb I_n} \otimes \mu_n)$ is homotopic to 
\begin{equation}
\nu_n \otimes \id_{SM_n} \colon \mathbb I_n \otimes SM_n \longrightarrow \mathbb C \otimes SM_n = SM_n.
\end{equation}
In particular, if $D$ is a separable $C^*$-algebra, then homotopy invariance of $KK$-theory implies
\begin{equation}\label{eq:app-homotopy}
    KK(\mathbb I_n, (\theta \circ (\mathrm{id}_{\mathbb I_n} \otimes \mu_n))\otimes\id_D) = KK(\mathbb I_n, \nu_n \otimes \mathrm{id}_{SM_n} \otimes \mathrm{id}_D).
\end{equation}

For any separable $C^*$-algebra $D$, we get\begin{equation}\label{eq:2coeffdefcompute}
\begin{array}{rcl}
&& \hspace*{-5em} 
\Omega_D \circ KK(\mathbb C, \mu_n \otimes \id_D) \\
&\stackrel{\eqref{omega}}{=}& KK(\mathbb I_n, \theta \otimes \mathrm{id}_D)
\circ \mathcal T^{(\mathbb C, \mathbb I_n \otimes D)}_{\mathbb I_n} \circ KK(\mathbb C, \mu_n \otimes \mathrm{id}_D)  \\
&=& KK(\mathbb I_n, \theta \otimes \mathrm{id}_D) \circ KK(\mathbb I_n, \mathrm{id}_{\mathbb I_n} \otimes \mu_n \otimes \mathrm{id}_D) \circ \mathcal T_{\mathbb I_n}^{(\mathbb C, SM_n \otimes D)}  \\
&=& KK(\mathbb I_n, (\theta \circ (\mathrm{id}_{\mathbb I_n} \otimes \mu_n)) \otimes \mathrm{id}_D) \circ \mathcal T_{\mathbb I_n}^{(\mathbb C, SM_n \otimes D) } \\
&\stackrel{\eqref{eq:app-homotopy}}{=}& KK(\mathbb I_n, \nu_n \otimes \mathrm{id}_{SM_n} \otimes \mathrm{id}_D)  \circ \mathcal T_{\mathbb I_n}^{(\mathbb C, SM_n \otimes D) } \\
&\stackrel{\text{Lem.~\ref{lem:app-tensor}}}{=}& KK(\mathrm{id}_{\mathbb C} \otimes \nu_n, \mathbb C \otimes SM_n \otimes D) \circ \mathcal T_{\mathbb C}^{(\mathbb C, SM_n \otimes D)} \\
&=& KK(\nu_n, SM_n \otimes D),
\end{array}
\end{equation}
using the naturality of $\mathcal T_{\mathbb I_n}^{(\mathbb C, \,\cdot\,)}$ for the second equality and the functoriality of $KK(\mathbb I_n, \,\cdot\,)$ for the third.
It follows that the diagram
\begin{equation}\label{eq:2coeffdef1}
\begin{tikzcd}[column sep = large]
    KK(\mathbb C, SM_n \otimes D) \arrow[equals]{r} \arrow{d}{\times n} & KK(\mathbb C, SM_n \otimes D) \arrow{d}{\times n} \\
    KK(\mathbb C, SM_n \otimes D) \arrow[equals]{r} \arrow{d}{KK(\mathbb C, \mu_n \otimes \mathrm{id}_D)} & KK(\mathbb C, SM_n \otimes D) \arrow{d}{KK(\nu_n SM_n \otimes D)} \\ KK(\mathbb C, \mathbb I_n \otimes D) \arrow{r}{\Omega_D} & KK(\mathbb I_n, SM_n \otimes D)
\end{tikzcd}
\end{equation}
commutes.

Similarly, $\theta \circ (\mu_n \otimes \id_{\mathbb I_n})$ is homotopic to
\begin{equation}
\sigma \otimes \nu_n \colon SM_n \otimes \mathbb I_n \to SM_n \otimes \mathbb C = SM_n,
\end{equation} 
where  $\sigma \colon SM_n \to SM_n$ is the $^*$-homomorphism $\sigma(f)(t) \coloneqq f(1-t)$ for $t\in [0,1]$ and $f\in SM_n$.  This yields
\begin{equation}\label{eq:app-homotopy2}
   KK(SM_n, (\theta \circ (\mu_n \otimes \mathrm{id}_{\mathbb I_n})) \otimes \mathrm{id}_D) = KK(SM_n, \sigma \otimes \nu_n \otimes \mathrm{id}_D)
\end{equation}
for every separable $C^*$-algebra $D$.
Further, note that $KK(SM_n, \sigma) = -\id_{KK(SM_n, SM_n)}$: this follows, for example, by using the UCT to identify the ring $KK(SM_n, SM_n)$ with $\mathrm{End}(K_1(SM_n))$ and noting that there is an isomorphism $K_1(SM_n) \rightarrow \mathbb Z$ given by sending a loop of unitaries to the winding number of the determinant.   Bott periodicity  then implies that tensoring with $\sigma$ induces a natural isomorphism in $KK$-theory; i.e.,
\begin{equation}\label{eq:app-eta}
    \eta^{(C, D)} \coloneqq KK(SM_n \otimes C, \sigma \otimes \mathrm{id}_D) \circ \mathcal T_{SM_n}^{(C, D)}
\end{equation}
is an isomorphism $KK(C, D) \xrightarrow\cong KK(SM_n \otimes C, SM_n \otimes D)$ for all $C^*$-algebras $C$ and $D$.

Fix a separable $C^*$-algebra $D$.  
We compute
\begin{equation}
\begin{array}{rcl}
&& \hspace*{-7em} KK(\mu_n, SM_n \otimes D) \circ \Omega_D \\
&\stackrel{\eqref{omega}}{=}& KK(\mu_n, SM_n \otimes D) \circ KK(\mathbb I_n, \theta \otimes \mathrm{id}_D) \circ \mathcal T_{\mathbb I_n}^{(\mathbb C, \mathbb I_n \otimes D)} \\
&=& KK( SM_n, \theta \otimes \mathrm{id}_D) \circ KK(\mu_n, \mathbb I_n \otimes \mathbb I_n \otimes D) \circ \mathcal T_{\mathbb I_n}^{(\mathbb C, \mathbb I_n \otimes D)} \\
&\stackrel{\text{Lem.~\ref{lem:app-tensor}}}{=}& KK(SM_n, \theta \otimes \mathrm{id}_D) \circ KK(SM_n, \mu_n \otimes \mathrm{id}_{\mathbb I_n} \otimes \mathrm{id}_D) \circ \mathcal T_{SM_n}^{(\mathbb C, \mathbb I_n \otimes D)} \\
&=& KK(SM_n, (\theta \circ (\mu_n \otimes \mathrm{id}_{\mathbb I_n})) \otimes \mathrm{id}_D) \circ \mathcal T_{SM_n}^{(\mathbb C, \mathbb I_n \otimes D)} \\
&\stackrel{\eqref{eq:app-homotopy2}}{=}& KK(SM_n, \sigma \otimes \nu_n \otimes \mathrm{id}_D) \circ \mathcal T_{SM_n}^{(\mathbb C, \mathbb I_n \otimes D)} \\
&=& KK(SM_n, \sigma \otimes \mathrm{id}_D) \circ KK(SM_n, \mathrm{id}_{SM_n} \otimes \nu_n \otimes \mathrm{id}_D) \circ T_{SM_n}^{(\mathbb C, \mathbb I_n \otimes D)} \\
&=& KK(SM_n, \sigma \otimes \mathrm{id}_D) \circ \mathcal T_{SM_n}^{(\mathbb C, D)} \circ KK(\mathbb C, \nu_n \otimes \mathrm{id}_D) \\
&\stackrel{\eqref{eq:app-eta}}{=}& \eta^{(\mathbb C, D)} \circ KK(\mathbb C, \nu_n \otimes \mathrm{id}_D),
\end{array}
\end{equation}
where the second equality follows from bifunctoriality of
$KK(\,\cdot\,,\,\cdot\,)$, the fourth and sixth equality follow from the functoriality of $KK(SM_n, \,\cdot\,)$, and the seventh equality follows from the naturality of $\mathcal T_{SM_n}^{(\mathbb C, \,\cdot\,)}$.  This implies
\begin{equation}
\label{eq:2coeffdef2}
\begin{tikzcd}[column sep = large, row sep = large]
    KK(\mathbb C, \mathbb I_n \otimes D) \arrow{r}{\Omega_D} \arrow{d}{KK(\mathbb C, \nu_n \otimes \mathrm{id}_D)} & KK(\mathbb I_n, SM_n \otimes D) \arrow{d}{KK(\mu_n, SM_n \otimes D)} \\
    KK(\mathbb C, D) \arrow{r}[swap]{\eta^{(\mathbb C, D)}}{\cong} \arrow{d}{\times n} & KK(SM_n, SM_n \otimes D) \arrow{d}{\times n} \\
    KK(\mathbb C, D) \arrow{r}[swap]{\eta^{(\mathbb C, D)}}{\cong} & KK(SM_n, SM_n \otimes D)
\end{tikzcd}
\end{equation}
commutes.

When concatenating the left columns of \eqref{eq:2coeffdef1} and \eqref{eq:2coeffdef2}, this becomes part of the six-term exact sequence \eqref{eq:bockstein-2} for the Bockstein maps, and thus this concatenation is exact. Similarly, concatenating the right columns gives an exact sequence, using the Puppe exact sequence in $KK(\,\cdot\,, SM_n \otimes D)$ applied to the diagonal inclusion $\mathbb C \rightarrow M_n$ (see \cite[Theorem~19.4.3]{Blackadar06}).  This follows by considering the six-term exact sequences in $KK(\,\cdot\,,SM_n \otimes D)$ corresponding to the rows of
\begin{equation}
    \begin{tikzcd} 0 \arrow{r} & SM_n \arrow{r} \arrow[equals]{d} & \mathbb I_n \arrow{r} \arrow{d} & \mathbb C \arrow{r} \arrow{d} & 0\phantom{.} \\ 0 \arrow{r} & SM_n \arrow{r} & CM_n \arrow{r} & M_n \arrow{r} & 0.
    \end{tikzcd}
\end{equation}
Using that the boundary map in the second row in an isomorphism, since the cone is contractible, we may identify the boundary map in the top row with multiplication by $n$. 

By the five lemma, it follows that $\Omega_D$ is an isomorphism.
\end{proof}

\subsection{Two definitions of total $K$-theory}\label{appendix.a.2}

In this section, we prove Proposition~\ref{prop:TotalKtheoryAgrees}. This amounts to showing that for any two sets of Bockstein operations, the corresponding maps are integer multiples of each other. To that end, 
we consider an arbitrary collection of natural transformations
\begin{equation}\label{eq:Lambda1}
  \begin{tikzcd}
    K_i(\,\cdot\,)\ar[r, rightarrow, "\mu'^{(n)}_i"] & K_i(\, \cdot \, ; \Zn{n}) \ar[r, rightarrow, "\nu'^{(n)}_i"] & K_{1-i}(\,\cdot\,)
  \end{tikzcd}
\end{equation}
and
\begin{equation}\label{eq:Lambda2}
  \begin{tikzcd}
    K_i (\, \cdot\, ; \Zn{n}) \ar[r, rightarrow, "\kappa'^{(nm,n)}_i"]
    &
    K_i (\, \cdot\, ; \Zn{nm}) \ar[r, rightarrow, "\kappa'^{(n,nm)}_i"] & K_i(\, \cdot\, ; \Zn{n})
  \end{tikzcd}
\end{equation} 
indexed by $i = 0, 1$ and $n, m \geq 2$.

\begin{remark}
If natural transformations as above are only defined on the functors restricted to the category of separable $C^*$-algebras or $\sigma$-unital $C^*$-algebras, then they can canonically be extended to the functors defined on the category of all $C^*$-algebras since the functors $K_i$ and $K_i(\,\cdot\, ; \Zn{n})$ preserve inductive limits.
Hence we may assume, for example, that the Bockstein operations in \cite{Dadarlat-Loring96} (which are only defined there for $\sigma$-unital $C^*$-algebras) are defined on all $C^*$-algebras.
\end{remark}

\begin{lemma}
\label{lem:BocksteinYoneda}
Given a collection of natural transformations as in \eqref{eq:Lambda1} and \eqref{eq:Lambda2}, suppose that 
$k\mu'^{(n)}_i$, $k \nu'^{(n)}_i$, $k \kappa'^{(nm,n)}_i$, and $k \kappa'^{(n,nm)}_i$
are non-zero for all $i=0,1$, $n,m\geq 2$, and $k=1,\dots, n-1$.
\begin{enumerate}
\item
\label{BocksteinYoneda.1}
Each map $\mu'^{(n)}_i$, $\nu'^{(n)}_i$, $\kappa'^{(nm,n)}_i$, and $\kappa'^{(n,nm)}_i$ and its corresponding Bockstein map $\mu^{(n)}_i$, $\nu^{(n)}_i$, $\kappa^{(nm,n)}_i$, and $\kappa^{(n,nm)}_i$ are integer multiples of each other.
\item 
\label{BocksteinYoneda.2}
For any $C^*$-algebras $D$ and $E$, a collection of homomorphisms $\alpha_i \colon K_i(D) \to K_i(E)$ and $\alpha_i^{(n)} \colon K_i(D;\Zn{n}) \to K_i(E;\Zn{n})$ forms a $\Lambda$-homomorphism if and only if these homomorphisms intertwine all the natural transformations $\mu'^{(n)}_i$, $\nu'^{(n)}_i$, $\kappa'^{(nm,n)}_i$, and $\kappa'^{(n,nm)}_i$.
\item 
\label{BocksteinYoneda.3}
These natural transformations induce natural exact sequences as in \eqref{eq:bockstein-2} and \eqref{eq:bockstein-commute}.
\end{enumerate}
\end{lemma}

\begin{proof}
Part \ref{BocksteinYoneda.2} follows from \ref{BocksteinYoneda.1}. For example, to show that if $\alpha_0$ and $\alpha_1^{(n)}$ intertwine $\nu'^{(n)}_1$ then they intertwine $\nu^{(n)}_1$, we start with $\alpha_0 \circ \nu'^{(n)}_{1,D} = \nu'^{(n)}_{1,E} \circ \alpha_1^{(n)}$, and, from \ref{BocksteinYoneda.1}, obtain some $j\in\mathbb Z$ such that $\nu^{(n)}_1=j\nu'^{(n)}_1$.
Then
\begin{equation}
\alpha_0 \circ \nu^{(n)}_{1,D} = j \alpha_0 \circ \nu'^{(n)}_{1,D} = j \nu'^{(n)}_{1,E} \circ \alpha_1^{(n)} = \nu^{(n)}_{1,E} \circ \alpha_1^{(n)}.
\end{equation}
The other cases are all similar.

Similarly, part \ref{BocksteinYoneda.3} follows from \ref{BocksteinYoneda.1}.
For example, as above, if $\nu^{(n)}_1 = j \nu'^{(n)}_1$ for some $j\in \mathbb Z$, then for any $C^*$-algebra $D$  we get $\ker \nu'^{(n)}_{i,D} \subseteq \ker \nu^{(n)}_{i,D}$ and $\im \nu'^{(n)}_{i,D} \subseteq \im \nu^{(n)}_{i,D}$. By symmetry, we obtain equalities. The same type of argument works for the other Bockstein operations.

Now, to prove \ref{BocksteinYoneda.1}, we focus on showing that $\nu^{(n)}_1$ is an integer multiple of $\nu'^{(n)}_1$; a similar argument proves the other direction.  Indeed,  the hypothesis holds for $\nu^{(n)}_1$ in place of $\nu'^{(n)}_1$ by noting that
\begin{equation*}
\nu^{(n)}_{1,S\mathbb I_n} \colon K_1(S\mathbb I_n; \Zn{n}) \to K_0(S\mathbb I_n)\cong \Zn{n}
\end{equation*}
is an isomorphism by exactness of \eqref{eq:bockstein-2}, using $K_1(\mathbb I_n) \cong \Zn{n}$ and $K_0(\mathbb I_n) =0$.

By continuity of the functors $K_0$ and $K_1(\,\cdot \,;\Zn{n})$, it suffices to find $j\in \mathbb N$ such that $\nu^{(n)}_{1,D}=j\nu'^{(n)}_{1,D}$ when $D$ is a separable $C^*$-algebra. When restricted to this setting, by associativity of the Kasparov product, the functors $K_0$ and $K_1(\,\cdot \,;\Zn{n})$ factorize through the $KK$-category, in which we have natural isomorphisms $K_1(\,\cdot\, ; \Zn{n}) \cong KK(S\mathbb I_n, \,\cdot\,)$ given in Proposition~\ref{p:Zntwodef}.  (Recall that the $KK$-category has all separable $C^*$-algebras as objects and $KK(D,E)$ as the set of morphisms from $D$ to $E$.) For any separable $C^*$-algebra $D$, the functor $KK(D,\,\cdot \,)$ is a $\Hom$-functor in the $KK$-category, so by the Yoneda lemma, the group of natural transformations from $KK(D,\,\cdot \,)$ to $F$ is naturally isomorphic to $F(D)$ for any additive functor $F$ from the $KK$-category to the category of abelian groups. Hence the group of natural transformations from $K_1(\,\cdot \,; \Zn{n})$ to $K_0$ is naturally isomorphic to $K_0(S\mathbb I_n) \cong \Zn{n}$. In particular, a natural transformation $\eta \colon K_1(\, \cdot\, ; \Zn{n}) \rightarrow K_0$ generates the entire group of natural transformation between these functors if and only if $k\eta \neq 0$ for every $k=1,\dots, n-1$. This is true for $\nu'^{(n)}_1$ by assumption.
Hence $\nu^{(n)}_1$ is an integer multiple of $\nu'^{(n)}_1$.
\end{proof}

\begin{proof}[Proof of Proposition~\ref{prop:TotalKtheoryAgrees}]
After identifying $KK^i(\mathbb I_n, \,\cdot \,)$ with $K_i(\, \cdot\, ; \Zn{n})$ via natural isomorphisms (which exist by Proposition~\ref{p:Zntwodef}), the Bockstein operations in \cite{Dadarlat-Loring96} induce natural transformations satisfying the criterion of Lemma~\ref{lem:BocksteinYoneda}.  
\end{proof}

\subsection{A split exact sequence relating ordinary and total $K$-theory}\label{sec:bodigheimer}

We end this appendix by showing how to use B\"odigheimer's work (\cite{Bodigheimer79,Bodigheimer80}) to prove Proposition~\ref{TotalKExtend} by means of an unnatural splitting of an exact sequence relating total $K$-theory with ordinary $K$-theory.  To ease notation, in this section, we use $K_i(D)/nK_i(D)$ rather that the isomorphic group $K_i(D)\otimes \Zn{n}$ as it allows us to more easily describe maps between such groups.

Recall that for an abelian group $G$ and $n \geq 2$, we identify $\mathrm{Tor}(G,\Zn{n})$ with the $n$-torsion group, $\{g\in G: ng=0\}$.
In particular, for any $m,n$, we have a natural inclusion $\mathrm{Tor}(G,\Zn{n}) \to \mathrm{Tor}(G,\Zn{mn})$, as well as a natural map $\mathrm{Tor}(G,\Zn{mn}) \to \mathrm{Tor}(G,\Zn{n})$ given by multiplication by $m$.

Let $n\geq 2$, and let $n=p_1^{r_1}\cdots p_k^{r_k}$ be its prime factorization.
For any $C^*$-algebra $D$, $i =0, 1$, and $j=1,\dots, k$, using commutativity of \eqref{eq:bockstein-commute}, we get a commuting diagram
\begin{equation}
  \begin{tikzcd}[row sep = 7ex]
    K_i(D)/ p_j^{r_j}K_i(D)
    \ar[r, rightarrowtail, "\overline{\mu}_{i,D}^{(p_j^{r_j})}"]
    \ar[d, two heads, "\times n/p_j^{r_j}"]
    &
    K_i(D; \Zn{p_j^{r_j}})
    \ar[d,  "\kappa_{i,D}^{(n, p_j^{r_j})}"]
    \ar[r, twoheadrightarrow, " \overline{\nu}_{i,D}^{(p_j^{r_j})}"]
    &
    \mathrm{Tor}(K_{1-i}(D), \Zn{p_j^{r_j}})
    \ar[d, rightarrowtail] \\
    K_i(D)/nK_i(D)
    \ar[r, rightarrowtail, "\overline{\mu}_{i,D}^{(n)}"]
    &
    K_i(D; \Zn{n})
    \ar[r, twoheadrightarrow, "\overline{\nu}_{i,D}^{(n)}"]
    &
    \mathrm{Tor}(K_{1-i}(D), \Zn{n}).
  \end{tikzcd}
\end{equation}
where the rows are as in \eqref{eq:bockstein-new}. Taking the direct sum of the top row for $j=1,\dots, k$, we get a commutative diagram
\begin{equation}\label{eq:bockstein-primefactorise}
  \begin{tikzcd}[row sep =7ex]
    \bigoplus_j K_i(D)/p_j^{r_j}K_i(D)
    \ar[r, rightarrowtail, "\quad \bigoplus_j \overline{\mu}_{i,D}^{(p_j^{r_j})}"]
    \ar[d, "\cong"]
    &
    \bigoplus_j K_i(D; \Zn{p_j^{r_j}})
    \arrow{d}[yshift = 1ex]{\bigoplus_j \kappa_{i,D}^{(n, p_j^{r_j})}}[swap]{\cong} 
    \ar[r, twoheadrightarrow, "\bigoplus_j \overline{\nu}_{i,D}^{(p_j^{r_j})}"]
    &
    \bigoplus_j \mathrm{Tor}(K_{1-i}(D), \Zn{p_j^{r_j}})
    \ar[d, "\cong"] \\
    K_i(D)/nK_i(D)
    \ar[r, rightarrowtail, "\overline{\mu}_{i,D}^{(n)}"]
    &
    K_i(D; \Zn{n})
    \ar[r, twoheadrightarrow, "\overline{\nu}_{i,D}^{(n)}"]
    &
    \mathrm{Tor}(K_{1-i}(D), \Zn{n}),
\end{tikzcd}
\end{equation}
where the left and right vertical maps are isomorphisms by the Chinese remainder theorem ($\Zn{n} \cong \bigoplus_j \Zn{p_j^{r_j}}$) together with additivity of $\otimes $ and $\mathrm{Tor}$, and the middle vertical map is therefore an isomorphism by the five lemma.

\begin{lemma}[B\"odigheimer]\label{l:Bodigheimer}
For any $C^*$-algebra $D$ there are (unnatural) homomorphisms
\begin{equation}
\mathrm{Tor}(K_{1-i}(D), \Zn{n}) \xrightarrow{s_{i,D}^{(n)}} K_i(D; \Zn{n})   \xrightarrow{\iota_{i,D}^{(n)}} K_i(D)/nK_i(D)
\end{equation}
for $n\geq 2$ such that 
\begin{equation}\label{eq:Bodigheimersplit}
\overline{\nu}_{i,D}^{(n)} \circ s_{i,D}^{(n)}  = \id_{\mathrm{Tor}(K_{1-i}(D),\Zn{n})},\qquad \iota_{i,D}^{(n)} \circ \overline \mu_{i,D}^{(n)} = \id_{K_i(D)/nK_i(D)},
\end{equation}
and the following diagram commutes:
\begin{equation}\label{eq:splittingbockstein}
  \begin{tikzcd}[row sep = 7ex]
    \mathrm{Tor}(K_{1-i}(D), \Zn{n})
    \ar[r, "s_{i,D}^{(n)}"] \ar[d]
    &
    K_i(D; \Zn{n}) \ar[d, "\kappa_{i,D}^{(nm,n)}"] \ar[r, "\iota_{i,D}^{(n)}"]
    &
    K_i(D)/nK_i(D)\phantom{.}
    \ar[d, "\times m"] \\
    \mathrm{Tor}(K_{1-i}(D), \Zn{nm})
    \ar[r, "s_{i,D}^{(nm)}"]
    \ar[d, "\times m"]
    &
    K_i(D; \Zn{nm})\phantom{.}
    \ar[d, "\kappa_{i,D}^{(n,nm)}"]
    \ar[r, "\iota_{i,D}^{(nm)}"]
    &
    K_i(D)/nmK_i(D) \ar[d]  \\
    \mathrm{Tor}(K_{1-i}(D), \Zn{n})
    \ar[r, "s_{i,D}^{(n)}"] & K_i(D; \Zn{n}) \ar[r, "\iota_{i,D}^{(n)}"]
    &
    K_i(D)/nK_i(D).
  \end{tikzcd}
\end{equation}
\end{lemma}
\begin{proof}
We start with the maps $s_{i,D}^{(n)}$. 
Recall that for $n, m \geq 2$, the maps
\begin{equation}\label{eq:app-kappa}
\begin{split}
    &\kappa_{*, D}^{(nm, n)} \colon K_*(D ; \mathbb Z/n) \longrightarrow K_*(D ; \mathbb Z/nm) \text{ and} \\
    &\kappa_{*,D}^{(n, nm)} \colon K_*(D ; \mathbb Z/nm) \longrightarrow K_*(D; \mathbb Z/n)
\end{split}
\end{equation}
are induced by the
canonical inclusions $\mathbb I_n \rightarrow \mathbb I_{nm}$ and  $\mathbb I_{nm} \rightarrow \mathbb I_n \otimes M_m$, respectively.  The composition $\mathbb I_n \rightarrow \mathbb I_{nm} \rightarrow \mathbb I_n \otimes M_m$ is the diagonal embedding, and so $\kappa_{*,D}^{(n, nm)} \circ \kappa_{*, D}^{(nm, n)}$ is multiplication by $m$. The composition $\mathbb I_{nm} \rightarrow \mathbb I_n \otimes M_m \rightarrow \mathbb I_{nm}\otimes M_m$ is  homotopic to the diagonal embedding and thus $\kappa_{*, D}^{(nm, n)} \circ \kappa_{*,D}^{(n, nm)}$ also induces multiplication by $m$. To see this, note that this composition takes values in $\mathbb I_{nm^2}$. Every unitary in $M_{nm^2}$ induces by conjugation an automorphism on $\mathbb I_{nm^2}$ homotopic to the identity, so by composing the previous composition with a unitary flipping the two copies of $M_m$, we obtain the desired homotopy. 

Further, note that $nx = 0$ for all $x \in K_*(D; \mathbb Z/n)$.  It follows that for a prime $p$, the diagram
\begin{equation}
\begin{tikzcd}[row sep = 8ex, column sep = 7.5ex]
    \vdots \arrow[shift left]{d} & \vdots \arrow[shift left]{d} & \vdots \arrow[shift left]{d} \\
    K_i(D)/ p^k K_i(D) \arrow[tail]{r}{\bar\mu_{i, D}^{(p^k)}} \arrow[shift left]{u} \arrow[shift left]{d} & K_i(D ; \mathbb Z/p^k) \arrow[shift left]{u} \arrow[two heads]{r}{\bar \nu_{i, D}^{(p^k)}} \arrow[shift left]{d}{\kappa_{i, D}^{(p^{k+1}, p^k)}} & \mathrm{Tor}(K_{1-i}(D), \mathbb Z/p^k) \arrow[shift left]{u} \arrow[shift left,"\times p"]{d} \\
    K_i(D)/p^{k+1} K_i(D) \arrow[tail]{r}{\bar\mu_{i, D}^{(p^{k+1})}} \arrow[shift left,"\times p"]{u} \arrow[shift left]{d}& K_i(D ; \mathbb Z/p^{k+1}) \arrow[shift left]{u}{\kappa_{i, D}^{(p^k, p^{k+1})}} \arrow[two heads]{r}{\bar \nu_{i, D}^{(p^{k+1})}} \arrow[shift left]{d} & \mathrm{Tor}(K_{1-i}(D), \mathbb Z/p^{k+1}) \arrow[shift left]{u} \arrow[shift left]{d} \\
    \vdots \arrow[shift left]{u} & \vdots \arrow[shift left]{u} & \vdots \arrow[shift left]{u}
\end{tikzcd}
\end{equation}
satisfies the hypotheses of the (purely algebraic) lemma in \cite[Section~2]{Bodigheimer80} for $i = 0, 1$ (note that $G[n] \coloneqq\mathrm{Tor}(G, \mathbb Z/n)$ in \cite{Bodigheimer80}), where the vertical maps on the left and right are induced by the canonical maps between $\mathbb Z/p^k$ and $\mathbb Z/p^{k+1}$. This lemma implies that the maps $s_{i,D}^{(n)}$ exist satisfying \eqref{eq:Bodigheimersplit} and such that \eqref{eq:splittingbockstein} commutes when $n$ and $m$ above are powers of the same prime $p$; we therefore fix such $s_{i,D}^{(p^r)}$ for all primes $p$ and $r \geq 1$. 

We now use the identifications in the right-hand square of \eqref{eq:bockstein-primefactorise} to define maps $s_{i,D}^{(n)}$ for general $n$.
For $n\geq 2$ arbitrary, let $n=p_1^{r_1} \cdots p_k^{r_k}$ be its prime factorization.
Then we define $s_{i,D}^{(n)}\colon \mathrm{Tor}(K_{1-i}(D);\Zn{n}) \to K_i(D;\Zn{n})$ such that the following commutes
\begin{equation}
\label{eq:bodigheimer-s}
  \begin{tikzcd}[row sep =8ex, column sep=8ex]
    \bigoplus_j \mathrm{Tor}(K_{1-i}(D), \Zn{p_j^{r_j}})
    \ar[d, "\cong"] \ar[r,"\bigoplus_j s_{i,D}^{(p_j^{r_j})}"] 
    &
    \bigoplus_j K_i(D; \Zn{p_j^{r_j}})
    \arrow{d}[yshift = 1ex]{\bigoplus_j \kappa_{i,D}^{(n, p_j^{r_j})}}[swap]{\cong}  \\
    \mathrm{Tor}(K_{1-i}(D), \Zn{n})
    \ar[r,"s_{i,D}^{(n)}"]
    &
    K_i(D; \Zn{n}).
\end{tikzcd}
\end{equation}

Since the top map, $\bigoplus_j s_{i,D}^{(p_j^{r_j})}$, is a right splitting for the top row of \eqref{eq:bockstein-primefactorise}, it follows that the $s_{i,D}^{(n)}$ is a right splitting for the bottom row, i.e., $\overline{\nu}_{i,D}^{(n)}\circ s_{i,D}^{(n)}  = \id_{\mathrm{Tor}(K_{1-i}(D),\Zn{n})}$.

To check commutativity of the left side of \eqref{eq:splittingbockstein}, we consider $m,n\geq 2$.
Let $n=p_1^{r_1} \cdots p_k^{r_k}$ and $m=p_1^{t_1}\cdots p_k^{t_k}$ be their prime factorizations, allowing some exponents to possibly be zero.
For convenience, we take $K_i(D;\Zn{1})$ to be $0$ by convention; then \eqref{eq:bockstein-primefactorise} and the definition of $s_{i,D}^{(n)}$ from \eqref{eq:bodigheimer-s} are still correct with some $r_j$ equal to $0$.

Since $\kappa_{i,D}^{(nm,p_j^{r_j+t_j})}\circ \kappa_{i,D}^{(p_j^{r_j+t_j},p_j^{r_j})}=\kappa_{i,D}^{(nm,p_j^{r_j})}=\kappa_{i,D}^{(nm,n)}\circ\kappa_{i,D}^{(n,p_j^{r_j})}$ (by the definition of these maps), the following diagram commutes:
\begin{equation}
  \begin{tikzcd}[row sep =8ex, column sep=13ex]
    \bigoplus_j K_{i}(D; \Zn{p_j^{r_j}})
    \arrow{d}[yshift = 1ex]{\bigoplus_j \kappa_{i,D}^{(n, p_j^{r_j})}}[swap]{\cong} 
    \ar[r,"\bigoplus_j \kappa_{i,D}^{(p_j^{r_j+t_j},p_j^{r_j})}"]
    &
    \bigoplus_j K_{i}(D; \Zn{p_j^{r_j+t_j}})
    \arrow{d}[yshift = 1ex]{\bigoplus_j \kappa_{i,D}^{(nm, p_j^{r_j+t_j})}}[swap]{\cong}  \\
    K_{i}(D; \Zn{n})
    \ar[r,"\kappa_{i,D}^{(nm,n)}"]
    &
    K_{i}(D; \Zn{nm}).
\end{tikzcd}
\end{equation}
From this and \eqref{eq:bodigheimer-s}, one can deduce commutativity of the top-left square of \eqref{eq:splittingbockstein} for general $n$ and $m$ from B\"odigheimer's commutativity of this square for the case that $n$ and $m$ are powers of the same prime.

Fix $j\in\{1,\dots,k\}$.
Then we have
\begin{equation}
\label{eq:kappafact2}
\kappa^{(n,nm)}_{i,D}\circ\kappa^{(nm,p_j^{r_j+t_j})}_{i,D}=\frac{m}{p_j^{t_j}} \kappa_{i,D}^{(n,p_j^{r_j})}\circ
\kappa_{i,D}^{(p_j^{r_j},p_j^{r_j+t_j})}
\end{equation}
since the left side comes from the inclusion $\mathbb I_{p_j^{r_j+t_j}} \to \mathbb I_n \otimes M_m$ given by $f \mapsto f \otimes 1_{M_{mn/p_j^{r_j+t_j}}}$, while the 
right side (ignoring the factor $m/p_j^{t_j}$) comes from the inclusion $\mathbb I_{p_j^{r_j+t_j}} \to \mathbb I_n \otimes M_{p_j^{t_j}}$ given by $f \mapsto f \otimes 1_{M_{n/p_j^{r_j}}}$.


Therefore, for an element $x$ of the direct summand $\mathrm{Tor}(K_i(D),\Zn{p_j^{r_j+t_j}})$ of $\mathrm{Tor}(K_i(D),\Zn{nm})$, we have
\begin{equation}
\begin{array}{rcl}
\kappa^{(n,nm)}_{i,D}\big(s^{(nm)}_{i,D}(x)\big)
&\stackrel{\eqref{eq:bodigheimer-s}}=& \kappa^{(n,nm)}_{i,D}\Big( \kappa^{(nm,p_j^{r_j+t_j})}_{i,D}\big(s^{(p_j^{r_j+t_j})}_{i,D}(x)\big)\Big) \\
&\stackrel{\eqref{eq:kappafact2}}=& \frac{m}{\,p_j^{t_j}} \Big(\kappa^{(n,p_j^{r_j})}_{i,D}\Big(\kappa^{(p_j^{r_j},p_j^{r_j+t_j})}_{i,D}\big( s^{(p_j^{r_j+t_j})}_{i,D}(x)\big)\Big)\Big) \\
&=& \frac{m}{\,p_j^{t_j}} \kappa^{(n,p_j^{r_j})}_{i,D}\Big( s^{(p_j^{r_j})}_{i,D}(p_j^{t_j}x)\Big) \\
&=& \kappa^{(n,p_j^{r_j})}_{i,D}\Big( s^{(p_j^{r_j})}_{i,D}(m x)\Big) \\
&\stackrel{\eqref{eq:bodigheimer-s}}=& s^{(n)}_{i,D}(mx),
\end{array}
\end{equation}
where the third equation comes from B\"odigheimer's commutativity of the bottom-left square of \eqref{eq:splittingbockstein} in the case that $n$ and $m$ are powers of $p_j$.
Since
\begin{equation}
\mathrm{Tor}(K_i(D),\Zn{nm})=\bigoplus_j
\mathrm{Tor}(K_i(D),\Zn{p_j^{r_j+t_j}})
\end{equation}
and the above shows commutativity of the bottom-left square of \eqref{eq:splittingbockstein} on each of these summands, it follows that the bottom-left square of \eqref{eq:splittingbockstein} commutes.

The splitting lemma provides the maps $\iota_{i,D}^{(n)}$.  A version that suffices for our purposes says that given a commutative diagram of abelian groups
  \begin{equation}
    \begin{tikzcd}
      0 \ar[r] & G_1 \ar[r] \ar[d] & H_1 \ar[r] \ar[d] & L_1 \ar[r]\ar[d] & 0 \\
      0 \ar[r] & G_2 \ar[r]        & H_2 \ar[r]        & L_2 \ar[r]       & 0
    \end{tikzcd}
  \end{equation}
  with exact rows, if there is a  splitting for the quotient maps in each row that commutes with the downward maps, then the corresponding  splittings for the inclusion maps also commute with the downward maps.  
  
  This last claim follows from the construction of the  splitting for the inclusion map in the proof of the splitting lemma. In more detail, if  $0 \rightarrow G \xrightarrow i H \xrightarrow p L \rightarrow 0$ is an exact sequence of abelian groups and $q \colon L \rightarrow H$ is a splitting of $p$, then a splitting $j \colon H \rightarrow G$ of $i$ is given explicitly as follows: the map $\mathrm{id}_H - q\circ p \colon H \rightarrow H$ composed with $p$ vanishes. Since $\ker p = \im i$ and $i$ is injective, there is a unique homomorphism $j \colon H \rightarrow G$ such that $i\circ j = \mathrm{id}_H - q\circ p$. Since $i$ is injective and $i\circ j \circ i = i$, it follows that $j\circ i = \mathrm{id}_G$. 
\end{proof}

\begin{proof}[Proof of Proposition \ref{TotalKExtend}]
Fix homomorphisms $\iota_{i,A}^{(n)}$ and $s_{i,B}^{(n)}$ as in Lemma~\ref{l:Bodigheimer}. 
Given $\alpha_\ast \colon K_\ast(A) \to K_\ast(B)$ we define group homomorphisms $\alpha_i^{(n)}  \colon K_i(A; \Zn{n}) \to K_i(B; \Zn{n})$ by
\begin{equation}
\alpha_i^{(n)} \coloneqq \overline{\mu}_{i,B}^{(n)} \circ(\alpha_i \otimes  \id_{\Zn{n}})\circ \iota_{i,A}^{(n)} + s_{i,B}^{(n)}\circ \mathrm{Tor}(\alpha_{1-i}, \Zn{n})\circ \overline{\nu}_{i,A}^{(n)}.
\end{equation}
It follows from the construction that the collection of maps $\alpha_i$ and $\alpha_i^{(n)}$ for $i=0,1$ and $n\geq 2$ intertwines the natural transformations $\mu_i$ and $\nu_i$, and by Lemma~\ref{l:Bodigheimer}, they also intertwine the $\kappa_{i}$. 
\end{proof}

\section{$KK$-theory outside the \texorpdfstring{$\sigma$}{sigma}-unital setting}
\label{sec:kkappendix}

In this appendix we give some details regarding the inductive limit description of $KK$-theory in the case when the second variable is not $\sigma$-unital, extending some standard results from the $\sigma$-unital case to the general framework.

\subsection{Cuntz's picture of $KK$-theory}

We will make use of Cuntz's picture of $KK$-theory from \cite{Cuntz87}. This (like its precursor in \cite{Cuntz83a}) works smoothly for non-$\sigma$-unital algebras in the second variable and allows for a very clean description of $KK(A,\theta)$ for a $^*$-homomorphism $\theta$. 

\begin{definition}[{Cuntz, \cite{Cuntz87}}]
\label{defn:KKc}
Let $A$ and $I$ be $C^*$-algebras with $A$ separable.
Let $qA$ be the kernel of the map $\id_A * \id_A\colon A * A \to A$ (where $*$ denotes the full free product).
Define
\begin{equation}
KK_c(A,I)\coloneqq [qA,I\otimes \mathcal K],
\end{equation}
(i.e., homotopy classes of $^*$-homomorphisms from $qA$ to $I \otimes \mathcal K$).
Write $[\rho]_{KK_c(A,I)} \in KK_c(A,I)$ for the homotopy class of $\rho\colon qA \to I \otimes \mathcal K$.

Given another $C^*$-algebra $J$ and a $^*$-homomorphism $\theta\colon I\to J$, the induced map
$KK_c(A,\theta)\colon KK_c(A,I) \to KK_c(A,J)$ is given by
\begin{equation}
\label{eq:KKcTheta}
KK_c(A,\theta)([\rho]_{KK_c(A,I)}) \coloneqq [(\theta\otimes \mathrm{id}_{\mathcal K})\circ\rho]_{KK_c(A,J)}.
\end{equation}
\end{definition}

The above definition of $KK_c(A,\theta)$ makes $KK_c(A,\,\cdot\,)$ a covariant functor, and indeed $KK_c(\,\cdot\, ,\,\cdot\,)$ is a bifunctor (though we do not need an explicit description of functoriality in the first variable). Cuntz used Higson's characterization of $KK$-theory for separable $C^*$-algebras (\cite{Higson87}) to show that $KK_c(\,\cdot\, ,\, \cdot\,)$ gives another realization of $KK$-theory in the separable setting. This was extended to the case when the second variable is $\sigma$-unital in \cite{Jensen-Thomsen91}, where the isomorphism between the Cuntz picture and Cuntz--Thomsen picture is described explicitly (the proof reveals that the isomorphism is natural) as in the following proposition. To set this up, note that given a Cuntz pair $(\phi,\psi)\colon A \rightrightarrows E \rhd I\otimes \mathcal K$, we have $(\phi * \psi)(qA) \subseteq I\otimes \mathcal K$ (see \cite[Lemma 5.14]{Jensen-Thomsen91}); we write
\begin{equation}
\label{eq:KKcClass}
[\phi,\psi]_{KK_c(A,I)} \coloneqq \big[(\phi * \psi)|_{qA}^{I\otimes \mathcal K}\big]_{KK_c(A,I)}.
\end{equation}

\begin{proposition}[cf.\ {\cite[Theorem 5.2.4]{Jensen-Thomsen91}}]
\label{prop:KKequalsKKc}
Let $A$ and $I$ be $C^*$-algebras with $A$ separable and $I$ $\sigma$-unital.
Then 
$[\phi,\psi]_{KK(A,I)} \mapsto [\phi,\psi]_{KK_c(A,I)}$ defines an isomorphism 
$KK(A,I)\cong KK_c(A,I)$
that is natural in both variables.
\end{proposition}

The clean description of $KK_c(A,\theta)$ above facilitates the following computation.  This calculation was performed in Kasparov's original Fredholm module picture, in the case of $I_2=E_1=E_2$ separable, in \cite[Proposition 2.1]{Schafhauser18}. Comparing the computations demonstrates the power of Cuntz's picture.

\begin{proposition}\label{appendix.corestrict}
Let $A$ be a separable $C^*$-algebra and let $\theta\colon I_1\to I_2$ be a $^*$-homomorphism between $C^*$-algebras.  Suppose that $E_i$ is a $C^*$-algebra containing $I_i$ as an ideal for $i=1,2$ and that $\theta$ extends to a $^*$-homomorphism $\bar\theta\colon E_1 \to E_2$. 
Given a Cuntz pair $(\phi,\psi)\colon A \rightrightarrows E_1 \rhd I_1$, we have
\begin{equation} 
\label{eq:appendix.corestrict}
KK_c(A,\theta)([\phi,\psi]_{KK_c(A,I_1)}) = [\bar\theta \circ \phi,\bar\theta\circ \psi]_{KK_c(A,I_2)}. \end{equation}
In particular, if $I_1$ and $I_2$ are $\sigma$-unital, then
\begin{equation}
KK(A,\theta)([\phi,\psi]_{KK(A,I_1)}) = [\bar\theta \circ \phi,\bar\theta\circ \psi]_{KK(A,I_2)}. \end{equation}
\end{proposition}

\begin{proof}
To shorten notation, we may assume (by stabilizing everything) that $I_1,I_2$ are stable. One may verify (using \cite[Lemma 5.1.2]{Jensen-Thomsen91}) that $((\bar\theta\circ\phi) * (\bar\theta\circ\psi))|_{qA} = \theta\circ(\phi*\psi)|_{qA}$.
We use this for the third equality below:
\begin{equation}\begin{array}{rcl}
&&\hspace*{-7em} KK_c(A,\theta)\big([\phi,\psi]_{KK_c(A,I_1)}\big) \\
&\stackrel{\eqref{eq:KKcClass}}=& KK_c(A,\theta)\big(\big[(\phi * \psi)|_{qA}^{I_1}\big]_{KK_c(A,I_1)}\big) \\
&\stackrel{\eqref{eq:KKcTheta}}=& \big[\theta \circ \big((\phi * \psi)|_{qA}^{I_1}\big)\big]_{KK_c(A,I_2)} \\
&=& \big[\big((\bar\theta\circ \phi)*(\bar\theta\circ \psi)\big)|_{qA}^{I_2}\big]_{KK_c(A,I_2)} \\
&=& [\bar\theta\circ\phi,\bar\theta\circ\psi]_{KK_c(A,I_2)}.
\end{array} \end{equation}

The last line of the proposition follows from Proposition~\ref{prop:KKequalsKKc}
\end{proof}

\subsection{The inductive limit picture of $KK$-theory}

Recall that  in Definition~\ref{KK-inductivelimit} for $C^*$-algebras $A$ and $I$ with $A$ separable, we defined
\begin{equation}\label{KK-inductivelimit-equation-appendix} KK(A,I) \coloneqq \varinjlim\limits_{I_0\text{ sep.}} KK(A,I_0),\end{equation}
with the limit taken over all separable $C^*$-subalgebras of $I$, ordered by inclusion.  We now justify the interpretation of the class of an $(A,I)$-Cuntz pair in this picture.

\begin{proposition}
\label{prop:KKwelldefined}
Let $A$ and $I$ be $C^*$-algebras with $A$ separable.
Given a Cuntz pair $(\phi,\psi)\colon A \rightrightarrows E \rhd I$, the class $[\phi,\psi]_{KK(A,I)}$ in Definition~\ref{KK-inductivelimit} is well-defined.
That is,
\begin{enumerate}
\item
\label{KKwelldefined.1}
 there exist separable $C^*$-subalgebras $E_0\subseteq E$ and $I_0\subseteq I$ such that $I_0\lhd E_0$, $\phi(A)\cup\psi(A) \subseteq E_0$, and $(\phi-\psi)(A)\subseteq I_0$, and
\item
\label{KKwelldefined.2}
the class of $[\phi|^{E_0},\psi|^{E_0}]_{KK(A,I_0)}$ in $\varinjlim\limits_{I_0\text{ \rm sep.}} KK(A,I_0)$ does not depend on the choice of $I_0$ and $E_0$ in \ref{KKwelldefined.1}.
\end{enumerate}
\end{proposition}

\begin{proof}
\ref{KKwelldefined.1}:
Let $E_0$ denote the $C^*$-algebra generated by $\phi(A) \cup \psi(A)$ and let $I_0$ be the ideal generated by $(\phi - \psi)(A)$ in $E_0$.

\ref{KKwelldefined.2}:
Suppose that $I_1\subseteq I$ and $E_1\subseteq E$ are separable $C^*$-subalgebras that also satisfy the conditions in \ref{KKwelldefined.1}.  Let $E_2$ be the $C^*$-algebra generated by $E_0 \cup E_1$ and let $I_2$ be the ideal generated by $I_0 \cup I_1$ in $E_2$.  These are separable $C^*$-subalgebras of $E$ and $I$, respectively.

Using Proposition~\ref{appendix.corestrict} twice (with $\theta$ equal to the inclusions of $I_0$ and $I_1$ into $I_2$ in turn), we have
\begin{equation}
\begin{split} 
KK(A,\iota_{I_0\subseteq I_2})\big([\phi|^{E_0},\psi|^{E_0}]_{KK(A,I_0)}\big)
&= [\phi|^{E_2},\psi|^{E_2}]_{KK(A,I_2)}\\
&= KK(A,\iota_{I_1\subseteq I_2})\big([\phi|^{E_1},\psi|^{E_1}]_{KK(A,I_1)}\big).\end{split}
\end{equation}
Therefore, $[\phi|^{E_0},\psi|^{E_0}]_{KK(A,I_0)}$ and $[\phi|^{E_1},\psi|^{E_1}]_{KK(A,I_1)}$ represent the same class in $\varinjlim\limits_{I_0\text{ sep.}} KK(A,I_0)$.
\end{proof}

\begin{proposition}
\label{prop:KKbifunc}
$KK(\,\cdot\,,\,\cdot\,)$ (as in Definition~\ref{KK-inductivelimit}) is a bifunctor.
\end{proposition}

\begin{proof}
Functoriality in the first variable is entirely straightforward.

For the second variable, given $C^*$-algebras $A$, $I$, and $J$ such that $A$ is separable, and a $^*$-homomorphism $\theta\colon I \to J$, we must define a group homomorphism $KK(A,\theta)\colon KK(A,I) \to KK(A,J)$.
First, fix a separable $C^*$-subalgebra $I_0\subseteq I$; then there exists a separable $C^*$-subalgebra $J_0\subseteq J$ such that $\theta(I_0)\subseteq J_0$.
Moreover, by functoriality of $KK(A,\,\cdot\,)$ in the separable case, we see that the map 
\begin{equation}
    KK(A,I_0) \to \varinjlim\limits_{J_0\text{ sep.}} KK(A,J_0)
\end{equation}
given by $KK(A,\theta|_{I_0}^{J_0})$ does not depend on the choice of $J_0$.
We may call this map $KK(A,\theta|_{I_0})\colon KK(A,I_0)\to KK(A,J)$.

The system of maps $(KK(A,\theta|_{I_0}))_{I_0\subseteq I\text{ sep.}}$ is compatible with the connecting maps defining $KK(A,I)$ (by functoriality for separable codomains) and therefore they induce a map $KK(A,\theta)\colon KK(A,I) \to KK(A,J)$.
It is a straightforward check that this definition makes $KK(A,\,\cdot\,)$ a covariant functor, using that it is a covariant functor in the separable case.
It is likewise straightforward to check that the functoriality in the first and second variable commute, so that $KK(\,\cdot\, ,\,\cdot\,)$ is a bifunctor.
\end{proof}

We end this subsection by noting that Cuntz's picture of $KK$-theory agrees with the inductive limit definition in the non-separable case.
(The left-hand side of \eqref{eq:KKequalsKKc-nonsep} below is, by definition, the inductive limit over separable subalgebras of $I$.)
This justifies that the inductive limit picture agrees with the usual picture in the case of $\sigma$-unital but non-separable algebras $I$.

\begin{proposition}
\label{prop:KKequalsKKc-nonsep}
Let $A$ and $I$ be $C^*$-algebras with $A$ separable.
Then
\begin{equation}
\label{eq:KKequalsKKc-nonsep}
KK(A,I) \cong KK_c(A,I),
\end{equation}
and the isomorphism is natural in both variables.

In particular, if $I$ is $\sigma$-unital, then the definitions of $KK(A,I)$ given in Definitions~\ref{defn:KK} and \ref{KK-inductivelimit} agree.
\end{proposition}

\begin{proof}
Using Proposition~\ref{prop:KKequalsKKc}, there is a natural identification 
\begin{equation}
    KK(A,I)\cong \varinjlim\limits_{I_0\text{ sep.}} KK_c(A,I_0).
\end{equation}
We will therefore prove the existence of a natural isomorphism between the right-hand side and $KK_c(A,I)$.

Given separable $C^*$-subalgebras $I_0\subseteq I_1 \subseteq I$, we have 
\begin{equation}
    KK_c(A,\iota_{I_1\subseteq I})\circ KK_c(A,\iota_{I_0\subseteq I_1}) = KK_c(A,\iota_{I_0\subseteq I}),
\end{equation}
since $KK_c(A,\,\cdot\,)$ is designed to be functorial without restriction on the codomain. Therefore, we may define
\begin{equation}
\Theta\coloneqq \varinjlim\limits_{I_0\text{ sep.}} KK_c(A,\iota_{I_0\subseteq I})\colon \varinjlim\limits_{I_0\text{ sep.}} KK_c(A,I_0)\to KK_c(A,I), \end{equation}
a homomorphism that is natural in both variables.

For surjectivity, consider $\rho\colon qA \to I\otimes\mathcal K$. Let $I_1$ be a separable $C^*$-subalgebra of $I$ such that $\rho(qA)\subseteq I_1\otimes\mathcal K$.
Then $\rho$ corestricts to $I_1\otimes\mathcal K$, thus giving an element $\kappa\coloneqq \big[\rho|^{I_1\otimes\mathcal K}\big]_{KK_c(A,I_1)} \in \varinjlim\limits_{I_0\text{ sep.}} KK_c(A,I_0)$.
We have
\begin{equation}
\Theta(\kappa) = KK_c(A,\iota_{I_1\subseteq I})\big(\big[\rho|^{I_1\otimes\mathcal K}\big]_{KK_c(A,I_1)}\big) = [\rho]_{KK_c(A,I)}
\end{equation}
by \eqref{eq:KKcTheta}.

For injectivity, suppose that $I_1$ is a separable $C^*$-subalgebra of $I$ and $\rho\colon qA \to I_1\otimes\mathcal K$ is a $^*$-homomorphism such that
\begin{equation}
\Theta([\rho]_{KK_c(A,I_1)}) = 0.
\end{equation}
That is to say, $\iota_{I_1\subseteq I}\circ\rho$ is homotopic to the zero $^*$-homomorphism within the space of $^*$-homomorphisms $qA \to I\otimes\mathcal K$.
Let $(\rho_t\colon qA \to I\otimes\mathcal K)_{t\in [0,1]}$ be such a homotopy, so that $\rho_0 = \iota_{I_1\subseteq I}\circ\rho$ and $\rho_1 = 0$.
Then let $I_2\subseteq I$ be a separable $C^*$-subalgebra containing $I_1$ and $\rho_t(A)\subseteq 
 I_2\otimes\mathcal K$ for $t\in [0,1]\cap\mathbb Q$. By density, it follows that $(\rho_t)_{t\in [0,1]}$ corestricts to a homotopy of $^*$-homomorphisms $qA \to I_2\otimes\mathcal K$, from $\iota_{I_1\subseteq I_2}\circ\rho$ to $0$.
Thus
\begin{equation}
KK_c(A,\iota_{I_1\subseteq I_2})\big([\rho]_{KK_c(A,I_1)}\big)
= [\iota_{I_1\subseteq I_2}\circ\rho]_{KK_c(A,I_2)} \\
= 0.
\end{equation}
This shows that $[\rho]_{KK_c(A,I_1)}$ represents $0$ in $\varinjlim\limits_{I_0\text{ sep.}} KK(A,I_0)$, as required.

The last statement follows using Proposition~\ref{prop:KKequalsKKc}.
\end{proof}

\begin{remark}\label{rmk:GammaAgree}
    If $A$ and $I$ are separable $C^*$-algebras and $J$ is a $\sigma$-unital and non-separable $C^*$-algebra, then the Kasparov product $KK(A,I)\times KK(I,J) \rightarrow KK(A, J)$ agrees with taking the limit of Kasparov products over separable subalgebras $J_0$ of $J$.
    To see this, let $\kappa \in KK(A,I)$ and $\kappa' \in KK(I,J)$.
    According to Definition~\ref{KK-inductivelimit}, we may find a separable subalgebra $J_0$ of $J$ such that $\kappa'$ comes from $\kappa_0' \in KK(A,J_0)$ --- that is, $\kappa' = KK(A,\iota_{J_0\subseteq J})(\kappa_0') = [\iota_{J_0\subseteq J}]_{KK(J_0,J)} \circ \kappa_0'.$
    Then
    \begin{equation}
    \begin{split}
        \kappa'\circ \kappa &= ([\iota_{J_0\subseteq J}]_{KK(J_0,J)}\circ \kappa'_0) \circ\kappa \\
        &= [\iota_{J_0\subseteq J}]_{KK(J_0,J)}\circ (\kappa'_0 \circ\kappa) \\
        &= KK(A,\iota_{J_0\subseteq J})(\kappa'_0\circ\kappa).
    \end{split}
    \end{equation}
    The above argument only uses that the Kasparov product is associative and that $KK(A,\iota_{J_0\subseteq J})(\,\cdot\,) = [\iota_{J_0\subseteq J}]_{KK(J_0,J)}\circ (\,\cdot\,)$, so it holds for the Kasparov product as defined in any picture. 
\end{remark}

\subsection{$KL$-theory outside the \texorpdfstring{$\sigma$}{sigma}-unital setting}

Next we turn to $KL$, aiming to prove Proposition~\ref{KL-inductivelimit} that $KL(A,I)$ can also be described as a limit over separable $C^*$-subalgebras.

\begin{lemma}\label{lem:ZKK-limit}
    Let $A$ and $I$ be $C^*$-algebras with $A$ separable.
    Then
    \begin{equation} Z_{KK(A,I)} = \varinjlim_{I_0 \text{ \rm sep.}} Z_{KK(A,I_0)}, \end{equation}
    inside $\varinjlim\limits_{I_0\text{ \rm sep.}} KK(A,I_0)$.
\end{lemma}

\begin{proof}
  Suppose $I_0\subseteq I$ is separable and $\kappa_0 \in Z_{KK(A,I_0)}$. Then there exists an element $\bar\kappa_0 \in KK(A,C(\bcN,I_0))$ such that $KK(A,\ev_n)(\bar\kappa_0)=0$ for $n\in\mathbb N$ and $KK(A,\ev_\infty)(\bar\kappa_0)=\kappa_0$.
  We see that $KK(A,\iota_{I_0\subseteq I})(\kappa_0) \in Z_{KK(A,I)}$ by pushing this forward to $C(\bcN,I)$.

On the other hand, suppose that $\kappa \in Z_{KK(A,I)}$. Then there exists an element $\bar\kappa \in KK(A,C(\bcN,I))$ such that $KK(A,\mathrm{ev}_n)(\bar\kappa)=0$ for all $n\in\mathbb N$ and $KK(A,\mathrm{ev}_\infty)(\bar\kappa)=KK(A,\iota_{I_0\subseteq I})(\kappa)$.

Using the definition of $KK(A,C(\bcN,I))$ as a limit, there exists a separable subalgebra $J_0\subseteq C(\bcN,I)$ and $\bar\kappa_0 \in KK(A,J_0)$ such that $KK(A,\iota_{J_0\subseteq C(\bcN,I)})(\bar\kappa_0)=\bar\kappa$.
By increasing $J_0$, we can arrange that $J_0$ has the form $C(\bcN,I_0)$ for some separable $I_0 \subseteq I$, and that $KK(A,\mathrm{ev}_n)(\bar\kappa_0)= 0$ in $KK(A,I_0)$ for all $n\in \mathbb N$.
Setting $\kappa_0\coloneqq KK(A,\ev_\infty)(\bar\kappa_0) \in Z_{KK(A,I_0)}$, we then have $\kappa=KK(A,\iota_{I_0\subseteq I})(\kappa_0)$, as required.
\end{proof}

\begin{proof}[Proof of Proposition~\ref{KL-inductivelimit}]
Functoriality of $KL$ descends from functoriality of $KK$ (Proposition~\ref{prop:KKbifunc}).
Given separable $C^*$-subalgebras $I_0\subseteq I_1\subseteq I_2\subseteq I$, 
functoriality applied to the inclusions $I_0\subseteq I_1 \subseteq I_2$ tells us that 
the limit $\varinjlim\limits_{I_0\text{ sep.}}KL(A,I_0)$ exist, and functoriality applied to $I_0\subseteq I_1 \subseteq I$ provides a map
\begin{equation} 
\Theta\coloneqq \varinjlim\limits_{I_0\text{ sep.}} KL(A,\iota_{I_0\subseteq I})\colon \varinjlim\limits_{I_0\text{ sep.}} KL(A,I_0) \to KL(A,I). \end{equation}
By naturality of the map from $KK$ to $KL$, we get a commuting diagram
\begin{equation}
\begin{tikzcd} 
\varinjlim\limits_{I_0\text{ sep.}} KK(A,I_0) \ar[r,equals] \ar[d, two heads, shorten <= -2.35ex] & KK(A,I) \ar[d, two heads] \\
\varinjlim\limits_{I_0\text{ sep.}} KL(A,I_0) \ar[r,"\Theta"] & KL(A,I).
\end{tikzcd}
\end{equation}
In particular, this shows that $\Theta$ is surjective.

For injectivity, let $I_0 \subseteq I$ and $\lambda \in KL(A,I_0)$ satisfy $KL(A,\iota_{I_0\subseteq I})(\lambda)=0$.
Let $\kappa \in KK(A,I_0)$ be a lift of $\lambda$.
Then $KK(A,\iota_{I_0\subseteq I})(\kappa) \in Z_{KK(A,I)}$, so by Lemma~\ref{lem:ZKK-limit}, there exists $I_1\subseteq I$ separable such that $I_0\subseteq I_1$ and $KK(A,\iota_{I_0\subseteq I_1})(\kappa)\in Z_{KK(A,I_1)}$.
Therefore,
\begin{equation}
    KL(A,\iota_{I_0\subseteq I_1})(\lambda)=0,
\end{equation}
as required.
\end{proof}

Now we recall and prove Propositions~\ref{prop:KK-facts} and \ref{prop:KL-facts}.

\begin{proposition}\label{prop:KasparovProdAppendix}
Let $A$ and $I$ be $C^*$-algebras with $A$ separable, let $E$ be a $C^*$-algebra containing $I$ as an ideal, and let $(\phi,\psi)\colon A \rightrightarrows E \rhd I$ be an $(A, I)$-Cuntz pair.
\begin{enumerate}
\item \label{KasparovProdAppendix.1}
Let $\iota^{(2)}_I\colon I \to M_2(I)$ be the top-left corner inclusion.  Then the induced maps $KK(A,\iota^{(2)}_I)\colon KK(A,I) \to KK(A,M_2(I))$ and $KL(A,\iota^{(2)}_I)\colon KL(A,I)\to KL(A,M_2(I))$ are isomorphisms. Moreover, 
\begin{equation}
\label{eq:MatrixInclusionFormula}
\begin{split}
KK(A,\iota^{(2)}_I)([\phi,\psi]_{KK(A,I)}) &= [\iota^{(2)}_{E}\circ\phi,\iota^{(2)}_E\circ\psi]_{KK(A,M_2(I))}, \quad\text{and} \\
KL(A,\iota^{(2)}_I)([\phi,\psi]_{KL(A,I)}) &= [\iota^{(2)}_{E}\circ\phi,\iota^{(2)}_E\circ\psi]_{KL(A,M_2(I))}.
\end{split}\end{equation}
\item \label{KasparovProdAppendix.3}
For any unitary $u \in I^\dagger$,
\begin{equation}\begin{split}
\label{eq:KK-UnitaryConjugacyAppendix}
[\phi,\psi]_{KK(A,I)}&=[\Ad u\circ\phi,\psi]_{KK(A,I)},\text{ and} \\
[\phi,\psi]_{KL(A,I)}&=[\Ad u\circ\phi,\psi]_{KL(A,I)}.
\end{split}\end{equation}
\item \label{KasparovProdAppendix.4}
$[\phi,\psi]_{KK(A,I)}+[\psi,\phi]_{KK(A,I)}=0$ and
$[\phi,\psi]_{KL(A,I)}+[\psi,\phi]_{KL(A,I)}=0$.
\item \label{KasparovProdAppendix.2}
If $J \lhd F$ and $\theta\colon I \to J$ is a $^*$-homomorphism that extends to $\bar\theta\colon E \to F$, then $KK(A,\theta)([\phi,\psi]_{KK(A,I)}) = [\bar\theta\circ \phi,\bar\theta\circ \psi]_{KK(A,J)}$ and similarly $KL(A,\theta)([\phi,\psi]_{KL(A,I)}) = [\bar\theta\circ \phi,\bar\theta\circ \psi]_{KL(A,J)}$.
\item \label{KasparovProdAppendix.5}
    Let $A$ be nuclear, and suppose
    \begin{equation}
      \mathsf{e}\colon
      0\longrightarrow I \stackrel{j_{\mathsf{e}}}\longrightarrow
      E\stackrel{q_{\mathsf{e}}}\longrightarrow D\longrightarrow 0
    \end{equation}
is an extension of $C^*$-algebras.
Then the sequence
    \begin{equation}
      KK(A,I) \xrightarrow{KK(A,j_{\mathsf{e}})}
      KK(A,E)\xrightarrow{KK(A,q_{\mathsf{e}})} KK(A,D) 
    \end{equation}
    is exact. The same holds when $A$ is $KK$-equivalent to a nuclear $C^*$-algebra.
\end{enumerate}
\end{proposition}

\begin{proof}
\ref{KasparovProdAppendix.1}:
In the case when $I$ is separable, the $KK$ isomorphism result follows from \cite[Example 17.8.2(c)]{Blackadar98}, for example.
The formula for $KK(A,\iota^{(2)}_I)$ follows from Proposition~\ref{appendix.corestrict}.
For general $I$, the result follows from the definition of $KK(A,\iota^{(2)}_I)$ as a limit of maps $KK(A,\iota^{(2)}_{I_0})$ (where $I_0\subseteq I$ is separable), in the same spirit as the proof of Proposition~\ref{prop:KKbifunc}. 

\ref{KasparovProdAppendix.3}:
First suppose that $I$ is separable.
We may assume that the scalar part of $u$ is $1_{I^\dag}$. 
Then $u\oplus u^*$ is homotopic to $1_{I^\dag}\oplus 1_{I^\dag}$ via
the unitary path
\begin{equation}
  [0,\pi/2] \ni t \mapsto
  \begin{pmatrix}
    \cos^2(t)(u-1_{I^\dag}) +1_{I^\dag} & - \cos(t)\sin(t) (u-1_{I^\dag}) \\
    \cos(t) \sin(t) (u^* -1_{I^\dag}) & \cos^2(t)(u^* -1_{I^\dag}) + 1_{I^\dag}
  \end{pmatrix}
\end{equation}
which is constantly $1_{I^\dag}\oplus 1_{I^\dag}$ modulo $M_2(I)$, since $u-1_{I^\dag}\in I$. Conjugating with this unitary path in the first entry induces a homotopy of $(A,M_2(I))$-Cuntz pairs from $((\Ad u\circ \phi)\oplus 0,\psi\oplus 0) = (\Ad(u\oplus u^*) \circ (\phi\oplus 0), \psi \oplus 0)$ to $(\phi\oplus 0,\psi\oplus 0)$.
Thus (using the identification in \ref{KasparovProdAppendix.1}), $[\phi,\psi]_{KK(A,I)}=[\Ad u\circ\phi,\psi]_{KK(A,I)}$.

For non-separable $I$, the result follows from the separable case and the definition $KK(A, I)$ as a direct limit (see Definition~\ref{KK-inductivelimit}).
The $KL$ version follows by taking quotients.

\ref{KasparovProdAppendix.4}:
When $I$ is separable, the $KK$ formula can be found in \cite[Proposition~4.1.5]{Jensen-Thomsen91}, for example.
For non-separable $I$, the result follows from the separable case, and the $KL$ version follows by taking quotients.

\ref{KasparovProdAppendix.2}:
When $I$ and $J$ are separable, this is the last part of Proposition~\ref{appendix.corestrict}.
For the non-separable case, set $E_0\coloneqq C^*(\phi(A)\cup\psi(A)) \subseteq E$ and $F_0\coloneqq C^*(\theta(I_0))\subseteq F$.  Then let $I_0 \lhd E_0$ and $J_0 \lhd F_0$ be the ideals generated by $(\phi-\psi)(A)$ and $\theta(I_0)$, respectively.
Then Proposition~\ref{appendix.corestrict} gives
\begin{equation}
KK(A,\theta|_{I_0}^{J_0})\big(\big[\phi|^{E_0},\psi|^{E_0}\big]_{KK(A,I_0)}\big)
= \left[(\bar\theta\circ \phi)|^{F_0},(\bar\theta\circ\psi)|^{F_0}\right]_{KK(A,J_0)}.
\end{equation}
By the definition of $KK(A,\theta)$ in the proof of Proposition~\ref{prop:KKbifunc}, it follows that $KK(A,\theta)([\phi,\psi]_{KK(A,I)})=[\bar\theta\circ\phi,\bar\theta\circ\psi]_{KK(A,J)}$.
The computation for $KL$ follows by taking quotients.

\ref{KasparovProdAppendix.5}:
This is standard when $E$ is separable, being a fragment of the six-term exact sequence (see \cite[Theorem 19.5.7 and Example 19.5.2(a)]{Blackadar98}, or \cite[Exercise 20.10.2(f)]{Blackadar98} for the case that $A$ is $KK$-equivalent to a nuclear $C^*$-algebra, for example).  
For the non-separable case, as $\operatorname{im} KK(A,j_{\mathsf{e}})\subseteq \ker KK(A,q_{\mathsf{e}})$ by functoriality, the only thing to check is that 
$\ker KK(A,q_{\mathsf{e}}) \subseteq \operatorname{im} KK(A,j_{\mathsf{e}})$. Fix $\kappa \in \ker KK(A,q_{\mathsf{e}})$.
As in the proof of Proposition~\ref{prop:KKbifunc}, there exist separable subalgebras $E_0$ and $D_0$ of $E$ and $D$, respectively, such that $q_{\mathsf{e}}(E_0)\subseteq D_0$, $\kappa$ comes from some element $\kappa_0 \in KK(A,E_0)$, and $KK(A,q_{\mathsf{e}}|_{E_0}^{D_0})(\kappa_0)=0$ in $KK(A,D_0)$.
Setting $I_0\coloneqq I\cap E_0$, we have an extension of separable $C^*$-algebras
\begin{equation}
 \mathsf{e}_0\colon  0 \rightarrow I_0 \stackrel{j_{\mathsf{e}_0}}\longrightarrow E_0\stackrel{q_{\mathsf{e}_0}}\longrightarrow D_0\to 0. 
\end{equation}
Thus by the separable case, there exists $\kappa_0' \in KK(A,I_0)$ such that $KK(A,j_{\mathsf{e}_0})(\kappa_0')=\kappa_0$.
Letting $\kappa' \in KK(A,I)$ be the image of $\kappa_0'$, it follows that $KK(A,j_{\mathsf{e}})(\kappa')=\kappa$.
\end{proof}

\subsection{The universal multicoefficient theorem}
\label{sec:nonsepUMCT}

Now we turn to the natural map
\begin{equation}
\Gamma^{(A,I)}\colon KK(A,I) \to \mathrm{Hom}(K_*(A),K_*(I)).
\end{equation}
In the case that $I$ is $\sigma$-unital, the natural identification of $KK(\mathbb C,S^iI)$ with $K_i(I)$ for $i=0,1$ (\cite[Proposition~17.5.5]{Blackadar98}) combined with the Kasparov product provides this map.
For $\phi\colon A \to I$, since $KK(\mathbb C,S^i\phi)$ identifies with $K_i(\phi)$ and since taking the Kasparov product with $[\phi]_{KK(A,I)}$ corresponds to applying $KK(\,\cdot\,,\phi)$ (\cite[Example 18.4.2(a)]{Blackadar98}), we have
\begin{equation}
  \Gamma^{(A,I)}_i([\phi]_{KK(A, I)}) = K_i(\phi).
\end{equation}

For $I$ non-separable, since $K_*(A)$ is countable, we have a natural isomorphism
\begin{equation}
\label{eq:HomKLimit}
\mathrm{Hom}(K_*(A),K_*(I)) \cong \varinjlim\limits_{I_0\text{ sep.}} \mathrm{Hom}(K_*(A),K_*(I_0)),
\end{equation}
(see \cite[Proposition~1.10]{Schafhauser18}, for example), and upon making this identification we may define
\begin{equation}
\label{eq:AppendixKKtoHom}
\Gamma^{(A,I)}\coloneqq \varinjlim_{I_0\text{ sep.}} \Gamma^{(A,I_0)}.
\end{equation}
When $I$ is both non-separable and $\sigma$-unital, a priori this gives two competing definitions for $\Gamma^{(A,I)}$. By Remark~\ref{rmk:GammaAgree}, they agree.

More generally, identifying $KK(\mathbb I_n,S^i I)$ naturally with $K_i(I;\Zn{n})$ as per the picture in \cite{Dadarlat-Loring96} (using continuity of $K_i(\,\cdot\,;\Zn{n})$ for non-separable $I$), we likewise have
\begin{equation}
\mathrm{Hom}_{\Lambda}(\underline{K}(A),\underline{K}(I)) \cong \varinjlim\limits_{I_0\text{ sep.}}\mathrm{Hom}_\Lambda(\underline{K}(A),\underline{K}(I_0)).
\end{equation}
Injectivity of the map 
$\varinjlim\limits_{I_0\text{ sep.}}\mathrm{Hom}_\Lambda(\underline{K}(A),\underline{K}(I_0))\to
\mathrm{Hom}_{\Lambda}(\underline{K}(A),\underline{K}(I))$
follows directly from injectivity of the corresponding map \eqref{eq:HomKLimit}.
For surjectivity, given an element $\underline{\alpha}$ of $\mathrm{Hom}_{\Lambda}(\underline{K}(A),\underline{K}(I))$, first by surjectivity of the corresponding map \eqref{eq:HomKLimit}, we can lift $\underline{\alpha}$ to a graded group homomorphism $\underline{K}(A) \to \underline{K}(I_0)$ for some separable $I_0\subseteq I$.
Then using injectivity of the map \eqref{eq:HomKLimit}, we can enlarge $I_0$ to ensure compatibility with each of the countably many Bockstein operations.

This gives a homomorphism
\begin{equation}
\label{eq:AppendixKKtoHom2}
\Gamma_\Lambda^{(A,I)}\colon KK(A,I) \to \mathrm{Hom}_\Lambda(\underline{K}(A),\underline{K}(I)).
\end{equation}
Just as with $\Gamma^{(A,I)}$, the two possible definitions of $\Gamma^{(A,I)}_\Lambda$ for $I$ $\sigma$-unital and non-separable agree.  When $A$ satisfies the UCT, Dadarlat and Loring's universal multicoefficient theorem says that $\Gamma_\Lambda^{(A,I)}$ descends to an isomorphism when taking the $KL$ quotient on the left.
While their statement only allows separable codomains, the result extends to the non-separable case.

\begin{proof}[Proof of Theorem~\ref{thm:the-umct}]
First let us justify that $\tilde\Gamma_\Lambda^{(A,I)}$ is well-defined, without assuming that $A$ satisfies the UCT.
When $I$ is separable, this is implicit in the paragraph following \cite[Eq.\ (13)]{Dadarlat05}.  The idea is that for $\kappa \in KK(A,C(\bcN,I))$, the sequence 
$$\left(\Gamma_\Lambda^{(A,I)}(KK(A,\ev_n)(\kappa))_i^{(m)} \colon K_i(A; \Zn{m}) \to K_i(I; \Zn{m})\right)_{n=1}^\infty$$ of homomorphisms is eventually equal to $\Gamma_\Lambda^{(A,I)}(KK(A,\ev_\infty)(\kappa))_i^{(m)}$. Hence any element of $Z_{KK(A,I)}$ is mapped to zero by $\Gamma_\Lambda^{(A,I)}$.

For $I$ non-separable, since $Z_{KK(A,I)}=\varinjlim\limits_{I_0\text{ sep.}} Z_{KK(A,I_0)}$ (Lemma~\ref{lem:ZKK-limit}), it follows from the separable case that $\Gamma_\Lambda^{(A,I)}$ vanishes on $Z_{KK(A,I)}$.

Now, suppose that $A$ satisfies the UCT.
If $I$ is separable, then $\Tilde\Gamma_\Lambda^{(A,I)}$ is an isomorphism by \cite[Theorem 4.1]{Dadarlat05} (see also \cite{Dadarlat-Loring96}).
For the non-separable case, we have 
\begin{equation}
\tilde\Gamma_\Lambda^{(A,I)} = \varinjlim_{I_0\text{ sep.}} \tilde\Gamma_\Lambda^{(A,I_0)}.
\end{equation}
Since $\tilde\Gamma_\Lambda^{(A,I_0)}\colon KL(A,I_0) \to \mathrm{Hom}_\Lambda(\underline{K}(A),\underline{K}(I_0))$ is an isomorphism for each $I_0$, it follows that $\tilde\Gamma_\Lambda^{(A,I)}$ is as well.
\end{proof}


\begin{thebibliography}{100}

\bibitem{Akemann-Pedersen-etal73}
C.~A. Akemann, G.~K. Pedersen, and J.~Tomiyama.
\newblock Multipliers of {$C^*$}-algebras.
\newblock {\em J. Funct. Anal.}, 13:277--301, 1973.

\bibitem{Ara-Perera-Toms11}
P.~Ara, F.~Perera, and A.~S. Toms.
\newblock {$K$}-theory for operator algebras. {C}lassification of
  {$C^*$}-algebras.
\newblock In {\em Aspects of operator algebras and applications}, volume 534 of
  {\em Contemp. Math.}, pages 1--71. Amer. Math. Soc., Providence, RI, 2011.

\bibitem{Archbold}
R.~J. Archbold.
\newblock A counterexample for commutation in tensor products of {$C^{*}
  $}-algebras.
\newblock {\em Proc. Amer. Math. Soc.}, 81(4):562--564, 1981.

\bibitem{Arveson69}
W.~B. Arveson.
\newblock Subalgebras of {$C^*$}-algebras.
\newblock {\em Acta Math.}, 123:141--224, 1969.

\bibitem{Barlak-Li17}
S.~Barlak and X.~Li.
\newblock Cartan subalgebras and the {UCT} problem.
\newblock {\em Adv. Math.}, 316:748--769, 2017.

\bibitem{Barlak-Szabo17}
S.~Barlak and G.~Szab\'{o}.
\newblock Rokhlin actions of finite groups on {UHF}-absorbing {${\rm
  C}^*$}-algebras.
\newblock {\em Trans. Amer. Math. Soc.}, 369(2):833--859, 2017.

\bibitem{BKP}
D.~J. Benson, A.~Kumjian and N.~C.~Phillips.
\newblock Symmetries of Kirchberg algebras.
\newblock {\em Canad. Math. Bull.},  46(4):509–528, 2003.

\bibitem{Blackadar88}
B.~Blackadar.
\newblock Comparison theory for simple {$C^*$}-algebras.
\newblock In {\em Operator algebras and applications, {V}ol.\ 1}, volume 135 of
  {\em London Math. Soc. Lecture Note Ser.}, pages 21--54. Cambridge Univ.
  Press, Cambridge, 1988.

\bibitem{Blackadar90}
B.~Blackadar.
\newblock Symmetries of the {CAR} algebra.
\newblock {\em Ann. of Math. (2)}, 131(3):589--623, 1990.

\bibitem{Blackadar98}
B.~Blackadar.
\newblock {\em {$K$}-theory for operator algebras}, volume~5 of {\em
  Mathematical Sciences Research Institute Publications}.
\newblock Cambridge University Press, Cambridge, second edition, 1998.

\bibitem{Blackadar06}
B.~Blackadar.
\newblock {\em Operator algebras}, volume 122 of {\em Encyclopaedia of
  Mathematical Sciences}.
\newblock Springer-Verlag, Berlin, 2006.
\newblock Theory of $C^*$-algebras and von Neumann algebras, Operator
  Algebras and Non-commutative Geometry, III.

\bibitem{Blackadar-Handelman82}
B.~Blackadar and D.~E. Handelman.
\newblock Dimension functions and traces on {$C^{*} $}-algebras.
\newblock {\em J. Funct. Anal.}, 45(3):297--340, 1982.

\bibitem{Blackadar-Kirchberg97}
B.~Blackadar and E.~Kirchberg.
\newblock Generalized inductive limits of finite-dimensional {$C^*$}-algebras.
\newblock {\em Math. Ann.}, 307(3):343--380, 1997.

\bibitem{Blackadar-Kirchberg01}
B.~Blackadar and E.~Kirchberg.
\newblock Inner quasidiagonality and strong {N}{F} algebras.
\newblock {\em Pacific J. Math.}, 198(2):307--329, 2001.

\bibitem{Blackadar-Kumjian-etal92}
B.~Blackadar, A.~Kumjian, and M.~R{\o}rdam.
\newblock Approximately central matrix units and the structure of
  noncommutative tori.
\newblock {\em $K$-Theory}, 6(3):267--284, 1992.

\bibitem{Blackadar-Rordam92}
B.~Blackadar and M.~R{\o}rdam.
\newblock Extending states on preordered semigroups and the existence of
  quasitraces on {$C^*$}-algebras.
\newblock {\em J. Algebra}, 152(1):240--247, 1992.

\bibitem{Blanchard-Kirchberg04}
E.~Blanchard and E.~Kirchberg.
\newblock Non-simple purely infinite {$C^*$}-algebras: the {H}ausdorff case.
\newblock {\em J. Funct. Anal.}, 207(2):461--513, 2004.

\bibitem{BRR08}
E.~Blanchard, R.~Rohde, and M.~R\o{}rdam.
\newblock Properly infinite {$C(X)$}-algebras and {$K_1$}-injectivity.
\newblock {\em J. Noncommut. Geom.}, 2(3):263--282, 2008.

\bibitem{Bodigheimer79}
C.-F. B\"{o}digheimer.
\newblock Splitting the {K}\"{u}nneth sequence in {$K$}-theory.
\newblock {\em Math. Ann.}, 242(2):159--171, 1979.

\bibitem{Bodigheimer80}
C.-F. B\"{o}digheimer.
\newblock Splitting the {K}\"{u}nneth sequence in {$K$}-theory. {II}.
\newblock {\em Math. Ann.}, 251(3):249--252, 1980.

\bibitem{Bosa-Brown-etal15}
J.~Bosa, N.~P. Brown, Y.~Sato, A.~Tikuisis, S.~White, and W.~Winter.
\newblock Covering dimension of {$C^*$}-algebras and 2-coloured classification.
\newblock {\em Mem. Amer. Math. Soc.}, 257(1233):vii+97, 2019.

\bibitem{Bouwen-Gabe24}
B.~Bouwen and J.~Gabe.
\newblock A unified approach for classifying simple nuclear {$C^*$}-algebras.
\newblock Preprint, arXiv:2412.15968, 2024.

\bibitem{Bratteli72}
O.~Bratteli.
\newblock Inductive limits of finite dimensional {$C^{*} $}-algebras.
\newblock {\em Trans. Amer. Math. Soc.}, 171:195--234, 1972.

\bibitem{Bratteli-Elliott-etal}
O.~Bratteli, G.~A. Elliott, D.~E. Evans, and A.~Kishimoto.
\newblock Homotopy of a pair of approximately commuting unitaries in a simple
  {$C^*$}-algebra.
\newblock {\em J. Funct. Anal.}, 160(2):466--523, 1998.

\bibitem{BSKR93}
O.~Bratteli, E.~St{\o}rmer, A.~Kishimoto, and M.~R{\o}rdam.
\newblock The crossed product of a {UHF} algebra by a shift.
\newblock {\em Ergodic Theory Dynam. Systems}, 13(4):615--626, 1993.

\bibitem{Brown82a}
L.~G. Brown.
\newblock Extensions of {AF} algebras: the projection lifting problem.
\newblock In {\em Operator algebras and applications, {P}art 1 ({K}ingston,
  {O}nt., 1980)}, Proc. Sympos. Pure Math., 38, pages 175--176. 1982.

\bibitem{Brown84}
L.~G. Brown.
\newblock The universal coefficient theorem for {${\rm Ext}$} and
  quasidiagonality.
\newblock In {\em Operator algebras and group representations, {V}ol. {I}
  ({N}eptun, 1980)}, volume~17 of {\em Monogr. Stud. Math.}, pages 60--64.
  Pitman, Boston, MA, 1984.

\bibitem{Brown-Douglas-etal77}
L.~G. Brown, R.~G. Douglas, and P.~A. Fillmore.
\newblock Extensions of {$C^*$}-algebras and {$K$}-homology.
\newblock {\em Ann. of Math. (2)}, 105(2):265--324, 1977.

\bibitem{Brown06}
N.~P. Brown.
\newblock Invariant means and finite representation theory of {$C^*$}-algebras.
\newblock {\em Mem. Amer. Math. Soc.}, 184(865):viii+105, 2006.

\bibitem{Brown-Perera-etal08}
N.~P. Brown, F.~Perera, and A.~S. Toms.
\newblock The {C}untz semigroup, the {E}lliott conjecture, and dimension
  functions on {$C^*$}-algebras.
\newblock {\em J. Reine Angew. Math.}, 621:191--211, 2008.

\bibitem{Brown-Winter}
N.~P. Brown, and W.~Winter.
\newblock Quasitraces are traces: a short proof of the finite-nuclear-dimension case.
\newblock {\em C. R. Math. Acad. Sci. Soc. R. Can.}, 33(2):44--49, 2011.

\bibitem{CCEGSTW}
J.~R. Carrión, J.~Castillejos, S.~Evington, J.~Gabe, C.~Schafhauser, A.~Tikuisis, and S.~White.
\newblock Tracially complete $C^*$-algebras.
\newblock arXiv:2310.20594.

\bibitem{Castillejos-Evington20}
J.~Castillejos and S.~Evington.
\newblock Nuclear dimension of simple stably projectionless {${C}^*$}-algebras.
\newblock {\em Anal. PDE}, 13(7):2205--2240, 2020.

\bibitem{CETW21}
J.~Castillejos, S.~Evington, A.~Tikuisis, and S.~White.
\newblock Classifying maps into uniform tracial sequence algebras.
\newblock {\em M\"{u}nster J. Math.}, 14(2):265--281, 2021.

\bibitem{CETW22}
J.~Castillejos, S.~Evington, A.~Tikuisis, and S.~White.
\newblock Uniform property {$\Gamma$}.
\newblock {\em Int. Math. Res. Not. IMRN}, (13):9864--9908, 2022.

\bibitem{CETW-corrigendum}
J.~Castillejos, S.~Evington, A.~Tikuisis, and S.~White.
\newblock Erratum to: {C}lassifying maps into uniform tracial sequence
  algebras.
\newblock {\em M\"{u}nster J. Math.}, 16(2):489--490, 2023.

\bibitem{Castillejos-Evington-etal21}
J.~Castillejos, S.~Evington, A.~Tikuisis, S.~White, and W.~Winter.
\newblock Nuclear dimension of simple {$C^*$}-algebras.
\newblock {\em Invent. Math.}, 224(1):245--290, 2021.

\bibitem{Choi-Effros76}
M.~D. Choi and E.~G. Effros.
\newblock The completely positive lifting problem for {$C^*$}-algebras.
\newblock {\em Ann. of Math. (2)}, 104(3):585--609, 1976.

\bibitem{Choi-Effros76a}
M.~D. Choi and E.~G. Effros.
\newblock Separable nuclear {$C^*$}-algebras and injectivity.
\newblock {\em Duke Math. J.}, 43(2):309--322, 1976.

\bibitem{Choi-Effros77}
M.~D. Choi and E.~G. Effros.
\newblock Nuclear {$C^*$}-algebras and injectivity: the general case.
\newblock {\em Indiana Univ. Math. J.}, 26(3):443--446, 1977.

\bibitem{Ciuperca-Giordano-etal13}
A.~Ciuperca, T.~Giordano, P.~W. Ng, and Z.~Niu.
\newblock Amenability and uniqueness.
\newblock {\em Adv. Math.}, 240:325--345, 2013.

\bibitem{Clark-Fletcher-etal}
L.~O. Clark, J.~Fletcher, and A.~an~Huef.
\newblock All classifiable {K}irchberg algebras are {$C^*$}-algebras of ample
  groupoids.
\newblock {\em Expo. Math.}, 38(4):559--565, 2020.

\bibitem{CJKMST-D-18}
C.~T. Conley, S.~C. Jackson, D.~Kerr, A.~S. Marks, B.~Seward, and R.~D.
  Tucker-Drob.
\newblock F\o lner tilings for actions of amenable groups.
\newblock {\em Math. Ann.}, 371(1-2):663--683, 2018.

\bibitem{Connes73}
A.~Connes.
\newblock Une classification des facteurs de type {${\rm III}$}.
\newblock {\em Ann. Sci. \'{E}cole Norm. Sup. (4)}, 6:133--252, 1973.

\bibitem{Connes75b}
A.~Connes.
\newblock  Classification of automorphisms of hyperfinite factors of type {II}$_1$ and {II}$_\infty$ and application to type {III} factors.
\newblock {\em Bull. Amer. Math. Soc.}, 81(6):1090--1092, 1975.

\bibitem{Connes75}
A.~Connes.
\newblock On hyperfinite factors of type {${\rm III}_{0}$} and {K}rieger's
  factors.
\newblock {\em J. Funct. Anal.}, 18:318--327, 1975.

\bibitem{Connes75a}
A.~Connes.
\newblock Outer conjugacy classes of automorphisms of factors.
\newblock {\em Ann. Sci. \'{E}cole Norm. Sup. (4)}, 8(3):383--419, 1975.

\bibitem{Connes76}
A.~Connes.
\newblock Classification of injective factors. {C}ases II$_1$, II$_\infty$, III$_\lambda$, $\lambda\neq1$.
\newblock {\em Ann. of Math. (2)}, 104(1):73--115, 1976.

\bibitem{Connes77}
A.~Connes.
\newblock Periodic automorphisms of the hyperfinite factor of type
              II$_1$,
\newblock {\em Acta Sci. Math. (Szeged)}, 39(1-2):39--66, 1977.


\bibitem{Connes-Feldman-etal81}
A.~Connes, J.~Feldman, and B.~Weiss.
\newblock An amenable equivalence relation is generated by a single
  transformation.
\newblock {\em Ergodic Theory Dynam. Systems}, 1(4):431--450 (1982), 1981.

\bibitem{Conway90}
J.~B. Conway.
\newblock {\em A course in functional analysis}, volume~96 of {\em Graduate
  Texts in Mathematics}.
\newblock Springer-Verlag, New York, second edition, 1990.

\bibitem{Coward-Elliott-etal08}
K.~T. Coward, G.~A. Elliott, and C.~Ivanescu.
\newblock The {C}untz semigroup as an invariant for {$C^*$}-algebras.
\newblock {\em J. Reine Angew. Math.}, 623:161--193, 2008.

\bibitem{Cuntz77}
J.~Cuntz.
\newblock Simple {$C^*$}-algebras generated by isometries.
\newblock {\em Comm. Math. Phys.}, 57(2):173--185, 1977.

\bibitem{Cuntz78}
J.~Cuntz.
\newblock Dimension functions on simple {$C^*$}-algebras.
\newblock {\em Math. Ann.}, 233(2):145--153, 1978.

\bibitem{Cuntz81Ann}
J.~Cuntz.
\newblock {$K$}-theory for certain {$C^{*} $}-algebras.
\newblock {\em Ann. of Math. (2)}, 113(1):181--197, 1981.

\bibitem{Cuntz82}
J.~Cuntz.
\newblock The internal structure of simple {$C^*$}-algebras.
\newblock In {\em Operator algebras and applications, {P}art 1 ({K}ingston,
  {O}nt., 1980)}, Proc. Sympos. Pure Math., 38, pages 85--115,. 1982.

\bibitem{Cuntz83a}
J.~Cuntz.
\newblock Generalized homomorphisms between {$C^{*} $}-algebras and
  {$KK$}-theory.
\newblock In {\em Dynamics and processes ({B}ielefeld, 1981)}, volume 1031 of
  {\em Lecture Notes in Math.}, pages 31--45. Springer, Berlin, 1983.

\bibitem{Cuntz86}
J.~Cuntz.
\newblock The classification problem for the {$C^*$}-algebras {${\mathcal O}_A$}.
\newblock In {\em Geometric methods in operator algebras ({K}yoto, 1983)},
  volume 123 of {\em Pitman Res. Notes Math. Ser.}, pages 145--151. Longman
  Sci. Tech., Harlow, 1986.

\bibitem{Cuntz87}
J.~Cuntz.
\newblock A new look at {$KK$}-theory.
\newblock {\em $K$-Theory}, 1(1):31--51, 1987.

\bibitem{Cuntz-Krieger80}
J.~Cuntz and W.~Krieger.
\newblock A class of {$C^{*} $}-algebras and topological {M}arkov chains.
\newblock {\em Invent. Math.}, 56(3):251--268, 1980.

\bibitem{Cuntz-Pedersen79}
J.~Cuntz and G.~K. Pedersen.
\newblock Equivalence and traces on {$C^{*} $}-algebras.
\newblock {\em J. Funct. Anal.}, 33(2):135--164, 1979.

\bibitem{Dadarlat00b}
M.~Dadarlat.
\newblock Approximate unitary equivalence and the topology of {${\rm
  Ext}(A,B)$}.
\newblock In {\em {$C^*$}-algebras ({M}\"unster, 1999)}, pages 42--60.
  Springer, Berlin, 2000.
\newblock Available at
  \url{https://www.math.purdue.edu/~mdd/Publications/echz.pdf}.

\bibitem{Dadarlat04}
M.~Dadarlat.
\newblock Morphisms of simple tracially {AF} algebras.
\newblock {\em Internat. J. Math.}, 15(9):919--957, 2004.

\bibitem{Dadarlat05}
M.~Dadarlat.
\newblock On the topology of the {K}asparov groups and its applications.
\newblock {\em J. Funct. Anal.}, 228(2):394--418, 2005.

\bibitem{Dadarlat-Eilers01}
M.~Dadarlat and S.~Eilers.
\newblock Asymptotic unitary equivalence in {$KK$}-theory.
\newblock {\em $K$-Theory}, 23(4):305--322, 2001.

\bibitem{Dadarlat-Eilers02}
M.~Dadarlat and S.~Eilers.
\newblock On the classification of nuclear {$C^*$}-algebras.
\newblock {\em Proc. Lond. Math. Soc. (3)}, 85(1):168--210, 2002.

\bibitem{Dadarlat-LoringAIF}
M.~Dadarlat and T.~A. Loring.
\newblock Extensions of certain real rank zero {$C^*$}-algebras.
\newblock {\em Ann. Inst. Fourier (Grenoble)}, 44(3):907--925, 1994.

\bibitem{Dadarlat-Loring96b}
M.~Dadarlat and T.~A. Loring.
\newblock Classifying {$C^*$}-algebras via ordered, mod-{$p$} {$K$}-theory.
\newblock {\em Math. Ann.}, 305(4):601--616, 1996.

\bibitem{Dadarlat-Loring96}
M.~Dadarlat and T.~A. Loring.
\newblock A universal multicoefficient theorem for the {K}asparov groups.
\newblock {\em Duke Math. J.}, 84(2):355--377, 1996.

\bibitem{Harpe-Skandalis84}
P.~de~la Harpe and G.~Skandalis.
\newblock D\'eterminant associ\'e \`a une trace sur une alg\'ebre de {B}anach.
\newblock {\em Ann. Inst. Fourier (Grenoble)}, 34(1):241--260, 1984.

\bibitem{Ding-Hadwin05}
H.~Ding and D.~Hadwin.
\newblock Approximate equivalence in von {N}eumann algebras.
\newblock {\em Sci. China Ser. A}, 48(2):239--247, 2005.

\bibitem{Dixmier64}
J.~Dixmier.
\newblock Traces sur les {$C^{\ast} $}-alg\`ebres. {II}.
\newblock {\em Bull. Sci. Math. (2)}, 88:39--57, 1964.

\bibitem{Dixmier67}
J.~Dixmier.
\newblock On some {$C^{*} $}-algebras considered by {G}limm.
\newblock {\em J. Funct. Anal.}, 1:182--203, 1967.

\bibitem{DZ-17}
T.~Downarowicz and G.~Zhang.
\newblock Symbolic extensions of amenable group actions and the comparison property.
\newblock {\em Mem. Amer. Math. Soc.}, 281(1390):vi+95, 2023.

\bibitem{EFHKKL16}
C.~J. Eagle, I.~Farah, B.~Hart, B.~Kadets, V.~Kalashnyk, and M.~Lupini.
\newblock Fra\"{\i}ss\'{e} limits of {$C^*$}-algebras.
\newblock {\em J. Symb. Log.}, 81(2):755--773, 2016.

\bibitem{Echterhoff90}
S.~Echterhoff.
\newblock On maximal prime ideals in certain group {$C^*$}-algebras and crossed
  product algebras.
\newblock {\em J. Operator Theory}, 23(2):317--338, 1990.

\bibitem{Eckhardt:adv}
C.~Eckhardt.
\newblock $C^*$-algebras generated by representations of virtually nilpotent groups.
\newblock {\em Adv. Math.}, 444:Paper No.\ 109628, 2024.

\bibitem{Eckhardt-Gillaspy-16}
C.~Eckhardt and E.~Gillaspy.
\newblock Irreducible representations of nilpotent groups generate classifiable
  {$C^*$}-algebras.
\newblock {\em M\"{u}nster J. Math.}, 9(1):253--261, 2016.

\bibitem{Eckhard-Gillaspy-McKenney-19}
C.~Eckhardt, E.~Gillaspy, and P.~McKenney.
\newblock Finite decomposition rank for virtually nilpotent groups.
\newblock {\em Trans. Amer. Math. Soc.}, 371(6):3971--3994, 2019.

\bibitem{EckhardtMcKenney18}
C.~Eckhardt and P.~McKenney.
\newblock Finitely generated nilpotent group {$C^*$}-algebras have finite
  nuclear dimension.
\newblock {\em J. Reine Angew. Math.}, 738:281--298, 2018.

\bibitem{Effros82}
E.~G. Effros.
\newblock On the structure theory of {$C^*$}-algebras: some old and new
  problems.
\newblock In {\em Operator algebras and applications, {P}art 1 ({K}ingston,
  {O}nt., 1980)}, Proc. Sympos. Pure Math., 38, pages 19--34. 1982.

\bibitem{Effros-Hahn67}
E.~G. Effros and F.~Hahn.
\newblock {\em Locally compact transformation groups and {$C^{*} $}- algebras}.
\newblock {\em Mem. Amer. Math. Soc.}, No. 75. American
  Mathematical Society, Providence, R.I., 1967.

\bibitem{Effros-Handelman-etal80}
E.~G. Effros, D.~E. Handelman, and C.~L. Shen.
\newblock Dimension groups and their affine representations.
\newblock {\em Amer. J. Math.}, 102(2):385--407, 1980.

\bibitem{EilersPhD}
S.~Eilers.
\newblock {\em Invariants for {AD} algebras}.
\newblock PhD thesis, University of Copenhagen, 1995.

\bibitem{Eilers96}
S.~Eilers.
\newblock A complete invariant for {$AD$} algebras with real rank zero and
  bounded torsion in {$K_1$}.
\newblock {\em J. Funct. Anal.}, 139(2):325--348, 1996.

\bibitem{Eilers97}
S.~Eilers.
\newblock K\"{u}nneth splittings and classification of {$C^*$}-algebras with
  finitely many ideals.
\newblock In {\em Operator algebras and their applications ({W}aterloo, {ON},
  1994/1995)}, volume~13 of {\em Fields Inst. Commun.}, pages 81--90. Amer.
  Math. Soc., Providence, RI, 1997.

\bibitem{Elliott74}
G.~A. Elliott.
\newblock Derivations of matroid {$C^{*} $}-algebras. {II}.
\newblock {\em Ann. of Math. (2)}, 100:407--422, 1974.

\bibitem{Elliott76}
G.~A. Elliott.
\newblock On the classification of inductive limits of sequences of semisimple
  finite-dimensional algebras.
\newblock {\em J. Algebra}, 38(1):29--44, 1976.

\bibitem{Elliott79}
G.~A. Elliott.
\newblock On totally ordered groups, and {$K_{0}$}.
\newblock In {\em Ring theory ({P}roc. {C}onf., {U}niv. {W}aterloo, {W}aterloo,
  1978)}, volume 734 of {\em Lecture Notes in Math.}, pages 1--49. Springer,
  Berlin, 1979.

\bibitem{Elliott93a}
G.~A. Elliott.
\newblock A classification of certain simple {$C^*$}-algebras.
\newblock In {\em Quantum and non-commutative analysis ({K}yoto, 1992)},
  volume~16 of {\em Math. Phys. Stud.}, pages 373--385. Kluwer Acad. Publ.,
  Dordrecht, 1993.

\bibitem{Elliott93}
G.~A. Elliott.
\newblock On the classification of {$C^*$}-algebras of real rank zero.
\newblock {\em J. Reine Angew. Math.}, 443:179--219, 1993.

\bibitem{Elliott95}
G.~A. Elliott.
\newblock The classification problem for amenable {$C^ *$}-algebras.
\newblock In {\em Proceedings of the {I}nternational {C}ongress of
  {M}athematicians, Vol.\ 1, 2 ({Z}{\"u}rich, 1994)}, pages 922--932, Basel,
  1995. Birkh{\"a}user.

\bibitem{Elliott96}
G.~A. Elliott.
\newblock An invariant for simple {$C^*$}-algebras.
\newblock In {\em Canadian {M}athematical {S}ociety. 1945--1995, {V}ol. 3},
  pages 61--90. Canadian Math. Soc., Ottawa, ON, 1996.

\bibitem{Elliott97}
G.~A. Elliott.
\newblock A classification of certain simple {$C^*$}-algebras. {II}.
\newblock {\em J. Ramanujan Math. Soc.}, 12(1):97--134, 1997.

\bibitem{Elliott22}
G.~A. Elliott.
\newblock $K$-theory and traces.
\newblock {\em C. R. Math. Acad. Sci. Soc. R. Can.}, 44(1):1--15, 2022.

\bibitem{Elliott-Evans93}
G.~A. Elliott and D.~E. Evans.
\newblock The structure of the irrational rotation {$C^*$}-algebra.
\newblock {\em Ann. of Math. (2)}, 138(3):477--501, 1993.

\bibitem{Elliott-Gong96}
G.~A. Elliott and G.~Gong.
\newblock On the classification of {$C^*$}-algebras of real rank zero. {II}.
\newblock {\em Ann. of Math. (2)}, 144(3):497--610, 1996.

\bibitem{Elliott-Gong-etal97}
G.~A. Elliott, G.~Gong, X.~Jiang, and H.~Su.
\newblock A classification of simple limits of dimension drop {$C^*$}-algebras.
\newblock In {\em Operator algebras and their applications ({W}aterloo, {ON},
  1994/1995)}, volume~13 of {\em Fields Inst. Commun.}, pages 125--143. Amer.
  Math. Soc., Providence, RI, 1997.

\bibitem{Elliott-Gong-etal07}
G.~A. Elliott, G.~Gong, and L.~Li.
\newblock On the classification of simple inductive limit {$C^*$}-algebras.
  {II}. {T}he isomorphism theorem.
\newblock {\em Invent. Math.}, 168(2):249--320, 2007.

\bibitem{Elliott-Gong-etal15}
G.~A. Elliott, G.~Gong, H.~Lin, and Z.~Niu.
\newblock The classification of simple separable unital {$\mathcal Z$}-stable locally {ASH} algebras.
\newblock {\em J. Funct. Anal.}, 272(12):5307--5359, 2017.

\bibitem{EGLN}
G.~A. Elliott, G.~Gong, H.~Lin, and Z.~Niu.
\newblock On the classification of simple amenable {$C^*$}-algebras with finite
  decomposition rank, {II}.
\newblock {\em J. Noncommut. Geom.}, 19(1):73--104, 2025.

\bibitem{Elliott-Kucerovsky01}
G.~A. Elliott and D.~Kucerovsky.
\newblock An abstract {V}oiculescu--{B}rown--{D}ouglas--{F}illmore absorption
  theorem.
\newblock {\em Pacific J. Math.}, 198(2):385--409, 2001.

\bibitem{Elliott-Niu16a}
G.~A. Elliott and Z.~Niu.
\newblock On the classification of simple amenable {$C^*$}-algebras with finite
  decomposition rank.
\newblock In {\em Operator algebras and their applications}, volume 671 of {\em
  Contemp. Math.}, pages 117--125. Amer. Math. Soc., Providence, RI, 2016.

\bibitem{Elliott-Niu-17}
G.~A. Elliott and Z.~Niu.
\newblock The {{$C^*$}}-algebra of a minimal homeomorphism of zero mean
  dimension.
\newblock {\em Duke Math. J.}, 166(18):3569--3594, 2017.

\bibitem{BanffReport17}
G.~A. Elliott, Z.~Niu, and A.~Tikuisis.
\newblock Future targets in the classification program for amenable
  {$C^*$}-algebras.
\newblock BIRS workshop report 17w5127, available at
  \url{https://www.birs.ca/workshops//2017/17w5127/report17w5127.pdf}, 2017.

\bibitem{Elliott-Robert-etal11}
G.~A. Elliott, L.~Robert, and L.~Santiago.
\newblock The cone of lower semicontinuous traces on a {$C^*$}-algebra.
\newblock {\em Amer. J. Math.}, 133(4):969--1005, 2011.

\bibitem{Elliott-Rordam95}
G.~A. Elliott and M.~R{\o}rdam.
\newblock Classification of certain infinite simple {$C^*$}-algebras. {II}.
\newblock {\em Comment. Math. Helv.}, 70(4):615--638, 1995.

\bibitem{Elliott-Toms08}
G.~A. Elliott and A.~S. Toms.
\newblock Regularity properties in the classification program for separable
  amenable {$C^*$}-algebras.
\newblock {\em Bull. Amer. Math. Soc. (N.S.)}, 45(2):229--245, 2008.

\bibitem{Evans-Kishimoto91}
D.~E. Evans and A.~Kishimoto.
\newblock Compact group actions on {UHF} algebras obtained by folding the
  interval.
\newblock {\em J. Funct. Anal.}, 98(2):346--360, 1991.

\bibitem{Evans-Kishimoto97}
D.~E. Evans and A.~Kishimoto.
\newblock  Trace scaling automorphisms of certain stable AF algebras. 
\newblock {\em Hokkaido Math. J.} 26(1):211-224, 1997.

\bibitem{Farah}
I.~Farah.
\newblock Between reduced powers and ultrapowers.
\newblock {\em J. Eur. Math. Soc. (JEMS)}, 25(11):4369--4394, 2023.

\bibitem{Farah23}
I.~Farah.
\newblock The {C}alkin algebra, {K}azhdan's property ({T}), and strongly self-absorbing {$C^*$}-algebras.
\newblock {\em Proc. Lond. Math. Soc.}, 127(6):1749--1774, 2023.

\bibitem{Farah-Katsura15}
I.~Farah and T.~Katsura.
\newblock Nonseparable {UHF} algebras {II}: {C}lassification.
\newblock {\em Math. Scand.}, 117(1):105--125, 2015.

\bibitem{FarahSzabo24}
I.~Farah and G.~Szab\'{o}.
\newblock Coronas and strongly self-absorbing {$C^*$}-algebras.
\newblock Preprint, arXiv:2411.02274, 2024.


\bibitem{Franks84}
J.~Franks.
\newblock Flow equivalence of subshifts of finite type.
\newblock {\em Ergodic Theory Dynam. Systems}, 4(1):53--66, 1984.

\bibitem{Gabe-Preprint}
J.~Gabe.
\newblock Classification of {$\mathcal O_\infty$}-stable {$C^*$}-algebras.
\newblock {\em Mem. Amer. Math. Soc.}, 293(1461):v+115, 2024.

\bibitem{Gabe16}
J.~Gabe.
\newblock A note on nonunital absorbing extensions.
\newblock {\em Pacific J. Math.}, 284(2):383--393, 2016.

\bibitem{Gabe17}
J.~Gabe.
\newblock Quasidiagonal traces on exact {$C^*$}-algebras.
\newblock {\em J. Funct. Anal.}, 272(3):1104--1120, 2017.

\bibitem{Gabe20}
J.~Gabe.
\newblock A new proof of {K}irchberg's {$\mathcal O_2$}-stable classification.
\newblock {\em J. Reine Angew. Math.}, 761:247--289, 2020.

\bibitem{Gabe-Szabo23}
J.~Gabe and G.~Szabó.
\newblock The dynamical Kirchberg-Phillips theorem.
\newblock {\em Acta Math.}, 232(1):1--77, 2024.

\bibitem{GGNV}
E.~Gardella, S.~Geffen, P.~Naryshkin, and A.~Vaccaro.
\newblock Dynamical comparison and {$\mathcal Z$}-stability for crossed products of simple
  {$C^*$}-algebras.
\newblock {\em Adv. Math.}, 438:Paper No.\ 109471, 2024.

\bibitem{Gardella-Hirshberg-Vaccaro}
E.~Gardella, I.~Hirshberg, and A.~Vaccaro.
\newblock Strongly outer actions of amenable groups on {$\mathcal{Z}$}-stable
  nuclear {$C^*$}-algebras.
\newblock {\em J. Math. Pures Appl. (9)}, 162:76--123, 2022.

\bibitem{Ghasemi15}
S.~Ghasemi.
\newblock {${SAW}^*$}-algebras are essentially non-factorizable.
\newblock {\em Glasg. Math. J.}, 57(1):1--5, 2015.

\bibitem{Ghasemi21}
S.~Ghasemi.
\newblock Strongly self-absorbing {$C^*$}-algebras and {F}ra\"{\i}ss\'{e}
  limits.
\newblock {\em Bull. Lond. Math. Soc.}, 53(3):937--955, 2021.

\bibitem{Giol-Kerr10}
J.~Giol and D.~Kerr.
\newblock Subshifts and perforation.
\newblock {\em J. Reine Angew. Math.}, 639:107--119, 2010.

\bibitem{Glimm60}
J.~G. Glimm.
\newblock On a certain class of operator algebras.
\newblock {\em Trans. Amer. Math. Soc.}, 95:318--340, 1960.

\bibitem{Gong02}
G.~Gong.
\newblock On the classification of simple inductive limit {$C^*$}-algebras,
  {I}. {T}he reduction theorem.
\newblock {\em Doc. Math.}, 7:255--461 (electronic), 2002.

\bibitem{Gong-OberwolfachReport}
G.~Gong.
\newblock Tracial approximation and classification of {$C^*$}-algebras of
  generalized tracial rank 1.
\newblock {\em Oberwolfach Reports}, 9(4):3147--3148, 2012.

\bibitem{Gong-Jiang-etal00}
G.~Gong, X.~Jiang, and H.~Su.
\newblock Obstructions to {$\mathcal Z$}-stability for unital simple
  {$C^*$}-algebras.
\newblock {\em Canad. Math. Bull.}, 43(4):418--426, 2000.

\bibitem{GLN-preprint}
G.~Gong, H.~Lin, and Z.~Niu.
\newblock Classification of finite simple amenable $\mathcal{Z}$-stable
  ${C}^*$-algebras.
\newblock arXiv:1501.00135.

\bibitem{Gong-Lin-etal23}
G.~Gong, H.~Lin, and Z.~Niu.
\newblock Homomorphisms into simple {$\mathcal{Z}$}-stable {C$^*$}-algebras,
  {II}.
\newblock \emph{J. Noncommut. Geom.,} to appear.

\bibitem{GLN-part1}
G.~Gong, H.~Lin, and Z.~Niu.
\newblock A classification of finite simple amenable {$\mathcal Z$}-stable
  {$C^*$}-algebras, {I}: {$C^*$}-algebras with generalized tracial rank one.
\newblock {\em C. R. Math. Acad. Sci. Soc. R. Can.}, 42(3):63--450, 2020.

\bibitem{GLN-part2}
G.~Gong, H.~Lin, and Z.~Niu.
\newblock A classification of finite simple amenable {$\mathcal Z$}-stable
  {$C^*$}-algebras, {II}: {$C^*$}-algebras with rational generalized tracial
  rank one.
\newblock {\em C. R. Math. Acad. Sci. Soc. R. Can.}, 42(4):451--539, 2020.

\bibitem{Goodearl86}
K.~R. Goodearl.
\newblock {\em Partially ordered abelian groups with interpolation}, volume~20
  of {\em Mathematical Surveys and Monographs}.
\newblock American Mathematical Society, Providence, RI, 1986.

\bibitem{Goodearl92}
K.~R. Goodearl.
\newblock Notes on a class of simple {$C^*$}-algebras with real rank zero.
\newblock {\em Publ. Mat.}, 36(2A):637--654 (1993), 1992.

\bibitem{Gromov99b}
M.~Gromov.
\newblock Topological invariants of dynamical systems and spaces of holomorphic
  maps. {I}.
\newblock {\em Math. Phys. Anal. Geom.}, 2(4):323--415, 1999.

\bibitem{Haagerup87}
U.~Haagerup.
\newblock Connes' bicentralizer problem and uniqueness of the injective factor
  of type {${\rm III}_1$}.
\newblock {\em Acta Math.}, 158(1-2):95--148, 1987.

\bibitem{Haagerup14}
U.~Haagerup.
\newblock Quasitraces on exact {$C^*$}-algebras are traces.
\newblock {\em C. R. Math. Acad. Sci. Soc. R. Can.}, 36(2-3):67--92, 2014.

\bibitem{Hatcher02}
A.~Hatcher.
\newblock {\em Algebraic topology}.
\newblock Cambridge University Press, Cambridge, 2002.

\bibitem{Herman-Jones82}
R.~H.~Herman and V.~F.~R.~Jones.
\newblock Period two automorphisms of UHF $C^*$-algebras.
\newblock {\em J. Funct. Anal.}, 45(2)2:169--176, 1982.

\bibitem{Higson87}
N.~Higson.
\newblock A characterization of {$KK$}-theory.
\newblock {\em Pacific J. Math.}, 126(2):253--276, 1987.

\bibitem{Higson88}
N.~Higson.
\newblock Algebraic {$K$}-theory of stable {$C^*$}-algebras.
\newblock {\em Adv. Math.}, 67(1):140, 1988.

\bibitem{Higson95}
N.~Higson.
\newblock {$C^*$}-algebra extension theory and duality.
\newblock {\em J. Funct. Anal.}, 129(2):349--363, 1995.

\bibitem{Higson98}
N.~Higson.
\newblock The {B}aum--{C}onnes conjecture.
\newblock In {\em Proceedings of the {I}nternational {C}ongress of
  {M}athematicians, {V}ol. {II} ({B}erlin, 1998)}, number Extra Vol. II, pages
  637--646, 1998.

\bibitem{Higson-Kasparov97}
N.~Higson and G.~G. Kasparov.
\newblock Operator {$K$}-theory for groups which act properly and isometrically
  on {H}ilbert space.
\newblock {\em Electron. Res. Announc. Amer. Math. Soc.}, 3:131--142
  (electronic), 1997.

\bibitem{Orovitz-Hirshberg13}
I.~Hirshberg and J.~Orovitz.
\newblock Tracially {$\mathcal{Z}$}-absorbing {$C^*$}-algebras.
\newblock {\em J. Funct. Anal.}, 265(5):765--785, 2013.

\bibitem{Hirshberg-Phillips22}
I.~Hirshberg and N.~C. Phillips.
\newblock Radius of comparison and mean cohomological independence dimension.
\newblock {\em Adv. Math.}, 406:Paper No. 108563, 32, 2022.

\bibitem{HWZ-15}
I.~Hirshberg, W.~Winter, and J.~Zacharias.
\newblock Rokhlin dimension and {$C^*$}-dynamics.
\newblock {\em Comm. Math. Phys.}, 335(2):637--670, 2015.

\bibitem{Hjelmborg-Rordam98}
J.~v.~B. Hjelmborg and M.~R{\o}rdam.
\newblock On stability of {$C^*$}-algebras.
\newblock {\em J. Funct. Anal.}, 155(1):153--170, 1998.

\bibitem{Izumi-Duke-04}
M.~Izumi.
\newblock Finite group actions on {$C^*$}-algebras with the {R}ohlin property. {I}.
\newblock {\em Duke Math. J.}, 122(2):233--280, 2004.

\bibitem{Izumi-Advances04}
M.~Izumi.
\newblock Finite group actions on $C^*$--algebras with the Rohlin property. II. 
\newblock {\em Adv. Math.}, 184(1):119–160, 2004.

\bibitem{Izumi-10}
M.~Izumi.
\newblock Group Actions on Operator Algebras. 
\newblock In  {\em Proceedings of the International Congress of  Mathematicians, Vol. 3}, pages 1528–1548. Hindustan Book Agency, New Delhi, 2010.

\bibitem{Jensen-Thomsen91}
K.~K. Jensen and K.~Thomsen.
\newblock {\em Elements of {$KK$}-theory}.
\newblock Mathematics: Theory \& Applications. Birkh{\"a}user Boston Inc.,
  Boston, MA, 1991.

\bibitem{Jiang97}
X.~Jiang.
\newblock Nonstable {K}-theory for {$\mathcal{Z}$}-stable {$C^*$}-algebras.
\newblock arXiv:math/9707228.

\bibitem{Jiang-Su99}
X.~Jiang and H.~Su.
\newblock On a simple unital projectionless {$C^ *$}-algebra.
\newblock {\em Amer. J. Math.}, 121(2):359--413, 1999.

\bibitem{Jones80}
V.~F.~R. Jones.
\newblock Actions of finite groups on the hyperfinite type II$_1$ factor.
\newblock {\em Mem. Amer. Math. Soc.} 28:237, 1980.

\bibitem{Jones83}
V.~F.~R. Jones.
\newblock Index for subfactors.
\newblock {\em Invent. Math.}, 72(1):1--25, 1983.

\bibitem{Jones-Morrison-etal14}
V.~F.~R. Jones, S.~Morrison, and N.~Snyder.
\newblock The classification of subfactors of index at most 5.
\newblock {\em Bull. Amer. Math. Soc. (N.S.)}, 51(2):277--327, 2014.

\bibitem{Kadison51}
R.~V. Kadison.
\newblock A representation theory for commutative topological algebra.
\newblock {\em Mem. Amer. Math. Soc.}, 7:39, 1951.

\bibitem{Kadison66}
R.~V. Kadison.
\newblock Derivations of operator algebras.
\newblock {\em Ann. of Math. (2)}, 83:280--293, 1966.

\bibitem{KR67}
R.~V. Kadison and J.~R. Ringrose.
\newblock Derivations and automorphisms of operator algebras.
\newblock {\em Comm. Math. Phys.}, 4:32--63, 1967.

\bibitem{Kaminker-Schochet19}
J.~Kaminker and C.~Schochet.
\newblock Spanier--{W}hitehead {$K$}-duality for {$C^*$}-algebras.
\newblock {\em J. Topol. Anal.}, 11(1):21--52, 2019.

\bibitem{Kasparov80a}
G.~G. Kasparov.
\newblock Hilbert {$C^{*} $}-modules: theorems of {S}tinespring and
  {V}oiculescu.
\newblock {\em J. Operator Theory}, 4(1):133--150, 1980.

\bibitem{Kasparov80}
G.~G. Kasparov.
\newblock The operator {$K$}-functor and extensions of {$C^{*} $}-algebras.
\newblock {\em Izv. Akad. Nauk SSSR Ser. Mat.}, 44(3):571--636, 719, 1980.
\newblock Translated as {\emph{Math. USSR. Izvestija.}}, 16(3):513--572, 1981.

\bibitem{Kasparov84}
G.~G. Kasparov.
\newblock Operator {$K$}-theory and its applications: elliptic operators, group
  representations, higher signatures, {$C^*$}-extensions.
\newblock In {\em Proceedings of the {I}nternational {C}ongress of
  {M}athematicians, {V}ol. 1, 2 ({W}arsaw, 1983)}, pages 987--1000. PWN,
  Warsaw, 1984.

\bibitem{Kasparov88}
G.~G. Kasparov.
\newblock Equivariant {$KK$}-theory and the {N}ovikov conjecture.
\newblock {\em Invent. Math.}, 91(1):147--201, 1988.

\bibitem{Katsura08}
T.~Katsura.
\newblock A construction of actions on {K}irchberg algebras which induce given actions on their {$K$}-groups
\newblock {\em J. Reine Angew. Math.}, 617:27--65, 2008.

\bibitem{Kerr20}
D.~Kerr.
\newblock Dimension, comparison, and almost finiteness.
\newblock {\em J. Eur. Math. Soc. (JEMS)}, 22(11):3697--3745, 2020.

\bibitem{Kerr-Naryshkin21}
D.~Kerr and P.~Naryshkin.
\newblock Elementary amenability and almost finiteness.
\newblock arXiv:2107.05273.

\bibitem{Kerr-Szabo-20}
D.~Kerr and G.~Szab\'{o}.
\newblock Almost finiteness and the small boundary property.
\newblock {\em Comm. Math. Phys.}, 374(1):1--31, 2020.

\bibitem{Kirchberg}
E.~Kirchberg.
\newblock The classification of purely infinite {$C^ *$}-algebras using
  {K}asparov's theory.
\newblock Unfinished manuscript available at \url{https://www.uni-muenster.de/imperia/md/content/MathematicsMuenster/ekneu1.pdf}

\bibitem{Kirchberg94}
E.~Kirchberg.
\newblock Commutants of unitaries in {UHF} algebras and functorial properties
  of exactness.
\newblock {\em J. Reine Angew. Math.}, 452:39--77, 1994.

\bibitem{Kirchberg95b}
E.~Kirchberg.
\newblock Exact {${C}^*$}-algebras, tensor products, and the classification of
  purely infinite algebras.
\newblock In {\em Proceedings of the {I}nternational {C}ongress of
  {M}athematicians, {V}ol.\ 1, 2 ({Z}\"urich, 1994)}, pages 943--954.
  Birkh\"auser, Basel, 1995.

\bibitem{Kirchberg95}
E.~Kirchberg.
\newblock On subalgebras of the {CAR}-algebra.
\newblock {\em J. Funct. Anal.}, 129(1):35--63, 1995.

\bibitem{Kirchberg06}
E.~Kirchberg.
\newblock Central sequences in {$C^*$}-algebras and strongly purely infinite
  algebras.
\newblock In {\em Operator {A}lgebras: {T}he {A}bel {S}ymposium 2004}, volume~1
  of {\em Abel Symp.}, pages 175--231. Springer, Berlin, 2006.

\bibitem{Kirchberg-Phillips00}
E.~Kirchberg and N.~C. Phillips.
\newblock Embedding of exact {$C^*$}-algebras in the {C}untz algebra {$\mathcal
  O\sb 2$}.
\newblock {\em J. Reine Angew. Math.}, 525:17--53, 2000.

\bibitem{Kirchberg-Rordam00}
E.~Kirchberg and M.~R{\o}rdam.
\newblock Non-simple purely infinite {$C^*$}-algebras.
\newblock {\em Amer. J. Math.}, 122(3):637--666, 2000.

\bibitem{Kirchberg-Rordam02}
E.~Kirchberg and M.~R{\o}rdam.
\newblock Infinite non-simple {$C^*$}-algebras: absorbing the {C}untz algebras
  {$\mathcal O_\infty$}.
\newblock {\em Adv. Math.}, 167(2):195--264, 2002.

\bibitem{Kirchberg-Rordam14}
E.~Kirchberg and M.~R{\o}rdam.
\newblock Central sequence {$C^*$}-algebras and tensorial absorption of the
  {J}iang--{S}u algebra.
\newblock {\em J. Reine Angew. Math.}, 695:175--214, 2014.

\bibitem{Kirchberg-Winter04}
E.~Kirchberg and W.~Winter.
\newblock Covering dimension and quasidiagonality.
\newblock {\em Internat. J. Math.}, 15(1):63--85, 2004.

\bibitem{Kishimoto77}
A.~Kishimoto.
\newblock On the fixed point algebra of a UHF algebra under a periodic automorphism of product type.
\newblock {\em Publ. Res. Inst. Math. Sci.}, 13(3):777-791, 1977/78. 

\bibitem{Kishimoto81}
A.~Kishimoto.
\newblock Outer automorphisms and reduced crossed products of simple {$C^{*}
  $}-algebras.
\newblock {\em Comm. Math. Phys.}, 81(3):429--435, 1981.

\bibitem{Kishimoto-Kumjian01}
A.~Kishimoto and A.~Kumjian.
\newblock The {E}xt class of an approximately inner automorphism. {II}.
\newblock {\em J. Operator Theory}, 46(1):99--122, 2001.

\bibitem{Krieger76}
W.~Krieger.
\newblock On ergodic flows and the isomorphism of factors.
\newblock {\em Math. Ann.}, 223(1):19--70, 1976.

\bibitem{Kucerovsky-Ng06}
D.~Kucerovsky and P.~W. Ng.
\newblock The corona factorization property and approximate unitary
  equivalence.
\newblock {\em Houston J. Math.}, 32(2):531--550, 2006.

\bibitem{LazarLindenstrauss}
A.~J.~~Lazar and J.~Lindenstrauss.
\newblock Banach spaces whose duals are $L_1$ spaces and their representing matrices.
\newblock {\em Acta. Math.}, 126:165-193, 1971.

\bibitem{Lee11}
H.~H. Lee.
\newblock Proper asymptotic unitary equivalence in {$KK$}-theory and projection
  lifting from the corona algebra.
\newblock {\em J. Funct. Anal.}, 260(1):135--145, 2011.

\bibitem{Li20}
X.~Li.
\newblock Every classifiable simple {$\rm C^*$}-algebra has a {C}artan
  subalgebra.
\newblock {\em Invent. Math.}, 219(2):653--699, 2020.

\bibitem{Lin01}
H.~Lin.
\newblock Classification of simple tracially {AF} {$C^*$}-algebras.
\newblock {\em Canad. J. Math.}, 53(1):161--194, 2001.

\bibitem{Lin01a}
H.~Lin.
\newblock Tracially {AF} {$C^*$}-algebras.
\newblock {\em Trans. Amer. Math. Soc.}, 353(2):693--722 (electronic), 2001.

\bibitem{Lin02}
H.~Lin.
\newblock Stable approximate unitary equivalence of homomorphisms.
\newblock {\em J. Operator Theory}, 47(2):343--378, 2002.

\bibitem{Lin03a}
H.~Lin.
\newblock Classification of simple {$C^*$}-algebras and higher dimensional
  noncommutative tori.
\newblock {\em Ann. of Math. (2)}, 157(2):521--544, 2003.

\bibitem{Lin03}
H.~Lin.
\newblock Simple {$AH$}-algebras of real rank zero.
\newblock {\em Proc. Amer. Math. Soc.}, 131(12):3813--3819, 2003.

\bibitem{Lin04}
H.~Lin.
\newblock Classification of simple {$C^*$}-algebras of tracial topological rank
  zero.
\newblock {\em Duke Math. J.}, 125(1):91--119, 2004.

\bibitem{Lin05}
H.~Lin.
\newblock An approximate universal coefficient theorem.
\newblock {\em Trans. Amer. Math. Soc.}, 357(8):3375--3405, 2005.

\bibitem{Lin07}
H.~Lin.
\newblock Classification of homomorphisms and dynamical systems.
\newblock {\em Trans. Amer. Math. Soc.}, 359(2):859--895, 2007.

\bibitem{Lin08}
H.~Lin.
\newblock Asymptotically unitary equivalence and asymptotically inner
  automorphisms.
\newblock {\em Amer. J. Math.}, 131(6):1589--1677, 2009.

\bibitem{Lin10}
H.~Lin.
\newblock Approximate homotopy of homomorphisms from {$C(X)$} into a simple
  {$C^*$}-algebra.
\newblock {\em Mem. Amer. Math. Soc.}, 205(963):vi+131, 2010.

\bibitem{Lin11}
H.~Lin.
\newblock Asymptotic unitary equivalence and classification of simple amenable
  {$C^*$}-algebras.
\newblock {\em Invent. Math.}, 183(2):385--450, 2011.

\bibitem{Lin12a}
H.~Lin.
\newblock Approximate unitary equivalence in simple {$C^*$}-algebras of tracial
  rank one.
\newblock {\em Trans. Amer. Math. Soc.}, 364(4):2021--2086, 2012.

\bibitem{Lin-OberwolfachReport}
H.~Lin.
\newblock Uniqueness and existence theorems.
\newblock {\em Oberwolfach Reports}, 9(4):3154--3156, 2012.

\bibitem{Lin15}
H.~Lin.
\newblock Crossed products and minimal dynamical systems.
\newblock {\em J. Topol. Anal.}, 10(2):447--469, 2018.

\bibitem{Lin-Niu08}
H.~Lin and Z.~Niu.
\newblock Lifting {$KK$}-elements, asymptotic unitary equivalence and
  classification of simple {$C^*$}-algebras.
\newblock {\em Adv. Math.}, 219(5):1729--1769, 2008.

\bibitem{Lin-Niu14}
H.~Lin and Z.~Niu.
\newblock Homomorphisms into simple {$\mathcal Z$}-stable {$C^*$}-algebras.
\newblock {\em J. Operator Theory}, 71(2):517--569, 2014.

\bibitem{Lin-Phillips95c}
H.~Lin and N.~C. Phillips.
\newblock Approximate unitary equivalence of homomorphisms from {$\mathcal
  O_\infty$}.
\newblock {\em J. Reine Angew. Math.}, 464:173--186, 1995.

\bibitem{Lin-Phillips10}
H.~Lin and N.~C. Phillips.
\newblock Crossed products by minimal homeomorphisms.
\newblock {\em J. Reine Angew. Math.}, 641:95--122, 2010.

\bibitem{Lindenstrauss-Weiss00}
E.~Lindenstrauss and B.~Weiss.
\newblock Mean topological dimension.
\newblock {\em Israel J. Math.}, 115:1--24, 2000.

\bibitem{LOS}
J.~Lindenstrauss, G.~Olsen and Y.~Sternfeld.
\newblock The Poulsen simplex.
\newblock \emph{Ann. Inst. Fourier (Grenoble)}, 28(1):91--114, 1978.

\bibitem{Loreaux-Ng20}
J.~Loreaux and P.~W. Ng.
\newblock Remarks on essential codimension.
\newblock {\em Integral Equations Operator Theory}, 92(1):Paper No. 4, 35,
  2020.

\bibitem{Loring97a}
T.~A. Loring.
\newblock {\em Lifting solutions to perturbing problems in {$C^*$}-algebras},
  volume~8 of {\em Fields Institute Monographs}.
\newblock American Mathematical Society, Providence, RI, 1997.

\bibitem{Masumoto17}
S.~Masumoto.
\newblock The {J}iang--{S}u algebra as a {F}ra\"{\i}ss\'{e} limit.
\newblock {\em J. Symb. Log.}, 82(4):1541--1559, 2017.

\bibitem{Matui11}
H.~Matui.
\newblock Classification of homomorphisms into simple {$\mathcal Z$}-stable
  {$C^*$}-algebras.
\newblock {\em J. Funct. Anal.}, 260(3):797--831, 2011.

\bibitem{Matui-12}
H.~Matui.
\newblock Homology and topological full groups of \'{e}tale groupoids on
  totally disconnected spaces.
\newblock {\em Proc. Lond. Math. Soc. (3)}, 104(1):27--56, 2012.

\bibitem{Matui-Sato12a}
H.~Matui and Y.~Sato.
\newblock {$\mathcal{Z}$}-stability of crossed products by strongly outer actions.
\newblock {\em Comm. Math. Phys.}, 314(1):193--228, 2012.

\bibitem{Matui-Sato12}
H.~Matui and Y.~Sato.
\newblock Strict comparison and {$\mathcal{Z}$}-absorption of nuclear
  {$C^*$}-algebras.
\newblock {\em Acta Math.}, 209(1):179--196, 2012.

\bibitem{Matui-Sato14}
H.~Matui and Y.~Sato.
\newblock Decomposition rank of {UHF}-absorbing {${C}^*$}-algebras.
\newblock {\em Duke Math. J.}, 163(14):2687--2708, 2014.

\bibitem{McDuff70}
D.~McDuff.
\newblock Central sequences and the hyperfinite factor.
\newblock {\em Proc. Lond. Math. Soc. (3)}, 21:443--461, 1970.

\bibitem{Meyer-Nest06}
R.~Meyer and R.~Nest.
\newblock The {B}aum-{C}onnes conjecture via localisation of categories.
\newblock {\em Topology}, 45(2):209--259, 2006.

\bibitem{Moore-Rosenberg76}
C.~C. Moore and J.~Rosenberg.
\newblock Groups with {$T_{1}$} primitive ideal spaces.
\newblock {\em J. Funct. Anal.}, 22(3):204--224, 1976.

\bibitem{Munkholm25}
M.~Munkholm.
\newblock Uniqueness of the zeta transformation in operator {K}-theory.
\newblock {\em J. Operator Theory}, to appear. arXiv:2512.16035

\bibitem{MvN.4}
F.~J. Murray and J.~von Neumann.
\newblock On rings of operators. {IV}.
\newblock {\em Ann. of Math. (2)}, 44:716--808, 1943.

\bibitem{Naryshkin23}
P.~Naryshkin.
\newblock Group extensions preserve almost finiteness.
\newblock arXiv:2304.02456.

\bibitem{Ng-Robert16}
P.~W. Ng and L.~Robert.
\newblock Sums of commutators in pure {$\rm C^*$}-algebras.
\newblock {\em M\"{u}nster J. Math.}, 9(1):121--154, 2016.

\bibitem{Nielsen99}
K.~E. Nielsen.
\newblock Homomorphisms into simple limits of circle algebras.
\newblock {\em Math. Scand.}, 84(1):93--118, 1999.

\bibitem{NielsenThomsen}
K.~E. Nielsen and K.~Thomsen.
\newblock Limits of circle algebras.
\newblock {\em Expo. Math.}, 14(1):17--56, 1996.

\bibitem{Niu19}
Z.~Niu.
\newblock Comparison radius and mean topological dimension: {R}okhlin
  property, comparison of open sets, and subhomogeneous {$C^*$}-algebras.
\newblock {\em J. Anal. Math.}, 146(2):595--672, 2022.

\bibitem{Ocneanu85}
A.~Ocneanu.
\newblock Actions of discrete amenable groups on von Neumann algebras.
\newblock {\em Lecture Notes in Math.}, 1138,
\newblock Springer-Verlag, Berlin, 1985.

\bibitem{OlesenPedersen74}
D.~Olesen and G.~K. Pedersen.
\newblock Derivations of {$C^{*} $}-algebras have semi-continuous generators.
\newblock {\em Pacific J. Math.}, 53:563--572, 1974.

\bibitem{Ortega-Perera-etal12}
E.~Ortega, F.~Perera, and M.~R{\o}rdam.
\newblock The corona factorization property, stability, and the {C}untz
  semigroup of a {$C^*$}-algebra.
\newblock {\em Int. Math. Res. Not. IMRN}, (1):34--66, 2012.

\bibitem{Ozawa13}
N.~Ozawa.
\newblock Dixmier approximation and symmetric amenability for {$C^*$}-algebras.
\newblock {\em J. Math. Sci. Univ. Tokyo}, 20(3):349--374, 2013.

\bibitem{Ozawa-Rordam-etal15}
N.~Ozawa, M.~R{\o}rdam, and Y.~Sato.
\newblock Elementary amenable groups are quasidiagonal.
\newblock {\em Geom. Funct. Anal.}, 25(1):307--316, 2015.

\bibitem{Paschke81}
W.~L. Paschke.
\newblock {$K$}-theory for commutants in the {C}alkin algebra.
\newblock {\em Pacific J. Math.}, 95(2):427--434, 1981.

\bibitem{Pasnicu87}
C.~Pasnicu.
\newblock Tensor products of {B}unce--{D}eddens algebras.
\newblock In {\em Operators in indefinite metric spaces, scattering theory and
  other topics (Bucharest, 1985)}, volume~24 of {\em Oper. Theory Adv. Appl.},
  pages 283--288. Birkh{\"a}user, Basel, 1987.

\bibitem{Patchell}
G.~Patchell.
\newblock Primeness of generalized wreath product {$\mathrm{II}_1$} factors.
\newblock {\em Math. Z.}, 309(3):Paper No.\ 43, 2025.

\bibitem{PearcyTopping71}
C.~Pearcy and D.~Topping.
\newblock On commutators in ideals of compact operators.
\newblock {\em Michigan Math. J.}, 18:247--252, 1971.

\bibitem{Pedersen79}
G.~K. Pedersen.
\newblock {\em {$C^{*} $}-algebras and their automorphism groups}, volume~14 of
  {\em London Mathematical Society Monographs}.
\newblock Academic Press Inc. [Harcourt Brace Jovanovich Publishers], London,
  1979.

\bibitem{Perera97}
F.~Perera.
\newblock The structure of positive elements for {$C^*$}-algebras with real
  rank zero.
\newblock {\em Internat. J. Math.}, 8(3):383--405, 1997.

\bibitem{Phillips00}
N.~C. Phillips.
\newblock A classification theorem for nuclear purely infinite simple
  {$C^*$}-algebras.
\newblock {\em Doc. Math.}, 5:49--114, 2000.

\bibitem{Pimsner83}
M.~Pimsner.
\newblock Embedding some transformation group {$C^{*} $}-algebras into
  {AF}-algebras.
\newblock {\em Ergodic Theory Dynam. Systems}, 3(4):613--626, 1983.

\bibitem{Pimsner-Popa-etal79}
M.~Pimsner, S.~Popa, and D.~Voiculescu.
\newblock Homogeneous {$C^{*} $}-extensions of {$C(X)\otimes K(H)$}. {I}.
\newblock {\em J. Operator Theory}, 1(1):55--108, 1979.

\bibitem{Pimsner-Voiculescu}
M.~Pimsner and D.~Voiculescu.
\newblock Exact sequence for $K$-groups and $\mathrm{Ext}$-groups of certain crossed products.
\newblock {\em J. Operator Theory}, 4(1):93--118, 1980.

\bibitem{Poguntke81}
D.~Poguntke.
\newblock Discrete nilpotent groups have a {$T_{1}$} primitive ideal space.
\newblock {\em Studia Math.}, 71(3):271--275, 1981/82.

\bibitem{Popa86}
S.~Popa.
\newblock A short proof of ``injectivity implies hyperfiniteness'' for finite
  von {N}eumann algebras.
\newblock {\em J. Operator Theory}, 16(2):261--272, 1986.

\bibitem{Popa94}
S.~Popa.
\newblock Classification of amenable subfactors of type {II}.
\newblock {\em Acta Math.}, 172(2):163--255, 1994.

\bibitem{Popa97}
S.~Popa.
\newblock On local finite-dimensional approximation of {$C^*$}-algebras.
\newblock {\em Pacific J. Math.}, 181(1):141--158, 1997.

\bibitem{Popa07}
S.~Popa.
\newblock Deformation and rigidity for group actions and von {N}eumann
  algebras.
\newblock In {\em International {C}ongress of {M}athematicians. {V}ol. {I}},
  pages 445--477. Eur. Math. Soc., Z\"urich, 2007.

\bibitem{Putnam89}
I.~F. Putnam.
\newblock The {$C^*$}-algebras associated with minimal homeomorphisms of the
  {C}antor set.
\newblock {\em Pacific J. Math.}, 136(2):329--353, 1989.

\bibitem{Renault08}
J.~Renault.
\newblock Cartan subalgebras in {$C^*$}-algebras.
\newblock {\em Irish Math. Soc. Bull.}, (61):29--63, 2008.

\bibitem{Rieffel87}
M.~A. Rieffel.
\newblock The homotopy groups of the unitary groups of noncommutative tori.
\newblock {\em J. Operator Theory}, 17(2):237--254, 1987.

\bibitem{Robert12}
L.~Robert.
\newblock Classification of inductive limits of 1-dimensional {NCCW} complexes.
\newblock {\em Adv. Math.}, 231(5):2802--2836, 2012.

\bibitem{Rohde09}
R.~Rohde.
\newblock {\em {$K_1$}-injectivity of {$C^*$}-algebras}.
\newblock PhD thesis, University of Copenhagen, 2009.

\bibitem{Rordam92}
M.~R{\o}rdam.
\newblock On the structure of simple {$C^*$}-algebras tensored with a
  {UHF}-algebra. {II}.
\newblock {\em J. Funct. Anal.}, 107(2):255--269, 1992.

\bibitem{Rordam93}
M.~R{\o}rdam.
\newblock Classification of inductive limits of {C}untz algebras.
\newblock {\em J. Reine Angew. Math.}, 440:175--200, 1993.

\bibitem{Rordam94}
M.~R{\o}rdam.
\newblock A short proof of {E}lliott's theorem: {${\mathcal O}_2\otimes{\mathcal
  O}_2\cong{\mathcal O}_2$}.
\newblock {\em C. R. Math. Acad. Sci. Soc. R. Can.}, 16(1):31--36, 1994.

\bibitem{Rordam95}
M.~R\o{}rdam.
\newblock Classification of certain infinite simple {$C^*$}-algebras.
\newblock {\em J. Funct. Anal.}, 131(2):415--458, 1995.

\bibitem{RordamKth95}
M.~R{\o}rdam.
\newblock Classification of {C}untz--{K}rieger algebras.
\newblock {\em $K$-Theory}, 9(1):31--58, 1995.

\bibitem{Rordam02}
M.~R{\o}rdam.
\newblock Classification of nuclear, simple {$C^ *$}-algebras.
\newblock In {\em Classification of nuclear {$C^ *$}-algebras. {E}ntropy in
  operator algebras}, volume 126 of {\em Encyclopaedia Math. Sci.}, pages
  1--145. Springer, Berlin, 2002.

\bibitem{Rordam03}
M.~R{\o}rdam.
\newblock A simple {$C^*$}-algebra with a finite and an infinite projection.
\newblock {\em Acta Math.}, 191(1):109--142, 2003.

\bibitem{Rordam04}
M.~R{\o}rdam.
\newblock The stable and the real rank of {$\mathcal Z$}-absorbing
  {$C^*$}-algebras.
\newblock {\em Internat. J. Math.}, 15(10):1065--1084, 2004.

\bibitem{Rordam-Larsen-etal00}
M.~R{\o}rdam, F.~Larsen, and N.~J. Laustsen.
\newblock {\em An introduction to {$K$}-theory for {$C^ *$}-algebras},
  volume~49 of {\em London Mathematical Society Student Texts}.
\newblock Cambridge University Press, Cambridge, 2000.

\bibitem{Rordam-Winter10}
M.~R{\o}rdam and W.~Winter.
\newblock The {J}iang--{S}u algebra revisited.
\newblock {\em J. Reine Angew. Math.}, 642:129--155, 2010.

\bibitem{Rosenberg-Schochet87}
J.~Rosenberg and C.~Schochet.
\newblock The {K}\"unneth theorem and the universal coefficient theorem for
  {K}asparov's generalized {$K$}-functor.
\newblock {\em Duke Math. J.}, 55(2):431--474, 1987.

\bibitem{Sakai66}
S.~Sakai.
\newblock Derivations of {$W^{*} $}-algebras.
\newblock {\em Ann. of Math. (2)}, 83:273--279, 1966.

\bibitem{Sakai98}
S.~Sakai.
\newblock {\em {$C^*$}-algebras and {$W^*$}-algebras}.
\newblock Classics in Mathematics. Springer-Verlag, Berlin, 1998.
\newblock Reprint of the 1971 edition.

\bibitem{Sarkowicz-Tikuisis}
P.~Sarkowicz and A.~Tikuisis.
\newblock Polar decomposition in algebraic {$K$}-theory.
\newblock {\em J. Operator Theory}, 93(2):413--434, 2025.

\bibitem{Sato12}
Y.~Sato.
\newblock Trace spaces of simple nuclear {$C^*$}-algebras with
  finite-dimensional extreme boundary.
\newblock arXiv:1209.3000.

\bibitem{Sato19}
Y.~Sato.
\newblock Actions of amenable groups and crossed products of {$\mathcal
  Z$}-absorbing {$C^*$}-algebras.
\newblock In {\em Operator algebras and mathematical physics}, volume~80 of
  {\em Adv. Stud. Pure Math.}, pages 189--210. Math. Soc. Japan, Tokyo, 2019.

\bibitem{Sato-White-etal15}
Y.~Sato, S.~White, and W.~Winter.
\newblock Nuclear dimension and {$\mathcal{Z}$}-stability.
\newblock {\em Invent. Math.}, 202(2):893--921, 2015.

\bibitem{Schafhauser17}
C.~Schafhauser.
\newblock A new proof of the {T}ikuisis--{W}hite--{W}inter theorem.
\newblock {\em J. Reine Angew. Math.}, 759:291--304, 2020.

\bibitem{Schafhauser18}
C.~Schafhauser.
\newblock Subalgebras of simple {AF}-algebras.
\newblock {\em Ann. of Math. (2)}, 192(2):309--352, 2020.

\bibitem{Schafhauser24}
C.~Schafhauser.
\newblock {$KK$}-rigidity of simple nuclear {$C^*$}-algebras.
\newblock Preprint, arXiv:2408.02745 [math.OA], 2024.

\bibitem{Schemaitat19}
A.~Schemaitat.
\newblock The {J}iang--{S}u algebra is strongly self-absorbing revisited.
\newblock {\em J. Funct. Anal.}, 282(6):39, 2022.

\bibitem{Schochet84}
C.~Schochet.
\newblock Topological methods for {$C^{*} $}-algebras. {IV}. {M}od {$p$}\
  homology.
\newblock {\em Pacific J. Math.}, 114(2):447--468, 1984.

\bibitem{Schochet02}
C.~Schochet.
\newblock The fine structure of the {K}asparov groups. {II}. {T}opologizing the
  {UCT}.
\newblock {\em J. Funct. Anal.}, 194(2):263--287, 2002.

\bibitem{Skandalis85}
G.~Skandalis.
\newblock On the group of extensions relative to a semifinite factor.
\newblock {\em J. Operator Theory}, 13(2):255--263, 1985.

\bibitem{Skandalis88}
G.~Skandalis.
\newblock Une notion de nucl\'earit\'e en {$K$}-th\'eorie (d'apr\`es {J}.\
  {C}untz).
\newblock {\em $K$-Theory}, 1(6):549--573, 1988.

\bibitem{Spielberg07}
J.~Spielberg.
\newblock Non-cyclotomic presentations of modules and prime-order automorphisms of Kirchberg algebras.
\newblock {\em J. Reine Angew. Math.} 613:211–230, 2007.

\bibitem{Strung}
K.~R. Strung.
\newblock {${\rm C}^*$}-algebras of minimal dynamical systems of the product of
  a {C}antor set and an odd dimensional sphere.
\newblock {\em J. Funct. Anal.}, 268(3):671--689, 2015.

\bibitem{Strung21}
K.~R. Strung.
\newblock {\em An introduction to {$C^*$}-algebras and the classification program}.
\newblock Advanced Courses in Mathematics. Birkh\~auser/Springer, Cham, 2021.

\bibitem{Szabo-15}
G.~Szab\'{o}.
\newblock The {R}okhlin dimension of topological {$\mathbb{Z}^m$}-actions.
\newblock {\em Proc. Lond. Math. Soc. (3)}, 110(3):673--694, 2015.

\bibitem{Szabo26}
G.~Szab\'{o}.
\newblock The uniqueness theorem for {K}asparov theory.
\newblock Preprint, arXiv:2601.23029, 2026.

\bibitem{SWZ-19}
G.~Szab\'{o}, J.~Wu, and J.~Zacharias.
\newblock Rokhlin dimension for actions of residually finite groups.
\newblock {\em Ergodic Theory Dynam. Systems}, 39(8):2248--2304, 2019.

\bibitem{Tatsuuma98}
N.~Tatsuuma, H.~Shimomura, and T.~Hirai.
\newblock On group topologies and unitary representations of inductive limits
  of topological groups and the case of the group of diffeomorphisms.
\newblock {\em J. Math. Kyoto Univ.}, 38(3):551--578, 1998.

\bibitem{Thomsen90}
K.~Thomsen.
\newblock Homotopy classes of {$^*$}-homomorphisms between stable
  {$C^*$}-algebras and their multiplier algebras.
\newblock {\em Duke Math. J.}, 61(1):67--104, 1990.

\bibitem{Thomsen94}
K.~Thomsen.
\newblock Inductive limits of interval algebras: the tracial state space.
\newblock {\em Amer. J. Math.}, 116(3):605--620, 1994.

\bibitem{Thomsen95}
K.~Thomsen.
\newblock Traces, unitary characters and crossed products by {${\mathbb Z}$}.
\newblock {\em Publ. Res. Inst. Math. Sci.}, 31(6):1011--1029, 1995.

\bibitem{Thomsen01}
K.~Thomsen.
\newblock On absorbing extensions.
\newblock {\em Proc. Amer. Math. Soc.}, 129(5):1409--1417 (electronic), 2001.

\bibitem{TikuisisWhite-OberwolfachReport}
A.~Tikuisis and S.~White.
\newblock Classifying $^*$-homomorphisms.
\newblock {\em Oberwolfach Reports}, 16(3):2263--2270, 2019.

\bibitem{TWW}
A.~Tikuisis, S.~White, and W.~Winter.
\newblock Quasidiagonality of nuclear {$C^*$}-algebras.
\newblock {\em Ann. of Math.}, 185:229--284, 2017.

\bibitem{Toms08}
A.~S. Toms.
\newblock An infinite family of non-isomorphic {$C^*$}-algebras with identical
  {$K$}-theory.
\newblock {\em Trans. Amer. Math. Soc.}, 360(10):5343--5354, 2008.

\bibitem{Toms08a}
A.~S. Toms.
\newblock On the classification problem for nuclear {$C^*$}-algebras.
\newblock {\em Ann. of Math. (2)}, 167(3):1029--1044, 2008.

\bibitem{Toms-11}
A.~S. Toms.
\newblock $K$-theoretic rigidity and slow dimension growth.
\newblock {\em Invent. Math.}, 183(2):225--244, 2011.

\bibitem{Toms-White-Winter15}
A.~S. Toms, S.~White, and W.~Winter.
\newblock {$\mathcal{Z}$}-stability and finite-dimensional tracial boundaries.
\newblock {\em Int. Math. Res. Not. IMRN}, (10):2702--2727, 2015.

\bibitem{Toms-Winter07}
A.~S. Toms and W.~Winter.
\newblock Strongly self-absorbing {$C^*$}-algebras.
\newblock {\em Trans. Amer. Math. Soc.}, 359(8):3999--4029, 2007.

\bibitem{Toms-Winter09}
A.~S. Toms and W.~Winter.
\newblock Minimal dynamics and the classification of {$C^*$}-algebras.
\newblock {\em Proc. Natl. Acad. Sci. USA}, 106(40):16942--16943, 2009.

\bibitem{Toms-Winter13}
A.~S. Toms and W.~Winter.
\newblock Minimal dynamics and {K}-theoretic rigidity: {E}lliott's conjecture.
\newblock {\em Geom. Funct. Anal.}, 23(1):467--481, 2013.

\bibitem{Tu99}
J.-L. Tu.
\newblock La conjecture de {B}aum--{C}onnes pour les feuilletages moyennables.
\newblock {\em $K$-Theory}, 17(3):215--264, 1999.

\bibitem{Tu05}
J.-L. Tu.
\newblock The gamma element for groups which admit a uniform embedding into
  {H}ilbert space.
\newblock In {\em Recent advances in operator theory, operator algebras, and
  their applications}, volume 153 of {\em Oper. Theory Adv. Appl.}, pages
  271--286. Birkh{\"a}user, Basel, 2005.

\bibitem{Ursu21}
D.~Ursu.
\newblock Characterizing traces on crossed products of noncommutative
  {$C^*$}-algebras.
\newblock {\em Adv. Math.}, 391:Paper No. 107955, 29, 2021.

\bibitem{Vaes18}
S.~Vaes.
\newblock Amenability versus non amenability: an introduction to von {N}eumann
  algebras.
\newblock In {\em European {C}ongress of {M}athematics}, pages 483--500. Eur.
  Math. Soc., Z\"{u}rich, 2018.

\bibitem{Villadsen98}
J.~Villadsen.
\newblock Simple {$C^*$}-algebras with perforation.
\newblock {\em J. Funct. Anal.}, 154(1):110--116, 1998.

\bibitem{Villadsen99}
J.~Villadsen.
\newblock On the stable rank of simple {$C^*$}-algebras.
\newblock {\em J. Amer. Math. Soc.}, 12(4):1091--1102, 1999.

\bibitem{Voiculescu76}
D.~Voiculescu.
\newblock A non-commutative {W}eyl--von {N}eumann theorem.
\newblock {\em Rev. Roumaine Math. Pures Appl.}, 21(1):97--113, 1976.

\bibitem{Voiculescu91}
D.~Voiculescu.
\newblock A note on quasi-diagonal {$C^*$}-algebras and homotopy.
\newblock {\em Duke Math. J.}, 62(2):267--271, 1991.

\bibitem{Weibel94}
C.~A. Weibel.
\newblock {\em An introduction to homological algebra}, volume~38 of {\em
  Cambridge Studies in Advanced Mathematics}.
\newblock Cambridge University Press, Cambridge, 1994.

\bibitem{White:ICM}
S.~White.
\newblock Abstract classification theorems for amenable {$C^*$}-algebras.
\newblock In {\em Proceedings of the {I}nternational {C}ongress of
  {M}athematicians 2022. Vol. 4} pp. 3314–3338. EMS Press, Berlin, 2023

\bibitem{Winter04}
W.~Winter.
\newblock Decomposition rank of subhomogeneous {$C^*$}-algebras.
\newblock {\em Proc. Lond. Math. Soc. (3)}, 89(2):427--456, 2004.

\bibitem{Winter06}
W.~Winter.
\newblock On the classification of simple {$\mathcal Z$}-stable {$C^*$}-algebras
  with real rank zero and finite decomposition rank.
\newblock {\em J. London Math. Soc. (2)}, 74(1):167--183, 2006.

\bibitem{Winter10a}
W.~Winter.
\newblock Decomposition rank and {$\mathcal Z$}-stability.
\newblock {\em Invent. Math.}, 179(2):229--301, 2010.

\bibitem{Winter11}
W.~Winter.
\newblock Strongly self-absorbing {$C^*$}-algebras are {$\mathcal Z$}-stable.
\newblock {\em J. Noncommut. Geom.}, 5(2):253--264, 2011.

\bibitem{Winter12}
W.~Winter.
\newblock Nuclear dimension and {$\mathcal{Z}$}-stability of pure {$C^*$}-algebras.
\newblock {\em Invent. Math.}, 187(2):259--342, 2012.

\bibitem{Winter14}
W.~Winter.
\newblock Localizing the {E}lliott conjecture at strongly self-absorbing
  {$C^*$}-algebras.
\newblock {\em J. Reine Angew. Math.}, 692:193--231, 2014.

\bibitem{Winter16a}
W.~Winter.
\newblock Classifying crossed product {$C^*$}-algebras.
\newblock {\em Amer. J. Math.}, 138(3):793--820, 2016.

\bibitem{Winter19}
W.~Winter.
\newblock Structure of nuclear {$C^*$}-algebras: {F}rom quasidiagonality to
  classification, and back again.
\newblock In {\em Proceedings of the {I}nternational {C}ongress of
  {M}athematicians, Vol.\ 2 ({R}io de {J}aneiro, 2018)}, pages 1819--1842.
  {W}orld {S}cientific, 2019.

\bibitem{Winter-Zacharias10}
W.~Winter and J.~Zacharias.
\newblock The nuclear dimension of {$C^*$}-algebras.
\newblock {\em Adv. Math.}, 224(2):461--498, 2010.

\bibitem{Yamasaki98}
A.~Yamasaki.
\newblock Inductive limit of general linear groups.
\newblock {\em J. Math. Kyoto Univ.}, 38(4):769--779, 1998.

\bibitem{Zhang18} 
Y.~Zhang. 
\newblock On a lifting problem of Blackadar.
\newblock {\em Ann. Funct. Anal.}, 9(4):485--499, 2018.

\bibitem{Zhang23}
Y.~Zhang. 
\newblock Symmetries of simple A$\mathbb T$-algebras.
\newblock {\em J. Noncommut. Geom.}, 17:439--468, 2023.

\end{thebibliography}
\end{document}